\documentclass[11pt, letterpaper,final]{amsart}
\usepackage{amsfonts,amsmath, amsthm, amssymb, latexsym, epsfig}
\usepackage[english]{babel}
\usepackage[utf8]{inputenc}
\usepackage[all]{xy}
\usepackage{xspace}
\usepackage{comment}
\usepackage{setspace}
\usepackage{enumerate}
\usepackage{stmaryrd}
\usepackage{xcolor}
\usepackage[mathscr]{eucal}
\usepackage[notcite,notref]{showkeys}

	\advance \topmargin by -\headheight
	\advance \topmargin by -\headsep

	\evensidemargin \oddsidemargin
	\newcommand{\Ext}{\ensuremath{\operatorname{Ext}}}
	\newcommand{\multialg}[1]{\mathcal{M}(#1)\xspace}
	\newcommand{\corona}[1]{\mathcal{Q}(#1)\xspace}	
	\newcommand{\Prim}{\ensuremath{\operatorname{Prim}}}
	
	\newcommand{\nuc}{\mathrm{nuc}}
	
	\renewcommand{\span}{\mathrm{span}}
	\newcommand{\Ad}{\ensuremath{\operatorname{Ad}}\xspace}
	
	\newcommand{\rnuc}{\mathrm{r\textrm{-}nuc}}
	\newcommand{\spec}{\mathrm{spec}}
	\newcommand{\Hom}{\mathrm{Hom}}
	\newcommand{\Mod}{\mathfrak{Mod}}
	\newcommand{\op}{\mathrm{op}}
	
	\newcommand{\reg}{\mathrm{reg}}
	
	\newenvironment{lbmatrix}{\left[\begin{smallmatrix}}{\end{smallmatrix}\right]}
	
	\theoremstyle{plain}
	\newtheorem{thm}{Theorem}[section]
	\newtheorem{lemma}[thm]{Lemma}
	\newtheorem{theorem}[thm]{Theorem}
	\newtheorem{proposition}[thm]{Proposition}
	\newtheorem{corollary}[thm]{Corollary}

	\theoremstyle{definition}
	\newtheorem{definition}[thm]{Definition}
	
	\newtheorem{remark}[thm]{Remark}
	\newtheorem{notation}[thm]{Notation}
	\newtheorem{example}[thm]{Example}
	
	\newtheorem{warning}[thm]{Warning}
	\newtheorem{question}[thm]{Question}
	\newtheorem{observation}[thm]{Observation}
	
	\numberwithin{equation}{section}

\begin{document}
	\title{Absorbing representations with respect to closed operator convex cones}
	\author{James Gabe}
        \address{Department of Mathematical Sciences \\
        University of Copenhagen\\
        Universitetsparken~5 \\
        DK-2100 Copenhagen, Denmark}
        \email{jamiegabe123@hotmail.com}

	\author{Efren Ruiz}
        \address{Department of Mathematics\\University of Hawaii,
Hilo\\200 W. Kawili St.\\
Hilo, Hawaii\\
96720-4091 USA}
        \email{ruize@hawaii.edu}
        \thanks{This work was partially supported by the Danish National Research Foundation through the Centre for Symmetry and Deformation (DNRF92) and a grant from the Simons Foundation (\#279369 to Efren Ruiz)}
        
\begin{abstract}
 We initiate the study of absorbing representations of $C^\ast$-algebras with respect to closed operator convex cones. 
 We completely determine when such absorbing representations exist, which leads to the question of characterising
 when a representation is absorbing, as in the classical Weyl--von Neumann type theorem of Voiculescu. 
 In the classical case, this was solved by Elliott and Kucerovsky who proved that a representation is nuclearly absorbing if and only if it induces a purely large extension.
 By considering a related problem for extensions of $C^\ast$-algebras, which we call the purely large problem, we ask when
 a purely largeness condition similar to the one defined by Elliott and Kucerovsky, implies absorption with respect to some given closed operator convex cone.
 
 We solve this question for a special type of closed operator convex cone induced by actions of finite topological spaces on $C^\ast$-algebras. 
 As an application of this result, we give $K$-theoretic classification for certain $C^\ast$-algebras containing a purely infinite, two-sided, closed ideal for which the quotient is an AF algebra. 
 This generalises a similar result by the second author, S.~Eilers and G.~Restorff in which all extensions had to be full.
\end{abstract}

\maketitle

\section{Introduction}

The study of absorbing representations dates back, in retrospect, more than a century to Weyl \cite{Weyl-WvN} who proved that any bounded self-adjoint operator on a separable Hilbert space is a compact perturbation of a diagonal operator. 
This implies that the only points in the spectrum of a self-adjoint operator which are affected by compact perturbation, are the isolated points of finite multiplicity.
The result was improved by von Neumann \cite{vonNeumann-WvN} who showed that the compact operator in question may be chosen to be a Hilbert--Schmidt operator of arbitrarily small norm. 
An important corollary, which is often referred to as \emph{the Weyl--von Neumann theorem}, is that any two bounded self-adjoint operators $S$ and $T$ on a separable Hilbert space are \emph{approximately unitarily equivalent} 
if and only if they have the same spectrum with the same multiplicity in all discrete points.
Approximate unitary equivalence means that there exists a sequence of unitaries $(U_n)$ such that $U_n^\ast S U_n - T$ is compact for each $n$ and tends to zero.

In 1970, Halmos \cite{Halmos-tenproblems} published ten open problems in Hilbert space theory. 
One of these problems was whether it was possible to generalise Weyl's theorem to normal operators. 
Shortly after, this question was answered in the positive, independently, by Sikonia \cite{Sikonia-WvN} and Berg \cite{Berg-WvN}. 
Consequently they also obtained a generalisation of the Weyl--von Neumann theorem, for normal operators instead of self-adjoint operators.

Another problem of Halmos was the following: can every operator on a separable, infinite dimensional Hilbert space $\mathcal H$ be approximated in norm by reducible operators? 
In 1976, Voiculescu \cite{Voiculescu-WvN} provided an affirmative answer to this question.
In proving this, he made a ground breaking generalisation of the other Weyl--von Neumann type theorems. 
He showed that any two unital representations $\mathfrak A \to \mathbb B(\mathcal H)$, which are faithful modulo the compacts,
are approximately unitarily equivalent, for \emph{any} separable, unital $C^\ast$-algebra $\mathfrak A$. 
This is equivalent to saying that any such representation $\Phi$ is \emph{absorbing}, i.e.~for any unital representation $\Psi$, the diagonal sum $\Phi \oplus \Psi$ is approximately unitarily equivalent to $\Phi$.

The progress of Sikonia and Berg, motivated Brown, Douglas and Fillmore \cite{BrownDouglasFillmore-unitaryeq} to classify normal operators in the Calkin algebra up to unitary equivalence.
They found that it was equivalent to classify unital \emph{extensions} of certain commutative $C^\ast$-algebras by the $C^\ast$-algebra of compact operators $\mathbb K$.
This led them to construct the semigroup $\Ext(\mathfrak A)$ which they construct first for separable, unital, commutative $C^\ast$-algebras $\mathfrak A$ and later in \cite{BrownDouglasFillmore-Ext} for any separable, unital $C^\ast$-algebra $\mathfrak A$.
It follows from Voiculescu's Weyl--von Neumann type theorem that $\Ext(\mathfrak A)$ always has a zero element.
In \cite{Arveson-extensions}, Arveson combined this with the dilation theorem of Stinespring \cite{Stinespring-positive}, 
to show that an extension of $\mathfrak A$ by the compact operators has an inverse in $\Ext(\mathfrak A)$ if and only if the extension has a completely positive splitting.
In particular, by the lifting theorem of Choi and Effros \cite{ChoiEffros-lifting}, $\Ext(\mathfrak A)$ is a group if $\mathfrak A$ is separable and nuclear.

The theory of studying extensions was revolutionised by the work of Kasparov in \cite{Kasparov-KKExt} when he constructed the commutative semigroup $\Ext(\mathfrak A, \mathfrak B)$ 
generated by extensions of $\mathfrak A$ by $\mathfrak B \otimes \mathbb K$. 
In \cite{Kasparov-Stinespring} Kasparov generalises the Stinespring theorem, using Hilbert modules instead of Hilbert spaces, and proves that there \emph{exists} a zero element in $\Ext(\mathfrak A, \mathfrak B)$ if every completely positive map
from $\mathfrak A$ to $\mathfrak B$ is \emph{nuclear}.
By the exact same methods developed by Arveson, it follows that $\Ext(\mathfrak A, \mathfrak B)$ is a group when $\mathfrak A$ is nuclear, and Kasparov also proves that this group is isomorphic to $KK^1(\mathfrak A, \mathfrak B)$.

Say that a completely positive map $\phi$ from $\mathfrak A$ to the multiplier algebra $\multialg{\mathfrak B}$ is \emph{weakly nuclear} if $b^\ast \phi(-)b \colon \mathfrak A \to \mathfrak B$ is nuclear for all $b\in \mathfrak B$.
What Kasparov actually proves when showing the existence of the zero element in $\Ext(\mathfrak A, \mathfrak B)$, is that there is a unital representation $\Phi \colon \mathfrak A \to \multialg{\mathfrak B \otimes \mathbb K}$ which is weakly nuclear, and
such that the Cuntz sum (sometimes referred to as the BDF sum)
with any weakly nuclear, unital representation $\mathfrak A \to \multialg{\mathfrak B \otimes \mathbb K}$ is approximately unitarily equivalent to $\Phi$.
Any $\Phi$ which satisfies this latter condition is called \emph{nuclearly absorbing}, since it absorbs any weakly nuclear representation.
However, where as Voiculescu's Weyl--von Neumann type theorem says that \emph{any} representation, which is faithful modulo the compacts, is absorbing, Kasparov's result only shows that a somewhat small class of representations are nuclearly absorbing.
Kasparov notes in \cite{Kasparov-Stinespring} that in general, a representation $\Phi \colon \mathfrak A \to \multialg{\mathfrak B\otimes \mathbb K}$ 
which is faithful modulo $\mathfrak B \otimes \mathbb K$, is not necessarily nuclearly absorbing, and remarks that it would be desirable to find conditions which determine when a representation is nuclearly absorbing.

Kirchberg showed in \cite{Kirchberg-simple}, exactly when the obvious generalisation of Voiculescu's theorem holds: 
Let $\mathfrak B$ be a $\sigma$-unital, stable $C^\ast$-algebra which is not $\mathbb K$. Then $\mathfrak B$ is purely infinite and simple if and only if
for any separable, unital $C^\ast$-algebra $\mathfrak A$, any unital representation $\Phi \colon \mathfrak A \to \multialg{\mathfrak B}$ which is faithful modulo $\mathfrak B$, is nuclearly absorbing.
This was used by Kirchberg to prove, amongst other things, that any separable, exact $C^\ast$-algebra embeds in the Cuntz algebra $\mathcal O_2$, and also played an important part in his classification of the Kirchberg algebras.

It was not until 2001 that Elliott and Kucerovsky \cite{ElliottKucerovsky-extensions} found general conditions which determine nuclear absorption. 
We say that an extension of separable $C^\ast$-algebras $0 \to \mathfrak B \to \mathfrak E \to \mathfrak A \to 0$ is \emph{purely large},
if for any $x\in \mathfrak E \setminus \mathfrak B$, the $C^\ast$-algebra $\overline{x^\ast \mathfrak B x}$ contains a stable $C^\ast$-subalgebra which is full in $\mathfrak B$. 
They prove that if $\mathfrak A$ and $\mathfrak B$ are separable, $\mathfrak A$ is unital and $\mathfrak B$ is stable, then a unital representation $\Phi \colon \mathfrak A \to \multialg{\mathfrak B}$ is nuclearly absorbing if and only if
the extension generated by $\Phi$ is purely large. 
To simplify the purely large condition, the concept of a $C^\ast$-algebra having the \emph{corona factorisation property} was introduced (see \cite{Kucerovsky-largeFredholm} and \cite{KucerovskyNg-corona}).
The following was proved: let $\mathfrak B$ be a separable, stable $C^\ast$-algebra. Then $\mathfrak B$ has the corona factorisation property if and only if
for any separable, unital $C^\ast$-algebra $\mathfrak A$, any unital, \emph{full} representation $\Phi \colon \mathfrak A \to \multialg{\mathfrak B}$ is nuclearly absorbing.
It turns out that many $C^\ast$-algebras of interest have the corona factorisation property, including every $\sigma$-unital $C^\ast$-algebra with finite nuclear dimension \cite{Robert-nucdim} and all separable $C^\ast$-algebras which absorb the Jiang--Su algebra
\cite{KirchbergRordam-cschar}. 

These Weyl--von Neumann type theorems have been used in the classification program to classify simple $C^\ast$-algebras as well as non-simple $C^\ast$-algebras.
Rørdam \cite{Rordam-classsixterm} used the Weyl--von Neumann theorem of Kirchberg to classify extensions of UCT Kirchberg algebras. 
Embarking on the same idea, the second author, Eilers and Restorff \cite{EilersRestorffRuiz-classext} classified a large class of extensions of $C^\ast$-algebras, by applying the Weyl--von Neumann theorem of Elliott and Kucerovsky.
A necessary condition for this method to work is that the extensions have to be \emph{full}. 
However, fullness of an extension is a huge limitation on the primitive ideal space of the extension algebra, and thus only very specific non-simple $C^\ast$-algebras can be classified in this way.
The main motivation for this paper, was to generalise the Weyl--von Neumann theorem of Elliott and Kucerovsky in order to apply this to classify extensions which are \emph{not} necessarily full.
We found that the right way of doing this, was to consider a much more general problem using \emph{closed operator convex cones}.

Closed operator convex cones were introduced by Kirchberg and have been important in the study of non-simple $C^\ast$-algebras, see e.g.~\cite{KirchbergRordam-zero}, \cite{KirchbergRordam-absorbingOinfty} and \cite{Kirchberg-Abel}.
A closed operator convex cone $\mathscr C$ is a collection of completely positive maps $\mathfrak A \to \mathfrak B$ satisfying certain conditions (see Definition \ref{d:opconcon}). 
Usually $\mathscr C$ will be induced from an action of a topological space $\mathsf X$ on $\mathfrak A$ and $\mathfrak B$, and will often consist only of maps which are nuclear in some suitable generalised sense.
An example, which should be thought of as ``the classical case'' is when $\mathscr C$ consists exactly of all nuclear maps from $\mathfrak A$ to $\mathfrak B$. 
In this case one has chosen a trivial action of $\mathsf X$ to construct the closed operator convex cone.
We say that a completely positive map $\phi \colon \mathfrak A \to \multialg{\mathfrak B}$ is \emph{weakly in $\mathscr C$} if $b^\ast \phi(-) b$ is in $\mathscr C$ for all $b\in \mathfrak B$.
In the classical case, this corresponds to a map being weakly nuclear. 
If $\mathfrak B$ is stable, it makes sense to say that a representation $\Phi \colon \mathfrak A \to \multialg{\mathfrak B}$ is \emph{$\mathscr C$-absorbing}, if it absorbs any representation which is weakly in $\mathscr C$.
Again, the classical case corresponds to $\Phi$ being nuclearly absorbing. Two natural questions arise: do there exist any $\mathscr C$-absorbing representations which are themselves weakly in $\mathscr C$? 
And is there a way of determining which representations are $\mathscr C$-absorbing, i.e.~can we obtain Weyl--von Neumann type theorems for $\mathscr C$-absorbing representations?

The first of these questions is completely answered in Theorem \ref{t:absrep} and relies only on whether or not $\mathscr C$ is countably generated as a closed operator convex cone.
If $\mathfrak A$ and $\mathfrak B$ are both separable, then any closed operator convex cone is countably generated.
Thus it remains to find Weyl--von Neumann type theorems which determine $\mathscr C$-absorption.
By considering $\mathscr C$-purely largeness, a generalised version of purely largeness which takes the closed operator convex cone into consideration, we obtain what we refer to as \emph{the purely large problem}.
The problem basically is: 
can we determine classes of closed operator convex cones $\mathscr C$, for which a representation $\Phi \colon \mathfrak A \to \multialg{\mathfrak B}$ is $\mathscr C$-absorbing if and only if the induced extension is $\mathscr C$-purely large.
We solve this problem (Theorem \ref{t:purelylargefinite}) for closed operator convex cones induced by actions of \emph{finite} topological spaces $\mathsf X$.
In these cases, we also obtain Weyl--von Neumann type theorems resembling those of Kucerovsky (Theorem \ref{t:cfpequiv}) and of Kirchberg (Theorem \ref{t:WvN}).

Finally, we apply our Weyl--von Neumann type results to classify certain non-simple $C^\ast$-algebras. 
What is basically proven is, that if $\mathfrak A$ is a $C^\ast$-algebra of real rank zero containing a two-sided, closed ideal $\mathfrak I$, such that $\mathfrak A/\mathfrak I$ is an AF algebra, 
and $\mathfrak I$ falls into a class of separable, nuclear, purely infinite $C^\ast$-algebras with finite primitive ideal spaces which are strongly classified by a $K$-theoretic invariant, then $\mathfrak A$ can be classified by a $K$-theoretic invariant.
Note that we are \emph{not} requiring that the extension $0 \to \mathfrak I \to \mathfrak A \to \mathfrak A/\mathfrak I \to 0$ is full, thus allowing $\mathfrak A$ to have any (finite) primitive ideal space. 
This generalises a similar result by the second author, Eilers and Restorff \cite{EilersRestorffRuiz-classfininf} where this extension had to be full.
We use this result to classify all graph $C^\ast$-algebras $\mathfrak A$ which contain a purely infinite two-sided, closed ideal $\mathfrak I$ with $\Prim \mathfrak I$ finite, such that $\mathfrak A/\mathfrak I$ is an AF algebra.
In \cite{EilersRestorffRuiz-four}, the second author, Eilers and Restorff consider the classification of graph $C^\ast$-algebras for which the primitive ideal spaces have at most four points.
Our method shows, that we have classification in the seperated cases, where the two-sided, closed ideal is purely infinite and the quotient is AF.

The paper is divided up as follows. In Section \ref{s:cocc} we prove many basic properties about closed operator convex cones, and in particular, we prove dilation theorems a la Kasparov and Stinespring. 
In Section \ref{s:rep} we prove all basic theorems regarding absorbing representations with respect to closed operator convex cones, including an existence theorem of such.
In Section \ref{s:plp} we consider the purely large problem with respect to closed operator convex cones.
In Section \ref{s:approxsection} we introduce and prove many basic properties about actions of topological spaces on $C^\ast$-algebras, including a generalised notion of fullness for completely positive $\mathsf X$-equivariant maps. 
We also prove that residually $\mathsf X$-nuclear maps satisfy a certain approximation property similar to the classical approximation property of nuclear maps.
In Section \ref{s:finiteplp} we prove that for a large class of $\mathsf X$-$C^\ast$-algebras, when $\mathsf X$ is finite, we obtain a Weyl--von Neumann type theorem a la Elliott and Kucerovsky, which determines when representations are
weakly residually $\mathsf X$-nuclearly absorbing. We also obtain a version using the corona factorisation property, where we may replace our generalised purely large condition, 
with the much simpler condition that the representation is full in the $\mathsf X$-equivariant sense.
Section \ref{s:KirchbergWvN} contains a generalised version of Kirchberg's Weyl--von Neumann type theorem, for $C^\ast$-algebras with a finite primitive ideal space.
Using this, we obtain a nice classification theorem using the ideal-related $KK^1$-classes for extensions by such $C^\ast$-algebras.
Finally, in Section \ref{s:applications} we use the above result to obtain $K$-theoretic classification of certain $C^\ast$-algebras, where we can separate the finite and the infinite part.

\subsection{Notation} As for general notation, we let $\multialg{\mathfrak B}$ denote the multiplier algebra of the $C^\ast$-algebra $\mathfrak B$, $\corona{\mathfrak B} := \multialg{\mathfrak B}/\mathfrak B$ the corona algebra, and
$\pi$ denote the quotient map $\multialg{\mathfrak B} \to \corona{\mathfrak B}$. Any other quotient map $\mathfrak E \twoheadrightarrow \mathfrak A$ of $C^\ast$-algebras will usually be denoted by $p$.

By $\mathbb K$ we denote the $C^\ast$-algebra of compact operators on a separable, infinite dimensional Hilbert space.

As is common in the literature, we write \emph{c.p.~map} instead of completely positive map. We let $CP(\mathfrak A, \mathfrak B)$ denote the cone of c.p.~maps from $\mathfrak A$ to $\mathfrak B$.

For a $C^\ast$-algebra $\mathfrak A$, we let $\mathfrak A^\dagger$ denote the \emph{forced unitisation} of $\mathfrak A$, i.e.~we add a unit to $\mathfrak A$ regardless if $\mathfrak A$ already had a unit.
If $\phi \colon \mathfrak A \to \mathfrak B$ is a contractive c.p.~map and $\mathfrak B$ is unital, then we let $\phi^\dagger \colon \mathfrak A^\dagger \to \mathfrak B$ denote the map given by
\[
 \phi^\dagger(a + \lambda 1) = \phi(a) + \lambda 1_\mathfrak{B}, \qquad \text{for } a \in \mathfrak A, \text{ and }\lambda \in \mathbb C.
\]
It is well-known that $\phi^\dagger$ is a unital c.p.~map, and that if $\phi$ is a $\ast$-homomorphism, then $\phi^\dagger$ is a unital $\ast$-homomorphism.

If $\mathfrak e : 0 \to \mathfrak B \to \mathfrak E \xrightarrow{p} \mathfrak A \to 0$ is an extension of $C^\ast$-algebras, then the \emph{unitised extension} $\mathfrak e^\dagger$ is the short exact sequence
$0 \to \mathfrak B \to \mathfrak E^\dagger \xrightarrow{p^\dagger} \mathfrak A^\dagger \to 0$.

We write $a \approx_\epsilon b$ if $\| a- b \| < \epsilon$. At times we will write $a \approx b$ to just mean that $a$ and $b$ can be approximated arbitrarily well. This will severely simplify notation in certain proofs.
An example of this could be, that if $\mathfrak A$ is a simple, purely infinite $C^\ast$-algebra, and $a,b\in \mathfrak A$ are non-zero, positive elements then there are $c\in \mathfrak A$ such that $a \approx c^\ast b c$.
We write $a\in_\epsilon S$ to mean that there is an $x\in S$ such that $a \approx_\epsilon x$.

We will say that $s_1,s_2\in \mathfrak C$ are \emph{$\mathcal O_2$-isometries} if $s_1$ and $s_2$ are isometries such that $s_1s_1^\ast + s_2s_2^\ast = 1_\mathfrak{C}$.


\section{Closed operator convex cones and dilation theorems}\label{s:cocc}

In this section we prove some very general things about closed operator convex cones. 
Given a closed operator convex cone in $CP(\mathfrak A, \mathfrak B)$ we consider two pictures of representations with respect to closed operator convex cones; a Hilbert module picture where we study maps
$\mathfrak A \to \mathbb B(E)$ for some (countably generated) Hilbert $\mathfrak B$-module $E$, and a multiplier algebra picture where we (mainly) study maps $\mathfrak A \to \multialg{\mathfrak B \otimes \mathbb K}$. 
We show that we have nice Kasparov--Stinespring dilation type theorems in both cases.

\begin{definition}\label{d:opconcon}
Let $\mathfrak A$ and $\mathfrak B$ be $C^\ast$-algebras and let $CP(\mathfrak A,\mathfrak B)$ denote the cone of all completely positive (c.p.) maps from $\mathfrak A$ to $\mathfrak B$. 
A subset $\mathscr C$ of $CP(\mathfrak A,\mathfrak B)$ is called a \emph{(matricially) operator convex cone} if it satisfies the following:
\begin{itemize}
\item[$(1)$] $\mathscr C$ is a cone,
\item[$(2)$] If $\phi \in \mathscr C$ and $b$ in $\mathfrak B$ then $b^\ast \phi(-) b \in \mathscr C$,
\item[$(3)$] If $\phi \in \mathscr C$, $a_1,\dots,a_n \in \mathfrak A$, and $b_1,\dots,b_n \in \mathfrak B$ then the map
\begin{equation}\label{eq:basicmap}
\sum_{i,j=1}^n b_i^\ast \phi(a_i^\ast (-) a_j)b_j
\end{equation}
is in $\mathscr C$.
\end{itemize}
We equip $\mathscr C$ with the point-norm topology, and say that it is a \emph{closed} operator convex cone, if it is closed as a subspace of $CP(\mathfrak A,\mathfrak B)$.

Given a subset $\mathscr S \subset CP(\mathfrak A, \mathfrak B)$ we let $K(\mathscr S)$ denote the smallest closed operator convex cone containing $\mathscr S$.

We will say that a closed operator convex cone $\mathscr C$ is \emph{countably generated} if there is a countable set $\mathscr S$, such that $\mathscr C= K(\mathscr S)$.
\end{definition}

We will mainly be considering closed operator convex cones.

\begin{example}\label{e:nuclearmaps}
 A c.p.~map is called \emph{factorable} if it factors through a matrix algebra by c.p.~maps. The set $CP_\textrm{fact}(\mathfrak A, \mathfrak B)\subset CP(\mathfrak A, \mathfrak B)$ of all factorable maps is an operator convex cone.
 
 A c.p.~map is called \emph{nuclear} if it can be approximated point-norm by factorable maps, i.e.~if it is in the point-norm closure of $CP_\textrm{fact}(\mathfrak A, \mathfrak B)$. 
 The set $CP_\nuc(\mathfrak A,\mathfrak B)$ of nuclear c.p.~maps is a closed operator convex cone.
\end{example}

The following lemma is well-known by many, but the authors have not been able to find a proof in the literature. Thus a proof is also provided.
The lemma mainly states, that the definition of nuclear maps given above, agrees with the classical notion of nuclear maps.

\begin{lemma}\label{l:classicalnuc}
 Let $\phi \colon \mathfrak A \to \mathfrak B$ be a nuclear map as defined in Example \ref{e:nuclearmaps}. Then $\phi$ may be approximated in point-norm by c.p.~maps of the form $\rho \circ \psi$ where $\psi \colon \mathfrak A \to M_n$ is a contractive c.p.~map
 and $\rho \colon M_n \to \mathfrak B$ is a c.p.~map with $\| \rho \| \leq \| \phi\|$.
\end{lemma}
\begin{proof}
Let $\psi_\alpha \colon \mathfrak A \to M_{n_\alpha}$ and $\rho_\alpha \colon M_{n_\alpha} \to \mathfrak B$ be nets of c.p.~maps such that $\rho_\alpha \circ \psi_\alpha$ converges point-norm to $\phi$.  
Start by assuming that $\mathfrak A$ is unital.
For each $\alpha$, $\psi_\alpha(1)$ generates a corner in $M_{n_\alpha}$, say $M_{k_\alpha}$, such that $\rho_\alpha \circ \psi_\alpha = \rho_{\alpha}|_{M_{k_\alpha}} \circ \psi_{\alpha}|^{M_{k_\alpha}}$. 
Hence we may assume, without loss of generality, that $\psi_\alpha(1)$ generates all of $M_{n_\alpha}$ as a corner, i.e.~that $x_\alpha := \psi_\alpha(1)$ is invertible.
By letting $\psi_\alpha^0 := x_\alpha^{-1/2} \psi_\alpha(-) x_\alpha^{-1/2}$ and $\rho_\alpha^0 := \rho_\alpha(x_\alpha^{1/2}(-)x_\alpha^{1/2})$ we get $\rho_\alpha \circ \psi_\alpha = \rho_\alpha^{0} \circ \psi_\alpha^0$. 
Thus we may in fact assume that $\psi_\alpha$ is unital,
in particular $\| \psi_\alpha\| = 1$.

Since $\| \phi\| = \|\phi(1)\| \approx \| \rho_\alpha(\psi_\alpha(1))\| = \| \rho_\alpha(1)\| = \| \rho_\alpha\|$, we may perturb $\rho_\alpha$ slightly to obtain $\| \rho_\alpha\| = \| \phi\|$ and we still get an approximation of $\phi$. 
This finishes the case when $\mathfrak A$ is unital.

If $\mathfrak A$ is not unital, let $\mathfrak A^\dagger$ denote the unitisation of $\mathfrak A$. Let $(a_\lambda)$ be an approximate identity (of positive contractions) in $\mathfrak A$. 
Construct a net of c.p.~maps $\phi_\lambda \colon \mathfrak A^\dagger \to \mathfrak B$ given by $\phi_\lambda(a + \mu 1) = \phi(a_\lambda a a_\lambda + \mu a_\lambda^2)$. 
We have $\| \phi_\lambda\| = \|\phi_\lambda(1)\|= \| \phi(a_\lambda^2)\| \leq \| \phi\|$ and clearly $\phi_\lambda|_\mathfrak{A}$ converges point-norm to $\phi$.
Let $\psi_{\alpha,\lambda} \colon \mathfrak A^\dagger \to M_{n_\alpha}$ be the c.p.~map given by $\psi_{\alpha,\lambda}(a + \mu 1) = \psi_\alpha(a_\lambda a a_\lambda + \mu a_\lambda^2)$.
Clearly $\rho_\alpha \circ \psi_{\alpha,\lambda}$ converges point-norm to $\phi_\lambda$ and thus $\phi_\lambda$ is nuclear.
By the unital case, $\phi_\lambda$ has the desired approximation, and so does the restriction to $\mathfrak A$ which approximates $\phi$.
\end{proof}

The following is a classical example where the closed operator convex cone is generated by an action of a topological space.
Many more examples, using the structure of $\mathsf X$-$C^\ast$-algebras, will be provided in Remarks \ref{r:Xeqcone} and \ref{r:rnuccone}.

\begin{example}\label{e:C(X)-maps}
 Let $\mathfrak A$ and $\mathfrak B$ be continuous $C_0(\mathsf X)$-algebras over some locally compact Hausdorff space $\mathsf X$. 
 The set $CP(\mathsf X; \mathfrak A, \mathfrak B)$ of all $C_0(\mathsf X)$-linear c.p.~maps is a closed operator convex cone.
\end{example}

\begin{remark}\label{r:conjmulti}
If $\mathscr C \subset CP(\mathfrak A,\mathfrak B)$ is point-norm closed, then condition $(3)$ implies condition $(2)$ in Definition \ref{d:opconcon}, by letting $n=1$, $b=b_1$, and letting $a_1$ run through an approximate identity in $\mathfrak A$. 

Moreover, if $\mathscr C$ is a closed operator convex cone, $\phi \in \mathscr C$, $a_1,\dots,a_n \in \multialg{\mathfrak A}$,
and $b_1,\dots,b_n \in \multialg{\mathfrak B}$, then by a similar argument as above, the map in equation \eqref{eq:basicmap} is in $\mathscr C$, since the unit ball in any $C^\ast$-algebra is strictly dense in the unit ball of its multiplier algebra.
\end{remark}

\begin{remark}\label{r:generatedcone}
 The point-norm closure of an operator convex cone is clearly a closed operator convex cone. Let $\mathscr S \subset CP(\mathfrak A, \mathfrak B)$ be a subset. It can easily be seen that the subset $C(\mathscr S) \subset CP(\mathfrak A, \mathfrak B)$ consisting
 of finite sums of maps of the type in equation \eqref{eq:basicmap} with $\phi\in \mathscr S$, is an operator convex cone. In general $C(\mathscr S)$ does \emph{not} contain $\mathscr S$, however every map in $\mathscr S$ is in the point-norm closure of
 $C(\mathscr S)$. Hence $K(\mathscr S)$ is the point-norm closure of $C(\mathscr S)$. In other words, any map in $K(\mathscr S)$ may be approximated
 point-norm by sums of maps of the form $\sum_{i,j=1}^n b_i^\ast \phi(a_i^\ast (-) a_j)b_j$ with $\phi \in \mathscr S$, $a_1, \dots,a_n\in \mathfrak A$ and $b_1,\dots,b_n\in \mathfrak B$.
\end{remark}

\begin{remark}
 Suppose that $E$ is a right Hilbert $\mathfrak B$-module and that $\phi \colon \mathfrak A \to \mathbb B(E)$ is a c.p.~map. Recall, e.g.~from \cite[Chapter 5]{Lance-book-Hilbertmodules}, that we may construct a right Hilbert $\mathfrak B$ module
 $\mathfrak A \otimes_\phi E$ by taking the algebraic tensor product $\mathfrak A \odot E$ equipped with the semi-inner product induced by $\langle a \otimes x, b \otimes y \rangle = \langle x, \phi(a^\ast b) y\rangle_E$, quotient out length zero vectors,
 and taking the completion. We will abuse notation slightly and write $a\otimes x \in \mathfrak A \otimes_\phi E$ for the element induced by the elementary tensor $a \otimes x \in \mathfrak A \odot E$. There is a canonical $\ast$-homomorphism
 $\omega \colon \mathfrak A \to \mathbb B(\mathfrak A \otimes_\phi E)$ given by left multiplication on the left tensor. We will often refer to this $\ast$-homomorphism as the \emph{dilating $\ast$-homomorphism}.
 
 Condition $(3)$ in the above definition is easily seen to be equivalent to that $\langle y, \omega(-)y \rangle_{\mathfrak A \otimes_\phi B}$ is in $\mathscr C$ for all $y = \sum_{i=1}^n a_i \otimes b_i \in \mathfrak A \otimes_\phi B$.
 
 If $\mathscr C$ is a point-norm closed subset of $CP(\mathfrak A,\mathfrak B)$ then condition $(3)$ is equivalent to that $\langle y, \omega(-)y \rangle_{\mathfrak A \otimes_\phi B}$ is in $\mathscr C$ for all $y \in \mathfrak A \otimes_\phi B$.
\end{remark}

The above remarks show, that closed operator convex cones behave nicely with respect to both multiplier algebras and Hilbert $C^\ast$-modules. 
We will first study the relations to Hilbert $C^\ast$-modules and use this to derive results in a multiplier algebra picture (for stable $C^\ast$-algebras).

\subsection{The Hilbert $C^\ast$-module picture}

\begin{definition}
 Let $\mathfrak A$ and $\mathfrak B$ be $C^\ast$-algebras and let $\mathscr C\subset CP(\mathfrak A,\mathfrak B)$ be a closed operator convex cone. Let $E$ be a Hilbert $\mathfrak B$-module. A c.p.~map 
 $\phi \colon \mathfrak A \to \mathbb B( E)$ is said to be \emph{weakly in $\mathscr C$} if the map
 \begin{equation}\label{eq:innerprodcp}
  \mathfrak A \ni a \mapsto \sum_{i,j=1}^n \langle x_i, \phi(a^\ast_i a a_j) x_j \rangle_E \in \mathfrak B
 \end{equation}
  is in $\mathscr C$ for every $x_1,\dots, x_n\in  E$ and $a_1,\dots, a_n \in \mathfrak A$.
\end{definition}

\begin{remark}\label{r:weaklyinB}
By condition $(3)$ above every $\phi \in \mathscr C$ is weakly in $\mathscr C$ when considered as a map $\mathfrak A \to B \cong \mathbb K(B)$. Moreover, it is obvious that if $E$ is a complemented Hilbert $\mathfrak B$-submodule of $F$, and
$\phi \colon \mathfrak A \to \mathbb B(E)$ is weakly in $\mathscr C$, then the composition $\mathfrak A \xrightarrow{\phi} \mathbb B(E) \hookrightarrow \mathbb B(F)$ is weakly in $\mathscr C$.

Let $\mathcal H_\mathfrak{B} = \bigoplus_{n=1}^\infty \mathfrak B$, and let $\iota$ be the map $\mathfrak B \cong \mathbb K(\mathfrak B) \hookrightarrow \mathbb K(\mathcal H_\mathfrak{B})$ 
induced by embedding $\mathfrak B \hookrightarrow \mathcal H_\mathfrak{B}$ in the first coordinate. 
By the above, a c.p.~map $\phi \colon \mathfrak A \to \mathfrak B$ is in $\mathscr C$ if and only if $\iota \circ \phi$ is weakly in $\mathscr C$.
\end{remark}

\begin{proposition}\label{p:starhomweakly}
 Let $\mathscr C \subset CP(\mathfrak A,\mathfrak B)$ be a closed operator convex cone, $E$ be a Hilbert $\mathfrak B$-module, and $\Phi \colon \mathfrak A \to \mathbb B(E)$ be a $\ast$-homomorphism. Then $\Phi$ is weakly in $\mathscr C$ if and only if
 $\langle x , \Phi(-) x \rangle_E$ is in $\mathscr C$ for every $x\in E$.
 
 In particular, a c.p.~map $\phi \colon \mathfrak A \to \mathbb B(E)$ is weakly in $\mathscr C$ if and only if the dilating $\ast$-homomorphism $\omega \colon \mathfrak A \to \mathbb B(\mathfrak A \otimes_\phi E)$ is weakly in $\mathscr C$.
\end{proposition}
\begin{proof}
 The ``if'' part follows from the observation that
 \[
  \sum_{i,j=1}^n \langle x_i, \Phi(a^\ast_i a a_j) x_j \rangle_E = \langle x , \Phi(a) x \rangle_E
 \]
 for $x = \sum_{i=1}^n \Phi(a_i)x_i$. For the converse, suppose $\Phi$ is weakly in $\mathscr C$, let $x\in E$ and $(a_\lambda)$ be an approximate identity in $\mathfrak A$. 
 Then the maps $a \mapsto \langle x, \Phi(a_\lambda^\ast a a_\lambda)x \rangle_E$ are in $\mathscr C$ and converge point-norm to $\langle x , \Phi(-) x \rangle_E$, which is thus in $\mathscr C$ since $\mathscr C$ is closed.
\end{proof}

The above proposition, although of interest itself, also allows us to prove a Kasparov--Stinespring type theorem (c.f.~\cite{Kasparov-Stinespring}) for c.p.~maps weakly in $\mathscr C$. 

\begin{remark}\label{r:infrepeat}
Let $\phi\colon \mathfrak A \to \mathbb B(E)$ be a c.p.~map, and let $\phi_\infty$ denote the infinite repeat $\mathfrak A \to \mathbb B(E^\infty)$ where $E^\infty := E \oplus E \oplus \dots$. 
If $\mathscr C$ is a closed operator convex cone and $\phi$ is weakly in $\mathscr C$, then $\phi_\infty$ is weakly in $\mathscr C$. 
In fact, let $\omega \colon \mathfrak A \to \mathbb B(\mathfrak A \otimes_\phi E)$ be the dilating $\ast$-homomorphism which is weakly in $\mathscr C$ by the above proposition. 
The infinite repeat $\omega_\infty$ of $\omega$ may be identified with the dilating $\ast$-homomorphism of $\phi_\infty$ in a canonical way, so by the above proposition it suffices to show that $\omega_\infty$ is weakly in $\mathscr C$.
If $(x_n) \in (\mathfrak A \otimes_{\phi} E)^\infty$ then $\langle (x_n), \omega_\infty(-) (x_n) \rangle = \sum_{n=1}^\infty \langle x_n, \omega(-) x_n\rangle$ is a point-norm
limit of maps in $\mathscr C$, and is thus also itself in $\mathscr C$. 
\end{remark}

Recall that a Hilbert $\mathfrak B$-module $E$ is called \emph{full} if $\overline{\span}\{ \langle x,x\rangle : x \in E\} = \mathfrak B$.

In our Kasparov--Stinespring theorem below, there is both a non-unital and a unital version. 
However, there is a clear obstruction, given a closed operator convex cone $\mathscr C$, for when there can exist a unital completely positive map $\phi \colon \mathfrak A \to \mathbb B(E)$ weakly in $\mathscr C$.
In fact, we must have that $\langle x, \phi(1) x\rangle = \langle x, x \rangle$ is a value obtained by some map in $\mathscr C$ for every $x$. 
Thus if $E$ is full then $\mathscr C$ can not factor through any proper two-sided, closed ideal $\mathfrak J$ in $\mathfrak B$, i.e.~there is a $\psi \in \mathscr C$ such that $\psi(\mathfrak A) \not \subset \mathfrak J$. 
It turns out that this is the only obstruction, as can be seen in the Kasparov--Stinespring theorem below. The observation motivates the following definition.

\begin{definition}\label{d:nondegcone}
 Let $\mathscr C \subset CP(\mathfrak A,\mathfrak B)$ be a closed operator convex cone. 
 We say that $\mathscr C$ is \emph{non-degenerate} if for any proper two-sided, closed ideal $\mathfrak J$ in $\mathfrak B$ there is a map $\phi\in \mathscr C$ such that $\phi(\mathfrak A) \not \subset \mathfrak J$.
\end{definition}

It is obvious that $\mathscr C$ is non-degenerate if and only if the two-sided, closed ideal generated by $\{ \phi(a) : a\in \mathfrak A , \phi \in \mathscr C\}$ is all of $\mathfrak B$.

In the case when $\mathfrak A$ is $\sigma$-unital and $h\in \mathfrak A$ is strictly positive, then $\mathscr C$ is non-degenerate if and only if the two-sided, closed ideal generated by $\{ \phi(h) : \phi \in \mathscr C\}$ is all of $\mathfrak B$.
This follows easily from the fact that any positive element in $\mathfrak A$ can be approximated by a positive element of the form $(a h^{1/2})^\ast (a h^{1/2}) \leq \| a\|^2 h$.

\begin{theorem}[Kasparov--Stinespring Theorem for Hilbert $C^\ast$-modules]\label{t:Stinespringmodules}
 Let $\mathfrak A$ and $\mathfrak B$ be $C^\ast$-algebras with $\mathfrak A$ separable and $\mathfrak B$ $\sigma$-unital, let $\mathscr C \subset CP(\mathfrak A,\mathfrak B)$ be a closed operator convex cone, let $E$ be a countably generated, full
 Hilbert $\mathfrak B$-module, and let $\phi \colon \mathfrak A \to \mathbb B(E)$ be a contractive completely positive map weakly in $\mathscr C$. 
 Then there is a representation $\Phi \colon \mathfrak A \to \mathbb B(\mathcal H_\mathfrak{B})$ weakly in $\mathscr C$ and isometries $V,W \in \mathbb B(E,\mathcal H_\mathfrak{B})$ with $VV^\ast + WW^\ast =1$, such that $V^\ast \Phi(-) V = \phi$.
  
  If, in addition, $\mathfrak A$ is unital and $\mathscr C$ is non-degenerate, then there is a \emph{unital} representation
  $\Phi \colon \mathfrak A \to \mathbb B(\mathcal H_\mathfrak{B})$ weakly in $\mathscr C$ and an element $V\in \mathbb B(E,\mathcal H_\mathfrak{B})$ such that $V^\ast \Phi(-) V = \phi$.
  If $\phi$ is unital, then such a $V$ exists, for which there is another isometry $W$ with $VV^\ast + WW^\ast = 1$.
\end{theorem}
\begin{proof}
 The construction of $\Phi$ follows the original construction of Kasparov \cite{Kasparov-Stinespring} very closely, however with subtle changes. Since we need this explicit construction in our proof we will repeat most of Kasparov's proof.
 
 We may extend $\phi$ to a unital c.p.~map $\phi^\dagger \colon \mathfrak A^\dagger \to \mathbb B(E)$ (see e.g.~\cite[Section 2.2]{BrownOzawa-book-approx}). Construct the Hilbert $\mathfrak B$-module $F = \mathfrak A^\dagger \otimes_{\phi^\dagger} E$, and let
 $\omega \colon \mathfrak A \to \mathbb B(F)$ be the $\ast$-homomorphism which is left multiplication on the left tensor. Since $\mathfrak A$ is separable, $F$ is countably generated and it is clearly full (since $\langle 1 \otimes x,1\otimes x \rangle_F =
 \langle x, x \rangle_E$). Also, since $\mathfrak B$ is $\sigma$-unital, it follows by \cite[Theorem 1.9]{MingoPhillips-Hilbertmodules} that $F^\infty := F \oplus F \oplus \dots \cong \mathcal H_\mathfrak{B}$. Hence it suffices to prove the result for $F^\infty$
 in place of $\mathcal H_\mathfrak{B}$. Let $V,W \in \mathbb B(E,F^\infty)$ be given by
 \[
  V(x) = (1\otimes x,0,0,\dots), \quad W(x) = (0,1\otimes x,1 \otimes x , \dots)
 \]
 which have adjoints induced by
 \[
  V^\ast(a \otimes x , y_2,y_3,\dots) = \phi^\dagger(a)x, \quad W^\ast(y_1,a_2\otimes x_2,\dots) = \sum_{k=1}^\infty \phi^\dagger(a_k)x_k.
 \]
 One should of course check that these in fact induce operators and that $V$ and $W$ are isometries for which $VV^\ast + WW^\ast = 1$. Now, let $\Phi = \omega_\infty \colon \mathfrak A \to \mathbb B(F^\infty)$ be the infinite repeat.
 Then clearly $V^\ast \Phi(-) V = \phi$. To see that $\Phi$ is weakly in $\mathscr C$, it suffices by Proposition \ref{p:starhomweakly} and Remark \ref{r:infrepeat} to show that $\langle y, \omega(-)y \rangle $ is in $\mathscr C$ for every $y\in F$. 
 It suffices to check for elements of the form $y=\sum_{i=1}^n (a_i + \lambda_i1) \otimes x_i$ where $x_i \in E$, $a_i \in \mathfrak A$ and $\lambda_i \in \mathbb C$. Letting $(c_\lambda)$ be an approximate identity in $\mathfrak A$ we get that
 \begin{eqnarray*}
  \langle y , \omega(a) y \rangle_F &=& \sum_{i,j=1}^n \langle x_i , \phi^\dagger((a_i + \lambda_i 1)^\ast a (a_j + \lambda_j1)^\ast) x_j\rangle_E \\
  &=& \lim_\lambda \sum_{i,j=1}^n \langle x_i , \phi((a_i + \lambda_i c_\lambda)^\ast a (a_j + \lambda_jc_\lambda)^\ast) x_j\rangle_E
 \end{eqnarray*}
 which is in $\mathscr C$ as a function of $a$, since $\mathscr C$ is closed. Thus $\Phi$ is weakly in $\mathscr C$.
 
 Now suppose $\mathfrak A$ is unital. If $\phi$ is unital, then a proof exactly as above (with $\mathfrak A$ instead of $\mathfrak A^\dagger$) yields a unital representation $\Phi$ weakly in $\mathscr C$, 
 and $\mathcal O_2$-isometries $V,W$ such that $V^\ast \Phi(-) V = \phi$. So it remains to prove the case where $\phi$ is not necessarily unital.
 
 We will start by proving that there exists a unital representation $\Psi \colon \mathfrak A \to \mathbb B(\mathcal H_\mathfrak{B})$ weakly in $\mathscr C$, by using that $\mathscr C$ is non-degenerate.
 As above, it suffices to show that there is a unital
 representation $\Psi \colon \mathfrak A \to \mathbb B(G)$ weakly in $\mathscr C$ where $G$ is countably generated and full. In fact, then the infinite repeat will do the trick, since $G^\infty \cong \mathcal H_\mathfrak{B}$.
 
 Fix $h\in \mathfrak B$ a strictly positive element. Since $\mathscr C$ is non-degenerate, we may for each $n$ find $\phi_n\in \mathscr C$ such that $h \in_{1/n} \overline{\mathfrak B \phi_n(1) \mathfrak B}$ (a priori we find a finite subset $\mathscr S_n$ of
 $\mathscr C$ such that $h\in_{1/n} \overline{\mathfrak B \{ \phi'(1):\phi' \in \mathscr S_n\} \mathfrak B}$, but then $\phi_n = \sum_{\phi' \in \mathscr S_n} \phi'$ also works). Letting $G_n = \mathfrak A \otimes_{\phi_n} \mathfrak B$, we clearly get that
 $h \in_{1/n} \langle G_n,G_n \rangle$. Hence if $G := \bigoplus G_n$ then $h \in \langle G, G \rangle$ and thus $G$ is full and also countably generated. 
 Clearly there is an induced unital representation $\Psi \colon \mathfrak A \to \mathbb B(G)$ weakly in $\mathscr C$.
 
 Let $F = \mathfrak A \otimes_\phi E$ which is countably generated, 
 and $\omega \colon \mathfrak A \to \mathbb B(F)$ be the canonical $\ast$-homomorphism which is weakly in $\mathscr C$ by Proposition \ref{p:starhomweakly}. By Kasparov's stabilisation theorem \cite[Theorem 2]{Kasparov-Stinespring} 
 $F \oplus \mathcal H_\mathfrak{B} \cong \mathcal H_\mathfrak{B}$, so it suffices to prove the result for $F \oplus \mathcal H_\mathfrak{B}$ in place of $\mathcal H_\mathfrak{B}$. 
 Let $\Psi \colon \mathfrak A \to \mathbb B(\mathcal H_\mathfrak{B})$ be a unital representation weakly in $\mathscr C$, and let $\Phi = \omega \oplus \Psi \colon \mathfrak A \to \mathbb B(F \oplus \mathcal H_\mathfrak{B})$ 
 which is a unital $\ast$-homomorphism weakly in $\mathscr C$ since both $\omega$ and $\Psi$ are unital and weakly in $\mathscr C$. 
 If $V \in \mathbb B(E , F \oplus \mathcal H_\mathfrak{B})$ is given by $V(x) = (1 \otimes x,0)$ then $V^\ast \Phi(-) V = \phi$.
\end{proof}

\begin{corollary}\label{c:unitalrepiff}
 Let $\mathfrak A$ be separable and unital, and $\mathfrak B$ be $\sigma$-unital and let $\mathscr C \subset CP(\mathfrak A, \mathfrak B)$ be a closed operator convex cone. 
 Then there exists a unital representation $\mathfrak A \to \mathbb B(\mathcal H_\mathfrak{B})$ weakly in $\mathscr C$ if and only if $\mathscr C$ is non-degenerate.
\end{corollary}

\subsection{The multiplier algebra picture}

In order to obtain a similar Kasparov--Stinespring result for multiplier algebras, we need to construct new operator convex cones as below.

\begin{lemma}\label{l:hilbertcone}
  Let $\mathfrak A$ and $\mathfrak B$ be $C^\ast$-algebras and let $\mathscr C\subset CP(\mathfrak A,\mathfrak B)$ be a closed operator convex cone, and let $E$ be a Hilbert $\mathfrak B$-module. Then $\mathscr C_E \subset CP(\mathfrak A,\mathbb K(E))$
  given by
  \[
   \mathscr C_E := \{ \phi \in CP(\mathfrak A,\mathbb K(E)) \mid \phi \colon \mathfrak A \to \mathbb K(E) \subset \mathbb B(E) \text{ is weakly in } \mathscr C\}
  \]
  is a closed operator convex cone.
  
  Moreover, a c.p.~map $\phi \colon \mathfrak A \to \mathbb B(E)$ is weakly in $\mathscr C$ if and only if $T^\ast \phi(-) T$ is in $\mathscr C_E$ for every $T\in \mathbb K(E)$.
\end{lemma}

\begin{proof}
 $\mathscr C_E$ is obviously point-norm closed, and clearly satisfies conditions $(1)$ of Definition \ref{d:opconcon}. Since $\mathscr C_E$ is closed it suffices to check condition $(3)$.
 
 Let $\phi \in \mathscr C_E$, and let $a_1,\dots ,a_n \in \mathfrak A$ and $T_1,\dots,T_n \in \mathbb K(E)$. We should show that
\[
 a \mapsto \Phi(a) := \sum_{i,j=1}^n T_i^\ast \phi(a_i^\ast a a_j) T_j
\]
is weakly in $\mathscr C$. Let $x_1,\dots ,x_m \in E$ and $c_1,\dots,c_m \in \mathfrak A$. Then
\[
 a \mapsto \sum_{k,l=1}^m \langle x_k , \Phi( c_k^\ast a c_l) x_l \rangle = \sum_{k,l=1}^m \sum_{i,j=1}^n \langle T_ix_k , \phi((c_k a_i)^\ast a (c_l a_j)) T_j x_l \rangle
\]
is in $\mathscr C$, since $\phi$ is weakly in $\mathscr C$. Thus $\mathscr C_E$ is a closed operator convex cone.

If $\phi$ is weakly in $\mathscr{C}$ then $T^\ast \phi(-) T$ is weakly in $\mathscr C_E$ for $T$ since
\[
 \sum_{i,j=1}^n \langle x_i, T^\ast \phi(a_i^\ast a a_j) T x_j\rangle = \sum_{i,j=1}^n \langle Tx_i, \phi(a_i^\ast a a_j) (T x_j)\rangle.
\]
If conversely $T^\ast \phi(-) T$ is weakly in $\mathscr C_E$ for every $T$, and $(T_\lambda)$ is an approximate identity in $\mathbb K(E)$, then $\phi$ is weakly in $\mathscr C$ since
\[
 \sum_{i,j=1}^n \langle x_i, \phi(a_i^\ast a a_j) x_j\rangle = \lim_\lambda \sum_{i,j=1}^n \langle x_i, T_\lambda^\ast \phi(a_i^\ast a a_j)T_\lambda  x_j)\rangle.
\]
\end{proof}

In light of Remark \ref{r:weaklyinB} we observe, that if we identify $\mathfrak B$ with $\mathbb K(\mathfrak B)$, then $\mathscr C_{\mathbb K(\mathfrak B)} = \mathscr C$. Thus, by Lemma \ref{l:hilbertcone}, it makes sense to make the following definition.

\begin{definition}
 Let $\mathscr C \subset CP(\mathfrak A , \mathfrak B)$ be a closed operator convex cone. A c.p.~map $\phi \colon \mathfrak A \to \multialg{\mathfrak B}$ is said to be \emph{weakly in $\mathscr C$} if $b^\ast \phi(-)b$ is in
 $\mathscr C$ for every $b\in \mathfrak B$.
\end{definition}

To be fair, this definition is the original definition. The name \emph{weakly} is due to the fact, that $\phi \colon \mathfrak A \to \multialg{\mathfrak B}$ is weakly in $\mathscr C$ exactly when the composition 
$\mathfrak A \xrightarrow{\phi} \multialg{\mathfrak B} \hookrightarrow \mathfrak B^{\ast \ast}$ is in the point-weak closure of $\mathscr C$ in $CP(\mathfrak A, \mathfrak B^{\ast \ast})$. The proof of this is a simple Hahn--Banach separation argument.

To obtain a multiplier algebra version of the Kasparov--Stinespring Theorem, we apply the version for Hilbert modules, and obtain a dilating $\ast$-homomorphism $\Phi \colon \mathfrak A \to \multialg{\mathfrak B \otimes \mathbb K}$ which would be weakly in 
$\mathscr C^s$, 
where $\mathscr C^s \subset CP(\mathfrak A, \mathfrak B \otimes \mathbb K)$ is the closed operator convex cone corresponding to $\mathscr C_{\mathcal H_\mathfrak B}$ when identifying $\mathfrak B \otimes \mathbb K$ with $\mathbb K(\mathcal H_\mathfrak{B})$. 
However, it seems more desirable, if $\mathfrak B$ is stable, to obtain a dilating $\ast$-homomorphism $\Phi \colon \mathfrak A \to \multialg{\mathfrak B}$ which is weakly in $\mathscr C$. 

First we will show a one-to-one correspondence between closed operator convex cones of $CP(\mathfrak A, \mathfrak B)$ and $CP(\mathfrak A, \mathfrak B \otimes \mathbb K)$ which preserves countably generated cones.
In fact, we prove something more general.

\begin{proposition}\label{p:stablecone}
 Let $\mathfrak A,\mathfrak B$ and $\mathfrak D$ be $C^\ast$-algebras such that $\mathfrak D$ is simple and nuclear. Then there is a bijection 
 \begin{eqnarray*}
  \left\{ \begin{array}{c} \text{closed operator convex cones} \\ \text{in } CP(\mathfrak A, \mathfrak B) \end{array} \right\} &\leftrightarrow &
  \left\{ \begin{array}{c} \text{closed operator convex cones} \\ \text{in } CP(\mathfrak A, \mathfrak B \otimes \mathfrak D) \end{array} \right\} \\
  \mathscr C & \mapsto & \mathscr C \otimes \mathfrak D := K(\{ \phi(-) \otimes d : \phi \in \mathscr C\}) \\
  \{ (id_\mathfrak B \otimes \eta) \circ \psi  : \psi \in \mathscr K, \eta \text{ is a state on } \mathfrak D \} & \mapsfrom & \mathscr K,
 \end{eqnarray*}
   where $d\in \mathfrak D$ is any fixed non-zero, positive element, and $id_\mathfrak{B} \otimes \eta \colon \mathfrak B \otimes \mathfrak D \to \mathfrak B$ is the slice map.
   
 Moreover, if $\mathfrak D$ is separable, then $\mathscr C$ is countably generated if and only if $\mathscr C \otimes \mathfrak D$ is countably generated.
\end{proposition}

To prove Proposition \ref{p:stablecone} we will need the following amazing result of Kirchberg. The result is a part of \cite[Theorem 9.3]{Kirchberg-permanence}. We have included a proof for completion.

\begin{theorem}[Kirchberg]\label{t:Kirchbergcone}
 Let $\mathscr S \subset CP(\mathfrak A, \mathfrak B)$ be a subset and $\phi \in CP(\mathfrak A, \mathfrak B)$. Then $\phi \in K(\mathscr S)$ if and only if for every positive $c\in C^\ast(\mathbb F_\infty) \otimes_{\max{}} \mathfrak A$, 
 the element $(id \otimes \phi)(c)$ is in the closed, two-sided ideal $\mathfrak I(c)$ of $C^\ast(\mathbb F_\infty) \otimes_{\max{}} \mathfrak B$ generated by
 \[
  \{ (id \otimes \psi)((1 \otimes a^\ast)c(1 \otimes a)) : \psi \in \mathscr S, a \in \mathfrak A\}.
 \]
\end{theorem}
\begin{proof}
 ``only if'': For any positive $c\in C^\ast(\mathbb F_\infty) \otimes_{\max{}} \mathfrak A $ let $\mathscr C_c \subset CP(\mathfrak A, \mathfrak B)$ be the set of all maps $\psi$ such that $(id \otimes \psi)((1 \otimes a^\ast)c(1 \otimes a))$ is in $\mathfrak I(c)$.
 By polar decomposition of Hermitian forms, it follows that $(id \otimes \psi)((1 \otimes a^\ast)c(1 \otimes b)) \in \mathfrak I(c)$ for all $a, b\in \mathfrak A$ and $\psi \in \mathscr C_c$. 
 Hence it easily follows that $\mathscr C_c$ is a closed operator convex cone containing $\mathscr S$, and thus $K(\mathscr S) \subset \mathscr C_c$.
 
 ``if'': Suppose that $\phi \notin K(\mathscr S)$.
 A standard Hahn--Banach argument implies that there are $a_1,\dots, a_n\in \mathfrak A$, $\epsilon >0$ and $f_1,\dots, f_n$ states on $\mathfrak B$ such that for any $\psi \in K(\mathscr S)$ we have 
 $|f_i(\phi(a_i)) - f_i(\psi(a_i))| \geq \epsilon$ for some $i=1,\dots, n$.
 
 By \cite[Lemma 7.17 (i)]{KirchbergRordam-absorbingOinfty} there is a representation $\pi \colon \mathfrak B \to \mathbb B(\mathcal H)$ with a cyclic vector $\xi$, and $c_1,\dots, c_n \in \pi(\mathfrak B)'$ such that 
 $f_i(b) = \langle \pi(b) c_i \xi, \xi \rangle$. Let $p \colon C^\ast(\mathbb F_\infty) \to C^\ast(1,c_1,\dots,c_n)$ be a $\ast$-epimorphism and $d_1,\dots,d_n \in C^\ast(\mathbb F_\infty)$ be such that $p(d_i) = c_i$.
 For any $\phi' \in CP(\mathfrak A, \mathfrak B)$ we may construct a positive linear functional $\rho_{\phi'}$ on $C^\ast(\mathbb F_\infty) \otimes_{\max{}} \mathfrak A$ given on elementary tensors by
 $\rho_{\phi'}(d \otimes a) = \langle \pi(\phi'(a)) p(d) \xi , \xi\rangle$. In fact, this is just the composition
 \begin{equation}\label{eq:composition}
  C^\ast(\mathbb F_\infty) \otimes_{\max{}} \mathfrak A \xrightarrow{id \otimes \phi'} C^\ast(\mathbb F_\infty) \otimes_{\max{}} \mathfrak B \xrightarrow{p \times \pi} \mathbb B(\mathcal H) \xrightarrow{\langle (-)\xi,\xi\rangle} \mathbb C.
 \end{equation}
 Let $\mathscr C_\rho$ be the weak-$\ast$ closure of $\{ \rho_{\psi'} : \psi' \in K(\mathscr S)\}$.
 Since $K(\mathscr S)$ is a closed operator convex cone and $\xi$ is cyclic for the image of $\pi$, one easily checks (as done in the proof of \cite[Lemma 7.18]{KirchbergRordam-absorbingOinfty}) that $\mathscr C_\rho$ is a cone such that
 $\rho'(z^\ast(-)z)$ is in $\mathscr C_\rho$ for all $z \in C^\ast(\mathbb F_\infty) \otimes_{\max{}} \mathfrak A$ and all $\rho' \in \mathscr C_\rho$. 
 
 Let $\mathfrak I = \{ z \in C^\ast(\mathbb F_\infty) \otimes_{\max{}} \mathfrak A : \rho'(z^\ast z) = 0, \text{ for all } \rho' \in \mathscr C_\rho\}$. 
 By \cite[Lemma 7.17 (ii)]{KirchbergRordam-absorbingOinfty} $\mathfrak I$ is a two-sided, closed ideal in $C^\ast(\mathbb F_\infty) \otimes_{\max{}} \mathfrak A$ such that any positive linear functional vanishing on $\mathfrak I$ is in $\mathscr C_\rho$.
 Thus if $\rho_\phi$ vanished on $\mathfrak I$, then $\rho_\phi \in \mathscr C_\rho$ which implies that there would be a $\psi \in K(\mathscr S)$ such that $|\rho_\phi(a_i\otimes d_i) - \rho_\psi(a_i\otimes d_i)| < \epsilon$ for $i=1,\dots,n$. 
 However, this would imply that 
 \[
  f_i(\phi(a_i)) = \langle \pi(\phi(a_i)) p(d_i) \xi , \xi \rangle = \rho_\phi(a_i \otimes d_i) \approx_\epsilon \rho_\psi(a_i \otimes d_i) = f_i(\psi(a_i)),
 \]
 a contradiction. Thus there is a $z\in \mathfrak I$ such that $\rho_\phi(z^\ast z) >0$. Since $\mathfrak I$ is a two-sided, closed ideal containing the generators of $\mathfrak I(z^\ast z)$, it follows that $\mathfrak I(z^\ast z) \subset \mathfrak I$, 
 and thus $(id \otimes \phi)(z^\ast z) \notin \mathfrak I(z^\ast z)$.
\end{proof}

\begin{corollary}\label{c:Kirchbergcone}
 Let $\mathscr C \subset CP(\mathfrak A, \mathfrak B)$ be a closed operator convex cone and $\phi \in CP(\mathfrak A, \mathfrak B)$. Then $\phi \in \mathscr C$ if and only if for every positive $c\in C^\ast(\mathbb F_\infty) \otimes_{\max{}} \mathfrak A$, 
 the element $(id \otimes \phi)(c)$ is in the two-sided, closed ideal $C^\ast(\mathbb F_\infty) \otimes_{\max{}} \mathfrak B$ generated by
 \[
  \{ (id \otimes \psi)(c) : \psi \in \mathscr C\}.
 \]
\end{corollary}
\begin{proof}
 This follow from Theorem \ref{t:Kirchbergcone} since $(id \otimes \psi)((1 \otimes a^\ast)c(1 \otimes a)) = (id \otimes \psi(a^\ast (-) a))(c)$ and $\psi(a^\ast (-) a)$ is in $\mathscr C$, for all $\psi \in \mathscr C$ and $a\in \mathfrak A$.
\end{proof}

Using that two-sided, closed ideals are hereditary $C^\ast$-subalgebras, we immediately obtain the following corollary. 
Although the corollary will not be used in this paper, it somehow illustrates how large a closed operator convex cone necessarily is, even though these are often generated by only a single c.p.~map (see e.g.~Corollary \ref{c:Cfullmap}).

\begin{corollary}\label{c:hercone}
 Any closed operator convex cone $\mathscr C$ is hereditary, i.e.~if $\phi,\psi$ are c.p.~maps such that $\phi + \psi \in \mathscr C$, then $\phi,\psi \in \mathscr C$.
\end{corollary}

\begin{proof}[Proof of Proposition \ref{p:stablecone}]
  Since $\mathfrak D$ is nuclear we have 
  \[
    C^\ast(\mathbb F_\infty) \otimes_{\max{}} (\mathfrak B \otimes \mathfrak D) = (C^\ast(\mathbb F_\infty) \otimes_{\max{}} \mathfrak B) \otimes \mathfrak D, 
  \]
  thus we will simply write 
  $C^\ast(\mathbb F_\infty) \otimes_{\max{}} \mathfrak B \otimes \mathfrak D$. Since $\mathfrak D$ is simple and nuclear, there is a lattice isomorphism between the lattice of two-sided, closed ideals in $C^\ast(\mathbb F_\infty) \otimes_{\max{}} \mathfrak B$ and in
  $C^\ast(\mathbb F_\infty) \otimes_{\max{}} \mathfrak B \otimes \mathfrak D$ given by $\mathfrak J \mapsto \mathfrak J \otimes \mathfrak D$.
  
  Let $\mathfrak C = C^\ast(\mathbb F_\infty) \otimes_{\max{}} \mathfrak B$. For a subset $\mathscr S \subset CP(\mathfrak A, \mathfrak B)$ and $c\in C^\ast(\mathbb F_\infty) \otimes_{\max{}} \mathfrak A$, 
  let $I_\mathscr{S}(c)$ be the two-sided, closed ideal in $\mathfrak C$ generated by 
  \[
   \{ (id \otimes \psi)((1 \otimes a^\ast)c(1 \otimes a)) : \psi\in \mathscr S, a\in \mathfrak A\}. 
  \]
  By Theorem \ref{t:Kirchbergcone} a c.p.~map $\phi \in CP(\mathfrak A, \mathfrak B)$ is in $K(\mathscr S)$ if and only if $(id \otimes \phi)(c) \in I_\mathscr{S}(c)$ for all positive $c \in C^\ast(\mathbb F_\infty) \otimes_{\max{}} \mathfrak A$.
  
  Similarly, for a subset $\mathscr S_0 \subset CP(\mathfrak A, \mathfrak B \otimes \mathfrak D)$ and $c\in C^\ast(\mathbb F_\infty) \otimes_{\max{}} \mathfrak A$, let $J_{\mathscr S_0}(c)$ be the two-sided, closed ideal in $\mathfrak C$ 
  (the existence of which follows from the lattice isomorphism above), such that $J_{\mathscr{S}_0}(c) \otimes \mathfrak D$
  is the two-sided, closed ideal generated by 
  \[
   \{ (id \otimes \psi)((1\otimes a^\ast)c(1 \otimes a)) : \psi \in \mathscr S_0, a\in \mathfrak A\}. 
  \]
  Again by Theorem \ref{t:Kirchbergcone} a c.p.~map $\psi \in CP(\mathfrak A, \mathfrak B\otimes \mathfrak D)$ is in $K(\mathscr S_0)$ if and only if $(id \otimes \psi)(c) \in J_{\mathscr{S}_0}(c)\otimes \mathfrak D$ 
  for all positive $c \in C^\ast(\mathbb F_\infty) \otimes_{\max{}} \mathfrak A$.
  
  Let $F$ be the map which takes a closed operator convex cone $\mathscr K$ in $CP(\mathfrak A, \mathfrak B \otimes \mathfrak D)$ to
  \[
   K(\{ (id_\mathfrak B \otimes \eta) \circ \psi  : \psi \in \mathscr K, \eta \in S(\mathfrak D) \}),
  \]
  where $S(\mathfrak D)$ is the space of states on $\mathfrak D$. Note that this is the closed operator convex cone \emph{generated} by the set which, in the statement of the proposition, is claimed to be a closed operator convex cone.
  
  We will show that $F(\mathscr C \otimes \mathfrak D) = \mathscr C$ and $F(\mathscr K) \otimes \mathfrak D = \mathscr K$.  
  We will use the following facts which we state without proof. 
  
  Fact 1: Let $\mathfrak S \subset \mathfrak C$ be a set of positive elements and $\mathfrak J(\mathfrak S)$ be the two-sided, closed ideal in $\mathfrak C$ generated by $\mathfrak S$.
  Then for any positive non-zero $d\in \mathfrak D$ (which is full in $\mathfrak D$ since $\mathfrak D$ is simple), $\mathfrak J(\mathfrak S) \otimes \mathfrak D$ is the two-sided, closed ideal generated by the set 
  \[
   \{ x \otimes d : x \in \mathfrak S\}.
  \]
    
  Fact 2: (Cf.~\cite[Corollary IV.3.4.2]{Blackadar-book-opalg}) 
  Similarly, suppose $\mathfrak S_0 \subset \mathfrak C \otimes \mathfrak D$. Let $\mathfrak J(\mathfrak S_0)$ be the two-sided, closed ideal in $\mathfrak C$ such that $\mathfrak J(\mathfrak S_0) \otimes \mathfrak D$ 
  is the two-sided, closed ideal in $\mathfrak C \otimes \mathfrak D$ generated by $\mathfrak S_0$ (the existence of $\mathfrak J(\mathfrak S_0)$ follows from what we noted above). 
  Then $\mathfrak J(\mathfrak S_0)$ is the two-sided, closed ideal generated by the set 
  \[
   \{ (id_{\mathfrak C} \otimes \eta)(y) : y\in \mathfrak S_0, \eta \in S(\mathfrak D)\}.
  \]
   If $\mathfrak D$ is separable, it suffices to take the $\eta$ above only in a countable dense subset of $S( \mathfrak D)$.
  
  From Fact 1 it follows that if $\mathscr C \subset CP(\mathfrak A, \mathfrak B)$ is a closed operator convex cone, then $I_\mathscr{C}(c) = J_{\mathscr C \otimes \mathfrak D}(c)$ for any positive $c\in C^\ast(\mathbb F_\infty) \otimes_{\max{}} \mathfrak A$.
  To see this, let $\mathscr C_d := \{ \phi \otimes d : \phi \in \mathscr C\}$ and note that by definition $\mathscr C \otimes \mathfrak D = K(\mathscr C_d)$. Thus from Fact 1 we get that
  \[
   I_\mathscr{C}(c) \otimes \mathfrak D = ideal(\{ (id \otimes \phi \otimes d)((1\otimes a)^\ast c (1 \otimes a)): \phi \in \mathscr C, a\in \mathfrak A\}) = J_{\mathscr C_d}(c) \otimes \mathfrak D = J_{\mathscr C \otimes \mathfrak D}(c) \otimes \mathfrak D,
  \]
  where $ideal(\mathfrak S_0)$ means the two-sided, closed ideal generated by $\mathfrak S_0$.
  
  Similarly, from Fact 2 it follows that if $\mathscr K \subset CP(\mathfrak A, \mathfrak B \otimes \mathfrak D)$ is a closed operator convex cone, 
  then $I_{F(\mathscr K)}(c) = J_{\mathscr K}(c)$ for any positive $c\in C^\ast(\mathbb F_\infty) \otimes_{\max{}} \mathfrak A$.
  Thus 
  \[
   I_{F(\mathscr C \otimes \mathfrak D)}(c) = J_{\mathscr C \otimes \mathfrak D}(c) = I_{\mathscr C}(c)
  \]
  and
  \[
   J_{F(\mathscr K) \otimes \mathfrak D}(c) = I_{F(\mathscr K)}(c) = J_\mathscr{K}(c),
  \]
  for all positive $c\in C^\ast(\mathbb F_\infty) \otimes_{\max{}} \mathfrak A$. Thus it follows that $\mathscr C = F(\mathscr C \otimes \mathfrak D)$ and $\mathscr K = F(\mathscr K) \otimes \mathfrak D$.
  
  It remains to show that $\mathscr S := \{ (id_\mathfrak B \otimes \eta) \circ \psi  : \psi \in \mathscr K, \eta \in S(\mathfrak D) \}$ is a closed operator convex cone for any closed operator convex cone $\mathscr K$,
  and that $\mathscr C$ is countably generated if and only if $\mathscr C \otimes \mathfrak D$ is countably generated.
  
  It follows from what we have already proven that $\mathscr K$ above is of the form $\mathscr C \otimes \mathfrak D$. Thus for any $\phi \in \mathscr C$, $\phi \otimes d \in \mathscr K$. 
  Choosing a state $\eta$ on $\mathfrak D$ such that $\eta(d) > 0$, it follows that 
  \[
   \phi = (id \otimes \eta)\circ ((\eta(d)^{-1}\phi) \otimes d) \in \mathscr S
  \]
  and thus $\mathscr C \subset \mathscr S$. Since $K(\mathscr S) = \mathscr C$ it follows that $\mathscr S = \mathscr C$.
  
  It follows from Fact 1 that $\mathscr C \otimes \mathfrak D$ is countably generated if $\mathscr C$ is countably generated. To see this, let $\mathscr S \subset \mathscr C$ be a countable subset generating $\mathscr C$.
  Then $I_{\mathscr S}(c) = I_{\mathscr C}(c) = J_{\mathscr C \otimes \mathfrak D}(c)$ for all positive $c\in C^\ast(\mathbb F_\infty) \otimes_{\max{}} \mathfrak A$.
  Let $\mathscr S_d = \{ \phi \otimes d : \phi \in \mathscr S\}$. It follows from Fact 1 that
  \[
   J_{\mathscr C \otimes \mathfrak D}(c) \otimes \mathfrak D = I_{\mathscr S}(c) \otimes \mathfrak D = ideal(\{ (id \otimes \phi \otimes d)((1 \otimes a)^\ast c(1 \otimes a)) : \phi \in \mathscr S, a\in \mathfrak A\}) = J_{\mathscr S_d}(c) \otimes \mathfrak D
  \]
  which implies that $\mathscr C \otimes \mathfrak D$ is generated by the countable set $\mathscr S_d$.
  
  A similar argument implies, that if $\mathfrak D$ is separable then $F(\mathscr K)$ is countably generated if $\mathscr K$ is countably generated.
\end{proof}

\begin{notation}
 Let $\mathscr C \subset CP(\mathfrak A, \mathfrak B)$ be a closed operator convex cone. We let $\mathscr C^s\subset CP(\mathfrak A, \mathfrak B \otimes \mathbb K)$ denote the closed operator convex cone generated by $\mathscr C$ by Proposition \ref{p:stablecone}. 
\end{notation}

In this section we will only apply the following lemma when $p$ is an isomorphism. However, the more general statement will be applied when proving a Choi--Effros type theorem, Corollary \ref{c:choieffros}.

\begin{lemma}\label{l:closedliftable}
Let $\mathfrak A, \mathfrak B$ and $\mathfrak C$ be $C^\ast$-algebras with $\mathfrak A$ separable, and let $p \colon \mathfrak B \to \mathfrak C$ be a surjective $\ast$-homomorphism.
Let $\mathscr C \subset CP(\mathfrak A,\mathfrak B)$ be a closed operator convex cone. Then
\[
 p(\mathscr C) := \{ p \circ \phi : \phi \in \mathscr C\}
\]
is a closed operator convex cone. Also, the set
\[
 \{ p \circ \phi : \phi \in \mathscr C, \| \phi \| \leq 1\}
\]
is point-norm closed.
\end{lemma}
\begin{proof}
 Clearly $p(\mathscr C)$ is an operator convex cone. The proof that the two sets above are point-norm closed is essentially the exact same proof as \cite[Theorem 6]{Arveson-extensions} (that the set of liftable (contractive) c.p.~maps is closed).
\end{proof}

Using Theorem \ref{t:Kirchbergcone} the following is obvious.

\begin{corollary}\label{c:stablecone2}
 Let $\mathscr C \subset CP(\mathfrak A,\mathfrak B)$ be a closed operator convex cone and suppose that $\mathfrak B = \mathfrak B_0 \otimes \mathbb K$. If $\Psi' \colon \mathbb K \otimes \mathbb K \to \mathbb K$ is an isomorphism and
 $\Psi = id_{\mathfrak{B}_0} \otimes \Psi' \colon \mathfrak B \otimes \mathbb K \to \mathfrak B$ is the induced isomorphism, then $\Psi(\mathscr C^s) = \mathscr C$. 
\end{corollary}

Later in the paper, we will also need the following lemma, which looks similar to Lemma \ref{l:closedliftable} but is very different in nature.

\begin{lemma}\label{l:quotcone}
 Let $\mathscr C \subset CP(\mathfrak A, \mathfrak B)$ be a closed operator convex cone, and let $p \colon \mathfrak C \to \mathfrak A$ be a surjective $\ast$-homomorphism. Then
 \[
  \mathscr C_0 := \{ \phi \circ p : \phi \in \mathscr C\}
 \]
 is a closed operator convex cone. Moreover, if $\mathscr S \subset \mathscr C$ generates $\mathscr C$, then
 \[
  \mathscr S_0 := \{ \phi \circ p : \phi \in \mathscr S\}
 \]
 generates $\mathscr C_0$.
\end{lemma}
\begin{proof}
 Clearly $\mathscr C_0$ is an operator convex cone. Let $(\phi_\lambda \circ p)$ be a net in $\mathscr C_0$, with $\phi_\lambda \in \mathscr C$, which converges point-norm to a c.p.~map $\psi \colon \mathfrak C \to \mathfrak B$. 
 Since every $\phi_\lambda \circ p$ vanishes on $\ker p$, so does $\psi$ and thus $\psi = \phi \circ p$ for some c.p.~map $\phi \colon \mathfrak A \to \mathfrak B$. Clearly $\phi_\lambda$ converges point-norm to $\phi$ so $\phi \in \mathscr C$.
 It follows that $\psi \in \mathscr C_0$, so $\mathscr C_0$ is point-norm closed.
 
 Suppose $\mathscr S$ generates $\mathscr C$. Let $\psi \in \mathscr C_0$ and write $\psi = \phi \circ p$ with $\phi \in \mathscr C$.
 By Remark \ref{r:generatedcone}, $\phi$ may be approximated point-norm by sums of maps of the form
 \[
  a \mapsto \sum_{i,j=1}^n b_i^\ast \phi'(a_i^\ast a a_j) b_j = \sum_{i,j=1}^n b_i^\ast \phi'(p(c_i)^\ast a p(c_j)) b_j
 \]
 with $\phi'\in \mathscr S$, and where $p(c_i) = a_i$ for each $i$. It easily follows that $\phi \circ p$ may be approximated point-norm by sums of maps of the form
 \[
  c \mapsto \sum_{i,j=1}^n b_i^\ast \phi'(p(c_i)^\ast p(c) p(c_j)) b_j = \sum_{i,j}^n b_i^\ast \psi'(c_i^\ast c c_j) b_j
 \]
 where $\psi' = \phi' \circ p \in \mathscr S_0$. Hence $\mathscr S_0$ generates $\mathscr C_0$.
\end{proof}

Recall that if $\mathcal H_\mathfrak{B} = \bigoplus_{\mathbb N} \mathfrak B$ is the canonical Hilbert $\mathfrak B$-module, then there is a canonical isomorphism $\eta \colon \mathbb B(\mathcal H_\mathfrak{B}) \to \multialg{\mathfrak B \otimes \mathbb K}$
for which the restriction $\mathbb K(\mathcal H_\mathfrak{B}) \to \mathfrak B \otimes \mathbb K$ is also an isomorphism, the restricted isomorphism given by $\theta_{(b)_n,(c)_m} \mapsto bc^\ast \otimes e_{nm}$. 
Here $(b)_n \in \mathcal H_\mathfrak{B}$ denotes the element which is $b$ on the $n$'th coordinate and zero everywhere else.

\begin{corollary}\label{c:ctblygeniff}
With the notation as above, we have
\[
 \mathscr C^s = \eta(\mathscr C_{\mathcal H_{\mathfrak B}}).
\]

Moreover, $\mathscr C$ is countably generated if and only if $\mathscr C_{\mathcal H_{\mathfrak B}}$ is countably generated.
\end{corollary}
\begin{proof}
 This follows easily from Proposition \ref{p:stablecone} and Theorem \ref{t:Kirchbergcone}.
\end{proof}

Thus combining the Kasparov--Stinespring Theorem for Hilbert modules with Corollaries \ref{c:stablecone2} and \ref{c:ctblygeniff} we get the following multiplier algebra version.

\begin{corollary}[Kasparov--Stinespring Theorem for multiplier algebras]
 Let $\mathfrak A$ and $\mathfrak B$ be $C^\ast$-algebras with $\mathfrak A$ separable and $\mathfrak B$ $\sigma$-unital and stable, let $\mathscr C \subset CP(\mathfrak A,\mathfrak B)$ be a closed operator convex cone, and let 
 $\phi \colon \mathfrak A \to \multialg{\mathfrak B}$ be a contractive completely positive map weakly in $\mathscr C$. 
 Then there is a $\ast$-homomorphism $\Phi \colon \mathfrak A \to \multialg{\mathfrak{B}}$ weakly in $\mathscr C$ and isometries $V,W \in \multialg{\mathfrak B}$ with $VV^\ast + WW^\ast =1$, such that $V^\ast \Phi(-) V = \phi$.
  
  If, in addition, $\mathfrak A$ is unital $\mathscr C$ is non-degenerate, then there is a \emph{unital} $\ast$-homomorphism 
  $\Phi \colon \mathfrak A \to \multialg{\mathfrak B}$ weakly in $\mathscr C$ and an element $V\in \multialg{\mathfrak B}$ such that $V^\ast \Phi(-) V = \phi$.
  If $\phi$ is unital, then such a $V$ exists, for which there is another isometry $W$ with $VV^\ast + WW^\ast = 1$.
\end{corollary}

We also get another corollary which will be applied by the first named author in \cite{Gabe-cplifting}.

\begin{corollary}
 Let $\mathfrak A$ be a separable, unital $C^\ast$-algebra, $\mathfrak B$ be a $\sigma$-unital $C^\ast$-algebra, and let $\mathscr C\subset CP(\mathfrak A, \mathfrak B)$ be a closed operator convex cone.
 Then there exists a \emph{unital} c.p.~map $\phi \colon \mathfrak A \to \multialg{\mathfrak B}$ weakly in $\mathscr C$ if and only if $\mathscr C$ is non-degenerate.
\end{corollary}
\begin{proof}
 If $\phi$ is unital and weakly in $\mathscr C$, and $h\in \mathfrak B$ is a strictly positive element, then $h^{1/2}\phi(-)h^{1/2}$ does not factor through any proper, two-sided, closed ideal in $\mathfrak J$. Thus $\mathscr C$ is non-degenerate.
 
 Conversely, it easily follows from Proposition \ref{p:stablecone} that $\mathscr C^s$ is also non-degenerate. 
 By Corollary \ref{c:unitalrepiff} and Corollary \ref{c:ctblygeniff} there is a unital $\ast$-homomorphism $\Phi \colon \mathfrak A \to \multialg{\mathfrak B \otimes \mathbb K}$ weakly in $\mathscr C^s$.
 Let $\psi = (1 \otimes e_{11})\Phi(-)(1 \otimes e_{11}) \colon \mathfrak A \to \multialg{\mathfrak B} \otimes e_{11} \subset \multialg{\mathfrak B \otimes \mathbb K}$. 
 Regarding this as a c.p.~map $\mathfrak A \to \multialg{\mathfrak B}$ it follows by Proposition \ref{p:stablecone} that this c.p.~map is weakly in $\mathscr C$. Since $\Phi$ is unital it also follows that $\psi$ is unital.
\end{proof}


\section{Absorbing representations}\label{s:rep}

We find necessary and sufficient conditions for when there exist representations weakly in $\mathscr C$ which absorb any representation weakly in $\mathscr C$. We obtain this result in a unital and a non-unital version.

\begin{definition}\label{d:unitaryeq}
Let $\mathfrak A$ and $\mathfrak B$ be $C^\ast$-algebras with $\mathfrak A$ separable and $\mathfrak B$ $\sigma$-unital, and let $\Phi \colon \mathfrak A \to \mathbb B(E)$ and $\Psi \colon \mathfrak A \to \mathbb B(F)$ be representations with $E$ and $F$
countably generated Hilbert $\mathfrak B$-modules. We say that $\Phi$ is \emph{approximately unitarily equivalent} to $\Psi$, written $\Phi \sim_{ap} \Psi$, if there is a sequence of unitaries $(U_n)$ in $\mathbb B(E,F)$ such that
\begin{itemize}
\item[$(i)$] $\Phi(a) - U_n^\ast \Psi(a) U_n \in \mathbb K(E)$, for all $a\in \mathfrak A$ and $n \in \mathbb N$,
\item[$(ii)$] $\lim_{n\to \infty} \| \Phi(a) - U_n^\ast \Psi(a) U_n \| = 0$, for all $a\in \mathfrak A$.
\end{itemize}
We say that $\Phi$ is \emph{asymptotically unitarily equivalent}\footnote{Sometimes called \emph{unitarily homotopic}.} to $\Psi$, written $\Phi \sim_{as} \Psi$, if there is a norm continuous path $U \colon [1,\infty) \to \mathbb B(E,F)$ of unitaries such that
\begin{itemize}
\item[$(i')$] $\Phi(a) - U_t^\ast \Psi(a) U_t \in \mathbb K(E)$, for all $a\in \mathfrak A$ and $t \in [1,\infty)$,
\item[$(ii')$] $\lim_{t\to \infty} \| \Phi(a) - U_t^\ast \Psi(a) U_t \| = 0$, for all $a\in \mathfrak A$.
\end{itemize}
\end{definition}

We will also need the following related notions.

\begin{definition}\label{d:approxdom}
 Let $\mathfrak A$ and $\mathfrak B$ be $C^\ast$-algebras with $\mathfrak A$ separable and $\mathfrak B$ $\sigma$-unital, let $\Phi \colon \mathfrak A \to \mathbb B(E)$ be a representation and $\phi \colon \mathfrak A \to \mathbb B(F)$ be a c.p.~map,
 where $E$ and $F$ are countably generated Hilbert $\mathfrak B$-modules. We say that \emph{$\Phi$ approximately dominates $\phi$} if there is a bounded sequence $(v_n)$ in $\mathbb B(F,E)$ such that
 \begin{itemize}
  \item[$(i)$] $v_n^\ast \Phi(a) v_n - \phi(a) \in \mathbb K(F)$ for all $a\in \mathfrak A$ and all $n \in \mathbb N$,
  \item[$(ii)$] $\lim_{n\to \infty} \| v_n^\ast \Phi(a) v_n - \phi(a)\| = 0$ for all $a\in \mathfrak A$.
 \end{itemize}
 Furthermore, we say that \emph{$\Phi$ strongly approximately dominates $\phi$} if we may pick a sequence $(v_n)$ satisfying $(i)$ and $(ii)$ above, and also
 \begin{itemize}
  \item[$(iii)$] $\lim_{n \to \infty} \| v_n^\ast T v_n \| = 0$ for all $T \in \mathbb K(E)$.
 \end{itemize}
 
 Moreover, if we may find a norm-continuous bounded path $(v_t)_{t\in [1,\infty)}$ satisfying the obvious analogues of $(i)$, $(ii)$ (and $(iii)$) above, then we say that \emph{$\Phi$ (strongly) asymptotically dominates $\phi$}.
\end{definition}

\begin{remark}\label{r:chooseisometries}
 If $\mathfrak A$ is unital and $\Phi$ and $\phi$ are both unital in the above definition, then we may always pick the bounded sequence $(v_n)$ (or family $(v_t)$) to consist of isometries. This follows easily since unitality implies that
 $v_n^\ast v_n - 1$ is compact and small in norm for large $n$. Thus for large $n$, $w_n := v_n (v_n^\ast v_n)^{-1/2}$ is an isometry identical to $v_n$ modulo compacts, and $w_n$ makes $(i)$, $(ii)$ (and $(iii)$) hold for $\Phi$ and $\phi$ if $v_n$ does.
\end{remark}

\begin{theorem}\label{t:domabs}
 Let $\mathfrak A$ and $\mathfrak B$ be $C^\ast$-algebras with $\mathfrak A$ separable and unital, and $\mathfrak B$ $\sigma$-unital. Let $\Phi,\Psi \colon \mathfrak A \to \mathbb B(\mathcal H_\mathfrak{B})$ be unital representations
 and let $\Psi_\infty \colon \mathfrak A \to \mathbb B(\mathcal H_\mathfrak{B}^\infty)$ be the infinite repeat of $\Psi$. The following are equivalent.
 \begin{itemize}
  \item[$(i)$] $\Phi$ strongly approximately dominates $T^\ast \Psi(-) T$ for every $T \in \mathbb K(\mathcal H_\mathfrak{B})$.
  \item[$(ii)$] $\Phi$ strongly approximately dominates $\Psi$,
  \item[$(iii)$] $\Phi$ strongly asymptotically dominates $\Psi$,
  \item[$(iv)$] There is a unitary $U \in \mathbb B(\mathcal H_\mathfrak{B} \oplus \mathcal H_\mathfrak{B}^\infty, \mathcal H_\mathfrak{B})$ such that
  \[
   U^\ast \Phi(a) U - \Phi(a)\oplus \Psi_\infty(a) \in \mathbb K(\mathcal H_\mathfrak{B} \oplus \mathcal H_\mathfrak{B}^\infty) \text{ for all } a\in \mathfrak A,
  \]
  \item[$(v)$] $\Phi \oplus \Psi_\infty \sim_{ap} \Phi$,
  \item[$(vi)$] $\Phi \oplus \Psi_\infty \sim_{as} \Phi$.

 \end{itemize}
\end{theorem}
\begin{proof}
 We will prove $(v) \Rightarrow (iv) \Rightarrow (iii) \Rightarrow (ii) \Rightarrow (i) \Rightarrow (v)$ and $(iii) \Rightarrow (vi) \Rightarrow (v)$.

 $(v) \Rightarrow (iv)$: Obvious.
 
 $(iv) \Rightarrow (iii)$: Let $U$ be as in $(iv)$. Let $V_n \in \mathbb B(\mathcal H_\mathfrak{B}, \mathcal H_\mathfrak{B} \oplus \mathcal H_\mathfrak{B}^\infty)$ be the isometry given as the inclusion into the $n$'th coordinate of $\mathcal H_\mathfrak{B}^\infty$.
 Let $V_t := (n+1-t)^{1/2} V_n + (t-n)^{1/2} V_{n+1}$ for $t\in [n,n+1]$. Then $(V_t)_{t\in [1,\infty)}$ is a continuous family of isometries such that $V_t^\ast T V_t \to 0$ for all $T\in \mathbb K(\mathcal H_\mathfrak{B} \oplus \mathcal H_\mathfrak{B}^\infty)$.
 Let $W_t := U V_t$. Since $V_t^\ast(\Phi(a) \oplus \Psi_\infty(a)) V_t = \Psi(a)$ for all $a\in \mathfrak A$, it follows $W^\ast_t \Phi(a) W_t - \Psi(a)$
 is in $\mathbb K(\mathcal H_\mathfrak{B} \oplus \mathcal H_\mathfrak{B}^\infty)$ and tends to $0$ as $t\to \infty$, for all $a\in \mathfrak A$.
 
 $(iii) \Rightarrow (ii)$: Obvious.
 
 $(ii) \Rightarrow (i)$: If $W_n$ implements a strong approximate domination of $\Psi$, and $T\in \mathbb K(\mathcal H_\mathfrak{B})$ is given, then $V_n = TW_n$ implements a strong approximate domination of $T^\ast \Psi(-) T$.
 
 $(i) \Rightarrow (v)$: This follows from \cite[Theorem 2.13]{DadarlatEilers-classification}.

 $(iii) \Rightarrow (vi)$: By $(iii) \Leftrightarrow (iv)$ as proven above, and by identifying $\Psi_\infty$ and $(\Psi_\infty)_\infty$ (which are clearly unitarily equivalent), it follows that $\Phi$ asymptotically dominates $\Psi_\infty$.
 Let $V_t \in \mathbb B(\mathcal H_\mathfrak{B}^\infty,\mathcal H_\mathfrak{B})$ be a bounded continuous family of elements such that $V_t^\ast \Phi(a) V_t - \Psi_\infty(a)$ is compact and tends to zero for every $a$ in $\mathfrak A$. 
 By a standard trick of Arveson \cite[Cf.~proof of Corollary 1]{Arveson-extensions} it follows that $V_t \Psi_\infty(a) - \Phi(a) V_t$ is compact and tends to zero for all $a$ in $\mathfrak A$. 
 By, once again, identifying $\Psi_\infty$ and $(\Psi_\infty)_\infty$, it follows from \cite[Lemma 2.16]{DadarlatEilers-classification} that $\Phi \oplus \Psi_\infty \sim_{as} \Phi$.
 
 $(vi) \Rightarrow (v)$: Obvious.
\end{proof}

\begin{definition}
 Let $\mathscr C \subset CP(\mathfrak A,\mathfrak B)$ be a closed operator convex cone. A (unital) representation $\Phi \colon \mathfrak A \to \mathbb B(E)$, is called \emph{(unitally) $\mathscr C$-absorbing},
 if for any (unital) representation $\Psi \colon \mathfrak A \to \mathbb B(F)$ weakly in $\mathscr C$, we have that $\Phi \oplus \Psi \sim_{as} \Phi$.
Here we implicitly assume that $E$ and $F$ are countably generated Hilbert $\mathfrak B$-modules.
\end{definition}

\begin{remark}
 Suppose that $\Phi \colon \mathfrak A \to \mathbb B(E)$ is a $\mathscr C$-absorbing representation. Since $\Phi$ absorbs the zero representation $\mathfrak A \to \mathbb B(\mathcal H_\mathfrak{B})$ it follows that 
 $E \cong E \oplus \mathcal H_\mathfrak{B} \cong \mathcal H_\mathfrak{B}$ by Kasparov's stabilisation theorem.
 
 Also, if $\Phi \colon \mathfrak A \to \mathbb B(E)$ is a unitally $\mathscr C$-absorbing representation and $\mathscr C$ is non-degenerate,
 then $E \cong \mathcal H_\mathfrak{B}$ since it absorbs a unital representation $\mathfrak A \to \mathbb B(\mathcal H_{\mathfrak B})$ weakly in $\mathscr{C}$, which exists by Corollary \ref{c:unitalrepiff}.
 
 Thus it is essentially no loss of generality only to consider the case when $E = \mathcal H_\mathfrak{B}$, which is also the case that corresponds to the multiplier algebra picture.
\end{remark}

\begin{remark}
 Note that any two (unital) $\mathscr C$-absorbing representations $\Phi$ and $\Phi'$ which are weakly in $\mathscr C$, are asymptotically unitarily equivalent since $\Phi \sim_{as} \Phi \oplus \Phi' \sim_{as} \Phi'$.
\end{remark}

\begin{remark}
 By Theorem \ref{t:domabs}, a unital representation $\Phi$ is unitally $\mathscr C$-absorbing if and only if $\Phi \sim_{ap} \Phi \oplus \Psi$ for any unital representation $\Psi$ weakly in $\mathscr C$.
 
 The same is true in the non-unital case, which follows easily from Proposition \ref{p:unitalvsnonunital} below.
\end{remark}

The following lemma shows, that given a closed operator convex cone in $CP(\mathfrak A, \mathfrak B)$ there is an induced closed operator convex cone in $CP(\mathfrak A^\dagger,\mathfrak B)$. 
Recall that $\mathfrak A^\dagger$ denotes the forced unitisation of $\mathfrak A$, i.e.~we add a unit to $\mathfrak A$ regardless if $\mathfrak A$ is unital or not.

\begin{lemma}\label{unitalcone}
Let $\mathscr C\subset CP(\mathfrak A,\mathfrak B)$ be a closed operator convex cone and define
\[
\mathscr C^\dagger := \{ \phi \in CP(\mathfrak A^\dagger , \mathfrak B) : \phi|_\mathfrak{A} \in \mathscr C\}.
\]
Then $\mathscr C^\dagger$ is a non-degenerate, closed operator convex cone.
\end{lemma}
\begin{proof}
Clearly $\mathscr C^\dagger$ is a cone. Let $\phi\in \mathscr C^\dagger$. Clearly $b^\ast \phi(-) b$ is in $\mathscr C^\dagger$ for every $b\in \mathfrak B$. Let $a_1,\dots,a_n\in \mathfrak A^\dagger$ and $b_1,\dots,b_n\in \mathfrak B$. 
We should show that if $\psi \colon \mathfrak A^\dagger \to \mathfrak B$ is given by
\[
\psi(a) = \sum_{i,j}^n b_i^\ast \phi(a_i^\ast a a_j) b_j,
\]
then $\psi|_\mathfrak{A} \in \mathscr C$. If $\mathfrak A$ is non-unital, the $\mathfrak A^\dagger \subset \multialg{\mathfrak A}$ and the result follows from Remark \ref{r:conjmulti}. If $\mathfrak A$ is unital, then
$\mathfrak A^\dagger = \mathfrak A \oplus \mathbb C$, and $a_i = a_i' \oplus \lambda_i$. Thus when we restrict $\psi$ to $\mathfrak A$, we may replace the $a_i$ with $a_i' \in \mathfrak A$, and thus $\psi$ is in $\mathscr C^\dagger$. It follows that
$\mathscr C^\dagger$ is an operator convex cone.
If $(\phi_n)$ is a sequence in $\mathscr C^\dagger$ which converges point-norm to a c.p. map $\phi\colon \mathfrak A^\dagger \to \mathfrak B$, then $\phi_n|_\mathfrak{A} \to \phi|_\mathfrak{A}$ point-norm and thus $\phi|_\mathfrak{A} \in \mathscr C$. 
Hence $\mathscr C$ is point-norm closed. To see that $\mathscr C^\dagger$ is non-degenerate, let for any positive $b\in \mathfrak B$, $\rho_b \colon \mathbb C \to \mathfrak B$ be the c.p.~map given by $\rho_b(1) = b$. Clearly the composition
\[
 \mathfrak A^\dagger \to \mathfrak A^\dagger/\mathfrak A \cong \mathbb C \xrightarrow{\rho_b} \mathfrak B
\]
is in $\mathscr C^\dagger$ for every positive $b \in \mathfrak B$, and thus $\mathscr C^\dagger$ is non-degenerate.
\end{proof}

The following lemma tells us, that in our quest to find absorbing representations, we may often restrict to the unital case and vice versa.

\begin{proposition}\label{p:unitalvsnonunital}
Let $\mathscr C \subset CP(\mathfrak A,\mathfrak B)$ be a closed operator convex cone. Then $\Phi \colon \mathfrak A \to \mathbb B(\mathcal H_\mathfrak{B})$ is a $\mathscr C$-absorbing representation if and only if 
$\Phi^\dagger \colon \mathfrak A^\dagger \to \mathbb B(\mathcal H_\mathfrak{B})$ is a unitally $\mathscr C^\dagger$-absorbing representation. 

Moreover, if $\mathfrak A$ is unital, $\mathscr C$ is non-degenerate, and $\Phi\colon \mathfrak A \to \mathbb B(\mathcal H_\mathfrak{B})$ is a unital representation, 
then $\Phi$ is unitally $\mathscr C$-absorbing if and only if $\Phi \oplus 0 \colon \mathfrak A \to \mathbb B(\mathcal H_\mathfrak{B} \oplus \mathcal H_\mathfrak{B})$ is $\mathscr C$-absorbing.
\end{proposition}
\begin{proof}
Suppose $\Phi^\dagger$ is $\mathscr C^\dagger$-absorbing. Then clearly $\Phi$ is $\mathscr C$-absorbing since every representation weakly in $\mathscr C$ extends to a unital representation weakly in $\mathscr C^\dagger$ and the
induced asymptotic unitary equivalences in the unital case clearly hold in the non-unital case, since we just restrict to $\mathfrak A$.

Now suppose that $\Phi$ is $\mathscr C$-absorbing. Let $\Psi \colon \mathfrak A^\dagger \to \mathbb B(E)$ be a unital representation weakly in $\mathscr C^\dagger$. 
Since $\Psi|_\mathfrak{A}$ is weakly in $\mathscr C$ we may find a path of unitaries $(U_t)$ in $\mathbb B(\mathcal H_\mathfrak{B},\mathcal H_\mathfrak{B}\oplus E)$ such that conditions $(i')$ and $(ii')$ of Definition \ref{d:unitaryeq} are satisfied. Then
\[
 U_t^\ast (\Phi^\dagger \oplus \Psi)(a+ \lambda 1) U_t - \Phi^\dagger(a + \lambda 1) = U_t^\ast (\Phi \oplus \Psi|_\mathfrak{A})(a) U_t - \Phi(a)
\]
and thus $\Phi^\dagger$ clearly absorbs $\Psi$. 

Now suppose that $\mathfrak A$ and $\Phi$ are unital, and that $\mathscr C$ is non-degenerate. Suppose that $\Phi$ is unitally $\mathscr C$-absorbing and let $\Psi\colon \mathfrak A \to \mathbb B(E)$ be a representation weakly in $\mathscr C$.
Letting $P= \Psi(1_\mathfrak{A})$ we get that $E = PE \oplus P^\perp E$, and clearly $\Psi = \Psi_1 \oplus 0 \colon \mathfrak A \to \mathbb B(PE \oplus P^\perp E)$ for a unital representation $\Psi_1$ weakly in $\mathscr C$.
Letting $(U_t)$ be a continuous path of unitaries in $\mathbb B(\mathcal H_\mathfrak{B},\mathcal H_\mathfrak{B} \oplus PE)$ implementing the asymptotic unitary equivalence of $\Phi \oplus \Psi_1$ and $\Phi$, 
we clearly have that 
\[
 U_t \oplus 1 \in \mathbb B(\mathcal H_\mathfrak{B} \oplus ( P^\perp E \oplus \mathcal H_\mathfrak{B}),(\mathcal H_\mathfrak{B} \oplus PE) \oplus ( P^\perp E \oplus \mathcal H_\mathfrak{B}))
\]
implements an asymptotic unitary equivalence of $\Phi \oplus \Psi \oplus 0_{\mathcal H_\mathfrak{B}}$ and 
$\Phi \oplus 0_{P^\perp E \oplus \mathcal H_\mathfrak{B}}$. The result now follows since $\Phi \oplus \Psi \oplus 0_{\mathcal H_\mathfrak{B}}$ is unitarily equivalent to $\Phi \oplus 0_{\mathcal H_{\mathfrak B}} \oplus \Psi$ and 
$\Phi \oplus 0_{P^\perp E \oplus \mathcal H_\mathfrak{B}}$ is unitarily equivalent to $\Phi \oplus 0_{\mathcal H_\mathfrak{B}}$ by Kasparov's stabilisation theorem.

Suppose $\Phi \oplus 0$ is $\mathscr C$-absorbing and let $\Psi \colon \mathfrak A \to \mathbb B(E)$ be a unital representation weakly in $\mathscr C$. 
We first prove the case where $E \cong \mathcal H_\mathfrak{B}$.
Consider the infinite repeat $\Psi_\infty \colon \mathfrak A \to \mathbb B(F)$ where $F= E^\infty$. We may find a continuous path of unitaries in 
$\mathbb B(\mathcal H_\mathfrak{B} \oplus \mathcal H_\mathfrak{B}, \mathcal H_\mathfrak{B} \oplus F \oplus \mathcal H_\mathfrak{B})$ such that
\[
 U_t^\ast ( \Phi(a) \oplus \Psi_\infty(a) \oplus 0) U_t - \Phi(a) \oplus 0
\]
is compact and tends to zero. Let $S \in \mathbb B(\mathcal H_\mathfrak{B} \oplus F, (\mathcal H_\mathfrak{B} \oplus F) \oplus \mathcal H_\mathfrak{B})$ and $T\in \mathbb B(\mathcal H_\mathfrak{B}, \mathcal H_\mathfrak{B} \oplus \mathcal H_\mathfrak{B})$ be the
isometries which are the embedding into the first summands. By letting $V_t = S^\ast U_t T$ we have a contractive continuous family such that
\[
 V_t^\ast (\Phi(a) \oplus \Psi_\infty(a)) V_t - \Phi(a) = T^\ast(U_t^\ast ( \Phi(a) \oplus \Psi_\infty(a) \oplus 0) U_t - \Phi(a) \oplus 0) T
\]
is compact and tends to zero. It easily follows that $\Phi$ strongly asymptotically dominates $\Psi$ and by Theorem \ref{t:domabs} $\Phi$ absorbs $\Psi$.

Now suppose that $E$ is not isomorphic to $\mathcal H_\mathfrak{B}$. By Corollary \ref{c:unitalrepiff} there is a unital representation $\Psi_1 \colon \mathfrak A \to \mathbb B(\mathcal H_\mathfrak{B})$ weakly in $\mathscr C$.
Since $\mathcal H_\mathfrak{B} \oplus E \cong \mathcal H_\mathfrak{B}$ by Kasparov's stabilisation theorem, it follows from what we proved above that $\Phi$ absorbs both $\Psi_1$ and $\Psi_1 \oplus \Psi$. Hence
\[
 \Phi \oplus \Psi \sim_{as} \Phi \oplus \Psi_1 \oplus \Psi \sim_{as} \Phi.\qedhere
\]
\end{proof}

\begin{lemma}\label{l:absdominates}
 Any $\mathscr C$-absorbing representation strongly asymptotically dominates every c.p. map weakly in $\mathscr C$.
 Also, if $\mathscr C$ is non-degenerate, then any unitally $\mathscr C$-absorbing representation strongly asymptotically dominates every c.p.~map weakly in $\mathscr C$.
\end{lemma}
\begin{proof}
 We do the unital case first. Suppose $\Phi \colon \mathfrak A \to \mathbb B(E)$ is a unitally $\mathscr C$-absorbing and $\psi \colon \mathfrak A \to \mathbb B(F)$ is a c.p.~map weakly in $\mathscr C$.
 By rescaling $\psi$ if necessary, we may assume that $\psi$ is contractive.
 Let $(\Psi,V)$ be a unital Kasparov--Stinespring dilation of $\psi$ weakly in $\mathscr C$ by Theorem \ref{t:Stinespringmodules}. 
 By Theorem \ref{t:domabs} $\Phi$ strongly asymptotically dominates $\Psi$, so let $W_t$ be a bounded family implementing this strong asymptotic domination. Then $W_tV$ implements a strong asymptotic domination of $\psi$.
 
 Now consider the non-unital case. Let $\Phi$ be $\mathscr C$-absorbing and $\psi$ be weakly in $\mathscr C$. By Proposition \ref{p:unitalvsnonunital} $\Phi^\dagger$ is unitally $\mathscr C^\dagger$-absorbing.
 By scaling $\psi$ we may assume that $\psi$ is contractive, and thus we may extend $\psi$ to a unital c.p.~map $\psi^\dagger$. Clearly $\psi^\dagger$ is weakly in $\mathscr C^\dagger$, and thus $\Phi^\dagger$ strongly asymptotically dominates $\psi^\dagger$.
 By restricting to $\mathfrak A \subset \mathfrak A^\dagger$, it follows that $\Phi$ strongly asymptotically dominates $\psi$.
\end{proof}

We may now prove an important theorem. In some sense, the following theorem has been a main ingredient in many absorption results to date. This will be explained in more detail in Remark \ref{r:classicalabs} below.

\begin{theorem}\label{t:domgen}
  Let $\mathfrak A$ and $\mathfrak B$ be $C^\ast$-algebras with $\mathfrak A$ separable and unital, and $\mathfrak B$ $\sigma$-unital, and let $\Phi \colon \mathfrak A \to \mathbb B(\mathcal H_\mathfrak{B})$ be a unital representation. Let 
  $\mathscr C \subset CP(\mathfrak A,\mathfrak B)$ be a non-degenerate, closed operator convex cone and suppose that the subset $\mathscr S \subset CP(\mathfrak A,\mathbb K(\mathcal H_\mathfrak{B}))$ generates $\mathscr C_{\mathcal H_\mathfrak{B}}$. 
  If $\Phi$ strongly approximately dominates every map in $\mathscr S$, then $\Phi$ strongly approximately dominates any completely positive map $\mathfrak A \to \mathbb B(\mathcal H_\mathfrak{B})$ weakly in $\mathscr C$.
  
  In particular, $\Phi$ is unitally $\mathscr C$-absorbing if and only if $\Phi$ strongly approximately dominates every map in $\mathscr S$.
\end{theorem}

\begin{proof}
 This is proven using methods from \cite{ElliottKucerovsky-extensions}. First, we will show that $\Phi$ strongly approximately dominates any map of the form
 \[
  \psi(a) = \sum_{i,j=1}^n T_i^\ast \phi(a_i^\ast a a_j)T_j
 \]
 where $T_i\in \mathbb K(\mathcal H_\mathfrak{B})$ and $\phi \in \mathscr S$. Pick a sequence $(v_k)$ in $\mathbb B(\mathcal H_\mathfrak{B})$ such that $\lim v_k^\ast \Phi(a) v_k = \phi(a)$ and $\lim v_k^\ast T v_k = 0$ for 
 $T\in \mathbb K(\mathcal H_\mathfrak{B})$. We get that
 \[
  \psi(a) = \lim_{k\to \infty} \sum_{i,j=1}^n T_i^\ast v_k^\ast \Phi(a_i^\ast a a_j) v_k T_j = \lim_{k\to \infty} w_k^\ast \Phi(a) w_k
 \]
 where $w_k = \sum_{i=1}^n \Phi(a_i)v_k T_i$. Since $\| w_k^\ast T w_k \| \leq \sum_{i,j=1}^n \|T_i\| \|T_j \| \| v_k^\ast (\Phi(a_i^\ast) T \Phi(a_j)) v_k \| \to 0$ for any $T$, it follows that
 $\Phi$ strongly approximately dominates $\psi$. This is obvious since condition $(i)$ of Definition \ref{d:approxdom} is trivially satisfied since $\psi(a),w_k \in \mathbb K(\mathcal H_\mathfrak{B})$.
 
 We will now prove that if $\psi_1,\dots,\psi_n \colon \mathfrak A \to \mathbb K(\mathcal H_\mathfrak{B})$ such that $\Phi$ strongly approximately dominates each $\psi_i$, then $\Phi$ strongly approximately dominates $\psi = \sum_{i=1}^n \psi_i$.
 First fix $h\in \mathbb K(\mathcal H_\mathfrak{B})$ strictly positive, $0<\epsilon \leq 1$ and a finite subset $F\subset \mathfrak A$ containing $1$. It suffices to find $v\in \mathbb K(\mathcal H_\mathfrak{B})$ with $\| v^\ast v \| \leq \| \psi\| +1$ such that
 $\| v^\ast \Phi(a) v - \psi(a) \|, \| v^\ast h v\| < \epsilon$ for $a\in F$.   
Note that any $v$ satisfying the above (with no norm consideration) will satisfy $\| v^\ast v - \psi(1)\| < \epsilon$ and thus
$\| v^\ast v \| \leq \| \psi(1)\| + \epsilon \leq \| \psi \| + 1$.
 Pick $v_1 \in \mathbb K(\mathcal H_\mathfrak{B})$ such that $\| v_1^\ast \Phi(a) v_1 - \psi_1(a) \|, \| v_1^\ast h v_1\| < \epsilon/n^2$ for $a\in F$. Recursively, pick $v_k \in \mathbb K(\mathcal H_\mathfrak{B})$ such that 
 \[
 \| v_k^\ast \Phi(a) v_k - \psi_k(a) \|, \| v_j^\ast T v_k\|,\| v_k^\ast h v_k \| < \epsilon/n^2,
 \]
 for $a\in F$, $j=1,\dots k-1$, and $T \in \{ h \} \cup \Phi(F \cup F^\ast)$. Letting $v= \sum_{i=1}^n v_i$ we get that
 \[
  \| v^\ast \Phi(a) v - \psi(a)\| \leq \sum_{i=1}^n \| v_i^\ast \Phi(a) v_i - \psi_i(a)\| + \sum_{1\leq j < k \leq n} (\| v_j^\ast \Phi(a) v_k \| + \| v_j^\ast \Phi(a^\ast) v_k\|) < \epsilon
 \]
 for $a\in F$. Moreover, $\| v^\ast h v \| \leq \sum_{i,j=1}^n \|v_i^\ast h v_j \| < \epsilon$ and thus $\Phi$ strongly approximately dominates $\psi$.
 
 Finally, let $\psi \colon \mathfrak A \to \mathbb B(\mathcal H_\mathfrak{B})$ be weakly in $\mathscr C$. We may assume that $\psi$ is a contraction. Let $(\Psi,V)$ be 
a unital Kasparov--Stinespring dilation of $\psi$ weakly in $\mathscr C$ by Theorem \ref{t:Stinespringmodules}.
It suffices to show that $\Phi$ strongly approximately dominates $\Psi$, since if $W_n$ implements a strong approximate domination of $\Psi$, then $W_nV$ implements a strong approximate domination of $\psi$.

By Theorem \ref{t:domabs} $\Phi$ strongly approximate dominates $\Psi$ if and only if $\Phi$ strongly approximately dominates $T^\ast \Psi(-)T$ for any $T\in \mathbb K(\mathcal H_\mathfrak{B})$. 
 Since $T^\ast \Psi(-)T$ is in $\mathscr C_{\mathcal H_\mathfrak{B}}$, which is generated by $\mathscr S$, 
 it follows from Remark \ref{r:generatedcone} that $T^\ast\Psi(-)T$ may be approximated point-norm by a sum of maps of the form considered in the first part of the proof. 
Thus by what we have already proven, $\Phi$ strongly approximately dominates $T^\ast\Psi(-) T$ and thus also $\psi$.
 
 For the ``in particular'', the ``if'' part follows almost immediately from what has already been proven. In fact, if $\Psi$ is a unital representation weakly in $\mathscr C$, then $\Phi$ strongly approximately dominates $\Psi$ and by Theorem \ref{t:domabs}
 \[
 \Phi \sim_{as} \Phi \oplus \Psi_\infty \sim_{as} \Phi \oplus \Psi_\infty \oplus \Psi \sim_{as} \Phi \oplus \Psi.
 \]
 The ``only if'' part follows from Lemma \ref{l:absdominates}.
\end{proof}

\begin{remark}\label{r:classicalabs}
 As mentioned earlier, the above theorem is in some sense the main ingredient, or trick, in the absorption theorems involving nuclear maps. In fact, if $\mathscr C = CP_\nuc(\mathfrak A,\mathfrak B)$ is the closed operator convex cone consisting of all
 nuclear maps, and $\mathfrak B$ is stable, then $\mathscr C$ is generated by the set 
\[
\mathscr S= \{ \mathfrak A \xrightarrow{b\rho(-)} \mathfrak B : \rho \text{ is a pure state on } \mathfrak A, b\in \mathfrak B_+ \}.
\]
 This was basically proven by Kirchberg in the pre-print \cite{Kirchberg-simple}, although not stated in this way, and a similar proof can be found in the proof of \cite[Lemma 10]{ElliottKucerovsky-extensions}.
Now, if $\mathfrak A$ is the quotient of a $C^\ast$-subalgebra $\mathfrak E \subset \multialg{\mathfrak B}$, then any pure state state
$\rho$ on $\mathfrak A$ lifts to a pure state $\tilde \rho$ of $\mathfrak E$. If the extension $0 \to \mathfrak B \to \mathfrak E \xrightarrow{p} \mathfrak A \to 0$ has nice comparison properties (cf.~Lemma \ref{Xpllemma}), then we may
 use excision of pure states in $\mathfrak E$ to show the inclusion $\iota\colon \mathfrak E \hookrightarrow \multialg{\mathfrak B}$ strongly approximately dominates every map in $\mathscr S_0 =\{ \phi \circ p : \phi \in \mathscr S\}$.
 When this is the case, then $\iota$ is $\mathscr C_0$ absorbing by Lemma \ref{l:quotcone} and the above theorem, where $\mathscr C_0 = \{ \phi \circ p : \phi \in CP_\nuc(\mathfrak A, \mathfrak B)\}$.
 This is essentially what will be done when proving Theorem \ref{t:purelylargefinite}.
\end{remark}

We can now prove the main theorem on existence of absorbing representations.

\begin{theorem}[Existence]\label{t:absrep}
 Let $\mathfrak A$ and $\mathfrak B$ be $C^\ast$-algebras, with $\mathfrak A$ separable and $\mathfrak B$ $\sigma$-unital, and let $\mathscr C\subset CP(\mathfrak A,\mathfrak B)$ be a closed operator convex cone. 
 Then there exists a $\mathscr C$-absorbing representation $\Phi \colon \mathfrak A \to \mathbb B(\mathcal H_\mathfrak{B})$ which is weakly in $\mathscr C$ if and only if $\mathscr C$ is countably generated.
 
 Similarly, if $\mathfrak A$ is unital and $\mathscr C$ is non-degenerate, then there exists a unitally $\mathscr C$-absorbing representation $\Phi \colon \mathfrak A \to \mathbb B(\mathcal H_\mathfrak{B})$ which is weakly in $\mathscr C$ 
 if and only if $\mathscr C$ is countably generated.
\end{theorem}
\begin{proof}
We will prove the unital case first. Since $\mathfrak B$ is $\sigma$-unital, so is $\mathbb K(\mathcal H_\mathfrak{B})$. 
Let $T\in \mathbb K(\mathcal H_\mathfrak{B})$ be strictly positive and suppose that $\Phi \colon \mathfrak A \to \mathbb B(\mathcal H_\mathfrak{B})$ is unitally $\mathscr C$-absorbing and weakly in $\mathscr C$.
We will show that $T \Phi(-) T$ generates $\mathscr C_{\mathcal H_\mathfrak{B}}$ as a closed operator convex cone, i.e.~that $K(\{ T \Phi(-) T\}) = \mathscr C_{\mathcal H_\mathfrak{B}}$.

Let $\psi\in \mathscr C_{\mathcal H_\mathfrak{B}}$. 
By Lemma \ref{l:absdominates} $\Phi$ strongly asymptotically dominates $\psi$ and thus $\psi$ can be approximated point-norm by maps of the from $V^\ast \Phi(-) V$ with $V \in \mathbb K(\mathcal H_\mathfrak{B})$.
Since $T$ is strictly positive any such $V$ may be approximated by $TW$ for some $W \in \mathbb K(\mathcal H_\mathfrak{B})$. Hence $\psi$ can be approximated point-norm by maps of the form $W^\ast T \Phi(-) T W \in K(\{ T^\ast \Phi(-) T\})$, and since
$K(\{T\Phi(-)T\})$ is point-norm closed, it contains $\psi$. Hence $\mathscr C_{\mathcal H_\mathfrak{B}} \subset K(\{ T \Phi(-)T\})$ and since the other inclusion is trivial, these two cones are equal. 
It follows that $\mathscr C_{\mathcal H_\mathfrak{B}}$ is countably generated (in fact, it is singly generated). Thus it follows from Corollary \ref{c:ctblygeniff} that $\mathscr C$ is countably generated.

Now suppose $\mathscr C$ is countably generated and non-degenerate. Then $\mathscr C_{\mathcal H_\mathfrak{B}}$ is countably generated by Corollary \ref{c:ctblygeniff}.
Let $(\phi_n)_{n=1}^\infty$ be a sequence of contractive maps in $\mathscr C_{\mathcal H_\mathfrak{B}}$ such that each map is repeated
infinitely often in the sequence, and such that the sequence generates $\mathscr C_{\mathcal H_\mathfrak{B}}$. For each $n\in \mathbb N$ let $(\Phi_n,V_n)$ be a unital Kasparov--Stinespring dilation weakly in $\mathscr C$ as in Theorem \ref{t:Stinespringmodules},
and let $\Phi = \prod \Phi_n \colon \mathfrak A \to \mathbb B(\mathcal H_\mathfrak{B}^\infty)$. It follows from Theorem \ref{t:domgen} that $\Phi$ is unitally $\mathscr C$ absorbing if $\Phi$ strongly approximately dominates each $\phi_n$.
For a fixed $n\in \mathbb N$ let $n_1 < n_2 < \dots$ such that $\phi_{n_k} = \phi_n$ for all $k$. Let $W_k \in \mathbb B(\mathcal H_\mathfrak{B},\mathcal H_\mathfrak{B}^\infty)$ be the isometry which is the embedding into the $n_k$'th coordinate of
$\mathcal H_\mathfrak{B}^\infty$. Then $(W_{n_k} V_{n_k})^\ast \Phi(a) W_{n_k} V_{n_k} = \phi_{n_k}(a) = \phi_n(a)$ for all $a\in \mathfrak A$ and $\|TW_{n_k} V_{n_k}\| \to 0 $ for all $T \in \mathbb K(\mathcal H_{\mathfrak B}^\infty)$.
Hence $\Phi$ strongly approximately dominates each $\phi_n$ which finishes the proof in the unital case.

If $\Phi$ is $\mathscr C$-absorbing, then the exact same proof as in the unital case shows that $\mathscr C_{\mathcal H_\mathfrak{B}}$ is singly generated and thus $\mathscr C$ is countably generated.

Suppose that $\mathscr C$ is countably generated, say by a sequence $(\phi_n)$ of contractive maps in $\mathscr C$. We will show that $\mathscr C^\dagger$ is countably generated, and since this is non-degenerate by Lemma \ref{unitalcone}, it follows from
the unital case that there is a unitally $\mathscr C^\dagger$-absorbing representation $\Psi$ weakly in $\mathscr C^\dagger$. 
Since $\Psi = \Phi^\dagger$ where $\Phi = \Psi|_\mathfrak{A}$, it follows from Proposition \ref{p:unitalvsnonunital} that $\Phi$ is $\mathscr C$-absorbing and weakly in $\mathscr C$ thus finishing
the proof. So it remains to show that $\mathscr C^\dagger$ is countably generated.

By the same argument as above, the sequence $(h \phi_n(-) h)$ also generates $\mathscr C$, where $h \in \mathfrak B$ is some strictly positive element.
For each $n$, let $\phi_n^\dagger \colon \mathfrak A \to \multialg{\mathfrak B}$ be the unital extension of $\phi_n$ which is weakly in $\mathscr C^\dagger$.
Clearly $\mathscr K:=K(\{ h \phi_0^\dagger(-)h,h \phi_1^\dagger(-)h,\dots\}) \subset \mathscr C^\dagger$ where $\phi_0 = 0$. Let $\psi\in \mathscr C^\dagger$ so that we wish to show that $\psi \in \mathscr K$.
We will apply Theorem \ref{t:Kirchbergcone}. By split exactness of the short exact sequence
\[
 0 \to C^\ast(\mathbb F_\infty)\otimes_{\max{}} \mathfrak A \to C^\ast(\mathbb F_\infty)\otimes_{\max{}} \mathfrak A^\dagger \to C^\ast(\mathbb F_\infty) \to 0,
\]
any positive element $C^\ast(\mathbb F_\infty) \otimes_{\max{}} \mathfrak A^\dagger$ may be decomposed into a linear combination of positive elements from $C^\ast(\mathbb F_\infty)\otimes_{\max{}} \mathfrak A$
and $C^\ast(\mathbb F_\infty)$. Hence it suffices to check the condition of Theorem \ref{t:Kirchbergcone} for positive elements in $C^\ast(\mathbb F_\infty)\otimes_{\max{}} \mathfrak A$ and in
$C^\ast(\mathbb F_\infty)$. If $a \in C^\ast(\mathbb F_\infty)\otimes_{\max{}} \mathfrak A$ then $(id \otimes \psi)(a) = (id \otimes \psi|_\mathfrak{A})(a)$ is in the two-sided, closed ideal generated by
\[
 \{ (id \otimes \phi)((1\otimes x)^\ast a (1\otimes x)) : \phi = h\phi_n^\dagger(-)h \text{ for some }n, x \in \mathfrak A^\dagger\}
\]
since $(h\phi_n(-)h)$ generates $\mathscr C$. If $c\in C^\ast(\mathbb F_\infty) \otimes \mathbb C$ is positive, then $(id \otimes \psi)(c) = c\otimes \psi(1)$ is in the two-sided, closed ideal generated by $(id \otimes (h\phi_0^\dagger h))(c) = c \otimes h^2$,
(recall that $\phi_0 = 0$). Thus by Theorem \ref{t:Kirchbergcone}, $\psi \in \mathscr K$, which finishes the proof since $\mathscr C^\dagger = \mathscr K$ is countably generated.
\end{proof}

We will prove that a closed operator convex cone $\mathscr C \subset CP(\mathfrak A, \mathfrak B)$ is countably generated (in fact separable) if $\mathfrak A$ and $\mathfrak B$ are both separable. 
This shows, that in most cases of interest, there always exist $\mathscr C$-absorbing representations which are weakly in $\mathscr C$, and unital ones if $\mathscr C$ is non-degenerate.
Also, this generalises the main result in \cite{Thomsen-absorbing}, which corresponds to the case where $\mathscr C = CP(\mathfrak A,\mathfrak B)$.

The following result should be well-known.

\begin{lemma}\label{l:denseseq}
Let $\mathfrak A$ and $\mathfrak B$ be separable $C^\ast$-algebras. Then $CP(\mathfrak A,\mathfrak B)$ is second countable in the point-norm topology. In particular, any subspace $\mathscr C \subset CP(\mathfrak A,\mathfrak B)$ is separable.
\end{lemma}
\begin{proof}
As done in the proof of \cite[Lemma 2.3]{Thomsen-absorbing}, $CP(\mathfrak A,\mathfrak B)$ is metrisable and separable. Thus it is second countable.
\end{proof}

Since a point-norm separable closed operator convex cone is clearly generated by a countable point-norm dense subset, we may apply the above lemma together with Theorem \ref{t:absrep} to get the following.

\begin{corollary}\label{c:seprep}
 Let $\mathfrak A$ and $\mathfrak B$ be separable $C^\ast$-algebras, and let $\mathscr C\subset CP(\mathfrak A,\mathfrak B)$ be a closed operator convex cone. Then there exists a $\mathscr C$-absorbing representation weakly in $\mathscr C$. 
 
 Moreover, if $\mathfrak A$ is unital and $\mathscr C$ is non-degenerate, then there exists a unitally $\mathscr C$-absorbing representation weakly in $\mathscr C$.
\end{corollary}

\begin{remark}\label{r:hilbertvsmulti}
Every result in this section has an analogue in the multiplier algebra picture, but we need some notation to see this. If $\multialg{\mathfrak B}$ contains $\mathcal O_2$-isometries $s_1,s_2$, i.e.~$s_1$ and $s_2$ are isometries
such that $s_1s_1^\ast + s_2s_2^\ast = 1_{\multialg{\mathfrak B}}$, then we define the Cuntz sum $\oplus_{s_1,s_2}$ in $\multialg{\mathfrak B}$ to be $a\oplus_{s_1,s_2} b = s_1 a s_1^\ast + s_2 b s_2^\ast$.
In particular, this construction can be made when $\mathfrak B$ is stable.
Such a Cuntz sum is unique up to unitary equivalence. In fact, if $t_1,t_2$ are also $\mathcal O_2$-isometries, then $u=t_1s_1^\ast + t_2s_2^\ast$ is a unitary such that $u^\ast(a\oplus_{t_1,t_2} b) u = a \oplus_{s_1,s_2} b$. 
Hence we will often refer to a Cuntz sum as \emph{the} Cuntz sum and simply write $\oplus$ instead of $\oplus_{s_1,s_2}$.
It is not hard to see that, up to some canonical isomorphisms, Cuntz sums correspond to identifying $\mathcal H_\mathfrak{B} \oplus \mathcal H_\mathfrak{B} \cong \mathcal H_\mathfrak{B}$ 
by some unitary.

Similarly, when $\mathfrak B$ is stable, there is an analogue of identifying $\mathcal H_\mathfrak{B}^\infty \cong \mathcal H_\mathfrak{B}$ by some unitary. 
To obtain this analogue, pick a sequence of isometries $t_1,t_2,\dots$ in $\multialg{\mathfrak B}$ such that $\sum_{k=1}^\infty t_k t_k^\ast = 1$,
convergence in the strict topology. Note that existence of such a sequence is equivalent to stability of $\mathfrak B$.
A bounded sequence of elements $(b_n)$ in $\multialg{\mathfrak B}$, which corresponds to a diagonal element in $\mathbb B(\mathcal H_\mathfrak{B}^\infty)$, then corresponds to $\sum_{n=1}^\infty t_n b_n t_n^\ast$ in $\multialg{\mathfrak B}$. 
An argument as above shows that this identification is unique up to unitary equivalence.
In particular, the infinite repeat of an element $m\in \multialg{\mathfrak B}$ will be $\sum_{n=1}^\infty t_n m t_n^\ast$.

Thus, in the multiplier algebra picture, we obtain complete analogues of all results in this section.
\end{remark}

We will end this section by applying the above remark to show, that if $\mathfrak A$ is separable, and $\mathfrak B$ is $\sigma$-unital and stable, then any countably generated, closed operator convex cone is singly generated in an exceptionally nice way.

\begin{corollary}\label{c:Cfullmap}
 Let $\mathfrak A$ and $\mathfrak B$ be $C^\ast$-algebras, with $\mathfrak A$ separable and $\mathfrak B$ is $\sigma$-unital and stable, and let $\mathscr C\subset CP(\mathfrak A,\mathfrak B)$ be a countably generated, closed operator convex cone.
 Then there exists a c.p.~map $\phi \in \mathscr C$ such that $\mathscr C$ is the point-norm closure of the set
 \[
  \{ b^\ast \phi(-) b : b \in \mathfrak B\}.
 \]
\end{corollary}
\begin{proof}
 By Theorem \ref{t:absrep} and Remark \ref{r:hilbertvsmulti} there is a $\mathscr C$-absorbing $\ast$-homomorphism $\Phi \colon \mathfrak A \to \multialg{\mathfrak B}$ which is weakly in $\mathscr C$. 
 Let $h\in \mathfrak B$ be strictly positive and $\phi = h \Phi(-) h$.
 It follows from Lemma \ref{l:absdominates} that for any $\psi\in \mathscr C$ there is $b_0\in \mathfrak B$ such that $b_0^\ast \Phi(-)b_0$ is approximately $\psi$. Since $b_0$ may be approximated by $hb$ for some $b\in \mathfrak B$,
 it follows that $b^\ast \phi(-)b$ approximates $\psi$.
\end{proof}


\section{The purely large problem}\label{s:plp}

When trying to determine when representations are absorbing, it is often useful to consider the induced extension of $C^\ast$-algebras.
This was the strategy of Elliott and Kucerovsky in \cite{ElliottKucerovsky-extensions}, in which they prove that a unital $\ast$-homomorphism $\Phi \colon \mathfrak A \to \multialg{\mathfrak B}$ absorbs any unital weakly nuclear $\ast$-homomorphism exactly when
the extension induced by $\Phi$ is \emph{purely large}.

Motivated by this, we turn our study to extensions of $C^\ast$-algebras with respect to closed operator convex cones. In doing this, it seems more natural to work in the multiplier algebra picture of the theory, which we will thus do.
As mentioned in Remark \ref{r:hilbertvsmulti}, all results in the previous section can be defined in the multiplier algebra picture as well. 
Here $\mathscr C_{\mathcal H_\mathfrak{B}}$ corresponds to $\mathscr C^s$, as in Proposition \ref{p:stablecone}, and if $\mathfrak B$ is stable it corresponds
to $\mathscr C$ by Corollary \ref{c:stablecone2}.

Recall, that a short exact sequence of $C^\ast$-algebras $\mathfrak e: 0 \to \mathfrak B \to \mathfrak E \to \mathfrak A \xrightarrow{p} 0$ is called an \emph{extension} of $\mathfrak A$ by $\mathfrak B$. 
Given such an extension there is an induced $\ast$-homomorphism $\tau \colon \mathfrak A \to \corona{\mathfrak B}$ called the \emph{Busby map}, where $\corona{\mathfrak B} = \multialg{\mathfrak B}/\mathfrak B$ is the corona algebra.
By the universal property of multiplier algebras, there is a canonical $\ast$-homomorphism $\sigma\colon \mathfrak E \to \multialg{\mathfrak B}$. 
The induced $\ast$-homomorphism to the pull-back $(p,\sigma) \colon \mathfrak E \to \mathfrak A \oplus_{\corona{\mathfrak B}} \multialg{\mathfrak B}$ is an isomorphism.
Thus when studying extensions of $\mathfrak A$ by $\mathfrak B$, one might as well study $\ast$-homomorphisms $\tau\colon \mathfrak A \to \corona{\mathfrak B}$. 

Given two Busby maps $\tau_1,\tau_2 \colon \mathfrak A \to \corona{\mathfrak B}$ we say that $\tau_1$ and $\tau_2$ are \emph{equivalent}, written $\tau_1 \sim \tau_2$, if there is a unitary $u\in \multialg{\mathfrak B}$
such that $\tau_1 = \Ad \pi(u) \circ \tau_2$. Such a unitary implementing an equivalence of Busby maps induces an isomorphism of the corresponding extensions.

When $\multialg{\mathfrak B}$ contains a unital copy of $\mathcal O_2$ (e.g.~when $\mathfrak B$ is stable), then we may define Cuntz sums of Busby maps by $\tau_1 \oplus_{s_1,s_2} \tau_2 = \pi(s_1)\tau_1(-)\pi(s_1)^\ast + \pi(s_2) \tau_2(-)\pi(s_2)^\ast$ 
where $s_1,s_2$ are $\mathcal O_2$-isometries in $\multialg{\mathfrak B}$, i.e.~isometries satisfying $s_1s_1^\ast + s_2s_2^\ast = 1$. 
As in Remark \ref{r:hilbertvsmulti} the Cuntz sum is unique up to equivalence of Busby maps, thus we will often refer to \emph{the} Cuntz sum.

We say $\tau_1$ \emph{absorbs} $\tau_2$ (or that the corresponding extension $\mathfrak e_1$ \emph{absorbs} $\mathfrak e_2$) if $\tau_1 \oplus \tau_2 \sim \tau_1$.

\begin{definition}
 Let $0 \to \mathfrak B \to \mathfrak E \to \mathfrak A \to 0$ be an extension of $C^\ast$-algebras and $\mathscr C \subset CP(\mathfrak A,\mathfrak B)$ be a closed operator convex cone. If $\mathfrak E$ is unital (resp.~not necessarily unital) we say that the 
 extension is a \emph{unital $\mathscr C$-extension (resp.~$\mathscr C$-extension)} if it has a unital c.p.~splitting (resp.~a c.p.~splitting) which is weakly in $\mathscr C$.
 
 Moreover, if the splitting can be chosen to be a $\ast$-homomorphism, then we say that the extension is a \emph{trivial, unital $\mathscr C$-extension (resp.~trivial $\mathscr C$-extension)}.
\end{definition}

Note that by assumption any $\mathscr C$-extension (even when $\mathscr C = CP(\mathfrak A,\mathfrak B)$) is assumed to have a c.p.~splitting.

\begin{remark}
A trivial, unital $\mathscr C$-extension is always assumed to have a unital splitting $\ast$-homomorphism.
Thus if we have a unital $\mathscr C$-extension which has a splitting $\ast$-homomorphism which is not necessarily unital, then it is a trivial $\mathscr C$-extension, which just happens to unital, 
but it is \emph{not necessarily} a trivial, unital $\mathscr C$-extension. So there is a difference between a trivial, unital $\mathscr C$-extension and a unital, trivial $\mathscr C$-extension.
However, in this paper we never consider unital, trivial $\mathscr C$-extensions (unless our trivial $\mathscr C$-extension simply happens to be unital, in which case we do not mention unitality). 
So whenever the reader sees the words ``trivial'' and ``unital'' near ``$\mathscr C$-extension'', we will always assume that the splitting $\ast$-homomorphism weakly in $\mathscr C$ can be chosen to be unital.
\end{remark}

As defined in \cite{ElliottKucerovsky-extensions}, we say that an extension $\mathfrak e: 0 \to \mathfrak B \to \mathfrak E \to \mathfrak A \to 0$ of $C^\ast$-algebras
is \emph{purely large} if for every $x\in \mathfrak E \setminus \mathfrak B$, there is a $\sigma$-unital, stable $C^\ast$-subalgebra $\mathfrak D \subset \overline{x^\ast \mathfrak B x}$ which is full in $\mathfrak B$.
Note that we have added the condition that $\mathfrak D$ be $\sigma$-unital. This was redundant in \cite{ElliottKucerovsky-extensions} since they almost always assume that $\mathfrak B$ is separable.

Recall the main result of \cite{ElliottKucerovsky-extensions}. Note that we have changed this in accordance to the correction made by the first named author in \cite{Gabe-nonunitalext}.

\begin{theorem}[\cite{ElliottKucerovsky-extensions}, Theorem 6 and Corollary 16]\label{t:classicpl}
Let $\mathfrak e: 0 \to \mathfrak B \to \mathfrak E \to \mathfrak A \to 0$ be an extension of separable $C^\ast$-algebras, with $\mathfrak B$ stable, and let $\mathscr C = CP_\nuc(\mathfrak A,\mathfrak B)$. Then $\mathfrak e$ absorbs any
trivial $\mathscr C$-extension if and only if $\mathfrak e$ absorbs the zero extension and is purely large.

Moreover, if $\mathfrak e$ is unital, then $\mathfrak e$ absorbs any trivial, unital $\mathscr C$-extension if and only if $\mathfrak e$ is purely large.
\end{theorem}

As we will see as a special case of Theorem \ref{t:purelylargefinite} below, the above theorem also holds when $\mathfrak B$ is only assumed to be $\sigma$-unital.

In this section we address the question, of whether or not we can get a condition similar to extensions being purely large, for extensions with a c.p. splitting which is weakly in $\mathscr C$. 
We will use the following notation.

\begin{notation}
Let $\mathscr C \subset CP(\mathfrak A,\mathfrak B)$ be a closed operator convex cone. We let $\mathbb I(\mathfrak B)$ denote the complete lattice of two-sided, closed ideals in $\mathfrak B$, and let
$\mathfrak B_\mathscr{C} \colon \mathfrak A \to \mathbb I(\mathfrak B)$ denote the map given by
\[
\mathfrak B_\mathscr{C}(a) = \overline{ \mathfrak B \{ \phi(a) : \phi \in \mathscr C\} \mathfrak B }.
\]
\end{notation}

\begin{example}\label{e:trivialcones}
Let $\mathscr C$ denote either $CP(\mathfrak A,\mathfrak B)$ or $CP_\nuc(\mathfrak A,\mathfrak B)$. Then 
\[
\mathfrak B_\mathscr{C}(a) = \left\{ \begin{array}{ll} \mathfrak B, & \text{ if }a \neq 0 \\ 0, & \text{ if } a=0. \end{array} \right.
\]
\end{example} 

\begin{example}
 As mentioned after Definition \ref{d:nondegcone}, if $\mathfrak A$ has a strictly positive element $h$, then $\mathscr C$ is non-degenerate if and only if $\mathfrak B_{\mathscr C}(h) = \mathfrak B$.
\end{example}

A somewhat surprising result of Kirchberg and Rørdam \cite[Proposition 4.2]{KirchbergRordam-zero} says that a closed operator convex cone $\mathscr C \subset CP(\mathfrak A,\mathfrak B)$, where $\mathfrak A$ is separable and nuclear, is 
uniquely determined by the map $\mathfrak B_\mathscr{C}$. This suggests that any closed operator convex cone consists of some given relation between the ideal structures of $\mathfrak A$ and $\mathfrak B$ and some nuclearity conditions on the maps.
However, we will not use their result in this paper, since we almost only consider $C^\ast$-algebras which are not necessarily nuclear.

\begin{lemma}\label{l:idealmap}
Let $\mathfrak A$ and $\mathfrak B$ be $C^\ast$-algebras, with $\mathfrak A$ separable, $\mathfrak B$ $\sigma$-unital and stable, and let $\mathscr C \subset CP(\mathfrak A,\mathfrak B)$ be a countably generated, closed operator convex cone. 
If $\Phi\colon \mathfrak A \to \multialg{\mathfrak B}$ is a $\mathscr C$-absorbing $\ast$-homomorphism weakly in $\mathscr C$, then
\[
\mathfrak B_\mathscr{C}(a) = \overline{\mathfrak B \Phi(a) \mathfrak B}
\]
for every $a\in \mathfrak A$.

Moreover, if $\mathfrak A$ is unital, $\mathscr C$ is non-degenerate and $\Phi \colon \mathfrak A \to \multialg{\mathfrak B}$ is a unitally $\mathscr C$-absorbing $\ast$-homomorphism weakly in $\mathscr C$, then
\[
\mathfrak B_\mathscr{C}(a) = \overline{\mathfrak B \Phi(a) \mathfrak B}
\]
for every $a\in \mathfrak A$.
\end{lemma}
\begin{proof}
This follows from Lemma \ref{l:absdominates}.
\end{proof}

\begin{remark}
 It can easily be seen by the above lemma, that if $a_1$ and $a_2$ in $\mathfrak A$ generate the same two-sided, closed ideal, then $\mathfrak B_\mathscr{C}(a_1) = \mathfrak B_\mathscr{C}(a_2)$ if $\mathfrak A$, $\mathfrak B$ and $\mathscr C$ satisfy the
 (modest) assumptions of the lemma. 
 Hence the map $\mathfrak B_\mathscr{C} \colon \mathfrak A \to \mathbb I(\mathfrak B)$ drops to an action $\mathbb O(\Prim \mathfrak A) \cong \mathbb I(\mathfrak A) \to \mathbb I(\mathfrak B)$ as described in Section \ref{s:Xaction}.
 Thus any such closed operator convex cone induces an action of $\Prim \mathfrak A$ on $\mathfrak B$ which uniquely determines $\mathfrak B_\mathscr{C}$.
\end{remark}

We give a generalised description of purely large extensions as follows.

\begin{definition}\label{d:Cpl}
Let $0 \to \mathfrak B \to \mathfrak E \xrightarrow{p} \mathfrak A \to 0$ be an extension of $C^\ast$-algebras, with $\mathfrak A$ separable, $\mathfrak B$ $\sigma$-unital and stable, and let $\mathscr C \subset CP(\mathfrak A,\mathfrak B)$ be a countably generated,
closed operator convex cone. We will say that the extension is \emph{$\mathscr C$-purely large} if for any $x\in \mathfrak E$, $\overline{x^\ast\mathfrak B x}$ contains a $\sigma$-unital and stable sub-$C^\ast$-algebra $\mathfrak D$ 
which is full in $\mathfrak B_\mathscr{C}(p(x))$.
\end{definition}

Since $\overline{x^\ast \mathfrak B x} = \overline{ x^\ast x \mathfrak B x^\ast x}$ and since $\mathfrak B_{\mathscr C}(p(x)) = \mathfrak B_{\mathscr C}(p(x^\ast x))$ by Lemma \ref{l:idealmap}, it suffices to check the above condition only for positive $x$.

Note that an extension being $\mathscr C$-purely large only depends on the map $\mathfrak B_\mathscr{C}$.

The case above when the map $\mathfrak B_\mathscr{C}$ is the map defined in Example \ref{e:trivialcones}, corresponds to the purely large condition defined in \cite{ElliottKucerovsky-extensions}. 

The following proposition, together with the main result of \cite{ElliottKucerovsky-extensions}, motivates the purely large problem.

\begin{proposition}\label{purelylargeprop}
Let $\mathfrak e: 0 \to \mathfrak B \to \mathfrak E \to \mathfrak A \to 0$ be an extension of $C^\ast$-algebras, with $\mathfrak A$ separable, $\mathfrak B$ $\sigma$-unital and stable, and let $\mathscr C \subset CP(\mathfrak A,\mathfrak B)$ be a countably generated,
closed operator convex cone. If $\mathfrak e$ absorbs any trivial $\mathscr C$-extension, then the extension is $\mathscr C$-purely large.

Moreover, the same holds if $\mathfrak e$ is unital, $\mathscr C$ is non-degenerate, and $\mathfrak e$ absorbs any trivial, unital $\mathscr C$-extension.
\end{proposition}
\begin{proof}
The proof is very similar to that of \cite[Theorem 17(iii)]{ElliottKucerovsky-extensions}. We will give the argument for completion and since it is not identical to their proof. We will only prove the non-unital case, since the unital case is identical. 

Let $\tau$ be the Busby map of $\mathfrak e$ and let $\Phi_\mathscr{C} \colon \mathfrak A \to \multialg{\mathfrak B}$ denote a $\mathscr C$-absorbing $\ast$-homo\-morphism weakly in $\mathscr C$ of Theorem \ref{t:absrep}. 
Fix isometries $t_1,t_2,\dots$ such that $\sum_{k=1}^\infty t_kt_k^\ast =1$.
Let $\Phi = \sum_{k=1}^\infty t_k \Phi_\mathscr{C}(-)t_k^\ast$ denote the infinite repeat, and $\mathfrak e'$ denote the extension with Busby map $\tau' = \pi \circ \Phi$. 
Since $\Phi$ is weakly in $\mathscr C$, $\mathfrak e$ absorbs $\mathfrak e'$. Thus, given $\mathcal O_2$-isometries $s_1,s_2\in \multialg{\mathfrak B}$, we may assume without loss of generality that $\mathfrak e$ has Busby map $\tau \oplus_{s_1,s_2} \tau'$. 

Let $\sigma \colon \mathfrak E \to \multialg{\mathfrak B}$ be the canonical map and let $x\in \mathfrak E$ so that we should show the condition in Definition \ref{d:Cpl}. 
We may assume that $\| x\| \leq 1$, and as remarked after the definition, we may also assmue that $x$ is positive. Let $P=0\oplus_{s_1,s_2} 1 \in \multialg{\mathfrak B}$ and $\Phi' = 0 \oplus_{s_1,s_2} \Phi$.
Define $m = \sigma(x)^{1/2}P\sigma(x)^{1/2} \in \multialg{\mathfrak B}$, $a=p(x)$ and note that $m - \Phi'(a) \in \mathfrak B$.  
We get that $0 \leq m':= m^{1/2} \Phi'(a) m^{1/2} \leq m \leq \sigma(x)$ and $0 \leq m'' := \Phi'(a)^{1/2} m \Phi'(a)^{1/2} \leq \Phi'(a)$.
Set $\mathfrak D := \overline{m' \mathfrak B m'} \subseteq \overline{x\mathfrak B x}$ and note that it is $\sigma$-unital as $\mathfrak B$ is $\sigma$-unital, and that $\mathfrak D \cong \overline{ m'' \mathfrak B m''}$. 
Hence it suffices to show that $\mathfrak D$ is full in $\mathfrak B_\mathscr{C}(a)$ and that $\overline{m'' \mathfrak B m''}$ is stable.

By Lemma \ref{l:idealmap} 
\[
\mathfrak B_\mathscr{C}(a) = \overline{\mathfrak B \Phi_\mathscr{C}(a) \mathfrak B} = \overline{\mathfrak B \Phi'(a) \mathfrak B} \supseteq \overline{\mathfrak B \mathfrak D \mathfrak B}.
\]
Note that $m' - \Phi'(a)^2 \in \mathfrak B$ and thus $(m' - \Phi'(a)^2)(0 \oplus_{s_1,s_2} t_k) \to 0$. Any positive element in $\mathfrak B_\mathscr{C}(a)$ is close to an element of the form
\begin{eqnarray*}
\sum_{i=1}^n b_i^\ast (0 \oplus_{s_1,s_2} \Phi_\mathscr{C}(a)^2) b_i &=& \sum_{i=1}^n b_i^\ast (0 \oplus_{s_1,s_2} t_k^\ast t_k \Phi_\mathscr{C}(a)^2 t_k^\ast t_k) b_i \\
&=& \sum_{i=1}^n b_i^\ast (0\oplus_{s_1,s_2} t_k^\ast) \Phi'(a)^2 (0 \oplus_{s_1,s_2} t_k)  b_i \\
&\approx & \sum_{i=1}^n b_i^\ast (0 \oplus_{s_1,s_2} t_k^\ast) m' (0\oplus_{s_1,s_2} t_k) b_i \in \overline{\mathfrak B \mathfrak D \mathfrak B},
\end{eqnarray*}
for large $k$, where $b_1,\dots,b_n\in \mathfrak B$. Hence $\mathfrak D$ is full in $\mathfrak B_\mathscr{C}(a)$.

It remains to show that $\mathfrak D \cong \overline{ m'' \mathfrak B m''}$ is stable. 
By \cite[Theorem 2.1 and Proposition 2.2]{HjelmborgRordam-stability} it suffices to show that for any positive $b\in \overline{ m'' \mathfrak B m''}$, there is a sequence $(b_k)$ of elements in $\overline{ m'' \mathfrak B m''}$ such that
$b_k^\ast b_k - b \to 0$ and $(b_k^\ast b_k)(b_kb_k^\ast) \to 0$.

Fix a positive $b\in \overline{ m'' \mathfrak B m''} \subseteq \overline{ \Phi'(a) \mathfrak B \Phi'(a)}$. Note that $\Phi'(a)^{2/k} b \to b$ and that $(m'')^{1/k} - \Phi'(a)^{2/k}\in \mathfrak B$ for all $k$.
As in \cite[Theorem 17 (iii)]{ElliottKucerovsky-extensions}, which is basically a trick from \cite{HjelmborgRordam-stability}, 
we may pick a sequence of unitaries $(u_n)$ in $\multialg{\mathfrak B}$ such that each $u_n$ commutes with $\Phi'$, and $b_1 u_n b_2 \to 0$ for $b_1,b_2 \in \mathfrak B$.
For each $k$ let $n_k$ be such that
\[
(m'')^{1/k} u_{n_k} b^{1/2} \approx_{1/k} \Phi'(a)^{2/k} u_{n_k} b^{1/2} = u_{n_k} \Phi'(a)^{2/k} b^{1/2}.
\] 
Let $b_k := (m'')^{1/k} u_{n_k} b^{1/2} \in \overline{m'' \mathfrak B m''}$ and note that the sequence $(b_k)$ is bounded. We have that $b_k - u_{n_k} b^{1/2} \to 0$ and thus $b_k^\ast b_k - b \to 0$ since $(b_k)$ is bounded. Since
\begin{eqnarray*}
(b_k^\ast b_k)(b_k b_k^\ast) &=& b_k^\ast (m'')^{1/k} u_{n_k} b^{1/2} (m'')^{1/k} u_{n_k} b^{1/2} b_k^\ast \\
&\approx& b_k^\ast (m'')^{1/k} u_{n_k} ( b^{1/2} u_{n_k} b^{1/2}) b_k^\ast \quad \text{(for large $k$)} \\
& \to & 0
\end{eqnarray*}
it follows that $\overline{m'' \mathfrak B m''}$ is stable by \cite[Theorem 2.1 and Proposition 2.2]{HjelmborgRordam-stability}.
\end{proof}

We raise the following question. 

\begin{question}[The purely large problem]\label{q:plp}
Let $\mathfrak A$ and $\mathfrak B$ be $C^\ast$-algebras with $\mathfrak A$ separable and $\mathfrak B$ stable and $\sigma$-unital. For which countably generated, closed operator convex cones
$\mathscr C \subset CP(\mathfrak A, \mathfrak B)$ do the following hold? For any extension $\mathfrak e$ of $\mathfrak A$ by $\mathfrak B$:
\begin{itemize}
\item[$(Q1)$] $\mathfrak e$ absorbs any trivial $\mathscr C$-extension if and only if $\mathfrak e$ absorbs the zero extension and is $\mathscr C$-purely large,
\item[$(Q2)$] If, in addition, $\mathfrak e$ is unital and $\mathscr C$ is non-degenerate, then $\mathfrak e$ absorb any trivial, unital $\mathscr C$-extension if and only if $\mathfrak e$ is $\mathscr C$-purely large.
\end{itemize}
\end{question}

\begin{definition}\label{d:plp}
 Let $\mathfrak A$, $\mathfrak B$ and $\mathscr C$ be as in Question \ref{q:plp} above. We say that \emph{$\mathscr C$ satisfies the purely large problem} if it satisfies $(Q1)$ above.
 
 If, in addition, $\mathfrak A$ is unital and $\mathscr C$ is non-degenerate, then we say that \emph{$\mathscr C$ satisfies the unital purely large problem} if it satisfies $(Q2)$ above.
\end{definition}

Using the above definition, one could reformulate Theorem \ref{t:classicpl} as saying that if $\mathfrak A$ and $\mathfrak B$ are separable $C^\ast$-algebras with $\mathfrak B$ stable (and $\mathfrak A$ unital), 
then $CP_\nuc(\mathfrak A, \mathfrak B)$ satisfies the (unital) purely large problem.

An affirmative answer to the purely large problem implies a Weyl--von Neumann type theorem as follows.

\begin{proposition}[Abstract Weyl--von Neumann theorem]
Let $\mathfrak A$ and $\mathfrak B$ be $C^\ast$-algebras with $\mathfrak A$ separable and $\mathfrak B$ stable and $\sigma$-unital.
Suppose that $\mathscr C \subset CP(\mathfrak A, \mathfrak B)$ is a countably generated, closed operator convex cone.
If $\mathscr C$ satisfies the purely large problem, then a $\ast$-homomorphism $\Phi \colon \mathfrak A \to \multialg{\mathfrak B}$ is $\mathscr C$-absorbing if and only if
the extension with Busby map $\pi \circ \Phi$ is $\mathscr C$-purely large and absorbs the zero extension.

If $\mathfrak A$ is unital, $\mathscr C$ is non-degenerate and satisfies the unital purely large problem, then a unital $\ast$-homomorphism $\Phi \colon \mathfrak A \to \multialg{\mathfrak B}$ is unitally $\mathscr C$-absorbing 
if and only if the extension with Busby map $\pi \circ \Phi$ is $\mathscr C$-purely large.
\end{proposition}
\begin{proof}
If $\Phi$ is (unitally) $\mathscr C$-absorbing then clearly the extension $\mathfrak e$ with Busby map $\pi\circ \Phi$ absorbs every trivial, (unital) $\mathscr C$-extension.
By Proposition \ref{purelylargeprop}, $\mathfrak e$ is $\mathscr C$-purely large. In the non-unital case, $\Phi$ absorbs the zero homomorphism, so $\mathfrak e$ absorbs the zero extension.

We first prove the unital case. Suppose that $\pi \circ \Phi$ is $\mathscr C$-purely large and thus absorbs any trivial, unital $\mathscr C$-extension. 
Let $\Psi \colon \mathfrak A \to \multialg{\mathfrak B}$ be a unital $\ast$-homomorphism weakly in $\mathscr C$, and let $\Psi_\infty$ be an infinite repeat. 
Then there is a unitary $U \in \multialg{\mathfrak B}$ such that $U^\ast(\Phi(a) \oplus \Psi_\infty(a))U - \Phi(a) \in \mathfrak B$ for all $a\in \mathfrak A$.
By Theorem \ref{t:domabs}, $\Phi$ absorbs $\Psi_\infty$ and thus it also absorbs $\Psi$, so $\Phi$ is unitally $\mathscr C$-absorbing.

In the non-unital case, let $\Psi \colon \mathfrak A \to \multialg{\mathfrak B}$ be a $\ast$-homomorphism weakly in $\mathscr C$.
As above we may find a unitary $U\in \multialg{\mathfrak B}$ such that $U^\ast(\Phi(a) \oplus \Psi_\infty(a))U - \Phi(a) \in \mathfrak B$ for all $a\in \mathfrak A$.
As in the proof of Proposition \ref{p:unitalvsnonunital} it follows that $U^\ast(\Phi^\dagger(a) \oplus \Psi^\dagger_\infty(a)) U - \Phi^\dagger(a) \in \mathfrak B$ for all $a\in \mathfrak A^\dagger$.
As above, $\Phi^\dagger$ absorbs $\Psi^\dagger$ which implies that $\Phi$ absorbs $\Psi$, so $\Phi$ is $\mathscr C$-absorbing.
\end{proof}

\begin{remark}
 The purely large problem does not always have an affirmative answer. If $\mathscr C$ satisfies the (unital) purely large problem and $\mathscr C'$ is some other closed operator convex cone with $\mathfrak B_\mathscr{C} = \mathfrak B_\mathscr{C'}$, then
 $\mathscr C'$ does \emph{not} satisfy the (unital) purely large problem.
 
 To see this, suppose that both $\mathscr C$ and $\mathscr C'$ satisfy the (unital) purely large problem, and that $\mathfrak B_\mathscr{C} = \mathfrak B_\mathscr{C'}$. 
 Let $\Phi_\mathscr{C}, \Phi_{\mathscr C'} \colon \mathfrak A \to \multialg{\mathfrak B}$ be the (unital) $\mathscr C$-absorbing and $\mathscr C'$-absorbing $\ast$-homomorphisms weakly in $\mathscr C$ and $\mathscr C'$ respectively.
 The extensions with Busby maps $\pi \circ \Phi_{\mathscr C}$ and $\pi \circ \Phi_{\mathscr C'}$ are both $\mathscr C$-purely large and $\mathscr C'$-purely large, since $\mathfrak B_\mathscr{C} = \mathfrak B_\mathscr{C'}$. 
 Thus it follows that $\Phi_\mathscr{C}$ and $\Phi_{\mathscr C'}$ absorb each other, and thus $\mathscr C = \mathscr C'$.
 
 In particular, if $\mathfrak A$ is a separable, stable, non-nuclear $C^\ast$-algebra, then $CP(\mathfrak A, \mathfrak A)$ does not satisfy the purely large problem by Example \ref{e:trivialcones} and Theorem \ref{t:classicpl}.
\end{remark}

We will prove that if $\mathsf X$ is a finite space, $\mathfrak A$ and $\mathfrak B$ are sufficiently nice $\mathsf X$-$C^\ast$-algebras (e.g.~if both are separable $C^\ast$-algebras over $\mathsf X$) 
and $\mathscr C=CP_\rnuc(\mathsf X;\mathfrak A,\mathfrak B)$ then $\mathscr C$ satisfies the purely large problem. If, in addition, $\mathfrak A$ is unital and $\mathscr C$ is non-degenerate, then $\mathscr C$ satisfies the unital purely large problem.
See Theorem \ref{t:purelylargefinite}.

As an additional, or perhaps even preliminary, question to the purely large problem, one can ask the following.

\begin{question}\label{q:computeB_C}
Is there an easy way to compute the map $\mathfrak B_\mathscr C$?
\end{question}

In the case described above, we have an affirmative answer to this question, see Proposition \ref{p:determineB_C}. A much more general solution to this question, in many cases of interest, will appear in \cite{Gabe-cplifting} by the first author.

It turns out that the purely large problem can always be answered by answering a related \emph{unital} purely large problem. We will need the following lemma.

\begin{lemma}\label{l:B_Cdagger}
 Let $\mathscr C \subset CP(\mathfrak A, \mathfrak B)$ be a closed operator convex cone. If $a\in \mathfrak A \subset \mathfrak A^\dagger$ then $\mathfrak B_{\mathscr C^\dagger}(a) = \mathfrak B_{\mathscr C}(a)$.
 If $a\in \mathfrak A^\dagger \setminus \mathfrak A$ then $\mathfrak B_{\mathscr C^\dagger}(a) = \mathfrak B$.
\end{lemma}
\begin{proof}
 Suppose $a\in \mathfrak A$. Clearly $\mathfrak B_{\mathscr C^\dagger}(a) \subset \mathfrak B_{\mathscr C}(a)$. Let $\phi \in \mathscr C$ and $(a_\lambda)$ be an approximate identity in $\mathfrak A$.
 Let $\phi_\lambda \in \mathscr C^\dagger$ be given by $\phi_\lambda(x + \mu 1) = \phi(a_\lambda x a_\lambda + \mu a_\lambda^2)$. Then $\phi_\lambda|_\mathfrak{A} \to \phi$ point-norm. 
 It follows that $\phi(a) \in \mathfrak B_{\mathscr C^\dagger}(a)$ for all $\phi\in \mathscr C$ and thus $\mathfrak B_\mathscr{C}(a) \subset \mathfrak B_{\mathscr C}(a)$.
 
 Let $a\in \mathfrak A^\dagger \setminus \mathfrak A$. For every positive $b\in \mathfrak B$ there is a map $\phi_b$ in $\mathscr C^\dagger$ given as the composition
 \[
  \mathfrak A^\dagger \twoheadrightarrow \mathfrak A^\dagger/\mathfrak A \cong \mathbb C \xrightarrow{\rho_b} \mathfrak B
 \]
 where $\rho_b(\mu) = \mu b$. Since $\phi_b(a) = \mu b$ for some non-zero $\mu \in \mathbb C$, it follows that $b\in \mathfrak B_{\mathscr C^\dagger}(a)$ for every positive $b\in \mathfrak B$. Thus $\mathfrak B_{\mathscr C^\dagger}(a) = \mathfrak B$.
\end{proof}

\begin{proposition}\label{p:plpunitimpliesnonunit}
Let $\mathfrak e: 0 \to \mathfrak B \to \mathfrak E \xrightarrow{p} \mathfrak A \to 0$ be an extension of $C^\ast$-algebras, with $\mathfrak A$ separable and $\mathfrak B$ stable and $\sigma$-unital, 
and let $\mathscr C\subset CP(\mathfrak A, \mathfrak B)$ be a countably generated, closed operator convex cone.
If $\mathfrak e$ is $\mathscr C$-purely large and absorbs the zero extension, then the unitised extension $\mathfrak e^\dagger : 0 \to \mathfrak B \to \mathfrak E^\dagger \xrightarrow{p^\dagger} \mathfrak A^\dagger \to 0$ is $\mathscr C^\dagger$-purely large.

In particular, if $\mathscr C^\dagger$ satisfies the unital purely large problem then $\mathscr C$ satisfies the purely large problem.
\end{proposition}
\begin{proof}
 Since $\mathfrak e$ absorbs the zero extension we may assume that the Busby map $\tau$ of $\mathfrak e$ is of the form $\tau = \tau_0 \oplus_{s_1,s_2} 0$ for some $\mathcal O_2$-isometries $s_1,s_2 \in \multialg{\mathfrak B}$. 
 Let $x \in \mathfrak E^\dagger$ so that we must show that $\overline{x^\ast \mathfrak B x}$ contains a stable, $\sigma$-unital $C^\ast$-subalgebra $\mathfrak D$ which is full in $\mathfrak B_\mathscr{C^\dagger}(p^\dagger(x))$.
 If $x\in \mathfrak E$, then $\mathfrak B_{\mathscr C^\dagger}(p^\dagger(x)) = \mathfrak B_\mathscr{C}(p(x))$ by Lemma \ref{l:B_Cdagger}, and thus such a $\mathfrak D$ exists, since $\mathfrak e$ is $\mathscr C$-purely large.
 So it remains to check this when $x \in \mathfrak E^\dagger \setminus \mathfrak E$. Clearly, it suffices to do this when $x = 1+y$ with $y\in \mathfrak E$. 
 Let $\sigma \colon \mathfrak E \to \multialg{\mathfrak B}$ be the canonical $\ast$-homomorphism and let $P = 0\oplus_{s_1,s_2} 1 \in \multialg{\mathfrak B}$. 
 Since $\pi(P\sigma(y)) = \pi(P)\tau(p(y)) = 0 \in \corona{\mathfrak B}$ it follows that $b:= P \sigma(y)  \in \mathfrak B$. Note that $\mathfrak B_{\mathscr C^\dagger}(p^\dagger(1+y)) = \mathfrak B$ by Lemma \ref{l:B_Cdagger}. Since 
 \[ 
  \overline{(P+b)^\ast \mathfrak B (P+b)} = \overline{(1+y)^\ast P \mathfrak B P(1+y)} \subset \overline{(1+y)^\ast \mathfrak B (1+y)}
 \]
 it suffices to show that $\overline{(P+b)^\ast \mathfrak B (P+b)}$ contains a stable, $\sigma$-unital $C^\ast$-subalgebra which is full in $\mathfrak B$. Consider the trivial extension $\mathfrak e_0$ of $\mathbb C$ by $\mathfrak B$ with splitting $1 \mapsto P$. 
 This extension clearly absorbs any trivial extension of $\mathbb C$ by $\mathfrak B$, since the Cuntz sum of $P$ with any projection is unitarily equivalent to $P$. By Proposition \ref{purelylargeprop} it follows that $\mathfrak e_0$ is 
 $CP(\mathbb C, \mathfrak B)$-purely large, i.e.~purely large in the classical sense. Since $P+b$ is an element of the extension algebra of $\mathfrak e_0$, it follows that $\overline{(P+b)^\ast \mathfrak B (P+b)}$ contains a stable, $\sigma$-unital
 $C^\ast$-subalgebra which is full in $\mathfrak B$. Hence $\mathfrak e^\dagger$ is $\mathscr C^\dagger$-purely large.
 
 Now, for the ``in particular'' part, let $\mathfrak f$ be a trivial $\mathscr C$-extension. 
 Then $\mathfrak f^\dagger$ is a trivial, unital $\mathscr C^\dagger$-extension and is thus absorbed by $\mathfrak e^\dagger$, since $\mathscr C^\dagger$ satisfies the unital purely large problem.
 It easily follows that $\mathfrak e$ absorbs $\mathfrak f$, and thus $\mathscr C$ satisfies the purely large problem.
\end{proof}

We end this section by applying results of Kirchberg and Rørdam to give a lot of non-trivial examples of cones satisfying the (unital) purely large problem.

\begin{proposition}\label{p:plpinfinite}
 Let $\mathfrak A$ be a separable, nuclear $C^\ast$-algebra and $\mathfrak B$ be a separable, stable, nuclear, strongly purely infinite $C^\ast$-algebra. Then any closed operator convex cone $\mathscr C \subset CP(\mathfrak A, \mathfrak B)$ satisfies the purely
 large problem.
 
 If, in addition, $\mathfrak A$ is unital, then any non-degenerate, closed operator convex cone $\mathscr C \subset CP(\mathfrak A, \mathfrak B)$ satisfies the unital purely large problem.
\end{proposition}
\begin{proof}
 By Proposition \ref{p:plpunitimpliesnonunit} it suffices to prove the unital case.
 By \cite[Theorem 6.3]{Kirchberg-permanence} $\multialg{\mathfrak B}$ is strongly purely infinite. For any unital $\mathscr C$-purely large extension $0 \to \mathfrak B \to \mathfrak E \xrightarrow{p} \mathfrak A \to 0$ 
 let $\sigma \colon \mathfrak E \to \multialg{\mathfrak B}$ be the induced $\ast$-homomorphism.
 Then $\mathfrak C := \sigma(\mathfrak E)$ is a separable, nuclear $C^\ast$-subalgebra of $\multialg{\mathfrak B}$, and by \cite[Theorem 7.21]{KirchbergRordam-absorbingOinfty} any approximately inner map
 $\mathfrak C \to \multialg{\mathfrak B}$ is approximately one-step inner. 
 
 We claim that $\mathscr C$-purely largeness implies that $\mathfrak B_\mathscr{C}(p(x)) \subset \overline{\mathfrak B x \mathfrak B}$ for any positive $x\in \mathfrak E$.
 To see this, let $\mathfrak D\subset \overline{x \mathfrak B x}$ be a $\sigma$-unital, stable $C^\ast$-subalgebra which is full in $\mathfrak B_\mathscr{C}(p(x))$.
 Then
 \[
  \mathfrak B_\mathscr{C}(p(x)) = \overline{\mathfrak B \mathfrak D \mathfrak B} \subset \overline{\mathfrak B x \mathfrak B x \mathfrak B} = \overline{\mathfrak B x \mathfrak B} = \overline{\mathfrak B \sigma(x) \mathfrak B}.
 \]
 It follows from \cite[Corollary 4.3]{KirchbergRordam-zero} that any map of the form $\phi \circ p$ with $\phi\in \mathscr C$ is approximately inner with respect to $\sigma$, 
 and thus also approximately one-step inner, i.e.~for any $\phi \in \mathscr C$,
 the c.p.~map $\phi \circ p$ is approximated point-norm by maps of the form $m^\ast \sigma(-) m$ for $m\in \multialg{\mathfrak B}$. 
 It easily follows that $\sigma$ strongly approximately dominates any map $\phi \circ p \colon \mathfrak E \to \mathfrak B$ where $\phi \in \mathscr C$.
 Thus it follows from Theorem \ref{t:domgen} that $\sigma$ absorbs $\Phi \circ p$ where $\Phi \colon \mathfrak A \to \multialg{\mathfrak B}$ is a unitally $\mathscr C$-absorbing $\ast$-homomorphism weakly in $\mathscr C$.
 It follows that $\tau$ absorbs $\pi \circ \Phi$ and thus $\tau$ absorbs any trivial, unital $\mathscr C$-extension.
\end{proof}

\begin{remark}
 Note that we may assume something much weaker than $\mathscr C$-purely largeness in the above proposition. In fact, the only thing that $\mathscr C$-purely largeness is used for, 
 is so that $\mathfrak B_\mathscr{C}(p(x)) \subset \overline{\mathfrak B x \mathfrak B}$ for every
 positive $x\in \mathfrak E$. Hence this may be assumed above instead of $\mathscr C$-purely largeness of the extensions.
 
 In the (classical) case where $\mathfrak B$ in addition is simple and $\mathscr C = CP(\mathfrak A, \mathfrak B)$, this is the same as noting that an extension is purely large exactly when it is essential.
 When $\mathfrak B$ is simple and $\sigma$-unital, then every essential extension by $\mathfrak B$ is purely large if and only if $\mathfrak B$ is either purely infinite or isomorphic to $\mathbb K$.
 This result will be generalised in Theorem \ref{t:WvN}.
\end{remark}


\section{Actions of topological spaces on $C^\ast$-algebras}\label{s:approxsection}

So far, we have only studied closed operator convex cones rather abstractly. In order to construct actual examples, we use the notions of actions of topological spaces on $C^\ast$-algebras.
Such actions have been studied quite extensively. In particular, continuous actions of finite spaces have been well studied, with respect to the classification of separable, nuclear, purely infinite $C^\ast$-algebras with finite primitive ideal spaces,
using $K$-theoretic invariants.

\subsection{The basics}\label{s:Xaction}

When $\mathfrak A$ is a $C^\ast$-algebra we let $\mathbb I(\mathfrak A)$ denote the complete lattice of two-sided, closed ideals in $\mathfrak A$.

If $\mathsf X$ is a topological space we let $\mathbb O(\mathsf X)$ denote the complete lattice of open subsets of $\mathsf X$. For a subset $\mathsf Y$ of $\mathsf X$ we let $\mathsf Y^\circ$ denote the interior of $\mathsf Y$. 
Recall, that if $\Prim \mathfrak A$ is the primitive ideal space of $\mathfrak A$ then there is a canonical complete lattice isomorphism $\mathbb O(\Prim \mathfrak A)\cong \mathbb I(\mathfrak A)$.

\begin{definition}
Let $\mathsf X$ be a topological space. An \emph{action of $\mathsf X$ on a $C^\ast$-algebra $\mathfrak A$} is an order preserving map
\[
\mathbb O(\mathsf X) \to \mathbb I(\mathfrak A), \quad \mathsf U \mapsto \mathfrak A(\mathsf U),
\]
i.e.~a map such that if $\mathsf U \subseteq \mathsf V$ in $\mathbb O(\mathsf X)$ then $\mathfrak A(\mathsf U) \subseteq \mathfrak A(\mathsf V)$.

If $\mathfrak A$ is a $C^\ast$-algebra with an action of $\mathsf X$, 
then we say that $\mathfrak A$ (together with this action) is an \emph{$\mathsf X$-$C^\ast$-algebra}. We will almost always suppress the action in the notation.

A map $\phi \colon \mathfrak A \to \mathfrak B$ of $\mathsf X$-$C^\ast$-algebras is called \emph{$\mathsf X$-equivariant} if $\phi(\mathfrak A(\mathsf U)) \subseteq \mathfrak B(\mathsf U)$ for every $\mathsf U\in \mathbb O(\mathsf X)$.

An \emph{isomorphism of $\mathsf X$-$C^\ast$-algebras} is an $\mathsf X$-equivariant $\ast$-isomorphism $\phi \colon \mathfrak A \to \mathfrak B$ such that the inverse $\phi^{-1}$ is also $\mathsf X$-equivariant.
\end{definition}

\begin{remark}\label{r:Xeqcone}
If $\mathfrak A$ and $\mathfrak B$ are $\mathsf X$-$C^\ast$-algebras, then the set $CP(\mathsf X; \mathfrak A, \mathfrak B)$ of $\mathsf X$-equivariant c.p.~maps is a closed operator convex cone.
\end{remark}

\begin{remark}\label{r:soberspace}
 As explained in \cite[Section 2.5]{MeyerNest-bootstrap}, for any topological space $\mathsf X$, we may find an induced \emph{sober} topological space $\hat{\mathsf X}$ such that there is an induced complete lattice isomorphism 
 $\mathbb O(\mathsf X) \cong \mathbb O(\hat{\mathsf X})$. If $\mathsf X$ is already sober then $\mathsf X = \hat{\mathsf X}$. In particular, we may always assume that $\mathsf X$ is a sober space.
 
 Since any sober space is a $T_0$-space (i.e.~for any two distinct points in $\mathsf X$, one of the points contains an open neighbourhood not containing the other point), we may assume that $\mathsf X$ is a $T_0$-space.
\end{remark}

It is often necessary to impose stronger conditions on our actions.

\begin{definition}
Let $\mathsf X$ be a topological space and $\mathfrak A$ be an $\mathsf X$-$C^\ast$-algebra. We say that $\mathfrak A$ is
\begin{itemize}
 \item \emph{finitely lower semicontinuous} if $\mathfrak A(\mathsf X) = \mathfrak A$, and if it respects finite infima, i.e.~for open subsets $\mathsf U$ and $\mathsf V$ of $\mathsf X$ we have 
 \[
 \mathfrak A(\mathsf U) \cap \mathfrak A(\mathsf V) = \mathfrak A(\mathsf U \cap \mathsf V),
 \]
 \item \emph{lower semicontinuous} if $\mathfrak A(\mathsf X) = \mathfrak A$, and if it respects arbitrary infima, i.e.~for any family $(\mathsf U_\alpha)$ of open subsets of $\mathsf X$ we have
 \[
 \bigcap_\alpha \mathfrak A(\mathsf U_\alpha) = \mathfrak A(\mathsf U),
 \]
 where $\mathsf U$ is the interior of $\bigcap_\alpha \mathsf U_\alpha$,
 \item \emph{finitely upper semicontinuous} if $\mathfrak A(\emptyset) = 0$ and if it respects finite suprema, i.e.~for open subsets $\mathsf U$ and $\mathsf V$ of $\mathsf X$ we have
 \[
  \mathfrak A(\mathsf U) + \mathfrak A(\mathsf V) = \mathfrak A(\mathsf U \cup \mathsf V),
 \]
 \item \emph{monotone upper semicontinuous} if it respects monotone suprema, i.e.~for any increasing net $(\mathsf U_\alpha)$ of open subsets of $\mathsf X$ we have
 \[
  \overline{ \bigcup_\alpha \mathfrak A(\mathsf U_\alpha)} = \mathfrak A(\bigcup_\alpha \mathsf U_\alpha),
 \]
 \item \emph{upper semicontinuous} if it is finitely and monotone upper semicontinuous,
 \item \emph{finitely continuous} if it is finitely lower semicontinuous and finitely upper semicontinuous,
 \item \emph{continuous} if it is lower semicontinuous and upper semicontinuous,
 \item \emph{tight} if the action $\mathbb O(\mathsf X) \to \mathbb I(\mathfrak A)$ is a complete lattice isomorphism.
\end{itemize}
\end{definition}

\begin{remark}
 Note that we require (finitely) lower semicontinuous $\mathsf X$-$C^\ast$-algebras $\mathfrak A$ to satisfy $\mathfrak A(\mathsf X) = \mathfrak A$, and (finitely) upper semicontinuous $\mathsf X$-$C^\ast$-algebras $\mathfrak B$ to satisfy $\mathfrak B(\emptyset) = 0$.
 
 Although we will not use it in this paper, the reason that we require this in the upper semicontinuous case, is so that the map
 \[
  \mathbb I(\mathfrak A) \to \mathbb O(\mathsf X), \qquad \mathfrak I \mapsto \bigcup \{\mathsf U \in \mathbb O(\mathsf X) : \mathfrak A(\mathsf U) \subset \mathfrak I\}
 \]
 is well-defined. The reason in the finitely upper semicontinuous case is basically the same, since the action of any finite filtration (which we define in Section \ref{s:approx}) will be upper semicontinouous.
\end{remark}

\begin{example}\label{e:XvsC(X)}
 If $\mathsf X$ is a locally compact Hausdorff space, then a continuous $\mathsf X$-$C^\ast$-algebra is essentially the same as a continuous $C_0(\mathsf X)$-algebra. 
 In fact, the action of a continuous $\mathsf X$-$C^\ast$-algebra induces a non-degenerate $\ast$-homomorphism from $C_0(\mathsf X)$ to
 the centre of $\multialg{\mathfrak B}$ which defines a continuous $C_0(\mathsf X)$-algebra structure, and vice versa.
 This is shown in \cite[Section 2.2]{MeyerNest-bootstrap}.
 It is also shown that a $\ast$-homomorphism is $\mathsf X$-equivariant if and only if it is $C(\mathsf X)$-linear. The same proof applies to c.p.~maps, thus a c.p.~map is $\mathsf X$-equivariant if and only if it is $C(\mathsf X)$-linear.
\end{example}

\begin{remark}[$\mathsf X$-$C^\ast$-algebras versus $C^\ast$-algebras over $\mathsf X$]
 A \emph{$C^\ast$-algebra over $\mathsf X$} is a $C^\ast$-algebra $\mathfrak A$ together with a continuous map $\Prim \mathfrak A \to \mathsf X$. It is shown in \cite{MeyerNest-bootstrap} that a $C^\ast$-algebra over $\mathsf X$ is essentially
 the same thing as a finitely lower semicontinuous, upper semicontinuous $\mathsf X$-$C^\ast$-algebra. 
 Thus we abuse notation slightly, by saying that a $C^\ast$-algebra over $\mathsf X$ is a finitely lower semicontinuous, upper semicontinuous $\mathsf X$-$C^\ast$-algebra.
 
 In \cite[Section 2.9]{MeyerNest-bootstrap} Meyer and Nest give good reasons for why they only want to consider $C^\ast$-algebras over $\mathsf X$. However, only considering $C^\ast$-algebras over $\mathsf X$ is a great limitation on the constructions we allow
 ourselves to do, and seems to be more of a restraint than a simplification. 
 
 As an example, let us consider one of the most elementary construction for $C^\ast$-algebras; the unitisation. 
 If $\mathfrak A$ is a $C^\ast$-algebra over $\mathsf X$, there is no natural way to give the forced unitisation $\mathfrak A^\dagger$ a $C^\ast$-algebra over $\mathsf X$ structure.
 However, if $\mathfrak A$ is an $\mathsf X$-$C^\ast$-algebra then we may give the forced unitisation $\mathfrak A^\dagger$ an $\mathsf X$-$C^\ast$-algebra structure by letting $\mathfrak A^\dagger(\mathsf U) = \mathfrak A(\mathsf U)$ for 
 $\mathsf U \neq \mathsf X$ and $\mathfrak A^\dagger(\mathsf X) = \mathfrak A^\dagger$. 
 One can easily verify that for $\mathsf X$-$C^\ast$-algebras $\mathfrak A$ and $\mathfrak B$ we have 
 \[
  CP(\mathsf X; \mathfrak A^\dagger, \mathfrak B) = CP(\mathsf X; \mathfrak A, \mathfrak B(\mathsf X))^\dagger,  
 \]
 where we used the notation from Lemma \ref{unitalcone}. 
 
 The forced unitisation always preserves lower semicontinuity and finite lower semicontinuity. 
 However, it preserves monotone upper semicontinuity if and only if $\mathsf X$ is compact, and it preserves finite upper semicontinuity if and only if $\mathsf X$ is \emph{not} the union of two proper open subsets. 
 Thus, for most spaces $\mathsf X$, the forced unitisation of a $C^\ast$-algebra over $\mathsf X$ is not a $C^\ast$-algebra over $\mathsf X$. 
\end{remark}


\subsection{Lower semicontinuous $\mathsf X$-$C^\ast$-algebras and full maps}\label{s:lscaction}

\begin{definition}
Let $\mathfrak A$ be an $\mathsf X$-$C^\ast$-algebra, $a\in \mathfrak A$ and $\mathsf U\in \mathbb O(\mathsf X)$. 
We say that $a$ is \emph{$\mathsf U$-full}, if $\mathsf U$ is the unique smallest open subset of $\mathsf X$ such that $a \in \mathfrak A(\mathsf U)$, i.e.~$a\in \mathfrak A(\mathsf U)$ 
and whenever $\mathsf V \in \mathbb O(\mathsf X)$ such that $a\in \mathfrak A(\mathsf V)$ then $\mathsf U \subseteq \mathsf V$.
\end{definition}

It turns out that every element in an $\mathsf X$-$C^\ast$-algebra being $\mathsf U$-full for some $\mathsf U$, is equivalent to the $\mathsf X$-$C^\ast$-algebra being lower semicontinuous.

\begin{proposition}
Let $\mathfrak A$ be an $\mathsf X$-$C^\ast$-algebra. Then $\mathfrak A$ a lower semicontinuous if and only if every element $a\in \mathfrak A$ is $\mathsf U$-full for some (unique) $\mathsf U\in \mathbb O(\mathsf X)$.
\end{proposition}
\begin{proof}
Suppose that $\mathfrak A$ is lower semicontinuous and let $a \in \mathfrak A$. Let $\mathcal U = \{ \mathsf V \in \mathbb O (\mathsf X) : a \in \mathfrak A(\mathsf V)\}$, which is non-empty since $\mathfrak A(\mathsf X) = \mathfrak A$, and define
\[
\mathsf U = ( \bigcap_{\mathsf V \in \mathcal U} \mathsf V)^\circ.
\]
By lower semicontinuity $\bigcap_{\mathsf V\in \mathcal U} \mathfrak A(\mathsf V) = \mathfrak A(\mathsf U)$ which contains $a$. Hence by construction $a$ is $\mathsf U$-full.

Suppose that every element in $\mathfrak A$ is $\mathsf U$-full for some $\mathsf U$, and let $(\mathsf U_\alpha)$ be a family of open subsets of $\mathsf X$.  
Define $\mathsf U_0 := (\bigcap_{\alpha} \mathsf U_\alpha)^\circ$ and note that we obviously have an inclusion $\mathfrak A(\mathsf U_0) \subseteq \bigcap \mathfrak A(\mathsf U_\alpha)$.
Let $a\in  \bigcap \mathfrak A(\mathsf U_\alpha)$ and $\mathsf U\in \mathbb O(\mathsf X)$ such that $a$ is $\mathsf U$-full. Since $a\in \mathfrak A(\mathsf U_\alpha)$ for all $\alpha$, $\mathsf U\subseteq \mathsf U_\alpha$ for each $\alpha$. 
Hence $\mathsf U\subseteq \mathsf U_0$, implying that $a\in \mathfrak A(\mathsf U_0)$. Thus $\mathfrak A(\mathsf U_0) = \bigcap \mathfrak A(\mathsf U_\alpha)$. Finally, if $a \in \mathfrak A\setminus \mathfrak A( \mathsf X)$ and $a$ is $\mathsf U$-full for some $\mathsf U \in \mathbb O(\mathsf X)$, then $a \in \mathfrak A(\mathsf U) \subset \mathfrak A(\mathsf X)$ which is a contradiction. Thus $\mathfrak A(\mathsf X) = \mathfrak A$ and thus $\mathfrak A$ is lower semicontinouous. 
\end{proof}

\begin{notation}\label{n:U_a-full}
 If $\mathfrak A$ is a lower semicontinuous $\mathsf X$-$C^\ast$-algebra, and $a \in \mathfrak A$, then we denote by $\mathsf U_a$ the unique open subset of $\mathsf X$ for which $a$ is $\mathsf U_a$-full.
\end{notation}

\begin{example}
 Let $\mathsf X$ be a locally compact Hausdorff space and $\mathfrak A$ be a continuous $C_0(\mathsf X)$-algebra. For $a\in \mathfrak A$ it is easily seen that
 \[
  \mathsf U_a = \{ x \in \mathsf X : \| a_x \|_x > 0 \}.
 \]
 where $a_x$ denotes the fibre at $x$.
\end{example}

\begin{observation}
 Let $\mathfrak A$ and $\mathfrak B$ be $\mathsf X$-$C^\ast$-algebras with $\mathfrak A$ lower semicontinuous. 
 Then a map $\phi \colon \mathfrak A \to \mathfrak B$ is $\mathsf X$-equivariant if and only if for all $a\in \mathfrak A$, $\phi(a) \in \mathfrak B(\mathsf U_a)$.
\end{observation}

Recall, that an element $b$ in a $C^\ast$-algebra $\mathfrak B$ is called \emph{full} if the closed, two-sided ideal generated by $b$ is all of $\mathfrak B$.

\begin{definition}[Full $\mathsf X$-equivariant maps]
 Let $\mathfrak A$ and $\mathfrak B$ be $\mathsf X$-$C^\ast$-algebras with $\mathfrak A$ lower semicontinuous, and let $\phi \colon \mathfrak A \to \mathfrak B$ be an $\mathsf X$-equivariant map. We say that $\phi$ is \emph{full} (or $\mathsf X$-full)
 if $\phi(a)$ is full in $\mathfrak B(\mathsf U_a)$ for all $a \in \mathfrak A$.
\end{definition}

Note that simply the existence of an $\mathsf X$-full map $\phi \colon \mathfrak A \to \mathfrak B$ puts a restraint on the action on $\mathfrak B$. 
In fact, we always have that $0 \in \mathfrak A$ is $\emptyset$-full, and thus $\phi(0) = 0$ must be full in $\mathfrak B(\emptyset)$. Thus $\mathfrak B(\emptyset) = 0$.
However, this is a very small requirement, and since in all our applications $\mathfrak B$ will be finitely upper semicontinuous, this obstruction will never be an issue.

\begin{example}
 Let $\mathsf X = \star$ be the one-point space. Every $C^\ast$-algebra is canonically a continuous $\mathsf X$-$C^\ast$-algebra, and any linear map $\phi \colon \mathfrak A \to \mathfrak B$ is canonically $\mathsf X$-equivariant.
 Clearly any non-zero element $a\in \mathfrak A$ is $\mathsf X$-full, and one always has that $0$ is $\emptyset$-full. Note that $\phi(0) = 0$ is full in $\mathfrak B(\emptyset) = 0$. 
 Hence $\phi$ is full in the $\mathsf X$-equivariant sense exactly when $\phi(a)$ is full in $\mathfrak B$ for every non-zero $a\in \mathfrak A$.
 This is exactly the classical definition of a linear map being full.
\end{example}

One of the applications is that if there exists a full map we get a description of the map $\mathfrak B_\mathscr{C} \colon \mathfrak A \to \mathbb I(\mathfrak B)$ c.f.~Question \ref{q:computeB_C}.

\begin{lemma}\label{l:fullcomputeB_C}
 Let $\mathfrak A$ and $\mathfrak B$ be $\mathsf X$-$C^\ast$-algebras with $\mathfrak A$ lower semicontinuous, and let $\mathscr C \subset CP(\mathsf X; \mathfrak A, \mathfrak B)$ be a closed operator convex cone. 
 Let $a\in \mathfrak A$. If there is $\phi \in \mathscr C$ such that $\phi(a)$ full in $\mathfrak B(\mathsf U_a)$, then $\mathfrak B_\mathscr{C}(a) = \mathfrak B(\mathsf U_a)$.
 
 In particular, if $\mathscr C$ contains an $\mathsf X$-full map, then $\mathfrak B_\mathscr{C}(a) = \mathfrak B(\mathsf U_a)$ for every $a\in \mathfrak A$.
\end{lemma}
\begin{proof}
 Since every map in $\mathscr C$ is $\mathsf X$-equivariant, we clearly have $\mathfrak B_\mathscr{C}(a) \subset \mathfrak B(\mathsf U_a)$ for all $a\in \mathfrak A$. 
 Let $a\in \mathfrak A$ for which there exists $\phi\in \mathscr C$ such that $\phi(a)$ is full in 
 $\mathfrak B(\mathsf U_a)$. Then $\mathfrak B(\mathsf U_a) = \overline{\mathfrak B \phi(a) \mathfrak B} \subset \mathfrak B_\mathscr{C}(a)$. Thus $\mathfrak B_\mathscr{C}(a) = \mathfrak B(\mathsf U_a)$. 
 The last part of the lemma follows immediately since if $\mathscr{C}$ contains an $\mathsf{X}$-full map $\phi$, then for every $a \in \mathfrak A$, we have that $\phi (a)$ is full in $\mathfrak B( \mathsf{U}_{a})$.
\end{proof}

\subsection{Basic $\mathsf X$-equivariant extension theory}

Given a two-sided, closed ideal $\mathfrak J$ in a $C^\ast$-algebra $\mathfrak B$ there are canonical induced two-sided, closed ideals in the multiplier algebra $\multialg{\mathfrak B}$ and the corona algebra $\corona{\mathfrak B}$, given by
\begin{eqnarray*}
\multialg{\mathfrak B ,\mathfrak J} &=& \{ m \in \multialg{\mathfrak B} : m \mathfrak B \subseteq \mathfrak J \}, \\
\corona{\mathfrak B, \mathfrak J} &=& (\multialg{\mathfrak B, \mathfrak J} + \mathfrak B)/\mathfrak B.
\end{eqnarray*}
Note that $\multialg{\mathfrak B, \mathfrak J} \cap \mathfrak B = \mathfrak J$.

Given an action of a space $\mathsf X$ on a $C^\ast$-algebra $\mathfrak B$, there are induced actions on $\multialg{\mathfrak B}$ and $\corona{\mathfrak B}$ given by $\multialg{\mathfrak B} (\mathsf U) = \multialg{\mathfrak B, \mathfrak B(\mathsf U)}$ and
$\corona{\mathfrak B}(\mathsf U) = \corona{\mathfrak B, \mathfrak B(\mathsf U)}$.

\begin{remark}\label{r:canoniciso}
The embedding $\multialg{\mathfrak B}(\mathsf U) \hookrightarrow \multialg{\mathfrak B}(\mathsf U) + \mathfrak B$ induces an isomorphism
\[
 \frac{\multialg{\mathfrak B}(\mathsf U)}{\mathfrak B(\mathsf U)} = \frac{\multialg{\mathfrak B}(\mathsf U)}{\multialg{\mathfrak B}(\mathsf U) \cap \mathfrak B} \xrightarrow{\cong} \frac{\multialg{\mathfrak B}(\mathsf U) + \mathfrak B}{\mathfrak B} =
 \corona{\mathfrak B}(\mathsf U),
\]
for every $\mathsf U\in \mathbb O(\mathsf X)$. This implies that any element $\corona{\mathfrak B}(\mathsf U)$ lifts to an element in $\multialg{\mathfrak B}(\mathsf U) \subset \multialg{\mathfrak B}$. 
Also, for any $\mathsf U\in \mathbb O(\mathsf X)$ this implies that the sequence $0 \to \mathfrak B(\mathsf U) \to \multialg{\mathfrak B}(\mathsf U) \to \corona{\mathfrak B}(\mathsf U) \to 0$ 
induced by $0 \to \mathfrak B \to \multialg{\mathfrak B} \to \corona{\mathfrak B} \to 0$ is a short exact sequence.
\end{remark}

We have the following.

\begin{lemma}\label{l:actmulticor}
 Let $\mathfrak B$ be an $\mathsf X$-$C^\ast$-algebra.
 \begin{itemize}
  \item[$(a)$] If $\mathfrak B$ is finitely lower semicontinuous, then $\multialg{\mathfrak B}$ and $\corona{\mathfrak B}$ are finitely lower semicontinuous.
  \item[$(b)$] If $\mathfrak B$ is lower semicontinuous, then $\multialg{\mathfrak B}$ is lower semicontinuous.
  \item[$(c)$] If $\mathfrak B$ is $\sigma$-unital and finitely upper semicontinuous, then $\multialg{\mathfrak B}$ and $\corona{\mathfrak B}$ are finitely upper semicontinouous.
 \end{itemize}
\end{lemma}
\begin{proof}
 Part $(b)$ is obvious, and part $(a)$ for the multiplier algebra is also clear. Thus it suffices to show that $\corona{\mathfrak B}$ is finitely lower semicontinuous whenever $\mathfrak B$ is, where we obviously have $\corona{\mathfrak B}(\mathsf X) = \corona{\mathfrak B}$.
 
 Let $\mathsf U, \mathsf V \in \mathbb O(\mathsf X)$. Clearly $\corona{\mathfrak B}(\mathsf U \cap \mathsf V) \subset \corona{\mathfrak B}(\mathsf U) \cap \corona{\mathfrak B}(\mathsf V)$. It suffices to show that the image of this inclusion is dense.
 Fix $x\in \corona{\mathfrak B}(\mathsf U) \cap \corona{\mathfrak B}(\mathsf V)$. By Remark \ref{r:canoniciso} we may lift $x$ to an element $y\in \multialg{\mathfrak B}(\mathsf U)$.
 Let $(z_\alpha)$ be an approximate identity of $\multialg{\mathfrak B}(\mathsf V)$. Then $yz_\alpha$ is in $\multialg{\mathfrak B}(\mathsf U \cap \mathsf V)$ and thus $\pi(yz_\alpha)$ is in $\corona{\mathfrak B}(\mathsf U \cap \mathsf V)$ and converges to $x$.
 
  Ad $(c)$: Clearly $\multialg{\mathfrak B}(\emptyset) = 0$ and $\corona{\mathfrak B}(\emptyset) = 0$.
  
  We start by showing that $\multialg{\mathfrak B}$ is finitely upper semicontinuous and use this to show that $\corona{\mathfrak B}$ is also finitely upper semicontinuous. Let $\mathsf U, \mathsf V\in \mathbb O(\mathsf X)$.
 Clearly $\multialg{\mathfrak B}(\mathsf U) + \multialg{\mathfrak B}(\mathsf V) \subset \multialg{\mathfrak B}(\mathsf U \cup \mathsf V)$. 
 Let $m \in \multialg{\mathfrak B}(\mathsf U \cup \mathsf V)$ be a positive contraction and let $(b_n)$ be a countable approximate identity in $\mathfrak B$ such that $b_{n+1}b_n = b_n b_{n+1} = b_n$. 
  By passing to a subsequence we may assume that
 \[
  \| (1-b_n)mb_k\| < 2^{-(n+k)},
 \]
 for all $k<n$. Let $d_n = b_n - b_{n-1}$ where we define $b_0 = 0$.
 
 We claim that by using the norm estimate above one gets that $m$ is in the two-sided, closed ideal
 \[
  \overline{\multialg{\mathfrak B}(\sum_{n=1}^\infty d_n m d_n)\multialg{\mathfrak B}} + \mathfrak B(\mathsf U \cup \mathsf V).  
 \]
 To see this note first that $\sum_{n=1}^\infty d_n = 1$ strictly. Thus
 \begin{eqnarray*}
  m &=& \sum_{k,l=1}^\infty d_k m d_l \\
  &=& \sum_{n=1}^\infty d_n m d_n + \sum_{n=1}^\infty (d_{n+1}md_n + d_n m d_{n+1}) + \sum_{l=3}^\infty \sum_{k=1}^{l-2} (d_k m d_l + d_l m d_k).
 \end{eqnarray*}
 The latter of these three sums is in $\mathfrak B(\mathsf U \cup \mathsf V)$ since $\| d_l m d_k \| = \| b_l (1-b_{l-1}) m d_k\| < 2^{-(k+l)}$. 
 Thus it suffices to show that the second sum is in the two-sided, closed ideal generated by the first sum plus something from
 $\mathfrak B(\mathsf U \cup \mathsf V)$. Using that $d_n d_l = 0$ for $n < l-1$, we get that
 \begin{eqnarray*}
   && \left( \sum_{n=1}^\infty d_{n+1}md_n \right)^\ast \left( \sum_{n=1}^\infty d_{n+1}md_n \right) \\
   &=& \sum_{n=1}^\infty d_n m d_{n+1}^2 m d_n + \sum_{n=1}^\infty (d_n m d_{n+1}d_{n+2} m d_{n+1} + d_{n+1} m d_{n+2}d_{n+1} m d_{n})
 \end{eqnarray*}
 Using that $d_{n+1}d_{n+2} = d_{n+2}d_{n+1}$, the same norm consideration as above implies that the second of these sums is in $\mathfrak B(\mathsf U \cup \mathsf V)$. Finally, noting that
 \[
  \sum_{n=1}^\infty d_n m d_{n+1}^2 m d_n \leq \sum_{n=1}^\infty d_n m^2 d_n \leq \sum_{n=1}^\infty d_n m d_n,
 \]
 the claim follows.
 
  By the claim, it suffices to show that $\sum_{n=1}^\infty d_n m d_n$ is in $\multialg{\mathfrak B}(\mathsf U) + \multialg{\mathfrak B}(\mathsf V)$. 
 For each $n$ we may find positive contractions $x_n \in \mathfrak B(\mathsf U)$, $y_n \in \mathfrak B(\mathsf V)$ such that
 $x_n + y_n = d_n^{1/2} m d_n^{1/2}$. Then
 \[
  \sum_{n=1}^\infty d_n m d_n = \sum_{n=1}^\infty d_n^{1/2}x_n d_n^{1/2} + \sum_{n=1}^\infty d_n^{1/2} y_n d_n^{1/2} \in \multialg{\mathfrak B}(\mathsf U) + \multialg{\mathfrak B}(\mathsf V).
 \]
 Hence $m \in \multialg{\mathfrak B}(\mathsf U) + \multialg{\mathfrak B}(\mathsf V)$, so $\multialg{\mathfrak B}$ is finitely upper semicontinuous.
 
 Now for the corona algebra. Clearly $\corona{\mathfrak B}(\mathsf U) + \corona{\mathfrak B}(\mathsf V) \subset \corona{\mathfrak B}(\mathsf U \cup \mathsf V)$. 
 By Remark \ref{r:canoniciso} any element in $\corona{\mathfrak B}(\mathsf U \cup \mathsf V)$ lifts to one in $\multialg{\mathfrak B}(\mathsf U \cup \mathsf V)$,
 which is $\multialg{\mathfrak B}(\mathsf U) + \multialg{\mathfrak B}(\mathsf V)$ by what we have already proven.
 Hence it follows that $\corona{\mathfrak B}$ is upper semicontinuous.
\end{proof}

\begin{example}
 If $\mathfrak B$ is a lower semicontinuous $\mathsf X$-$C^\ast$-algebra, then $\corona{\mathfrak B}$ is not necessarily lower semicontinuous. As an example, take $\mathfrak B = C_0((0,1])$ as a tight $(0,1]$-$C^\ast$-algebra.
 Recall that $\multialg{\mathfrak B}$ may naturally be identified with $C_b((0,1])$, the $C^\ast$-algebra of bounded continuous functions on $(0,1]$.
 Then $\corona{\mathfrak B}((0,1/n)) = \corona{\mathfrak B}$ for all $n$. However, since $\bigcap_{n\in \mathbb N} (0,1/n)$ is empty, if $\corona{\mathfrak B}$ was lower semicontinuous then $\bigcap \corona{\mathfrak B}((0,1/n))$ should be $0$, which it is not.
 
 Similarly, if $\mathfrak B$ is a monotone upper semicontinuous $\mathsf X$-$C^\ast$-algebra, then $\multialg{\mathfrak B}$ and $\corona{\mathfrak B}$ are not necessarily monotone upper semicontinouous. To see this let $\mathfrak B$ be as above.
 Then $\multialg{\mathfrak B}((1/n,1]) = \mathfrak B((1/n,1])$. Hence $\overline{\bigcup_{n=1}^\infty \multialg{\mathfrak B}((1/n,1])} = \mathfrak B \neq \multialg{\mathfrak B}$, even though $\bigcup (1/n,1] = (0,1]$. This also works for the corona algebra,
 since $\corona{\mathfrak B}((1/n,1]) = 0$, but $\corona{\mathfrak B}((0,1]) = \corona{\mathfrak B}$.
\end{example}

\begin{definition}
An extension $0 \to \mathfrak B \to \mathfrak E \to \mathfrak A \to 0$ of $C^\ast$-algebras, for which $\mathfrak A$, $\mathfrak B$ and $\mathfrak E$ are $\mathsf X$-$C^\ast$-algebras is called \emph{$\mathsf X$-equivariant}
if the $\ast$-homomorphisms are $\mathsf X$-equivariant and the induced sequence $0 \to \mathfrak B(\mathsf U) \to \mathfrak E(\mathsf U) \to \mathfrak A(\mathsf U) \to 0$ is exact for each $\mathsf U\in \mathbb O(\mathsf X)$.

When such an extension is $\mathsf X$-equivariant, we will say that it is an \emph{extension of $\mathsf X$-$C^\ast$-algebras}.
\end{definition}

An example of this was given in Remark \ref{r:canoniciso}. Here we saw, that if $\mathfrak B$ is an $\mathsf X$-$C^\ast$-algebra, and $\multialg{\mathfrak B}$ and $\corona{\mathfrak B}$ are given the induced $\mathsf X$-$C^\ast$-algebra structures,
then the extension $0 \to \mathfrak B \to \multialg{\mathfrak B} \to \corona{\mathfrak B} \to 0$ is $\mathsf X$-equivariant.

If a Busby map $\tau \colon \mathfrak A \to \corona{\mathfrak B}$ is $\mathsf X$-equivariant, then the pull-back $\mathfrak A \oplus_{\corona{\mathfrak B}} \multialg{\mathfrak B} = \{ (a,m) \in \mathfrak A \oplus \multialg{\mathfrak B}: \tau(a) = \pi(m) \}$ 
gets an induced action of $\mathsf X$ by
\[
 (\mathfrak A \oplus_{\corona{\mathfrak B}} \multialg{\mathfrak B})(\mathsf U) = (\mathfrak A \oplus_{\corona{\mathfrak B}} \multialg{\mathfrak B}) \cap (\mathfrak A(\mathsf U) \oplus \multialg{\mathfrak B}(\mathsf U))
\]
for $\mathsf U\in \mathbb O(\mathsf X)$.

\begin{proposition}\label{p:basicXext}
Let $0 \to \mathfrak B \to \mathfrak E \xrightarrow{p} \mathfrak A \to 0$ be an extension of $C^\ast$-algebras such that $\mathfrak A$, $\mathfrak B$ and $\mathfrak E$ are $\mathsf X$-$C^\ast$-algebras and the $\ast$-homomorphisms are $\mathsf X$-equivariant.
The following are equivalent.
\begin{itemize}
\item[$(i)$] The extension is $\mathsf X$-equivariant,
\item[$(ii)$] $\mathfrak B(\mathsf U) = \mathfrak B \cap \mathfrak E(\mathsf U)$ and $\mathfrak A(\mathsf U) = p(\mathfrak E(\mathsf U))$ for all $\mathsf U\in \mathbb O(\mathsf X)$,
\item[$(iii)$] The Busby map $\tau$ is $\mathsf X$-equivariant, and the induced isomorphism 
\[
 \mathfrak E \xrightarrow{\cong} \mathfrak A \oplus_{\corona{\mathfrak B}} \multialg{\mathfrak B}
\]
 is an isomorphism of $\mathsf X$-$C^\ast$-algebras.
\end{itemize}
\end{proposition}
\begin{proof}
$(i) \Rightarrow (ii)$: If the extension is $\mathsf X$-equivariant then clearly the map $p$ induces an isomorphism $(\mathfrak E(\mathsf U) + \mathfrak B)/\mathfrak B \xrightarrow{\cong} \mathfrak A(\mathsf U)$. 
The inclusion $\mathfrak E(\mathsf U) \hookrightarrow \mathfrak E(\mathsf U) + \mathfrak B$ induces an isomorphism such that $p|_{\mathfrak E(\mathsf U)} \colon \mathfrak E(\mathsf U) \to \mathfrak A(\mathsf U)$ is the composition
\[
 \mathfrak E(\mathsf U) \twoheadrightarrow \frac{\mathfrak E(\mathsf U)}{\mathfrak E(\mathsf U) \cap \mathfrak B} \xrightarrow{\cong} \frac{\mathfrak E(\mathsf U) + \mathfrak B}{\mathfrak B} \xrightarrow{\cong} \mathfrak A(\mathsf U).
\]
It follows from exactness that $\mathfrak B(\mathsf U) = \ker p|_{\mathfrak E (\mathsf U)} = \mathfrak E(\mathsf U) \cap \mathfrak B$.

$(ii) \Rightarrow (i)$: Since $p \colon \mathfrak E(\mathsf U) \to \mathfrak A(\mathsf U)$ is surjective by assumption, it suffices to show that the kernel is $\mathfrak B(\mathsf U)$. 
As above, the kernel is $\mathfrak E(\mathsf U) \cap \mathfrak B = \mathfrak B(\mathsf U)$.

$(ii) \Rightarrow (iii)$: Let $\sigma \colon \mathfrak E \to \multialg{\mathfrak B}$ be the canonical $\ast$-homomorphism.
Note that 
\[
\mathfrak E(\mathsf U) \cdot \mathfrak B = \mathfrak B(\mathsf U).
\] 
Thus for any $x \in \mathfrak E(\mathsf U)$, we have $x\mathfrak B \subset \mathfrak B(\mathsf U)$,
which implies that $\sigma(x) \in \multialg{\mathfrak B, \mathfrak B(\mathsf U)} = \multialg{\mathfrak B}(\mathsf U)$. Hence $\sigma$ is $\mathsf X$-equivariant.
Recall that $\tau$ is constructed as follows: let $a\in \mathfrak A$, lift it to any element $x\in \mathfrak E$, and then let $\tau(a) = \pi \circ \sigma(a)$. 
If $a \in \mathfrak A(\mathsf U)$, then by $(ii)$ we may lift $a$ to an element $x \in \mathfrak E(\mathsf U)$. Since $\sigma$ and $\pi$ are both $\mathsf X$-equivariant, $\tau(a) = \pi \circ \sigma(x) \in \corona{\mathfrak B}(\mathsf U)$.
Thus $\tau$ is $\mathsf X$-equivariant.
Since $\mathfrak E \xrightarrow{\cong} \mathfrak A \oplus_{\corona{\mathfrak B}} \multialg{\mathfrak B}$ is induced by the $\mathsf X$-equivariant maps $p$ and $\sigma$, it follows that this isomorphism is $\mathsf X$-equivariant.

It remains to show that the inverse is also $\mathsf X$-equivariant. Suppose $(a,m) \in (\mathfrak A \oplus_{\corona{\mathfrak B}} \multialg{\mathfrak B})(\mathsf U)$, and let $x\in \mathfrak E$ be such that $p(x) = a$ and $\sigma(x) = m$.
We must show that $x \in \mathfrak E(\mathsf U)$.
Lift $a$ to an element $y\in \mathfrak E(\mathsf U)$. Then $b:= \sigma(x-y) = m - \sigma(y) \in \mathfrak B \cap \multialg{\mathfrak B, \mathfrak B(\mathsf U)} = \mathfrak B(\mathsf U)$.
Since $\mathfrak B(\mathsf U) \subset \mathfrak E(\mathsf U)$ we have that $b+y \in \mathfrak E(\mathsf U)$, and satisfies $p(b+y) = a$ and $\sigma(b+ y) = m$.
Thus $x=b+y \in \mathfrak E(\mathsf U)$.

$(iii) \Rightarrow (ii)$: Suppose that $\tau$ is $\mathsf X$-equivariant, and that when identifying $\mathfrak E$ with the pull-back $\mathfrak A \oplus_{\corona{\mathfrak B}} \multialg{\mathfrak B}$, then
\[
\mathfrak E(\mathsf U) = (\mathfrak A \oplus_{\corona{\mathfrak B}} \multialg{\mathfrak B}) \cap (\mathfrak A(\mathsf U) \oplus \multialg{\mathfrak B}(\mathsf U))
\]
for all $\mathsf U\in \mathbb O(\mathsf X)$. We have $\mathfrak B(\mathsf U) = \mathfrak B \cap \multialg{\mathfrak B}(\mathsf U) = \mathfrak B \cap \mathfrak E(\mathsf U)$. 
Clearly $p(\mathfrak E(\mathsf U)) \subseteq \mathfrak A(\mathsf U)$ since $p$ is $\mathsf X$-equivariant by assumption.
If $a\in \mathfrak A(\mathsf U)$ we may lift $\tau(a) \in \corona{\mathfrak B}(\mathsf U)$ to an element $m\in \multialg{\mathfrak B}(\mathsf U)$ by Remark \ref{r:canoniciso}.
Hence $(a,m)\in \mathfrak E(\mathsf U)$ which implies that $p(\mathfrak E(\mathsf U)) = \mathfrak A(\mathsf U)$.
\end{proof}

Note that it was proven in $(ii) \Rightarrow (iii)$ above, that if $0 \to \mathfrak B \to \mathfrak E \to \mathfrak A \to 0$ is an extension of $\mathsf X$-$C^\ast$-algebras, 
then the induced $\ast$-homomorphism $\sigma \colon \mathfrak E \to \multialg{\mathfrak B}$ is $\mathsf X$-equivariant.

As an analogue of full extensions of $C^\ast$-algebras, we make a similar definition for $\mathsf X$-equivariant extensions when the action on the quotient algebra is lower semicontinuous.
These will play an important part later, especially in Theorem \ref{t:cfpequiv} and its corollaries.

\begin{definition}
Let $0 \to \mathfrak B \to \mathfrak E \to \mathfrak A \to 0$ be an extension of $\mathsf X$-$C^\ast$-algebras, such that $\mathfrak A$ is lower semicontinuous.
We say that the extension is \emph{full} (or $\mathsf X$-full) if the Busby map is $\mathsf X$-full.
\end{definition}


\subsection{An approximation property for residually $\mathsf X$-nuclear maps}\label{s:approx}

It follows from Theorem \ref{t:domgen}, that in order to show that a closed operator convex cone $\mathscr C$ satisfies the purely large problem, it suffices to find a generating subset $\mathscr S$ consisting of suitably well-behaved maps.
E.g.~in the classical case where $\mathscr C = CP_\nuc(\mathfrak A, \mathfrak B)$, one uses that these maps have the classical approximation property of nuclear maps, i.e.~that the maps approximately factor through matrix algebras.

In order to solve our purely large problem for $\mathsf X$-$C^\ast$-algebras with $\mathsf X$ finite, we will study residually $\mathsf X$-nuclear maps, which we do for not necessarily finite spaces $\mathsf X$. 
In fact, we prove an approximation property for all such maps, which can also be used to prove other nice results, such as a Choi--Effros type lifting theorem for residually $\mathsf X$-nuclear maps, when $\mathsf X$ is finite.

\begin{notation}
Whenever $\Psi \colon \mathfrak A \to \mathfrak B$ is an $\mathsf X$-equivariant linear map and $\mathsf U\in \mathbb O(\mathsf X)$ we let $[\Psi]_\mathsf{U} \colon \mathfrak A/\mathfrak A(\mathsf U) \to \mathfrak B/\mathfrak B(\mathsf U)$ be the induced map.
\end{notation}

\begin{definition}
Let $\mathfrak A$ and $\mathfrak B$ be $\mathsf X$-$C^\ast$-algebras, and let $\Psi \colon \mathfrak A \to \mathfrak B$ be an $\mathsf X$-equivariant nuclear map. We say that $\Psi$ is \emph{residually $\mathsf X$-nuclear} 
if $[\Psi]_\mathsf{U}$ is nuclear for any $\mathsf U \in \mathbb O(\mathsf X)$.

We let $CP_\rnuc(\mathsf X;\mathfrak A,\mathfrak B)$ denote the set of residually $\mathsf X$-nuclear maps.
\end{definition}

\begin{remark}\label{r:rnuccone}
 It is straight forward to verify that $CP_\rnuc(\mathsf X; \mathfrak A, \mathfrak B)$ is a closed operator convex cone.
\end{remark}

\begin{remark}\label{r:exactresnuc}
 Note that if $\mathfrak A$ or $\mathfrak B$ is nuclear, then any $\mathsf X$-equivariant c.p.~map is residually $\mathsf X$-nuclear.
 Moreover, if $\mathfrak A$ is exact (or even
 locally reflexive), then an $\mathsf X$-equivariant c.p.~map $\mathfrak A \to \mathfrak B$ is residually $\mathsf X$-nuclear exactly when it is (non-equivariantly) nuclear, i.e.~
 \[
  CP_\rnuc(\mathsf X;\mathfrak A,\mathfrak B) = CP(\mathsf X;\mathfrak A,\mathfrak B) \cap CP_\nuc(\mathfrak A,\mathfrak B).
 \] 
 This follows from \cite[Proposition 3.2]{Dadarlat-qdmorphisms}.
\end{remark}

\begin{notation}
Let $\mathsf X$ be a topological space and $\mathsf Y$ be a finite space. Then any surjective continuous map $\mathsf f \colon \mathsf X \to \mathsf Y$ is called a \emph{finite filtration} of $\mathsf X$. 
We will often just say that $\mathsf Y$ is a finite filtration of $\mathsf X$.

Given a finite filtration $\mathsf Y$ of $\mathsf X$, then any $\mathsf X$-$C^\ast$-algebra $\mathfrak A$ gets an induced $\mathsf Y$-$C^\ast$-algebra structure, 
by letting $\mathfrak A(\mathsf U) = \mathfrak A(\mathsf f^{-1}(\mathsf U))$ for any $\mathsf X$-$C^\ast$-algebra $\mathfrak A$ and any $\mathsf U\in \mathbb O(\mathsf Y)$.
\end{notation}

When given a topological space $\mathsf X$, a $C^\ast$-algebra $\mathfrak A$, and a non-empty subset $\mathsf A \subset \mathsf X$, we may give $\mathfrak A$ an induced $\mathsf X$-$C^\ast$-algebra structure by letting
$\mathfrak A(\mathsf U)=\mathfrak A$ if $\mathsf A \subseteq \mathsf U$ and $0$ if $\mathsf A \not \subseteq \mathsf U$. We will let $i_\mathsf{A}(\mathfrak A)$ denote $\mathfrak A$ with this $\mathsf X$-$C^\ast$-algebra structure.

\begin{lemma}\label{l:finitedim}
Let $\mathsf Y$ be a finite $T_0$ space and $\mathfrak D$ is a finite dimensional, continuous $\mathsf Y$-$C^\ast$-algebra. Then $\mathfrak D \cong \bigoplus_{y\in Y} i_{\{ y\}} (\mathfrak D_y)$ as $\mathsf Y$-$C^\ast$-algebras 
for some (necessarily finite dimensional) $C^\ast$-algebras $\mathfrak D_y$. 
\end{lemma}
\begin{proof}
Let $\mathsf U^y$ denote the smallest open subset of $\mathsf Y$ which contains $y$. Note that $\mathsf U^y \setminus \{y\}$ is also open.
Let $\mathfrak D_y = \mathfrak D(\mathsf U^y)/\mathfrak D(\mathsf U^y \setminus\{ y\})$. Since $\mathfrak D$ is finite dimensional, any two-sided, closed ideal is a direct summand. Thus there is a canonical induced $\mathsf Y$-equivariant injective 
$\ast$-homomorphism $\iota_y \colon i_{\{y\}}(\mathfrak D_y) \to \mathfrak D$. 
Since $\mathfrak D$ is lower semicontinuous and $\mathfrak D(\emptyset ) = 0$, it follows that all $\iota_y$ have orthogonal images. 
Thus these induce a $\mathsf Y$-equivariant injective $\ast$-homomorphism $\iota \colon \bigoplus_{y\in Y} i_{\{ y\}} (\mathfrak D_y) \to \mathfrak D$. 
Since $\mathfrak D$ is upper semicontinuous, $\mathfrak D(\mathsf U) = \sum_{y\in \mathsf U} \mathfrak D(\mathsf U^y)$ which is canonically isomorphic to $\bigoplus_{y\in \mathsf U}\mathfrak D_y$ for any $\mathsf U\in \mathsf Y$. 
Thus it easily follows that $\iota$ induces an isomorphism of $\mathsf Y$-$C^\ast$-algebras.
\end{proof}

\begin{lemma}\label{l:liftfinite}
Let $\mathsf Y$ be a finite space, $\mathfrak D$ be a finite dimensional $\mathsf Y$-$C^\ast$-algebra, and $0\to \mathfrak J \to \mathfrak B \xrightarrow{p} \mathfrak B/\mathfrak J \to 0$ be an extension of $\mathsf Y$-$C^\ast$-algebras.
Suppose that $\rho \colon \mathfrak D \to \mathfrak B/\mathfrak J$ is a $\mathsf Y$-equivariant contractive c.p.~map. 
Then there exists a $\mathsf Y$-equivariant contractive c.p.~map $\tilde \rho \colon \mathfrak D \to \mathfrak B$ such that $p \circ \tilde \rho = \rho$.
\end{lemma}
\begin{proof}
By \cite[Sections 2.5 and 2.9]{MeyerNest-bootstrap} we may assume (by replacing $\mathsf Y$ with a possibly bigger, but still finite, space) that $\mathsf Y$ is $T_0$ and that all given $\mathsf Y$-$C^\ast$-algebras are continuous 
(i.e.~are $C^\ast$-algebras over $\mathsf Y$ since $\mathsf Y$ is finite). 
Write $\mathfrak D = \bigoplus_{y\in \mathsf Y} i_{\{y\}}(\mathfrak D_y)$ with $\mathfrak D_y$ finite dimensional $C^\ast$-algebras. 
Let $\mathsf U^y$ be the smallest open set containing $y$. Since $\rho$ is $\mathsf Y$-equivariant, the restriction $\rho_y$ to $\mathfrak D_y$ factors through $(\mathfrak B/\mathfrak J)(\mathsf U^y) \cong \mathfrak B(\mathsf U^y)/\mathfrak J(\mathsf U^y)$. 
Hence this contractive c.p.~map lifts to a contractive c.p.~map $\widetilde \rho_y \colon \mathfrak D_y \to \mathfrak B(\mathsf U^y)$. Let $\hat \rho \colon \mathfrak D \to \mathfrak B$ be given by
\[
\hat \rho (\oplus_{y\in \mathsf Y}d_y) = \sum_{y\in \mathsf Y} \widetilde \rho_y(d_y),\qquad \text{ for }d_y \in \mathfrak D_y.
\]
This map is clearly $\mathsf Y$-equivariant and completely positive but not necessarily contractive. Let $v=f(\hat \rho(1_\mathfrak{D}))$, defined by functional calculus in the minimal unitisation $\widetilde{\mathfrak B}$, where $f(t) = \max\{1,t\}$. 
Clearly $p(v^{-1/2} b v^{-1/2}) = p(b)$ for all $b\in \mathfrak B$, and thus the $\mathsf Y$-equivariant contractive c.p.~map $\widetilde \rho \colon \mathfrak D \to \mathfrak B$
given by $\widetilde \rho(d) = v^{-1/2} \hat \rho(d) v^{-1/2}$ is also a lift of $\rho$.
\end{proof}

The above trick of using the results of \cite{MeyerNest-bootstrap} to assume that our $\mathsf X$-$C^\ast$-algebras are continuous $\mathsf Y$-$C^\ast$-algebras, for some larger space $\mathsf Y$, will in general not work, 
since the functors taking $\mathsf X$-$C^\ast$-algebras to continuous $\mathsf Y$-$C^\ast$-algebras, will in general \emph{not} take residually $\mathsf X$-nuclear maps to residually $\mathsf Y$-nuclear maps.

Using the notion of finite filtrations of spaces, we get the following approximately $\mathsf X$-equivariant approximation property of residually $\mathsf X$-nuclear maps.

\begin{theorem}\label{t:nucapprox}
Let $\mathfrak A$ and $\mathfrak B$ be $\mathsf X$-$C^\ast$-algebras, with $\mathfrak B$ finitely continuous, and let $\phi \colon \mathfrak A \to \mathfrak B$ be a c.p.~map. Then $\phi$ is residually $\mathsf X$-nuclear 
if and only if the following holds: for any finite subset $F\subseteq \mathfrak A$, any $\epsilon >0$, and any finite filtration $\mathsf Y$ of $\mathsf X$, there are $\mathsf Y$-equivariant c.p.~maps
\[
\psi\colon \mathfrak A \to \mathfrak D, \qquad \rho \colon \mathfrak D \to \mathfrak B,
\]
such that
\[
\| \phi(a) - \rho(\psi(a)) \| < \epsilon, \qquad \text{for all } a \in F,
\]
where $\mathfrak D$ is a finite dimensional $\mathsf Y$-$C^\ast$-algebra, $\| \psi \| \leq 1$ and $\| \rho\| \leq \| \phi\|$.

Moreover, we may choose $\mathfrak D$ to be continuous.
\end{theorem}
\begin{proof}
For the ``if'' we will start by showing that $\phi$ is $\mathsf X$-equivariant. Let $\mathsf U \in \mathbb O(\mathsf X)$ and $a\in \mathfrak A(\mathsf U)$ so that we should show that $\phi(a) \in \mathfrak B(\mathsf U)$. If $\mathsf U = \mathsf X$, then $\phi(a) \in \mathfrak B = \mathfrak B(\mathsf X)$, since $\mathfrak B$ is finitely continuous. 
Thus we assume that $\mathsf U \neq \mathsf X$. Suppose $\mathsf U \neq \emptyset$.
Let $\mathsf Y = \{ 1,2\}$ be the two-point space with open subsets $\emptyset, \{2\},\mathsf Y$. The map $\mathsf f \colon \mathsf X \to \mathsf Y$ given by $\mathsf f(x) = 2$ if $x\in \mathsf U$ and $\mathsf f(x) = 1$ otherwise, is a continuous surjection.
Hence $\mathsf f \colon \mathsf X \to \mathsf Y$ is a finite filtration. Then choosing any element $a\in \mathfrak A(\mathsf U)$ and any $\epsilon >0$ we can find maps $\rho$ and $\psi$ as above which are $\mathsf Y$-equivariant. 
Hence $\psi(\rho(a)) \in \mathfrak B(\{2\}) = \mathfrak B(\mathsf U)$. Since $\mathfrak B(\mathsf U)$ is closed we get that $\phi(a)\in \mathfrak B(\mathsf U)$. 
If $\mathsf U = \emptyset$ let $\mathsf Y =\{ 1\}$ and $\mathsf f \colon \mathsf X \to \mathsf Y$ be the unique map. The same argument as before shows that $\phi(a) \in \mathfrak B(\emptyset) = 0$.

Clearly $\phi$ is nuclear. In order to prove that $\phi$ is residually $\mathsf X$-nuclear we should show that $[\phi]_\mathsf{U}$ is nuclear. Suppose that $\mathsf U \neq \emptyset$ and let $\mathsf f\colon \mathsf X \to \mathsf Y$ be defined as above. 
Let $F' \subseteq \mathfrak A/ \mathfrak A(\mathsf U)$ be a finite subset and 
$\epsilon >0$. Lift each element $a'$ of $F'$ to an element $a$ of $\mathfrak A$ and obtain a finite subset $F$ of $\mathfrak A$. By the approximation property there are $\mathsf Y$-equivariant $\rho$ and $\psi$ as described above. We get that
\[
\| [\phi]_\mathsf{U}(a') - [\psi]_{\{2\}}( [\rho]_{\{ 2\}} (a')) \| \leq \| \phi(a) - \psi(\rho(a)) \| < \epsilon, \quad \text{for all } a' \in F'.
\]
Hence $[\phi]_\mathsf{U}$ is nuclear. If $\mathsf U = \emptyset$ then a similar argument implies that $[\phi]_\mathsf{U}$ is nuclear, and thus $\phi$ is residually $\mathsf X$-nuclear.

Now for the ``only if''. We may assume, without loss of generality, that $\phi$ is contractive. We will prove the result by induction on the number of elements of $\mathbb O(\mathsf Y)$. 
If $|\mathbb O(\mathsf Y)|=2$ then $\mathbb O(\mathsf Y) = \{ \emptyset,\mathsf Y\}$. Since $\mathfrak B$ is finitely continuous, $\mathfrak B(\emptyset) = 0$. Thus $\phi(\mathfrak A(\emptyset)) = 0$. Since the map $[\phi]_\emptyset$ is nuclear, we may
approximate it by maps of the form $\rho \circ \psi$ where $\psi \colon \mathfrak A/\mathfrak A(\emptyset) \to M_n$ and $\rho \colon M_n \to \mathfrak B$ are contractive by Lemma \ref{l:classicalnuc}. Consider $M_n$ as a continuous $\mathsf Y$-$C^\ast$-algebra
and let $\overline \psi$ be the composition $\mathfrak A \to \mathfrak A/\mathfrak A(\emptyset) \xrightarrow{\psi} M_n$. Clearly $\overline \psi$ and $\rho$ are $\mathsf Y$-equivariant contractive c.p.~maps, and $\rho \circ \overline \psi$ approximates $\phi$.
This shows the case $|\mathbb O(\mathsf Y)| = 2$.

Suppose that there is an approximation, as described in the theorem, for any residually $\mathsf Y$-nuclear map between $\mathsf Y$-$C^\ast$-algebras $\mathfrak A$ and $\mathfrak B$, with $\mathfrak B$ continuous,
for any space $\mathsf Y$ with $|\mathbb O(\mathsf Y)|<n$, where the finite dimensional $\mathfrak D$ is a $\mathsf Y$-$C^\ast$-algebra. Let $F,\epsilon$ and $\mathsf Y$ be given with $|\mathbb O(\mathsf Y)|=n$. 
Since $\mathbb O(\mathsf Y) \cong \mathbb O(\hat{\mathsf Y})$ for a $T_0$-space $\hat{\mathsf Y}$ by Remark \ref{r:soberspace}, we may assume that $\mathsf Y$ is a $T_0$ space. 

Let $\mathsf U\in \mathbb O(\mathsf Y)$ be a non-empty, minimal open set (since $\mathsf Y$ is a finite $T_0$ space this means that $\mathsf U = \{ y \}$ for some $y\in \mathsf Y$).
Equip $\mathfrak C/\mathfrak C(\mathsf U)$ with a $(\mathsf Y\setminus \mathsf U)$-$C^\ast$-algebra structure by defining 
$(\mathfrak C/\mathfrak C(\mathsf U))(\mathsf W) := \mathfrak C(\mathsf W \cup \mathsf U)/\mathfrak C(\mathsf U)$ for $\mathsf W \in \mathbb O(\mathsf Y\setminus \mathsf U)$, and $\mathfrak C \in \{ \mathfrak A,\mathfrak B\}$. It is easily seen that 
$\mathfrak B/\mathfrak B(\mathsf U)$ is a continuous $(\mathsf Y\setminus \mathsf U)$-$C^\ast$-algebra, and that $[\phi]_\mathsf{U}$ is $(\mathsf Y\setminus \mathsf U)$-equivariant. Moreover, for any $\mathsf W \in \mathbb O(\mathsf Y\setminus \mathsf U)$ we
may canonically identify $[[\phi]_\mathsf{U}]_\mathsf{W}$ and $[\phi]_{\mathsf W \cup \mathsf U}$, and thus $[\phi]_\mathsf{U}$ is residually $(\mathsf Y\setminus \mathsf U)$-nuclear.

Hence by assumption there exists a finite dimensional, continuous $(\mathsf Y \setminus \mathsf U)$-$C^\ast$-algebra $\mathfrak D_1$ and $(\mathsf Y\setminus \mathsf U)$-equivariant contractive c.p.~maps
\[
\psi_1 \colon \mathfrak A/\mathfrak A(\mathsf U) \to \mathfrak D_1, \qquad \rho_1 \colon \mathfrak D_1 \to \mathfrak B/\mathfrak B(\mathsf U)
\]
such that
\[
\| [\phi]_{ \mathsf U } (p_\mathfrak{A}(a)) - \rho_1 ( \psi_1 (p_\mathfrak{A} (a))) \| < \epsilon/6
\]
for all $a \in F$, where $p_\mathfrak{A} \colon \mathfrak A \to \mathfrak A/\mathfrak A(\mathsf U)$ is the natural quotient map. Similarly, let $p_\mathfrak{B} \colon \mathfrak B \to \mathfrak B /\mathfrak B(\mathsf U)$ be the natural quotient map. 

Now we consider $\mathfrak C / \mathfrak C(\mathsf U)$ as $\mathsf Y$-$C^\ast$-algebras by $(\mathfrak C / \mathfrak C(\mathsf U))(\mathsf V) = \mathfrak C(\mathsf V \cup \mathsf U) /\mathfrak C(\mathsf U)$ for $\mathsf V \in \mathbb O(\mathsf Y)$,
and $\mathfrak C \in \{ \mathfrak A,\mathfrak B\}$. Also, give $\mathfrak D_1$ a $\mathsf Y$-$C^\ast$-algebra structure, by $\mathfrak D_1(\mathsf V) := \mathfrak D_1(\mathsf V \setminus \mathsf U)$. This action is easily seen to be continuous, and clearly
$\rho_1$, $\psi_1$, $p_\mathfrak{A}$ and $p_\mathfrak{B}$ are $\mathsf Y$-equivariant.

Since $\mathfrak B$ is continuous, we get
\[
\mathfrak B(\mathsf V\cup \mathsf U) /\mathfrak B(\mathsf U) = (\mathfrak B(\mathsf V) + \mathfrak B(\mathsf U))/ \mathfrak B(\mathsf U) \cong \mathfrak B(\mathsf V)/(\mathfrak B(\mathsf U) \cap\mathfrak B(\mathsf V)) = 
\mathfrak B(\mathsf V) / \mathfrak B(\mathsf U \cap \mathsf V).
\]
Hence $0 \to \mathfrak B(\mathsf U) \to \mathfrak B \xrightarrow{p_\mathfrak{B}} \mathfrak B/\mathfrak B(\mathsf U) \to 0$ is a $\mathsf Y$-equivariant extension, where we equip $\mathfrak B(\mathsf U)$ with the $\mathsf Y$-$C^\ast$-algebra structure
$\mathfrak B(\mathsf U)(\mathsf V) = \mathfrak B(\mathsf U \cap \mathsf V)$. 
By Lemma \ref{l:liftfinite} we may lift $\rho_1$ to a $\mathsf Y$-equivariant contractive c.p.~map $\widetilde{\rho_1} \colon \mathfrak D_1 \to \mathfrak B$. Observe that
\[
\| p_\mathfrak{B}(\phi(a) - \widetilde{\rho_1} ( \psi_1(p_\mathfrak{A}(a))) ) \| = \| [\phi]_{ \mathsf U } (p_\mathfrak{A}(a)) - \rho_1 ( \psi_1 (p_\mathfrak{A} (a))) \| < \epsilon/6
\]
for all $a\in F$. We may find a positive contractions $b \in \ker p_{\mathfrak B} = \mathfrak B(\mathsf U)$, e.g. by picking $b$ as an element of a quasi-central approximate identity, such that
\begin{eqnarray*}
\| (1-b) (\phi(a) - \widetilde{\rho_1} ( \psi_1(p_\mathfrak{A}(a))) ) \| &<& \epsilon/6 \\
\| b \phi (a) - b^{1/2} \phi (a) b^{1/2} \| & < & \epsilon/6 \\
\| (1-b) \widetilde{\rho_1} ( \psi_1(p_\mathfrak{A}(a))) - (1-b)^{1/2} \widetilde{\rho_1} ( \psi_1(p_\mathfrak{A}(a))) (1-b)^{1/2} \| & < & \epsilon/6
\end{eqnarray*}
for all $a\in F$. Let $\widetilde \rho = (1-b)^{1/2} \widetilde{\rho_1}(-) (1-b)^{1/2}$ which is also a $\mathsf Y$-equivariant contractive c.p.~lift of $\rho_1$. By the above inequalities we have that
\[
\| \phi(a) - (b^{1/2} \phi(a) b^{1/2} + \widetilde \rho (\psi_1(p_\mathfrak{A}(a))) ) \| < \epsilon/2,
\]
for all $a\in F$. 
Note that the contractive c.p.~map $b^{1/2}\phi(-)b^{1/2}$ factors through $\mathfrak B(\mathsf U)$. Let $\mathsf V := \mathsf Y \setminus \overline{\mathsf U}$. Since $\mathfrak B(\mathsf V) \cap \mathfrak B(\mathsf U) = 0$ by continuity of $\mathfrak B$,
it follows that $b^{1/2}\phi(-)b^{1/2}$ also has a factor $[\phi]_\mathsf{V} \colon \mathfrak A/\mathfrak A(\mathsf V) \to \mathfrak B/\mathfrak B(\mathsf V)$ which is nuclear. Hence we may pick contractive c.p.~maps 
\[
\psi_2 \colon \mathfrak A \to \mathfrak D_2, \qquad \rho_2 \colon \mathfrak D_2 \to \mathfrak B(\mathsf U),
\]
with $\mathfrak D_2$ a finite dimensional $C^\ast$-algebra such that
\[
\| b^{1/2}\phi(a)b^{1/2} - \rho_2(\psi_2(a)) \| < \epsilon/2,
\]
for all $a\in F$, such that $\rho_2(1_{\mathfrak{D}_2}) \leq b$ and such that $\psi_2$ factors through $\mathfrak A/\mathfrak A(\mathsf V)$. 
Now, let $\mathfrak D := \mathfrak D_1 \oplus i_{\mathsf U}(\mathfrak D_2)$. Note that $i_\mathsf{U}(\mathfrak D_2)$ is a continuous $\mathsf Y$-$C^\ast$-algebra since $\mathsf U$ is a non-empty, minimal open set.
It follows that $\mathfrak D$ is a continuous $\mathsf Y$-$C^\ast$-algebra, since it is the direct sum of two continuous $\mathsf Y$-$C^\ast$-algebras. Define
\[
\psi := (\psi_1\circ p_\mathfrak{A}, \psi_2) \colon \mathfrak A \to \mathfrak D, \qquad \rho = \langle \widetilde \rho, \rho_2 \rangle \colon \mathfrak D \to \mathfrak B.
\]
By the above estimates it clearly follows that
\[
\| \phi(a) - \rho(\psi(a))\| < \epsilon
\]
for all $a\in F$, and thus it remains to show that $\psi$ and $\rho$ are $\mathsf Y$-equivariant contractive c.p.~maps. 
Clearly $\psi$ is a contractive c.p.~map since each $\psi_1\circ p_\mathfrak{A}$ and $\psi_2$ is a contractive c.p.~map. 
Moreover, $\rho$ is clearly completely positive and since
\[
\rho(1_\mathfrak{D}) = \widetilde{\rho}(1) + \rho_2(1) = (1-b)^{1/2}\widetilde{\rho_1}(1) (1-b)^{1/2} + \rho_2(1) \leq (1-b) + b =1, 
\]
it follows that $\rho$ is contractive. It suffices to show that each $\psi_1\circ p_\mathfrak{A},\psi_2,\widetilde \rho$ and $\rho_2$ is $\mathsf Y$-equivariant. 

Clearly $\psi_1 \circ p_\mathfrak{A}$ is $\mathsf Y$-equivariant since each $\psi_1$ and $p_\mathfrak A$ is $\mathsf Y$-equivariant. Moreover, $\widetilde \rho$ is $\mathsf Y$-equivariant as noted earlier in the proof. 
Since $\mathfrak B$ is continuous (and thus a $C^\ast$-algebra over $\mathsf Y$) it follows from \cite[Lemma 2.22]{MeyerNest-bootstrap} (which clearly also holds for c.p.~maps) that $\rho_2$ is $\mathsf Y$-equivariant 
exactly if it factors through $\mathfrak B(\mathsf U)$. 
This is satisfied by how we constructed the map. 
Since $\psi_2$ factors through $i_{\mathsf U}(\mathfrak D_2)$, it is $\mathsf Y$-equivariant if and only if $\psi_2(\mathfrak A(\mathsf W)) = 0$ for all 
$\mathsf W \in \mathbb O(\mathsf Y)$ for which $\mathsf U \not \subset \mathsf W$ if and only if $\psi_2(\mathfrak A(\mathsf V)) = 0$ since $\mathsf U$ is minimal. 
This is again satisfied by how we chose $\psi_2$, which finishes the proof.
\end{proof}

\begin{corollary}\label{c:nucapproxfinite}
Let $\mathsf X$ be a finite space, let $\mathfrak A$ and $\mathfrak B$ be $\mathsf X$-$C^\ast$-algebras, with $\mathfrak B$ continuous, and let $\phi \colon \mathfrak A \to \mathfrak B$ be a c.p.~map. 
Then $\phi$ is residually $\mathsf X$-nuclear if and only if the following holds: for any finite subset $F\subseteq \mathfrak A$, and any $\epsilon >0$ there are $\mathsf X$-equivariant c.p.~maps
\[
\psi\colon \mathfrak A \to \mathfrak D, \qquad \rho \colon \mathfrak D \to \mathfrak B,
\]
such that
\[
\| \phi(a) - \rho(\psi(a)) \| < \epsilon, \qquad \text{for all } a \in F,
\]
where $\mathfrak D$ is a finite dimensional, continuous $\mathsf X$-$C^\ast$-algebra, $\| \psi\| \leq 1$ and $\| \rho\| \leq \| \phi\|$.
\end{corollary}

Arveson's method of proving the Choi--Effros lifting theorem \cite{Arveson-extensions}, can also be used in the $\mathsf X$-equivariant case, if $\mathsf X$ is finite.

\begin{corollary}\label{c:choieffros}
Let $\mathsf X$ be a finite space, $\mathfrak A$ be a separable $\mathsf X$-$C^\ast$-algebra, $0\to \mathfrak J \to \mathfrak B \xrightarrow{p} \mathfrak B/\mathfrak J \to 0$ be an extension of $\mathsf X$-$C^\ast$-algebras with $\mathfrak B/\mathfrak J$ 
continuous.
Suppose that $\phi \colon \mathfrak A \to \mathfrak B/\mathfrak J$ is a contractive residually $\mathsf X$-nuclear map. Then there exists a contractive residually $\mathsf X$-nuclear map $\tilde \phi \colon \mathfrak A \to \mathfrak B$
such that $p \circ \tilde \phi = \phi$.
\end{corollary}
\begin{proof}
 Consider the cone 
 \[
  \mathscr C_1:= \{ p \circ \psi \mid \psi \in CP_\rnuc(\mathsf X; \mathfrak A, \mathfrak B), \| \psi \| \leq 1\},
 \]
 i.e.~the cone of all c.p.~maps $\mathfrak A \to \mathfrak B/\mathfrak J$ which lift to a contractive residually $\mathsf X$-nuclear map $\mathfrak A \to \mathfrak B$.
 By Lemma \ref{l:closedliftable} this is point-norm closed. 
 So it suffices to show that $\phi$ can be point-norm approximated by maps in $\mathscr C_1$.
 By Corollary \ref{c:nucapproxfinite} we may approximate $\phi$ with $\mathsf X$-equivariant c.p.~maps which factor $\mathsf X$-equivariantly through finite dimensional $\mathsf X$-$C^\ast$-algebras by contractive maps.
 It follows from Lemma \ref{l:liftfinite} that these maps are in $\mathscr C_1$, and thus $\phi$ is in $\mathscr C_1$.
\end{proof}

The above Choi--Effros type lifting theorem implies the following.

\begin{corollary}\label{c:liftingcor}
 The following hold.
\begin{itemize}
\item[(1)] Suppose that $\mathsf X$ is a finite space and that $0 \to \mathfrak B \to \mathfrak E \to \mathfrak A \to 0$ is an extension of $\mathsf X$-$C^\ast$-algebras. If $\mathfrak A$ is separable and nuclear, 
then the extension has an $\mathsf X$-equivariant contractive c.p.~splitting.
\item[(2)] Suppose that $\mathsf X$ is a finite space and that $\mathfrak A$ is a separable, nuclear $C^\ast$-algebra over $\mathsf X$. Then the functors $KK(\mathsf X;\mathfrak A,-) \cong E(\mathsf X;\mathfrak A,-)$ are naturally isomorphic.
\end{itemize}
\end{corollary}
\begin{proof}
 Ad (1): by applying the trick of Meyer-Nest (see proof of Lemma \ref{l:liftfinite}), we may assume that our $\mathsf X$-$C^\ast$-algebras are continuous $\mathsf Y$-$C^\ast$-algebras, for some larger (but still finite) space $\mathsf Y$. 
 Thus Corollary \ref{c:choieffros} implies the existence of a $\mathsf Y$-equivariant contractive c.p.~splitting, which is thus also an $\mathsf X$-equivariant contractive c.p.~splitting.
 Ad (2): this follows immediately from (1) above and \cite[Corollary 5.3]{DadarlatMeyer-E-theory}.
\end{proof}


\section{Purely large extensions over finite spaces}\label{s:finiteplp}

In this section we prove that for sufficiently nice $\mathsf X$-$C^\ast$-algebras $\mathfrak A$ and $\mathfrak B$ (e.g.~when both are separable and continuous) 
the closed operator convex cone $CP_\rnuc(\mathsf X; \mathfrak A, \mathfrak B)$ satisfies the (unital) purely large problem from Section \ref{s:plp}, under the additional assumption that $\mathsf X$ is finite.
In the case where every $\mathfrak B(\mathsf U)$ has the corona factorisation property we prove that purely largeness may be replaced by the weaker condition of the extensions being $\mathsf X$-full.

\subsection{Absorbing $\ast$-homomorphisms are $\mathsf X$-full}\label{s:fullmaps}

We start by showing that if $\mathsf X$ is finite and $\mathfrak B$ is suitably nice, then $CP_\rnuc(\mathsf X; \mathfrak A, \mathfrak B)$ is generated by pure states in a certain sense. We need the following definition.

\begin{definition}
 Let $\mathfrak B$ be a $\sigma$-unital $\mathsf X$-$C^\ast$-algebra. We say that $\mathfrak B$ is \emph{$\mathsf X$-$\sigma$-unital} if $\mathfrak B(\mathsf U)$ is $\sigma$-unital for every $\mathsf U \in \mathbb O(\mathsf X)$.
\end{definition}

In particular, any separable $\mathsf X$-$C^\ast$-algebra is $\mathsf X$-$\sigma$-unital.

Recall, that if $\mathsf X$ is a finite $T_0$ space, and $x\in \mathsf X$, then we let $\mathsf U^x$ denote the smallest open subset of $\mathsf X$ containing $x$.

\begin{lemma}\label{l:statesgen}
 Let $\mathsf X$ be finite $T_0$ space, and let $\mathfrak B$ be a stable, $\mathsf X$-$\sigma$-unital and continuous $\mathsf X$-$C^\ast$-algebra. Fix strictly positive elements $h_x\in \mathfrak B(\mathsf U^x)$ for $x\in \mathsf X$.
 Then
 \[
  \mathscr S := \{ \mathfrak A \xrightarrow{h_x\rho_x(-)} \mathfrak B \mid x\in \mathsf X, \rho_x \text{ is a pure state on } \mathfrak A, \rho_x(\mathfrak{A}(\mathsf X \setminus \overline{\{ x\}})) = 0 \}
 \]
 generates $CP_\rnuc(\mathsf X;\mathfrak A,\mathfrak B)$ as a closed operator convex cone.
 
 In particular, if $\mathfrak A$ is separable then $CP_\rnuc(\mathsf X;\mathfrak A,\mathfrak B)$ is countably generated.
\end{lemma}
\begin{proof}
Let $h_x\rho_x(-)\in \mathscr S$, and let $a\in \mathfrak A(\mathsf U)$. If $x \notin \mathsf U$, then $\mathsf U \subset \mathsf X \setminus \overline{\{ \mathsf x\}}$ and thus $h_x \rho_x(a) = 0 \in \mathfrak B(\mathsf U)$.
If $x\in \mathsf U$ then $\mathsf U^x \subset \mathsf U$, so $h_x \rho_x(a) \in \mathfrak B(\mathsf U^x) \subset \mathfrak B(\mathsf U)$. Thus $h_x \rho_x(-)$ is $\mathsf X$-equivariant, and it is clearly residually $\mathsf X$-nuclear,
so $\mathscr S \subset CP_\rnuc(\mathsf X; \mathfrak A, \mathfrak B)$.

By Corollary \ref{c:nucapproxfinite} any map in $CP_\rnuc(\mathsf X; \mathfrak A, \mathfrak B)$ can be approximated by a composition $\mathfrak A \xrightarrow{\phi} \mathfrak D \xrightarrow{\psi} \mathfrak B$, where $\mathfrak D$ is a finite dimensional, continuous 
$\mathsf X$-$C^\ast$-algebra and $\phi$ and $\psi$ are $\mathsf X$-equivariant c.p.~maps. By Lemma \ref{l:finitedim} we may decompose $\mathfrak D = \bigoplus_{x\in \mathsf X} i_x(\mathfrak D_x)$. 
Hence we may decompose $\phi$ and $\psi$ into $\mathsf X$-equivariant c.p.~maps 
$\phi_x \colon \mathfrak A \to i_x(\mathfrak D_x)$ and $\psi_x \colon i_x(\mathfrak D_x) \to \mathfrak B$. It follows that $\psi \circ \phi = \sum_{x\in \mathsf X} \psi_x \circ \phi_x$. Hence it suffices to show that $\psi_x \circ \phi_x$ are in the closed operator
convex cone generated by $\mathscr S$. Recall, as in \cite[Lemma 2.22]{MeyerNest-bootstrap}, that $CP(\mathsf X; \mathfrak A, i_x(\mathfrak D_x)) \cong CP(\mathfrak A/\mathfrak A(\mathsf X\setminus \overline{\{ x\}}), \mathfrak D_x)$ and
$CP(\mathsf X;i_x(\mathfrak D_x), \mathfrak B) \cong CP(\mathfrak D_x, \mathfrak B(\mathsf U^x))$ naturally. As in the classical case of nuclearity, we may assume that $\mathfrak D_x$ is a matrix algebra, say $M_{n_x}$. 
Let $\psi_x'\colon \mathfrak A/\mathfrak A(\mathsf X\setminus \overline{\{ x\}}) \to M_{n_x}$ and $\psi_x' \colon M_{n_x} \to \mathfrak B(\mathsf U^x)$ be the induced maps.

As in the first part of the proof of \cite[Lemma 10]{ElliottKucerovsky-extensions}, we may approximate $\psi_x' \circ \phi_x'$ by a sum of c.p.~maps of the form
\begin{equation}\label{eq:statesgen}
 a' \mapsto \sum_{i,j=1}^{n_x} {b'}_i^\ast \rho_x'(a_i'^\ast a' a_j') b'_j,
\end{equation}
with $a'_1,\dots,a'_{n_x} \in \mathfrak A/\mathfrak A(\mathsf X \setminus \overline{\{ x\} })$, $b'_1,\dots, b'_n \in \mathfrak B(\mathsf U^x)$ and $\rho_x$ a pure state on $\mathfrak A/\mathfrak A(\mathsf X \setminus \overline{\{ x\} })$.
Since $h_x$ is strictly positive in $\mathfrak B(\mathsf U^x)$ we may assume, by replacing $b_i'$ with a small perturbation, that $b'_i = h_x^{1/2} b_i$ for some $b_i \in \mathfrak B(\mathsf U^x)$. 
Let $a_i \in \mathfrak A$ lifts of $a'_i$, and let $\rho_x$ be a pure state on $\mathfrak A$ which vanishes on $\mathfrak A(\mathsf X \setminus \overline{ \{ x\} })$ and lifts $\rho'_x$. 
Then the map in \eqref{eq:statesgen} corresponds to the $\mathsf X$-equivariant c.p.~map $\mathfrak A \to \mathfrak B$ given by
\[
 a \mapsto \sum_{i,j=1}^{n_x} b_i^\ast h_x\rho_x(a_i^\ast a a_j) b_j,
\]
which is in the operator convex cone generated by $\mathscr S$. It follows that $\psi_x \circ \phi_x$ can be approximated by sums of maps on the above form, which in turn are in the closed operator convex cone generated by $\mathscr S$. 

The ``in particular'' part follows since the pure state space of a separable $C^\ast$-algebra is separable in the weak-$\ast$ topology.
Hence $\mathscr S$ generates the same closed operator convex cone as a set where we only choose countably many weak-$\ast$ dense pure states.
\end{proof}

The following proposition shows that in nice cases when $\mathscr C = CP_\rnuc(\mathsf X; \mathfrak A, \mathfrak B)$, we may describe the map $\mathfrak B_\mathscr{C}$ in terms of the actions of $\mathsf X$ on $\mathfrak A$ and $\mathfrak B$.
This gives a solution to Question \ref{q:computeB_C} for a large class of closed operator convex cones.

\begin{proposition}\label{p:determineB_C}
Let $\mathsf X$ be a finite space and let $\mathfrak A$ and $\mathfrak B$ be $\mathsf X$-$C^\ast$-algebras with $\mathfrak A$ separable and lower semicontinuous, and $\mathfrak B$ stable, $\mathsf X$-$\sigma$-unital, and continuous, 
and let $\mathscr C = CP_\rnuc(\mathsf X;\mathfrak A, \mathfrak B)$. Then $\mathfrak B_\mathscr{C}(a) = \mathfrak B(\mathsf U_a)$ for all $a \in \mathfrak A$.

Moreover, $\mathfrak B_\mathscr{C} = \mathfrak B_{CP(\mathsf X; \mathfrak A, \mathfrak B)}$.
\end{proposition}

\begin{proof}
Let $a\in \mathfrak A$. By Lemma \ref{l:statesgen}, $\mathscr C$ is countably generated. Thus it follows from Lemma \ref{l:idealmap} that $\mathfrak B_\mathscr{C}(a) = \mathfrak B_\mathscr{C}(a^\ast a)$. 
Since $a$ and $a^\ast a$ are contained in exactly the same two-sided, closed ideals in $\mathfrak A$, it follows that $\mathsf U_a = \mathsf U_{a^\ast a}$, so we may assume, without loss of generality, that $a\in \mathfrak A$ is positive. 

We may assume that $\mathsf X$ is a $T_0$ space by Remark \ref{r:soberspace}.
For $x\in \mathsf X$ let $\mathsf U^x$ denote the smallest open set which contains $x$, and let $\mathsf V^x$ denote the largest open set which does not contain $x$. For each $x$ fix a strictly positive element $b_x$ in $\mathfrak B(\mathsf U^x)$, 
and a state $\phi_x \colon \mathfrak A \to \mathbb C$ such that $\phi_x(\mathfrak A(\mathsf V^x)) = 0$ and such that the induced state on $\mathfrak A/\mathfrak A(\mathsf V^x)$ is faithful if $\mathfrak A/\mathfrak A(\mathsf V^x)$ is non-zero. 
Now define a map
\[
\phi \colon \mathfrak A \to \mathfrak B, \qquad \phi(c) = \sum_{x\in \mathsf X} \phi_x(c) b_x.
\]
As $\phi_x(-)b_x = b_x^{1/2} \phi_x(-) b_x^{1/2}$ is completely positive for each $x$, $\phi$ is completely positive. 
For any $x\notin \mathsf U_c$ we have $\mathsf U_c \subset \mathsf V^x$. Hence $c \in \mathfrak A(\mathsf V^x)$ which implies $\phi_x(c) = 0$. 
It follows that
\[
\phi(c) = \sum_{x\in \mathsf U_c} \phi_x(c) b_x \in \sum_{x\in \mathsf U_c} \mathfrak B(\mathsf U^x) = \mathfrak B(\bigcup_{x\in \mathsf U_c} \mathsf U^x) = \mathfrak B(\mathsf U_c).
\]
Hence $\phi$ is $\mathsf X$-equivariant. We want to show that $\phi(a)$ is full in $\mathfrak B(\mathsf U_a)$.

If $\mathsf U_a = \emptyset$ then $\mathfrak B(\mathsf U_a) = 0$, so $\phi(a) = 0$ is full in $\mathfrak B(\mathsf U_a)$. Thus we may consider the case where $\mathsf U_a$ is non-empty.
Suppose for contradiction that there is $x\in \mathsf U_a$ such that $\mathfrak A(\mathsf V^x) = \mathfrak A$. Then $a \in \mathfrak A(\mathsf U_a) \cap \mathfrak A(\mathsf V^x) = \mathfrak A(\mathsf U_a \cap \mathsf V^x)$.
However, since $\mathsf U_a \cap \mathsf V^x$ does not contain $x$, it is a proper subset of $\mathsf U_a$, which contradicts that $a$ is $\mathsf U_a$-full. Hence $\mathfrak A/\mathfrak A(\mathsf V^x)$ is non-zero for all $x\in \mathsf U_a$.
By construction $\phi_x(a) >0$ for all $x\in \mathsf U_a$ and thus $\phi(a)$ is full in $\mathfrak B(\mathsf U_a)$. Moreover, $\phi$ is obviously residually $\mathsf X$-nuclear.
Hence $\mathfrak B(\mathsf U_a) = \mathfrak B_\mathscr{C}(a)$ by Lemma \ref{l:fullcomputeB_C}.

That $\mathfrak B_\mathscr{C} = \mathfrak B_{CP(\mathsf X; \mathfrak A, \mathfrak B)}$ follows, since
\[
 \mathfrak B_\mathscr{C}(a) \subset \mathfrak B_{CP(\mathsf X; \mathfrak A, \mathfrak B)}(a) \subset \mathfrak B(\mathsf U_a) = \mathfrak B_\mathscr{C}(a)
\]
for all $a\in \mathfrak A$.
\end{proof}

The rest of this subsection is dedicated to showing, that $\mathsf X$-equivariant $CP_\rnuc(\mathsf X; \mathfrak A, \mathfrak B)$-absorbing $\ast$-homomorphisms are $\mathsf X$-full.

Note that if $\mathfrak J$ is a two-sided, closed ideal in $\mathfrak B$, then $\mathfrak J$ is an essential ideal in $\multialg{\mathfrak B, \mathfrak J}$. Hence there is an induced injective $\ast$-homomorphism 
$\iota \colon \multialg{\mathfrak B, \mathfrak J} \hookrightarrow \multialg{\mathfrak J}$.

\begin{lemma}\label{l:multiher}
 The image $\iota(\multialg{\mathfrak B, \mathfrak J}) \subset \multialg{\mathfrak J}$ is a hereditary $C^\ast$-subalgebra. 
\end{lemma}
\begin{proof}
 If $m\in \multialg{\mathfrak J}$ and $m_1,m_2 \in \multialg{\mathfrak B, \mathfrak J}$ let $x \in \multialg{\mathfrak B, \mathfrak J}$ be defined by
 \[
  x b = m_1 (m (m_2 b)), \qquad b x = ((b m_1) m) m_2
 \]
 for $b\in \mathfrak B$. It is easily seen that $x$ is well-defined and that $\iota(x) = \iota(m_1) m \iota(m_2)$. Hence $\iota(\multialg{\mathfrak B, \mathfrak J}) \subset \multialg{\mathfrak J}$ is a hereditary $C^\ast$-subalgebra.
\end{proof}

\begin{lemma}\label{l:multifullher}
 Let $\mathfrak B$ be a $\sigma$-unital, stable $C^\ast$-algebra, let $m\in \multialg{\mathfrak B}$, and let $m_\infty$ be an infinite repeat of $m$.
 Then $m_\infty$ is full in $\multialg{\mathfrak B, \overline{\mathfrak B m \mathfrak B}}$ and $\iota(m_\infty)$ is full in $\multialg{\overline{\mathfrak B m \mathfrak B}}$. 
 
 In particular, if $\mathfrak J$ is a two-sided, closed ideal in $\mathfrak B$, and there exists an $m\in \multialg{\mathfrak B}$ with $\mathfrak J = \overline{\mathfrak B m \mathfrak B}$, then $\iota(\multialg{\mathfrak B, \mathfrak J})$ is a full,
 hereditary $C^\ast$-subalgebra of $\multialg{\mathfrak J}$.
\end{lemma}
\begin{proof}
 By replacing $m$ with $m^\ast m$ we may assume that $m$ is positive. 
 The infinite repeat $m_\infty$ is clearly an element in $\multialg{\mathfrak B, \mathfrak J}$. By \cite[Lemma 14]{Kucerovsky-ideals} it follows that $\iota(m_\infty)$ is full in $\multialg{\mathfrak J}$,
 since it is also an infinite repeat in $\multialg{\mathfrak J}$.
 Since $\iota(\multialg{\mathfrak B, \mathfrak J})$ is a hereditary $C^\ast$-subalgebra of $\multialg{\mathfrak J}$ by Lemma \ref{l:multiher} it follows that $m_\infty$ is full in $\multialg{\mathfrak B, \mathfrak J}$.
 
 The ``in particular'' part follows from Lemma \ref{l:multiher} and since $\iota(\multialg{\mathfrak B, \mathfrak J})$ contains an element which is full in $\multialg{\mathfrak J}$.
\end{proof}

Very similar to Lemma \ref{l:idealmap} we have the following lemma.

\begin{lemma}\label{l:multiidealmap}
Let $\mathfrak A$ and $\mathfrak B$ be $C^\ast$-algebras, with $\mathfrak A$ separable, $\mathfrak B$ $\sigma$-unital and stable, and let $\mathscr C \subset CP(\mathfrak A,\mathfrak B)$ be a countably generated, closed operator convex cone. 
If $\Phi\colon \mathfrak A \to \multialg{\mathfrak B}$ is a $\mathscr C$-absorbing $\ast$-homomorphism weakly in $\mathscr C$, then
\[
\multialg{\mathfrak B, \mathfrak B_\mathscr{C}(a)} = \overline{\multialg{\mathfrak B} \Phi(a) \multialg{\mathfrak B}}
\]
for every $a\in \mathfrak A$.

Moreover, if $\mathfrak A$ is unital, $\mathscr C$ is non-degenerate and $\Phi \colon \mathfrak A \to \multialg{\mathfrak B}$ is a unitally $\mathscr C$-absorbing $\ast$-homomorphism weakly in $\mathscr C$, then
\[
\multialg{\mathfrak B, \mathfrak B_\mathscr{C}(a)} = \overline{\multialg{\mathfrak B} \Phi(a) \multialg{\mathfrak B}}
\]
for every $a\in \mathfrak A$.
\end{lemma}
\begin{proof}
 The unital and non-unital case are identical, so we only prove the non-unital case. 
 Since $\Phi(a)$ and $\Phi(a^\ast a)$ generate the same two-sided, closed ideal, and since $\mathfrak B_\mathscr{C}(a) = \mathfrak B_{\mathscr C}(a^\ast a)$ by Lemma \ref{l:idealmap}, we may restrict to the case where $a\in \mathfrak A$ is positive.
 
 Let $t_1,t_2,\dots$ be isometries in $\multialg{\mathfrak B}$ such that $\sum_{n=1}^\infty t_nt_n^\ast$ converges strictly to $1$.
 The infinite repeat $\Phi_\infty = \sum_{n=1}^\infty t_n \Phi(-) t_n^\ast$ is also $\mathscr C$-absorbing and weakly in $\mathscr C$, and
 \[
  \overline{\multialg{\mathfrak B} \Phi(a) \multialg{\mathfrak B}} = \overline{\multialg{\mathfrak B} \Phi_\infty(a) \multialg{\mathfrak B}},
 \]
 for all $a\in \mathfrak A$, since $\Phi$ and $\Phi_\infty$ are asymptotically unitarily equivalent. 
 Note that $\mathfrak B_\mathscr{C}(a) = \overline{\mathfrak B \Phi_\infty(a) \mathfrak B}$ by Lemma \ref{l:idealmap}. 
 By Lemma \ref{l:multifullher}, $\Phi_\infty(a)$ is full in $\multialg{\mathfrak B, \mathfrak B_\mathscr{C}(a)}$.
\end{proof}

\begin{corollary}\label{c:Xfullmap}
 Let $\mathsf X$ be a finite space and let $\mathfrak A$ and $\mathfrak B$ be $\mathsf X$-$C^\ast$-algebras with $\mathfrak A$ separable and lower semicontinuous, and $\mathfrak B$ stable, $\mathsf X$-$\sigma$-unital, and continuous, 
and let $\mathscr C = CP_\rnuc(\mathsf X;\mathfrak A, \mathfrak B)$. 
Then there exists a $\mathscr C$-absorbing $\ast$-homomorphism $\Phi\colon \mathfrak A \to \multialg{\mathfrak B}$ weakly in $\mathscr C$, and any such $\ast$-homomorphism is $\mathsf X$-full.

Moreover, if $\mathfrak A$ is unital and $\mathfrak B(\mathsf U_{1_\mathfrak{A}}) = \mathfrak B$, then there exists a unitally $\mathscr C$-absorbing $\ast$-homomorphism $\Phi \colon \mathfrak A \to \multialg{\mathfrak B}$ weakly in $\mathscr C$,
and any such $\ast$-homomorphism is $\mathsf X$-full.
\end{corollary}
\begin{proof}
 For the unital case, note that $\mathfrak B_{\mathscr C}(1_\mathfrak{A}) = \mathfrak B(\mathsf U_{1_\mathfrak{A}}) = \mathfrak B$ by Proposition \ref{p:determineB_C}. This is equivalent to $\mathscr C$ being non-degenerate.
 
 Thus the existence of a (unitally) $\mathscr C$-absorbing $\ast$-homomorphism $\Phi$ weakly in $\mathscr C$ follows from Lemma \ref{l:statesgen} and Theorem \ref{t:absrep}. By Proposition \ref{p:determineB_C} and Lemma \ref{l:multiidealmap} it follows that
 \[
  \overline{\multialg{\mathfrak B} \Phi(a) \multialg{\mathfrak B}} = \multialg{\mathfrak B, \mathfrak B_\mathscr{C}(a)} = \multialg{\mathfrak B, \mathfrak B(\mathsf U_a)} = \multialg{\mathfrak B}(\mathsf U_a)
 \]
 for every $a\in \mathfrak A$. Thus $\Phi$ is $\mathsf X$-full.
\end{proof}


\subsection{The purely large problem over finite spaces}

In \cite{ElliottKucerovsky-extensions} an important part of the proof of the main theorem was, that every purely large extension had a certain purely infinite type comparison property. 
The following lemma seems to be the closest we can get to a similar result with respect to closed operator convex cones.

\begin{lemma}\label{Xpllemma}
Let $0 \to \mathfrak B \to \mathfrak E \xrightarrow{p} \mathfrak A \to 0$ be a $\mathscr C$-purely large extension. 
Let $\epsilon >0$, $x\in \mathfrak E$ be positive with $\| x \|=1$, and let that $b\in \mathfrak B$ with $\| b\| =1$. Suppose that $b \in \mathfrak B_\mathscr{C} (p(g(x)))$ for any positive, continuous function $g\colon [0,1] \to [0,1]$ with $g(0)=0$ and $g(1)=1$. 
Then there is a $b_0 \in \mathfrak B$ with $\| b_0 \| \leq 1$ such that
\[
\| b - b_0^\ast x b_0\| < \epsilon.
\]
\end{lemma}
\begin{proof}
Let $f,g \colon [0,1] \to [0,1]$ be the continuous functions
\[
f(t) = \left\{ \begin{array}{ll} 0, & \text{ for } t=0 \\ \text{affine}, & \text{ for }0 \leq t \leq 1-\epsilon/3 \\ 1 & \text{ for } 1-\epsilon/3 \leq t \leq 1 \end{array} \right. , 
\qquad g(t) = \left\{ \begin{array}{ll} 0, & \text{ for } 0 \leq t \leq 1-\epsilon/3 \\ \text{affine}, & \text{ for } 1-\epsilon/3 \leq t \leq 1 \\ 1, & \text{ for } t=1. \end{array} \right.
\]
Then $\| f(x) - x\| \leq \epsilon/3$ and $f(x)g(x) = g(x)f(x) = g(x)$. Since the extension is purely large with respect to $\mathfrak B_\mathscr{C}$, the $C^\ast$-algebra $\overline{g(x) \mathfrak B g(x)}$ contains a stable, $\sigma$-unital
$C^\ast$-subalgebra $\mathfrak D$
which is full in $\mathfrak B_\mathscr{C}(p(g(x)))$. Since $b\in \mathfrak B_\mathscr{C}(p(g(x)))$ we may, as in the proof of \cite[Lemma 7]{ElliottKucerovsky-extensions}, find $d\in \mathfrak D_+$, $b_1,\dots ,b_n\in \mathfrak B$ such 
that $\| b - \sum b_i^\ast d b_i \| < \epsilon/3$, and $V_1,\dots ,V_n, P\in \multialg{\mathfrak D}$ such that $P$ is a projection for which $Pd = d$, and $V_i^\ast V_j = \delta_{ij} P$. Define
\[
b_0' := \sum_{i=1}^n V_i d^{1/2} b_i, \qquad b_0 = b_0' /\|b_0'\|.
\]
We get that
\[
b_0'^\ast b_0' = \sum_{i,j=1}^n b_i^\ast d^{1/2} V_i^\ast V_j d^{1/2} b_j = \sum_{i=1}^n b_i^\ast d b_i.
\]
Thus $\|b_0'^\ast b_0' - b \| < \epsilon/3$ which implies, since $\| b\| = 1$, that $| \| b_0'^\ast b_0' \| - 1 | < \epsilon/3$. Hence 
\[
\| b_0^\ast b_0 - b_0'^\ast b_0' \| = \left\| b_0'^\ast b_0' \left( \frac{1}{\| b_0'^\ast b_0'\|} - 1 \right) \right\| \leq | 1- \|b_0'^\ast b_0'\|| < \epsilon/3.
\]

Note that $f(x) d' = d' f(x) = d'$ for any $d' \in \mathfrak D$. Hence $b_0^\ast f(x) b_0 = b_0^\ast b_0$ and thus
\[
b_0^\ast x b_0 \approx_{\epsilon/3}  b_0^\ast f(x) b_0 = b_0^\ast b_0 \approx_{\epsilon/3} b_0'^\ast b_0' \approx_{\epsilon/3} b.
\]
\end{proof}

Given an $\mathsf X$-$C^\ast$-algebra $\mathfrak A$, we may give the forced unitisation $\mathfrak A^\dagger$ an $\mathsf X$-$C^\ast$-algebra structure by letting $\mathfrak A^\dagger(\mathsf U)= \mathfrak A(\mathsf U)$ for $\mathsf U\in \mathbb O(\mathsf X)$,
when $\mathsf U \neq \mathsf X$, and $\mathfrak A^\dagger(\mathsf X) = \mathfrak A^\dagger$. Note that this construction may ruin certain properties which the action of $\mathsf X$ on $\mathfrak A$ had, e.g.~(finite) upper semicontinuity. However,
lower semicontinuity will be preserved.

We need the following lemma. For notation see Lemma \ref{unitalcone}.

\begin{lemma}\label{l:unitisationcone}
 Let $\mathsf X$ be a topological space, and let $\mathfrak A$ and $\mathfrak B$ be $\mathsf X$-$C^\ast$-algebras. Suppose that $\mathfrak B (\mathsf X) = \mathfrak B$. Then 
 \[
  CP_\rnuc(\mathsf X;\mathfrak A,\mathfrak B)^\dagger = CP_\rnuc(\mathsf X;\mathfrak A^\dagger,\mathfrak B).
 \]
\end{lemma}
\begin{proof}
 If $\phi \in CP_\rnuc(\mathsf X;\mathfrak A^\dagger,\mathfrak B)$ then $\phi |_\mathfrak{A}$ is residually $\mathsf X$-nuclear. In fact, if $\mathsf U\in \mathbb O(\mathsf X)\setminus \{ \mathsf X\}$, then the induced map 
 $\mathfrak A^\dagger/ \mathfrak A^\dagger(\mathsf U) \to \mathfrak B/ \mathfrak B(\mathsf U)$ is nuclear. Since $\mathfrak A^\dagger/\mathfrak A^\dagger(\mathsf U) \cong (\mathfrak A/\mathfrak A(\mathsf U)) ^\dagger$, the restriction to
 $\mathfrak A/\mathfrak A(\mathsf U)$ is nuclear. Hence $\phi \in CP_\rnuc(\mathsf X;\mathfrak A,\mathfrak B)^\dagger$.
 
 Now let us show, that a positive map $\phi \colon \mathfrak A^\dagger \to \mathfrak B$ is nuclear if $\phi|_\mathfrak{A}$ is nuclear. Given a finite subset
 $F \subset \mathfrak A^\dagger$ and $\epsilon>0$, we may assume that $F = F' \cup \{ 1\}$ where $F'\subset \mathfrak A$. Let $a\in \mathfrak A$ be a positive contraction such that $\| a b a - b \| < \epsilon/(2\| \phi\|)$ for all $b \in F'$. Let
 $\tilde \phi \colon \mathfrak A \to \mathfrak B$ be a c.p.~map factoring by c.p.~maps through a matrix algebra, such that $\| \tilde \phi \| \leq \| \phi \|$ and $\| \tilde \phi(aba) - \phi(aba)\| < \epsilon/2$ for all $b \in F$. 
 Now let $\psi$ be the composition $\mathfrak A^\dagger \twoheadrightarrow \mathbb C \to \mathfrak B$ given by $\psi(1) = \phi(1 - a^2)$, which is completely positive since $\phi$ is positive. 
 Then the sum $(\tilde \phi \circ \Ad a) + \psi$ factors through a finite dimensional $C^\ast$-algebra and approximates $\phi$ up to $\epsilon$ on $F$. Hence $\phi$ is nuclear.
 
 Let $\phi \in CP_\rnuc(\mathsf X;\mathfrak A,\mathfrak B)^\dagger$ and $\mathsf U\in \mathbb O(\mathsf X) \setminus \{ \mathsf X\}$. Clearly $\phi$ is $\mathsf X$-equivariant. Since $\phi|_\mathfrak{A}$ is residually $\mathsf X$-nuclear, 
 it follows that the restriction of $[\phi]_{\mathsf U} \colon \mathfrak A^\dagger/\mathfrak A^\dagger(\mathsf U) \to \mathfrak B/\mathfrak B(\mathsf U)$ to $\mathfrak A/\mathfrak A(\mathsf U)$ is nuclear. 
 Since $\mathfrak A^\dagger/\mathfrak A^\dagger(\mathsf U) \cong (\mathfrak A/\mathfrak A(\mathsf U))^\dagger$, it follows by what we proved above, that $[\phi]_{\mathsf U}$ is nuclear. Hence $\phi$ is residually $\mathsf X$-nuclear.
\end{proof}

\begin{definition}
We will say that an extension of $\mathsf X$-$C^\ast$-algebras $\mathfrak A$ by $\mathfrak B$ is \emph{$\mathsf X$-purely large} if it is $CP(\mathsf X;\mathfrak A, \mathfrak B)$-purely large.

Moreover, we say that an extension of $\mathfrak A$ by $\mathfrak B$ of $\mathsf X$-$C^\ast$-algebras is \emph{weakly residually $\mathsf X$-nuclear} if it is a $CP_\rnuc(\mathsf X; \mathfrak A, \mathfrak B)$-extension.
\end{definition}

It follows immediately from Proposition \ref{p:determineB_C} that if $\mathsf X$ is finite, $\mathfrak A$ is a separable, lower semicontinuous $\mathsf X$-$C^\ast$-algebra, and $\mathfrak B$ is a stable, $\mathsf X$-$\sigma$-unital, continuous
$\mathsf X$-$C^\ast$-algebra, then an extension of $\mathfrak A$ by $\mathfrak B$ is $\mathsf X$-purely large if and only if it is $CP_\rnuc(\mathsf X; \mathfrak A, \mathfrak B)$-purely large.

The following is one of our main theorems. It gives a solution to the purely large problem (Question \ref{q:plp}) for a large class of closed operator convex cones, showing that they \emph{do} satisfy the (unital) purely large problem. 
Note that we do \emph{not} assume any infinity nor nuclearity criteria on our $C^\ast$-algebras (as done in Proposition \ref{p:plpinfinite}).

\begin{theorem}\label{t:purelylargefinite}
 Let $\mathsf X$ be a finite space, and let $\mathfrak e: 0 \to \mathfrak B \to \mathfrak E \to \mathfrak A \to 0$ be an extension of $C^\ast$-algebras.
 Suppose that $\mathfrak A$ is a separable, lower semicontinuous $\mathsf X$-$C^\ast$-algebra, and $\mathfrak B$ is a
 stable, $\mathsf X$-$\sigma$-unital, continuous $\mathsf X$-$C^\ast$-algebra.

Then $\mathfrak e$ absorbs any trivial, weakly residually $\mathsf X$-nuclear extension if and only if $\mathfrak e$ is $\mathsf X$-purely large and absorbs the zero extension.

If, in addition, $\mathfrak e$ is unital and $\mathfrak B(\mathsf U_{1_\mathfrak{A}}) = \mathfrak B$, then $\mathfrak e$ absorbs any trivial, unital weakly residually $\mathsf X$-nuclear extension if and only if $\mathfrak e$ is $\mathsf X$-purely large.
\end{theorem}
\begin{proof}
Let $\mathscr C:= CP_\rnuc(\mathsf X; \mathfrak A, \mathfrak B)$.
The ``only if'' part follows from Proposition \ref{purelylargeprop}, where we note that, in the unital case, $\mathfrak B_\mathscr{C}(1_\mathfrak{A}) = \mathfrak B(\mathsf U_{1_\mathfrak{A}}) = \mathfrak B$ 
by Proposition \ref{p:determineB_C}, which is equivalent to $\mathscr C$ being non-degenerate.

To prove the ``if'' part, it follows from Proposition \ref{p:plpunitimpliesnonunit}, Lemma \ref{l:unitisationcone} and the fact that $\mathfrak A^\dagger$ is also lower semicontinuous, that it suffices to prove the unital version of the theorem. 
Thus assume that $\mathfrak e$ is unital and $\mathfrak B(\mathsf U_{1_\mathfrak{A}}) = \mathfrak B$ and that the extension is $\mathsf X$-purely large.

Let $\sigma\colon \mathfrak E \to \multialg{\mathfrak B}$ and $\tau \colon \mathfrak A \to \corona{\mathfrak B}$ be the canonical $\ast$-homomorphisms induced by the extension $\mathfrak e$.
Let $\mathfrak C \subset \mathfrak E$ be a separable $C^\ast$-subalgebra with $1_{\mathfrak E} \in \mathfrak C$, containing a strictly positive element of $\mathfrak B$, such that $\mathfrak C + \mathfrak B = \mathfrak E$,
and let $p_0 = p|_\mathfrak{C} \colon \mathfrak C \to \mathfrak A$, which is surjective.

We may assume that $\mathsf X$ is a $T_0$ space by Remark \ref{r:soberspace}. Let $\mathscr S$ be as in Lemma \ref{l:statesgen}. By Lemma \ref{l:quotcone}
\[
 \mathscr{C}_0 = \{ \phi \circ p_0 : \phi \in \mathscr C\} \subset CP(\mathfrak C, \mathfrak B)
\]
is a closed operator convex cone generated by the set $\mathscr S_0 := \{ \phi' \circ p_0 : \phi' \in \mathscr S\}$.
We will show that the unital $\ast$-homomorphism $\sigma_0 := \sigma|_\mathfrak{C} \colon \mathfrak C \to \multialg{\mathfrak B}$ approximately dominates any map $\phi\in \mathscr S_0$.

Let $\phi = \phi' \circ p_0$ with $\phi' \in \mathscr S$. Write $\phi' = h_x \rho_x'(-)$ for some $x\in \mathsf X$ with $h_x$ a strictly positive contraction in $\mathfrak B(\mathsf U^x)$, and $\rho_x'$ a pure state on $\mathfrak A$ vanishing on 
$\mathfrak A(\mathsf X \setminus \overline{\{ x\}})$. Note that $\rho_x = \rho'_x \circ p_0$ is a pure state on $\mathfrak C$ since $p_0$ is a surjective $\ast$-homomorphism. Fix $F \subset \mathfrak C$ a finite set of contractions and $\epsilon >0$.
By \cite{AkemannAndersonPedersen-excision}, we may excise $\rho_x$ as follows. There is a positive $d \in \mathfrak C$ such that $\rho_x(d) = \| d\| = 1$, and
\[
 \| d^2\rho_x(c) - dcd \| < \epsilon/2,
\]
for all $c\in F$. Since $\rho_x(d)=\| d\| = 1$ it follows (e.g.~by considering the GNS representation induced by $\rho_x$), that for any continuous function $g\colon [0,1] \to [0,1]$ for which $g(0)=0$ and $g(1)=1$, we have $\rho_x(g(d)) > 0$. 
Since $\mathfrak A$ is a lower semicontinuous $\mathsf X$-$C^\ast$-algebra, it follows that $p_0(g(d^2))$ is $\mathsf U_g$-full, for some unique open subset $\mathsf U_g$ of $\mathsf X$. 
Note that $\mathfrak B_\mathscr{C}(p_0(g(d))) = \mathfrak B(\mathsf U_g)$ by Proposition \ref{p:determineB_C}. 
If we suppose that $\mathsf U^x \not \subset \mathsf U_g$ for some $g$, i.e.~$x\notin \mathsf U_g$,
then $\mathsf U_g \subset \mathsf X \setminus \overline{\{ x\}}$ which would imply that $p_0(g(d^2)) \in \mathfrak A(\mathsf X \setminus \overline{\{ x\}})$ and thus $\rho_x(g(d^2))= 0$. 
However, this is false, so we must have that $\mathsf U^x \subset \mathsf U_g$ for all $g$.
Since $h_x \in \mathfrak B(\mathsf U^x) \subset \mathfrak B(\mathsf U_g)$ for any $g$ as above, we may apply Lemma \ref{Xpllemma} to obtain a contraction $b_0\in \mathfrak B$ such that
$h_x \approx_{\epsilon/2} b_0^\ast d^2b_0$. Defining $b := db_0$ we get that $\| b\| \leq 1$ and
\[
 h_x\rho_x(c) \approx_{\epsilon/2} b_0^\ast d \rho_x(c) d b_0 \approx_{\epsilon/2} b_0^\ast d c d b_0 = b^\ast c b = b^\ast \sigma_0(c) b
\]
for all $c\in F$. Hence $\sigma_0$ approximately dominates $\phi = h_x \rho_x(-)$.

Let $(b_n)$ be a bounded sequence in $\mathfrak B$ such that $\| b_n^\ast \sigma_0(c) b_n - \phi(c) \| \to 0$ for every $c\in \mathfrak C$.
Since $\mathfrak C$ contains a strictly positive element $h$ for $\mathfrak B$, and since $\phi(h)= (h_x \rho'_x(-)) \circ p_0(h) = 0$, we get that $\| b_n^\ast \sigma_0(h) b_n \| = \| b_n^\ast h b_n\| \to 0$.
It follows that $\|b_n^\ast b b_n \| \to 0$ for every $b\in \mathfrak B$, and thus $\sigma_0$ \emph{strongly} approximately dominates $\phi$.

Since $\sigma_0$ strongly approximately dominates any c.p.~map in $\mathscr S_0$, it follows from Theorems \ref{t:domgen} and \ref{t:domabs} 
that $\sigma_0 \oplus \Psi \sim_{ap} \sigma_0$ for any unital $\ast$-homomorphism $\Psi \colon \mathfrak C \to \multialg{\mathfrak B}$ weakly in $\mathscr C_0$.

Let $\Phi \colon \mathfrak A \to \multialg{\mathfrak B}$ be a unital $\mathscr C$-absorbing $\ast$-homomorphism weakly in $\mathscr C$ of Theorem \ref{t:absrep} which exists by Lemma \ref{l:statesgen}.
Then $\Phi \circ p_0$ is unital and weakly in $\mathscr C_0$ and thus $\sigma_0 \oplus (\Phi \circ p_0) \sim_{ap} \sigma_0$.
In particular, there is a unitary $u\in \multialg{\mathfrak B}$ such that 
\[
\tau(p_0(c)) = \pi (\sigma_0(c)) = \pi(u^\ast(\sigma_0(c)\oplus \Phi(p_0(c)))u) = \pi(u)^\ast ( \tau(p_0(c)) \oplus (\pi \circ \Phi)(p_0(c))) \pi(u)
\]
for all $c\in \mathfrak C$. Since $p_0$ is surjective, we get that $\mathfrak e$ absorbs the trivial, unital $\mathscr C$-extension with Busby map $\pi \circ \Phi$. Since this latter extension absorbs all trivial, unital $\mathscr C$-extensions,
it follows that $\mathfrak e$ absorbs all trivial, unital $\mathscr C$-extensions.
\end{proof}

\begin{remark}
If $\mathfrak A$ is exact (or even locally reflexive), then the conditions that $\mathfrak A$ is lower semicontinuous and $\mathfrak B$ is continuous, are redundant. 

In fact, as in \cite[Section 2.9]{MeyerNest-bootstrap} we may find a larger, but still finite, space $\mathsf Y$ acting on $\mathfrak A$ and $\mathfrak B$ continuously, such that $CP(\mathsf X; \mathfrak A, \mathfrak B) = CP(\mathsf Y; \mathfrak A, \mathfrak B)$.
Thus an extension is $\mathsf X$-purely large if and only if it is $\mathsf Y$-purely large.
By Remark \ref{r:exactresnuc} it follows that the residually $\mathsf X$-nuclear maps are exactly the residually $\mathsf Y$-nuclear maps. 
Since finite sums and finite intersections of $\sigma$-unital, two-sided, closed ideals are again $\sigma$-unital, it follows that if $\mathfrak B$ is $\mathsf X$-$\sigma$-unital then it is also $\mathsf Y$-$\sigma$-unital.
Thus we may replace $\mathsf X$ with $\mathsf Y$ and the result follows.
\end{remark}

For applications of the above theorem in the non-unital case, it can be hard to determine whether or not an extension absorbs the zero extension. 
However, when the quotient algebra is sufficiently non-unital with a sufficiently nice action of $\mathsf X$, we get this for free by knowing that the extension is $\mathsf X$-purely large.
Note that we assume in the following corollary that the extension is $\mathsf X$-equivariant, which was \emph{not} a part of the assumptions in Theorem \ref{t:purelylargefinite}. 

\begin{corollary}\label{c:nonunitalquotient}
  Let $\mathsf X$ be a finite space, let $\mathfrak e: 0 \to \mathfrak B \to \mathfrak E \to \mathfrak A \to 0$ be an extension of $\mathsf X$-$C^\ast$-algebras such that $\mathfrak A$ is separable and 
  lower semicontinuous, and $\mathfrak B$ is
 stable, $\mathsf X$-$\sigma$-unital and continuous. Suppose that $\mathfrak A/\mathfrak A(\mathsf U)$ is non-zero and non-unital for all 
 $\mathsf U\in \mathbb O(\mathsf X) \setminus \{ \mathsf X\}$.
 
 Then $\mathfrak e$ absorbs any trivial, weakly residually $\mathsf X$-nuclear extension if and only if $\mathfrak e$ is $\mathsf X$-purely large.
\end{corollary}
\begin{proof}
 Let $\mathscr C := CP_\rnuc(\mathsf X; \mathfrak A, \mathfrak B)$. 
 One part follows from Proposition \ref{purelylargeprop} so suppose that $\mathfrak e$ is $\mathsf X$-purely large.
 By Theorem \ref{t:purelylargefinite}, Proposition \ref{p:plpunitimpliesnonunit} and Lemma \ref{l:unitisationcone}, $\mathfrak e$ absorbs any trivial $\mathscr C$-extension if and only if
 the unitised extension $\mathfrak e^\dagger : 0 \to \mathfrak B \to \mathfrak E^\dagger \xrightarrow{p^\dagger} \mathfrak A^\dagger \to 0$ is $\mathsf X$-purely large. So we will show that this is the case.
 
 Let $y \in \mathfrak E^\dagger$. Note that if $a\in \mathfrak A \subset \mathfrak A^\dagger$, then $a$ is $\mathsf U_a$-full for the same $\mathsf U_a$ regardless of whether we consider $a$ as an element of the $\mathsf X$-$C^\ast$-algebra $\mathfrak A$
 or of the $\mathsf X$-$C^\ast$-algebra $\mathfrak A^\dagger$.
 Thus if $y \in \mathfrak E$, then 
 \[
  \mathfrak B_\mathscr{C}(p(y)) = \mathfrak B_\mathscr{C^\dagger}(p^\dagger(y)) = \mathfrak B(\mathsf U_{p^\dagger(y)}) = \mathfrak B(\mathsf U_{p(y)})
 \]
 by Lemma \ref{l:B_Cdagger} and Proposition \ref{p:determineB_C}. 
 It follows that $\overline{y \mathfrak B y^\ast}$ contains a stable $\sigma$-unital $C^\ast$-subalgebra which is full in $\mathfrak B(\mathsf U_{p^\dagger(y)})$, since $\mathfrak e$ is $\mathsf X$-purely large.
 
 Hence it remains to consider the case $y \in \mathfrak E^\dagger \setminus \mathfrak E$. Without loss of generality we may assume that $y = 1 - x$ for an $x\in \mathfrak E$. Note that $\mathsf U_{p^\dagger(1-x)} = \mathsf X$.
 Suppose that $(1-x)\mathfrak E \subset \mathfrak B + \mathfrak E(\mathsf U)$ for some $\mathsf U\neq \mathsf X$. 
 Then $x + (\mathfrak B + \mathfrak E(\mathsf U))$ would be a unit for $\mathfrak E/(\mathfrak B + \mathfrak E(\mathsf U)) \cong \mathfrak A/\mathfrak A(\mathsf U)$, which is a contradiction to the assumption that $\mathfrak A/\mathfrak A(\mathsf U)$ is
 non-zero and non-unital for $\mathsf U \in \mathbb O(\mathsf X) \setminus \{\mathsf X\}$. 
 Hence we may find positive elements $x'_\mathsf{U}$ in $\mathfrak E$ such that $(1-x)x'_\mathsf{U} \notin \mathfrak B + \mathfrak E(\mathsf U)$. Let $x' = \sum x'_\mathsf{U}$. It is easily seen that $z:=(1-x)x' \notin \mathfrak B + \mathfrak E(\mathsf U)$ for any 
 $\mathsf U \neq \mathsf X$, and thus $p(z) \notin \mathfrak A(\mathsf U)$. Hence $\mathsf U_{p(z)} = \mathsf U_{p^\dagger(1-x)} = \mathsf X$, and thus $\overline{z \mathfrak B z^\ast}$ contains a stable $\sigma$-unital $C^\ast$-subalgebra which is full
 in $\mathfrak B$ since $\mathfrak e$ is $\mathsf X$-purely large. Since $\overline{z \mathfrak B z^\ast} \subset \overline{(1-x) \mathfrak B (1-x)^\ast}$ it follows that $\mathfrak e^\dagger$ is $\mathsf X$-purely large.
\end{proof}

\begin{remark}
 One might find it interesting to compare the assumptions on $\mathfrak A$ above to the similar assumptions made by Kirchberg in \cite[Hauptsatz 4.2]{Kirchberg-non-simple}. 
 Here Kirchberg always assumes that $\mathfrak A$ is separable, stable (and thus has no unital quotients), that $\mathfrak A(\mathsf U) = \mathfrak A$ only when $\mathsf U = \mathsf X$ (corresponding to $\mathfrak A/\mathfrak A(\mathsf U)$ is non-zero
 for $\mathsf U \neq \mathsf X$) and that the action of $\mathsf X$ on $\mathfrak A$ is monotone continuous (which is the same as lower semicontinuous when $\mathsf X$ is finite).
 Kirchberg also assumes that $\mathfrak A$ is exact, which is not needed above.
\end{remark}


\subsection{The corona factorisation property}

Checking that an extension is $\mathsf X$-purely large can be very hard in general. This was the motivation, in the classical case,
to introduce the \emph{corona factorisation property} (see \cite{Kucerovsky-largeFredholm} and \cite{KucerovskyNg-corona}).
Recall that a stable $C^\ast$-algebra $\mathfrak B$ has the corona factorisation property if any norm-full projection 
$P\in \multialg{\mathfrak B}$ is properly infinite, or equivalently, Murray--von Neumann equivalent to $1_{\multialg{\mathfrak B}}$. 

Under the additional assumption on $\mathfrak B$ that $\mathfrak B(\mathsf U)$ has the corona factorisation property for each $\mathsf U\in \mathbb O(\mathsf X)$, we may, in the unital case, 
replace the assumption of $\mathsf X$-purely largeness in Theorem \ref{t:purelylargefinite} with the condition that the extension is $\mathsf X$-full. 
This is in general much easier to verify. In the non-unital case however, fullness of the extension will not be enough to guarantee absorption. This is the motivation for the following definition.

\begin{definition}
 Let $\mathfrak A$ be a lower semicontinuous $\mathsf X$-$C^\ast$-algebra and $\mathfrak D$ be a unital $\mathsf X$-$C^\ast$-algebra such that $\mathfrak D(\mathsf X) = \mathfrak D$.
 An $\mathsf X$-equivariant $\ast$-homomorphism $\phi \colon \mathfrak A \to \mathfrak D$ is said to be \emph{unitisably full} (or unitisably $\mathsf X$-full)
 if the induced unital, $\mathsf X$-equivariant $\ast$-homomorphism $\phi^\dagger \colon \mathfrak A^\dagger \to \mathfrak D$ is full.
 
 An $\mathsf X$-equivariant extension of $\mathfrak A$ by $\mathfrak B$ is said to be \emph{unitisably full} if the Busby map $\tau \colon \mathfrak A \to \corona{\mathfrak B}$ is unitisably full.
\end{definition}

Clearly a unitisably full $\ast$-homomorphism is necessarily full. Moreover, if $\mathfrak A$ is unital, then $\phi \colon \mathfrak A \to \mathfrak D$ is unitisably full if and only if $\phi$ is full and $1_\mathfrak{D} - \phi(1_\mathfrak{A})$ is
full in $\mathfrak D$. In particular, a unital $\ast$-homomorphism can never be unitisably full.

The following shows that the $\mathsf X$-purely large extensions we are interested in are (unitisably) $\mathsf X$-full.

\begin{lemma}\label{l:Xfullext}
 Let $\mathsf X$ be a finite space, and let $\mathfrak e: 0 \to \mathfrak B \to \mathfrak E \to \mathfrak A \to 0$ be an extension of $\mathsf X$-$C^\ast$-algebras.
 Suppose that $\mathfrak A$ is a separable, lower semicontinuous $\mathsf X$-$C^\ast$-algebra, and $\mathfrak B$ is a
 stable, $\mathsf X$-$\sigma$-unital, continuous $\mathsf X$-$C^\ast$-algebra.
 
 If $\mathfrak e$ is $\mathsf X$-purely large, then it is $\mathsf X$-full. Moreover, if $\mathfrak e$ in addition absorbs the zero extension, then it is unitisably $\mathsf X$-full.
\end{lemma}
\begin{proof}
 If $\mathfrak e$ is $\mathsf X$-purely large, then so is its sum with the zero extension. Also, $\mathfrak{e}$ is $\mathsf{X}$-full if and only if $\mathfrak{e} \oplus 0$ is $\mathsf{X}$-full.
 Thus it suffices to show that when $\mathfrak e$ absorbs the zero extension, then it is unitisably $\mathsf X$-full, so we should show that the unitised extension
 $\mathfrak e^\dagger$ is $\mathsf X$-full. Note that $\mathfrak e^\dagger$ is $\mathsf X$-purely large by Proposition \ref{p:plpunitimpliesnonunit}.
 
 Let $\mathscr C := CP_\rnuc(\mathsf X; \mathfrak A^\dagger, \mathfrak B)$. By Corollary \ref{c:Xfullmap} there is a unitally $\mathscr C$-absorbing $\ast$-homo\-morphism $\Phi\colon \mathfrak A^\dagger \to \multialg{\mathfrak B}$ weakly in $\mathscr C$
 which is $\mathsf X$-full.
 Since $\corona{\mathfrak B}(\mathsf U) = \pi (\multialg{\mathfrak B}(\mathsf U))$ for all $\mathsf U\in \mathbb O(\mathsf X)$, it follows that $\pi \circ \Phi \colon \mathfrak A^\dagger \to \corona{\mathfrak B}$ is $\mathsf X$-full.
 Moreover, $\pi \circ \Phi$ is the Busby map of a trivial, unital $\mathscr C$-extension $\mathfrak f$, and thus $\mathfrak e^\dagger$ absorbs $\mathfrak f$.
 Since the Cuntz sum of an $\mathsf X$-equivariant map with an $\mathsf X$-full map is clearly $\mathsf X$-full, it follows that $\mathfrak e^\dagger$ is an $\mathsf X$-full extension.
\end{proof}

If we assume that $\mathfrak B(\mathsf U)$ has the corona factorisation property for every $\mathsf U\in \mathbb O(\mathsf X)$ then we also get a converse of the above lemma.
We will need a few intermediate results.

Recall, that if $\mathfrak J$ is a two-sided, closed ideal in $\mathfrak B$, then there is a canonical injective $\ast$-homomorphism $\iota \colon \multialg{\mathfrak B, \mathfrak J} \hookrightarrow \multialg{\mathfrak J}$.
By the diagram
\[
 \xymatrix{
 0 \ar[r] & \mathfrak J \ar@{=}[d] \ar[r] & \multialg{\mathfrak B, \mathfrak J} \ar[d]^{\iota} \ar[r] & \corona{\mathfrak B, \mathfrak J} \ar[r] \ar@{.>}[d]^{\overline \iota} & 0 \\
 0 \ar[r] & \mathfrak J \ar[r] & \multialg{\mathfrak J} \ar[r]^\pi & \corona{\mathfrak J} \ar[r] & 0
 }
\]
which has exact rows, there is an induced $\ast$-homomorphism $\overline  \iota \colon \corona{\mathfrak B, \mathfrak J} \to \corona{\mathfrak J}$. A diagram chase above shows that $\overline \iota$ is injective.

\begin{lemma}\label{l:coronafullher}
 Let $\mathfrak B$ be a stable, $\sigma$-unital $C^\ast$-algebra, and $\mathfrak J$ be a two-sided, closed ideal in $\mathfrak B$ which contains a full element. 
 Then $\overline \iota(\corona{\mathfrak B, \mathfrak J})$ is a full hereditary $C^\ast$-subalgebra of $\corona{\mathfrak J}$.
\end{lemma}
\begin{proof}
 It follows from Lemma \ref{l:multifullher} that $\iota(\multialg{\mathfrak B, \mathfrak J})$ is a full hereditary $C^\ast$-subalgebra of $\multialg{\mathfrak J}$.
 Since any $\ast$-epimorphism maps full hereditary $C^\ast$-subalgebras onto full hereditary $C^\ast$-subalgebras, 
 $\overline \iota (\corona{\mathfrak B, \mathfrak J}) = \pi(\iota(\multialg{\mathfrak B, \mathfrak J}))$ is a full hereditary $C^\ast$-subalgebra of $\corona{\mathfrak J}$.
\end{proof}

\begin{lemma}\label{l:coronafullelement}
 Let $\mathfrak B$ be a stable $\sigma$-unital $C^\ast$-algebra, and let $\mathfrak J$ be a $\sigma$-unital, two-sided, closed ideal in $\mathfrak B$.
 Let $m \in \multialg{\mathfrak B, \mathfrak J}$. Then $m$ is full in $\multialg{\mathfrak B, \mathfrak J}$ if and only if $\pi(m)$ is full in $\corona{\mathfrak B, \mathfrak J}$.
\end{lemma}
\begin{proof}
 One direction is trivial, so suppose that $\pi(m)$ is full in $\corona{\mathfrak B, \mathfrak J}$. By Lemma \ref{l:coronafullher}, $\overline \iota (\pi(m))$ is full in $\corona{\mathfrak J}$.
 Since $\pi(\iota(m)) = \overline \iota(\pi(m))$, it follows from \cite[Proposition 3.3]{KucerovskyNg-corona} (clearly one does not need the positivity requirement in the proposition), that $\iota(m)$ is full in $\multialg{\mathfrak J}$.
 It follows from Lemma \ref{l:multifullher} that $m$ is full in $\multialg{\mathfrak B, \mathfrak J}$.
\end{proof}

\begin{lemma}\label{l:fullpropinfproj}
 Let $\mathfrak B$ be a stable $\sigma$-unital $C^\ast$-algebra, and let $\mathfrak J$ be a $\sigma$-unital, two-sided, closed ideal in $\mathfrak B$.  Then $\multialg{\mathfrak B, \mathfrak J}$ has a full, properly infinite projection.
 \end{lemma}

\begin{proof}
 The Hilbert $\mathfrak B$-module $E:= \mathfrak{J}$ is countably generated since $\mathbb K(E) \cong \mathfrak{J}$ is $\sigma$-unital. 
 By Kasparov's stabilisation theorem and since $\mathcal H_\mathfrak{B} \cong \mathfrak B$, there is a projection $P' \in \mathbb B(\mathfrak B) \cong \multialg{\mathfrak B}$ such that $P' \mathfrak B \cong E$.
 Let $\eta\colon \mathbb B(\mathfrak B) \xrightarrow{\cong} \multialg{\mathfrak B}$ be the canonical isomorphism. 
 It easily follows that $\overline{\mathfrak B \eta(P') \mathfrak B} = \mathfrak J$. It follows from Lemma \ref{l:multifullher} that any infinite repeat $P$ of $\eta(P')$ is full in $\multialg{\mathfrak B, \mathfrak J}$.
 Since $P$ is an infinite repeat, and is non-zero, it is unitarily equivalent any Cuntz sum $P \oplus P$, and is thus properly infinite.
 \end{proof}

\begin{proposition}\label{p:cfpimpliesplp}
 Let $\mathsf X$ be a finite space, and let $\mathfrak e: 0 \to \mathfrak B \to \mathfrak E \xrightarrow{p} \mathfrak A \to 0$ be an extension of $\mathsf X$-$C^\ast$-algebras.
 Suppose that $\mathfrak A$ is a separable, lower semicontinuous $\mathsf X$-$C^\ast$-algebra, and $\mathfrak B$ is a
 stable, $\mathsf X$-$\sigma$-unital, continuous $\mathsf X$-$C^\ast$-algebra. 
 Suppose that $\mathfrak B(\mathsf U)$ has the corona factorisation property for each $\mathsf U\in \mathbb O(\mathsf X)$.
 If $\mathfrak e$ is $\mathsf X$-full, then it is $\mathsf X$-purely large.
\end{proposition}
\begin{proof}
Let $x\in \mathfrak E$, and let $\mathsf U \in \mathbb O(\mathsf X)$ be such that $p(x)$ is $\mathsf U$-full. 
We must show that $\overline{x \mathfrak B x^\ast}$ contains a $\sigma$-unital, stable $C^\ast$-subalgebra $\mathfrak D$ which is full in $\mathfrak B(\mathsf U)$.
Let $y$ be a strictly positive contraction in $\mathfrak E(\mathsf U)$, which exists since $0 \to \mathfrak B(\mathsf U) \to \mathfrak E(\mathsf U) \to \mathfrak A(\mathsf U) \to 0$ is an extension with $\mathfrak A(\mathsf U)$ and $\mathfrak B(\mathsf U)$
$\sigma$-unital. Then $p(y)$ is strictly positive in $\mathfrak A(\mathsf U)$. Since $p(x) \in \mathfrak A(\mathsf U)$, $p(xy)$ generates the same two-sided, closed ideal as $p(x)$, which implies that $p(xy)$ is $\mathsf U$-full.

Let $\sigma\colon \mathfrak E \to \multialg{\mathfrak B}$ be the induced $\ast$-homomorphism. By Proposition \ref{p:basicXext}, $\sigma$ is $\mathsf X$-equivariant so $\sigma(xy)\in \multialg{\mathfrak B}(\mathsf U)$.
Since the extension is $\mathsf X$-full, $\pi(\sigma(xy)) = \tau(p(xy))$ is full in $\corona{\mathfrak B}(\mathsf U)$, so by Lemma \ref{l:coronafullelement} $\sigma(xy)$ is full in $\multialg{\mathfrak B}(\mathsf U)$.
By Lemma \ref{l:multifullher}, $\iota(\sigma(xy))$ is full in $\multialg{\mathfrak B(\mathsf U)}$ so by \cite[Lemma 3.2]{KucerovskyNg-corona} there is a full, multiplier projection $P \in \multialg{\mathfrak B(\mathsf U)}$ such that
\[
 \mathfrak D := \overline{xy \mathfrak B (xy)^\ast} = \overline{\iota(\sigma(xy)) \mathfrak B(\mathsf U) \iota(\sigma(xy))^\ast} \cong P\mathfrak B(\mathsf U) P.
\]
Since $\mathfrak B(\mathsf U)$ has the corona factorisation property it follows that $\mathfrak D$ is stable, and it is clearly full in $\mathfrak B(\mathsf U)$. 
Thus $\mathfrak e$ is $\mathsf X$-purely large since $\mathfrak D$ is obviously $\sigma$-unital and a $C^\ast$-subalgebra of $\overline{x \mathfrak B x^\ast}$.
\end{proof}

In order to get some results in the spirit of Voiculescu's Weyl--von Neumann type absorption theorem \cite{Voiculescu-WvN}, we will need the following definitions.

\begin{definition}
 Let $\mathfrak D$ be an $\mathsf X$-$C^\ast$-algebra. We say that $\mathfrak C$ is an \emph{$\mathsf X$-$C^\ast$-subalgebra of $\mathfrak D$} if $\mathfrak C$ is an $\mathsf X$-$C^\ast$-algebra such that the underlying $C^\ast$-algebra is a $C^\ast$-subalgebra
 of $\mathfrak D$, and $\mathfrak C(\mathsf U) = \mathfrak C \cap \mathfrak D(\mathsf U)$ for every $\mathsf U\in \mathbb O(\mathsf X)$.
 
 Moreover, we say that $\mathfrak C$ is a \emph{full} $\mathsf X$-$C^\ast$-subalgebra of $\mathfrak D$ if $\mathfrak C$ is an $\mathsf X$-$C^\ast$-subalgebra of $\mathfrak D$ and the inclusion map $\mathfrak C \hookrightarrow \mathfrak D$ is $\mathsf X$-full.
\end{definition}

We can now prove another main theorem. This theorem says that under the assumption of the corona factorisation property on each two-sided, closed ideal $\mathfrak B(\mathsf U)$, we get much nicer absorption results.

\begin{theorem}\label{t:cfpequiv}
 Let $\mathsf X$ be a finite space and $\mathfrak B$ be a stable, $\mathsf X$-$\sigma$-unital, continuous $\mathsf X$-$C^\ast$-algebra.\footnote{I.e.~a $C^\ast$-algebra over $\mathsf X$, since $\mathsf X$ is finite.} The following are equivalent.
 \begin{itemize}
  \item[$(i)$] $\mathfrak B(\mathsf U)$ has the corona factorisation property for every $\mathsf U \in \mathbb O(\mathsf X)$,
  \item[$(ii)$] for every separable, unital, lower semicontinuous $\mathsf X$-$C^\ast$-algebra $\mathfrak A$ with $\mathfrak B(\mathsf U_{1_\mathfrak{A}}) = \mathfrak B$, it holds that every unital, 
  $\mathsf X$-full extension of $\mathfrak A$ by $\mathfrak B$, absorbs any trivial, unital weakly residually $\mathsf X$-nuclear extension of 
  $\mathfrak A$ by $\mathfrak B$, 
  \item[$(iii)$] for every separable, lower semicontinuous $\mathsf X$-$C^\ast$-algebra $\mathfrak A$, it holds that every unitisably $\mathsf X$-full extension of $\mathfrak A$ by $\mathfrak B$, absorbs any trivial, weakly residually $\mathsf X$-nuclear 
  extension of $\mathfrak A$ by $\mathfrak B$,
  \item[$(iv)$] for every separable, lower semicontinuous $\mathsf X$-$C^\ast$-algebra $\mathfrak A$, $KK_\rnuc^1(\mathsf X;\mathfrak A,\mathfrak B)$ is (isomorphic to) the group of unitisably $\mathsf X$-full, weakly residually $\mathsf X$-nuclear extensions 
  of $\mathfrak A$ by $\mathfrak B$ under multiplier unitary equivalence.
  \item[$(v)$] for any separable, unital $\mathsf X$-$C^\ast$-subalgebra $\iota \colon \mathfrak C \hookrightarrow \multialg{\mathfrak B}$ for which $\mathfrak C/\mathfrak B$ is a full $\mathsf X$-$C^\ast$-subalgebra of $\corona{\mathfrak B}$, it holds that any
  unital, weakly residually $\mathsf X$-nuclear map $\phi \colon \mathfrak C \to \multialg{\mathfrak B}$ for which $\phi(\mathfrak C \cap \mathfrak B) = 0$, is strongly asymptotically dominated by $\iota$.
  \item[$(vi)$] for any separable, unital $\mathsf X$-$C^\ast$-subalgebra $\iota \colon \mathfrak C \hookrightarrow \multialg{\mathfrak B}$ for which $\mathfrak C/\mathfrak B$ is a full $\mathsf X$-$C^\ast$-subalgebra of $\corona{\mathfrak B}$, it holds for any
  unital, weakly residually $\mathsf X$-nuclear $\ast$-homomorphism $\Phi \colon \mathfrak C \to \multialg{\mathfrak B}$ for which $\Phi(\mathfrak C \cap \mathfrak B) = 0$, that $\iota \oplus \Phi \sim_{as} \iota$.
 \item[$(vii)$] for any separable, unital, lower semicontinuous $\mathsf X$-$C^\ast$-algebra $\mathfrak A$ with $\mathfrak B(\mathsf U_{1_\mathfrak{A}}) = \mathfrak B$, it holds that every unital, 
  $\mathsf X$-full $\ast$-homomorphism $\Phi \colon \mathfrak A \to \multialg{\mathfrak B}$ is unitally weakly residually $\mathsf X$-nuclearly absorbing.
 \end{itemize}
\end{theorem}
\begin{proof}
 $(i)\Rightarrow (ii)$: Let $\mathfrak A$ be as in $(ii)$ and let $\mathfrak e: 0 \to \mathfrak B \to \mathfrak E \to \mathfrak A \to 0$ be a unital, $\mathsf X$-full extension. 
 By Theorem \ref{t:purelylargefinite} it suffices to show that $\mathfrak e$ is
 $\mathsf X$-purely large. This follows from Proposition \ref{p:cfpimpliesplp}.
 
 $(ii)\Rightarrow (iii)$: Let $\mathfrak A$ be as in $(iii)$ and let $\mathfrak e : 0 \to \mathfrak B \to \mathfrak E \to \mathfrak A \to 0$ be a unitisably $\mathsf X$-full extension. By Proposition \ref{p:plpunitimpliesnonunit} it suffices to show that the
 unitised $\mathsf X$-equivariant extension $\mathfrak e^\dagger$ absorbs any trivial, unital weakly residually $\mathsf X$-nuclear extension of $\mathfrak A^\dagger$ by $\mathfrak B$. This follows from $(ii)$ since $\mathfrak e^\dagger$ is $\mathsf X$-full.
 
 $(iii) \Rightarrow (iv)$: It follows from \cite{Kirchberg-non-simple} that $KK_\rnuc^1(\mathsf X; \mathfrak A, \mathfrak B)$ is canonically isomorphic to the group of absorbing $CP_\rnuc(\mathsf X; \mathfrak A, \mathfrak B)$-extensions
 under multiplier unitary equivalence. By absorbing we mean extensions which absorb any trivial $CP_\rnuc(\mathsf X; \mathfrak A, \mathfrak B)$-extensions.
 By $(iii)$ these extensions are exactly the unitisably $\mathsf X$-full, weakly residually $\mathsf X$-nuclear extensions of $\mathfrak A$ by $\mathfrak B$.
 
 $(iv) \Rightarrow (i)$: Suppose that $\mathsf U\in \mathbb O(\mathsf X)$ such that $\mathfrak B(\mathsf U)$ does not have the corona factorisation property, and let $\mathbb C_\mathsf{U}$ be the $\mathsf X$-$C^\ast$-algebra with underlying
 $C^\ast$-algebra $\mathbb C$, and action of $\mathsf X$ given by
 \[
  \mathbb C_\mathsf{U}(\mathsf V) = \left\{ \begin{array}{rcl} \mathbb C, & \text{ if }\mathsf U \subset \mathsf V \\ 0 , & \text{ otherwise.} \end{array} \right.
 \]
 Clearly $\mathbb C_\mathsf{U}$ is lower semicontinuous. 
 
 By Lemma \ref{l:fullpropinfproj}, there there exists a multiplier projection $P$ which is properly infinite and full in $\multialg{\mathfrak B, \mathfrak B(\mathsf U)}$.
 Now $P \mathfrak B P \cong \mathfrak B(\mathsf U)$ and by \cite[Lemma 11]{Kucerovsky-ideals}
 $P\multialg{\mathfrak B}P \cong \multialg{P\mathfrak B P}$ canonically. Since $\mathfrak B(\mathsf U)$ does not have the corona factorisation property it follows that there is a projection $Q_1$ in $P\multialg{\mathfrak B}P$ which is not properly infinite,
 but which is full in $P \multialg{\mathfrak B}P$, and thus also full in $\multialg{\mathfrak B, \mathfrak B(\mathsf U)}$. Let $Q=Q_1 \oplus 0$ be a Cuntz sum in $\multialg{\mathfrak B}$. 
 Then $Q$ is a full projection in $\multialg{\mathfrak B, \mathfrak B(\mathsf U)}$ which is not properly infinite. 
 Moreover, since $0\oplus 1 \leq 1-Q$ it follows that $1-Q$ is a full, properly infinite projection in $\multialg{\mathfrak B}$.
 
 Let $\mathfrak e$ denote the trivial extension of $\mathbb C_\mathsf{U}$ by $\mathfrak B$ with $\mathsf X$-equivariant splitting $\sigma \colon \mathbb C_\mathsf{U} \to \multialg{\mathfrak B}$ given by $\sigma(1) = Q$. 
 Since $Q$ is full in $\multialg{\mathfrak B, \mathfrak B(\mathsf U)}$ and $1-Q$ is full in $\multialg{\mathfrak B}$ it easily follows that $\mathfrak e$ is unitisably $\mathsf X$-full. 
 Let $\mathscr C = CP_\rnuc(\mathsf X; \mathbb C_\mathsf{U},\mathfrak B)$.
 By Corollary \ref{c:Xfullmap} there exists a trivial $\mathscr C$-extension $\mathfrak f$ which absorbs any trivial $\mathscr C$-extension, so this is $\mathsf X$-purely large by Theorem \ref{t:purelylargefinite} and thus 
 unitisably $\mathsf X$-full by Lemma \ref{l:Xfullext}. 
 Any trivial $\mathscr C$-extension induces the zero element in $KK^1_\rnuc(\mathsf X; \mathbb C_\mathsf{U},\mathfrak B)$, so in particular $[\mathfrak e]= [\mathfrak f]=0$.
 It suffices to show that $\mathfrak e$ is does \emph{not} absorb every trivial $\mathscr C$-extension, so assume for contradiction that $\mathfrak e$ \emph{does} absorb every trivial $\mathscr C$-extension.
 
 Consider the trivial extension $\mathfrak e_0$ with $\mathsf X$-equivariant splitting $\sigma_0 \colon \mathbb C_\mathsf{U} \to \multialg{\mathfrak B}$ given by $\sigma_0(1) = P$, where $P$ is as above. 
 Since $\mathfrak e$ absorbs $\mathfrak e_0$ we may find a unitary $u\in \multialg{\mathfrak B}$ such that $u^\ast Q u - (Q \oplus P) \in \mathfrak B$. 
 Note that $u^\ast Q u - (Q \oplus P) \in \mathfrak B \cap \multialg{\mathfrak B, \mathfrak B(\mathsf U)} = \mathfrak B(\mathsf U)$.
 It is easily seen that $Q\oplus P\in\multialg{\mathfrak B, \mathfrak B(\mathsf U)}$ and dominates the properly infinite projection $0\oplus P$, 
 which is full in $\multialg{\mathfrak B, \mathfrak B(\mathsf U)}$. It follows that $Q\oplus P$ is properly infinite. 
 
 Let $\iota\colon \multialg{\mathfrak B} \to \multialg{\mathfrak B(\mathsf U)}$ be the canonical $\ast$-homomorphism by considering each multiplier of $\mathfrak B$ as a multiplier of $\mathfrak B(\mathsf U)$. 
 Note that the restriction of $\iota$ to $\multialg{\mathfrak B, \mathfrak B(\mathsf U)}$ is the canonical injective $\ast$-homomorphism $\multialg{\mathfrak B, \mathfrak B(\mathsf U)} \hookrightarrow \multialg{\mathfrak B(\mathsf U)}$
 considered in Section \ref{s:fullmaps}.
 By Lemma \ref{l:multifullher} $\iota(u^\ast Q u)$ and $\iota(Q\oplus P)$ are full in $\multialg{\mathfrak B(\mathsf U)}$. 
 Since $Q \oplus P$ is properly infinite in $\multialg{\mathfrak B,\mathfrak B(\mathsf U)}$ and since the restriction of $\iota$ to $\multialg{\mathfrak B, \mathfrak B(\mathsf U)}$ is injective, 
 it follows that $\iota(Q \oplus P)$ is properly infinite in $\multialg{\mathfrak B(\mathsf U)}$. 
 Thus there is an isometry $v\in \multialg{\mathfrak B(\mathsf U)}$ with $vv^\ast = \iota(Q \oplus P)$. Since $\iota(u^\ast Q u) - \iota(Q \oplus P) = u^\ast Q u - (Q \oplus P) \in \mathfrak B(\mathsf U)$ we get
 $b:=v^\ast \iota(u^\ast Q u) v - 1_{\multialg{\mathfrak B(\mathsf U)}} \in \mathfrak B(\mathsf U)$. Since $\mathfrak B(\mathsf U)$ is stable we may find an isometry $w\in \multialg{\mathfrak B(\mathsf U)}$ such that $w^\ast b w < 1$. Hence 
 \[
  \| (\iota(u)vw)^\ast \iota(Q) (\iota(u)vw) - 1_{\multialg{\mathfrak B(\mathsf U)}} \| < 1.
 \]
 This implies that $1_{\multialg{\mathfrak B(\mathsf U)}}$ is Murray--von Neumann subequivalent to $\iota(Q)$ and thus $\iota(Q)$ is properly infinite. 
 Since $\iota(\multialg{\mathfrak B, \mathfrak B(\mathsf U)}$ is a hereditary $C^\ast$-subalgebra of $\multialg{\mathfrak B}$ by Lemma \ref{l:multiher}, this implies that $Q$ is properly infinite, which is a contradiction to how we chose $Q$.
 Hence $(iv) \Rightarrow (i)$.
 
 $(ii) \Rightarrow (v)$: Let $\mathfrak C$ and $\phi$ be as in $(v)$, and let $\mathfrak A := \mathfrak C/\mathfrak B$. Since the inclusion $\tau \colon \mathfrak A \hookrightarrow \corona{\mathfrak B}$ is unital and $\mathsf X$-full, 
 it induces an $\mathsf X$-full, unital extension of $\mathfrak A$ by $\mathfrak B$, say $\mathfrak e$, which has extension algebra $\mathfrak E = \mathfrak C + \mathfrak B$. Since $\mathfrak A$ is an $\mathsf X$-$C^\ast$-subalgebra of $\corona{\mathfrak B}$
 it follows that $\mathsf U_{1_{\mathfrak{A}}} = \mathsf U_{1_{\corona{\mathfrak B}}}$, so $\mathfrak B(\mathsf U_{1_\mathfrak{A}}) = \mathfrak B$. Hence $\mathscr C = CP_\rnuc(\mathsf X; \mathfrak A, \mathfrak B)$ is non-degenerate, and by Theorem \ref{t:absrep}
 there exists a unitally $\mathscr C$-absorbing $\ast$-homomorphism $\Psi \colon \mathfrak A \to \multialg{\mathfrak B}$ weakly in $\mathscr C$. 
 
 By $(ii)$ the extension $\mathfrak e$ absorbs the extension with Busby map $\pi \circ \Psi \colon \mathfrak A \to \corona{\mathfrak B}$. So there are $\mathcal O_2$-isometries $s_1, s_2$ and a unitary $u$ in $\multialg{\mathfrak B}$ such that
 \[
  u^\ast x u - (s_1 xs_1^\ast + s_2 \Psi (\pi(x)) s_2^\ast) \in \mathfrak B,
 \]
 for all $x\in \mathfrak E$. Note that $\phi$ induces a map $\tilde \phi\colon \mathfrak A \to \multialg{\mathfrak B}$ which is weakly in $\mathscr C$, and thus $\tilde \phi$ is strongly asymptotically dominated by $\Psi$. 
 Let $(v_t)_{t\in [1,\infty)}$ be a family of isometries implementing this strong approximate domination, and let $w_t := u s_2 v_t$. Then $\| w_t^\ast b w_t \| \to 0$ for all $b\in \mathfrak B$ and
 $w_t^\ast c w_t - \phi (c)$ is in $\mathfrak B$ for all $c \in \mathfrak C$ and $t\in [1,\infty)$, and tends to zero as $t\to \infty$.

 $(v) \Rightarrow (vi)$: This follows from Theorem \ref{t:domabs}.
  
 $(ii) \Rightarrow (vii)$: Let $\mathfrak A$ and $\Phi$ be as in $(vii)$. Let $\Psi \colon \mathfrak A \to \multialg{\mathfrak B}$ be a unital, weakly residually $\mathsf X$-nuclear $\ast$-homomorphism.
 Let $\Psi_\infty$ be an infinite repeat of $\Psi$. Since $\Phi$ is $\mathsf X$-full, it follows that the extension with Busby map $\pi \circ \Phi$ is $\mathsf X$-full. 
 Thus by $(ii)$ it absorbs the extension with Busby map $\pi \circ \Psi_\infty$, so there is a unitary $U\in \multialg{\mathfrak B}$ such that
 \[
  U^\ast(\Phi(a) \oplus \Psi_\infty(a) ) U - \Phi(a) \in \mathfrak B,
 \]
 for all $a\in \mathfrak A$. By Theorem \ref{t:domabs} $\Phi$ absorbs $\Psi_\infty$ and thus also $\Psi$.
 
 $(vii) \Rightarrow (i)$: Suppose that $\mathsf U\in \mathbb O(\mathsf X)$ such that $\mathfrak B(\mathsf U)$ does not have the corona factorisation property. Let $P,Q\in \multialg{\mathfrak B, \mathfrak B(\mathsf U)}$ be the same projections as when proving
 $(iv) \Rightarrow (i)$, and let $\mathfrak C := C^\ast(Q, 1_{\multialg{\mathfrak B}})$ be the given $\mathsf X$-$C^\ast$-subalgebra of $\multialg{\mathfrak B}$. 
 There is an isomorphism of $\mathsf X$-$C^\ast$-algebras $\mathbb C_\mathsf{U}^\dagger \to \mathfrak C$ given
 by $\mathbb C_\mathsf{U}^\dagger = \mathbb C^2 \ni (\lambda, \mu) \mapsto \lambda Q + \mu(1-Q)$. 
 Thus, by how we chose $Q$, $\mathfrak C$ is a full $\mathsf X$-$C^\ast$-subalgebra of $\multialg{\mathfrak B}$. Moreover, $\mathfrak C$ is lower semicontinuous.
  
 Let $\Phi \colon \mathbb C_\mathsf{U}^\dagger \to \multialg{\mathfrak B}$ be given by $(\lambda, \mu) \mapsto \lambda P + \mu(1-P)$. This is a unital,
 $\mathsf X$-equivariant $\ast$-homomorphism which is obviously residually $\mathsf X$-nuclear. Statement $(vii)$ would imply that there is a unitary $u\in \multialg{\mathfrak B}$ such that
 $\| u^\ast Q u - P\oplus Q\| < 1$ which would imply that $Q$ and $P\oplus Q$ are Murray--von Neumann equivalent. As seen in the proof of $(iv) \Rightarrow (i)$, $P \oplus Q$ is properly infinite. Since $Q$ was chosen not to be properly infinite, this implies
 that $(vii)$ does not hold, and thus $(vii) \Rightarrow (i)$.
 
 $(vi) \Rightarrow (i)$: The proof is identical to that of $(vii) \Rightarrow (i)$, where one simply notes that $\mathfrak C/\mathfrak B$ is a full $\mathsf X$-$C^\ast$-subalgebra of $\corona{\mathfrak B}$ and $\mathfrak C \cap \mathfrak B = 0$.
\end{proof}

As a corollary more suited for classification, we obtain the following.

\begin{corollary}\label{c:nucquot}
  Let $\mathsf X$ be a finite space and $\mathfrak B$ be a stable, $\mathsf X$-$\sigma$-unital, continuous $\mathsf X$-$C^\ast$-algebra such that $\mathfrak B(\mathsf U)$ has the corona factorisation property for each $\mathsf U\in \mathbb O(\mathsf X)$. 
  Let $\mathfrak A$ be a separable, nuclear, lower semicontinuous $\mathsf X$-$C^\ast$-algebra. 
  
  Then $KK^1(\mathsf X;\mathfrak A,\mathfrak B)$ is the group of unitisably $\mathsf X$-full extensions of $\mathfrak A$ by $\mathfrak B$ under multiplier unitary equivalence.
  
  If, in addition, $\mathfrak A/\mathfrak A(\mathsf U)$ is non-zero and non-unital for every $\mathsf U \in \mathbb O(\mathsf X) \setminus \{ \mathsf X\}$, 
  then $KK^1(\mathsf X;\mathfrak A,\mathfrak B)$ is the group of $\mathsf X$-full extensions of $\mathfrak A$ by $\mathfrak B$ under multiplier unitary equivalence.
\end{corollary}
\begin{proof}
 We have that $KK^1(\mathsf X; \mathfrak A, \mathfrak B) = KK_\rnuc^1(\mathsf X; \mathfrak A, \mathfrak B)$ when $\mathfrak A$ is nuclear. 
 Moreover, by Corollary \ref{c:liftingcor} any $\mathsf X$-equivariant extension of $\mathfrak A$ by $\mathfrak B$ has an $\mathsf X$-equivariant contractive c.p.~split, which is residually $\mathsf X$-nuclear since $\mathfrak A$ is nuclear. 
 Hence the first part follows from Theorem \ref{t:cfpequiv}.
 
 Suppose that $\mathfrak A/\mathfrak A(\mathsf U)$ is non-zero and non-unital for every $\mathsf U \in \mathbb O(\mathsf X) \setminus \{ \mathsf X\}$.
 Recall, that $KK_\rnuc^1(\mathsf X; \mathfrak A, \mathfrak B)$ is canonically isomorphic to the group of absorbing $CP_\rnuc(\mathsf X; \mathfrak A, \mathfrak B)$-extensions under multiplier unitary equivalence.
 Since any such extension absorbs an $\mathsf X$-full extension, it follows that any such extension is $\mathsf X$-full.
 By Proposition \ref{p:cfpimpliesplp} any $\mathsf X$-full extension is $\mathsf X$-purely large and thus it absorbs any trivial $CP_\rnuc(\mathsf X; \mathfrak A, \mathfrak B)$-extension by Corollary \ref{c:nonunitalquotient}, which finishes the proof.
\end{proof}


\section{A Weyl--von Neumann theorem a la Kirchberg}\label{s:KirchbergWvN}

In this section we prove a Weyl--von Neumann type theorem, in the sense of Kirchberg in \cite{Kirchberg-simple}, for $C^\ast$-algebras with a finite primitive ideal space.
We also obtain a related result regarding extensions by such $C^\ast$-algebras.

Recall, that an $\mathsf X$-$C^\ast$-algebra $\mathfrak B$ is called \emph{tight} if the action $\mathbb O(\mathsf X) \to \mathbb I(\mathfrak B)$ is a complete lattice isomorphism.

\begin{definition}
We say that an $\mathsf X$-$C^\ast$-algebra $\mathfrak B$ has a \emph{tight corona algebra}, if the induced action of $\mathsf X$ on $\corona{\mathfrak B}$ is tight.
\end{definition}

The following lemma shows that if $\mathfrak B$ is a stable and tight $\mathsf X$-$C^\ast$-algebra, then the action $\mathbb O(\mathsf X) \to \mathbb I(\corona{\mathfrak B})$ is a lattice embedding.

\begin{lemma}\label{l:coronaideals}
 Let $\mathfrak B$ be a stable $C^\ast$-algebra, and let $\mathfrak I$ and $\mathfrak J$ be two-sided, closed ideals in $\mathfrak B$.
 Then $\mathfrak I \subset \mathfrak J$ if and only if $\corona{\mathfrak B, \mathfrak I} \subset \corona{\mathfrak B, \mathfrak J}$.
\end{lemma}
\begin{proof}
 One implication is trivial. Suppose that $\corona{\mathfrak B, \mathfrak I} \subset \corona{\mathfrak B, \mathfrak J}$.
 Let $x\in \mathfrak I$, and let $t_1,t_2,\dots \in \multialg{\mathfrak B}$ be isometries such that $\sum_{n=1}^\infty t_nt_n^\ast$ converges strictly to $1$.
 Then 
 \[
  x_\infty = \sum_{n=1}^\infty t_n x t_n^\ast \in \multialg{\mathfrak B, \mathfrak I}.
 \]
 Since the inclusion $\multialg{\mathfrak B, \mathfrak J} \hookrightarrow \multialg{\mathfrak B, \mathfrak J} + \mathfrak B$ induces an isomorphism $\multialg{\mathfrak B, \mathfrak J}/\mathfrak J \cong \corona{\mathfrak B, \mathfrak J}$, 
 and since $\pi(x_\infty) \in \corona{\mathfrak B, \mathfrak J}$, there is a $y\in \mathfrak J$ such that $x_\infty + y \in \multialg{\mathfrak B, \mathfrak J}$. Thus, 
 \[
  t_n^\ast (x_\infty + y) t_n  = x + t_n^\ast y t_n \in \mathfrak B \cap \multialg{\mathfrak B, \mathfrak J} = \mathfrak J.
 \]
 Since $y \in \mathfrak B$, it follows that $t_n^\ast y t_n \to 0$ and thus $x \in \mathfrak J$. This implies that $\mathfrak I \subset \mathfrak J$.
\end{proof}

\begin{definition}
 Let $\mathfrak A$ and $\mathfrak B$ be $\mathsf X$-$C^\ast$-algebra with $\mathfrak A$ lower semicontinuous. A c.p.~map $\phi \colon \mathfrak A \to \multialg{\mathfrak B}$ is called \emph{weakly $\mathsf X$-full} if
 \[
  \overline{\mathfrak B \phi(a) \mathfrak B} = \mathfrak B( \mathsf U_a)
 \]
 for every $a\in \mathfrak A$.
\end{definition}

\begin{example}
 If $\mathfrak B$ is lower semicontinuous, and $\mathfrak A \subset \multialg{\mathfrak B}$ is an $\mathsf X$-$C^\ast$-subalgebra, then the inclusion is weakly $\mathsf X$-full.
\end{example}

Kirchberg proved the following theorem in \cite{Kirchberg-simple} in the case where $\mathsf X$ is the one-point space. That conditions $(i)$ and $(ii)$ below are equivalent in the one-point space case, was originally proven by Lin in 
\cite{Lin-contscalecorona}. This was also proven by Rørdam in \cite{Rordam-multialg} in the case where $\mathfrak B$ was assumed to contain a full projection.

\begin{theorem}\label{t:WvN}
 Let $\mathsf X$ be a finite space and let $\mathfrak B$ be a stable, $\mathsf X$-$\sigma$-unital, tight $\mathsf X$-$C^\ast$-algebra. Then the following are equivalent.
 \begin{itemize}
  \item[$(i)$] every simple subquotient of $\mathfrak B$ is purely infinite or isomorphic to $\mathbb K$,
  \item[$(ii)$] $\mathfrak B$ has a tight corona algebra,
  \item[$(iii)$] for every unital, separable $\mathsf X$-$C^\ast$-subalgebra $\iota \colon \mathfrak C \hookrightarrow \multialg{\mathfrak B}$ it holds that every unital, weakly residually $\mathsf X$-nuclear map 
  $\phi\colon \mathfrak C \to \multialg{\mathfrak B}$ such that $\phi(\mathfrak C \cap \mathfrak B) = 0$, is approximately dominated by $\iota$,
  \item[$(iv)$] for every unital, separable $\mathsf X$-$C^\ast$-subalgebra $\iota \colon \mathfrak C\hookrightarrow \multialg{\mathfrak B}$ it holds for every unital, weakly residually $\mathsf X$-nuclear $\ast$-homomorphism 
  $\Phi\colon \mathfrak C \to \multialg{\mathfrak B}$ such that $\Phi(\mathfrak C \cap \mathfrak B) = 0$, that $\iota \oplus \Phi \sim_{as} \iota$,
  \item[$(v)$] for every unital, separable, lower semicontinuous $\mathsf X$-$C^\ast$-algebra $\mathfrak A$, satisfying that $\mathfrak B(\mathsf U_{1_\mathfrak{A}}) = \mathfrak B$,
  it holds that any weakly $\mathsf X$-full $\ast$-homomorphism $\Phi \colon \mathfrak A \to \multialg{\mathfrak B}$ for which $\Phi(\mathfrak A) \cap \mathfrak B = 0$, is weakly residually $\mathsf X$-nuclearly absorbing.
 \end{itemize}
\end{theorem}
\begin{proof}
 $(i) \Rightarrow (ii)$: Suppose that $\mathfrak I,\mathfrak J$ are two-sided, closed ideals in $\mathfrak B$. 
 By Lemma \ref{l:coronaideals}, $\mathfrak I \subset \mathfrak J$ if and only if $\corona{\mathfrak B, \mathfrak I} \subset \corona{\mathfrak B , \mathfrak J}$. 
 Hence, if the map $\mathbb I (\mathfrak B) \ni \mathfrak J \mapsto \corona{\mathfrak B, \mathfrak J} \in \mathbb I(\corona{\mathfrak B})$ is a one-to-one correspondence, then it is a lattice bijection, and thus $\mathfrak B$ has a tight corona algebra.
 That the map is injective follows from Lemma \ref{l:coronaideals} since $\mathfrak B$ is stable. To show that it is surjective it suffices to show that
 $|\Prim \mathfrak B| = |\Prim \corona{\mathfrak B}|$, which we will prove by induction. 
 Note that this is equivalent to $2\cdot |\Prim \mathfrak B| = |\Prim \multialg{\mathfrak B}|$. 
 If $\mathfrak B$ is simple the result follows from \cite{Rordam-multialg}.
 For a general $\mathfrak B$ with $|\Prim \mathfrak B|$ finite, fix a maximal, two-sided, closed ideal $\mathfrak I$ in $\mathfrak B$. Note that $|\Prim \mathfrak I| = |\Prim \mathfrak B| -1$. 
 It follows from Lemma \ref{l:multifullher} that $\multialg{\mathfrak B, \mathfrak I}$ embeds canonically as a full hereditary $C^\ast$-subalgebra of $\multialg{\mathfrak I}$. Because of the short exact sequence
 \[
  0 \to \multialg{\mathfrak B, \mathfrak I} \to \multialg{\mathfrak B} \to \multialg{\mathfrak B/\mathfrak I} \to 0,
 \]
 it follows that $|\Prim \multialg{\mathfrak B}| = |\Prim \multialg{\mathfrak B, \mathfrak I}| + |\Prim \multialg{\mathfrak B/\mathfrak I}| = |\Prim \multialg{\mathfrak I}| + 2$. By an induction argument it follows that 
 $|\Prim \multialg{\mathfrak B}| = 2 \cdot (| \Prim \mathfrak B|-1) + 2 = |\Prim \mathfrak B|$ and thus $|\Prim \mathfrak B | = |\Prim \corona{\mathfrak B}|$. 
 
 $(ii) \Rightarrow (i)$: Again we show the result by induction on $|\Prim \mathfrak B|$.
 If $\mathfrak B$ is simple, then $\corona{\mathfrak B}$ is simple by $(ii)$, so by \cite[Theorem 3.8]{Lin-contscalecorona} $\mathfrak B$ is either purely infinite or $\mathbb K$.
 Let $n\in \mathbb N$ and suppose that for any space $\mathsf X$ with $|\mathsf X|< n$, we have $(ii)\Rightarrow (i)$. Suppose $|\mathsf X| = n$.
 For a general $\mathfrak B$ as in the theorem, let $\mathfrak I \subset \mathfrak J \subset \mathfrak B$ be two-sided, closed ideals such that $\mathfrak J/\mathfrak I$ is simple. 
 Suppose that $\mathfrak J \neq \mathfrak B$. Since $\mathfrak B$ is tight, there is a proper open subset $\mathsf Y$ of $\mathsf X$, such that $\mathfrak B(\mathsf Y) = \mathfrak J$.
 Then $\mathfrak J$ and $\corona{\mathfrak B, \mathfrak J}$ get an induced $\mathsf Y$-$C^\ast$-algebra structure by restricting the action of $\mathsf X$.
 Clearly $\mathfrak J$ and $\corona{\mathfrak B, \mathfrak J}$ are tight $\mathsf Y$-$C^\ast$-algebra and in particular $| \mathbb O(\mathsf Y)| = |\mathbb I(\corona{\mathfrak B, \mathfrak J})|$.
 By Lemma \ref{l:coronafullher}, $\corona{\mathfrak B, \mathfrak J}$ is a full hereditary $C^\ast$-subalgebra $\corona{\mathfrak J}$, and thus $\mathbb I(\corona{\mathfrak B, \mathfrak J}) \cong \mathbb I(\corona{\mathfrak J})$.
 Hence $|\mathbb O(\mathsf Y)| = |\mathbb I(\corona{\mathfrak J})|$. Since the map
 \[
  \mathbb O(\mathsf Y) \ni \mathsf U \mapsto \corona{\mathfrak J, \mathfrak J(\mathsf U)} \in \mathbb I(\corona{\mathfrak J})
 \]
 is an injective order embedding by Lemma \ref{l:coronaideals}, and since $| \mathbb O(\mathsf Y)| = |\mathbb I(\corona{\mathfrak J})|$, it follows that this action is tight. 
 Hence $\mathfrak J$ has a tight corona algebra. 
 Since $\mathfrak J \neq \mathfrak B$, it follows that $|\mathsf Y | < |\mathsf X|$, and by our induction hypothesis it follows that every simple subquotient of $\mathfrak J$, and thus $\mathfrak J/\mathfrak I$, is either purely infinite or $\mathbb K$.
 
 It remains to show the case when $\mathfrak J = \mathfrak B$. 
 We have a short exact sequence $0 \to \corona{\mathfrak B, \mathfrak I} \to \corona{\mathfrak B} \to \corona{\mathfrak B/\mathfrak I} \to 0$, 
 and by Lemma \ref{l:coronafullher}, $|\Prim \corona{\mathfrak B, \mathfrak I}| = |\Prim \corona{\mathfrak I}| = |\Prim \mathfrak B| - 1$.
 Thus
 \[
  |\Prim \mathfrak B| = |\Prim \corona{\mathfrak B}|  = |\Prim \mathfrak B| - 1 + |\Prim \corona{\mathfrak B/\mathfrak J}|
 \]
 so $\corona{\mathfrak B/\mathfrak J}$ is simple. It follows from the case $n=1$ that $\mathfrak B/\mathfrak J$ is purely infinite or $\mathbb K$.
 
 $(i) \Rightarrow (iii)$: It follows from \cite[Proposition 4.1.1]{Rordam-book-classification} that any simple, stable, purely infinite $C^\ast$-algebra has the corona factorisation property, and clearly $\mathbb K$ has the corona factorisation property.
 Hence, by \cite[Theorem 3.1]{KucerovskyNg-Sregularity} and the remark following that theorem, it follows that $\mathfrak B(\mathsf U)$ has the corona factorisation property for every $\mathsf U \in \mathbb O(\mathsf X)$.
 Let $\mathfrak C$ and $\phi$ be as in condition $(iii)$. By Theorem \ref{t:cfpequiv} it suffices to show that $\mathfrak C/\mathfrak B \subset \corona{\mathfrak B}$ is full $\mathsf X$-$C^\ast$-subalgebra. 
 Since any $\mathsf X$-$C^\ast$-subalgebra of a tight $\mathsf X$-$C^\ast$-algebra is full, this follows from $(i) \Leftrightarrow (ii)$.
 
 $(iii) \Rightarrow (iv)$: This follows from Theorem \ref{t:domabs}.
 
 $(iv) \Rightarrow (i)$: Suppose that $(i)$ does \emph{not} hold. Then there are two-sided, closed ideals $\mathfrak I \subset \mathfrak J$ in $\mathfrak B$ such that $\mathfrak J/\mathfrak I$ is simple, but is neither purely infinite nor $\mathbb K$.
 We will prove something slightly more general than is needed for the proof, so that the proof for $(v) \Rightarrow (i)$ carries over with the same proof. We need the following.
 
 Claim: there is a positive element $x\in \multialg{\mathfrak B, \mathfrak J}$ which is \emph{not} full, such that $C^\ast(x) \cap \mathfrak B = 0$. 
 
 Let us first finish the proof of $(iv) \Rightarrow (i)$ under the assumption of this claim, and then prove the claim afterwards.
 
 Let $h\in \mathfrak J$ be strictly positive. By replacing $x$ with a Cuntz sum $x \oplus h \oplus 0$, we may furthermore assume that $\overline{\mathfrak B x \mathfrak B} = \mathfrak J$ and $C^\ast(1,x) \cap \mathfrak B = 0$. 
 Let $\mathfrak C = C^\ast(1,x)$ which we consider as an $\mathsf X$-$C^\ast$-subalgebra of $\multialg{\mathfrak B}$, and let $\Phi \colon \mathfrak C \to \multialg{\mathfrak B}$ be the infinite repeat of the inclusion $\iota$.
 Since $\iota$ is $\mathsf X$-equivariant (by definition) so is $\Phi$, and since $\mathfrak C$ is nuclear, $\Phi$ is weakly residually $\mathsf X$-nuclear. Also $\Phi(\mathfrak C \cap \mathfrak B) = \Phi(0) = 0$.
 Suppose that $(iv)$ holds. Then $\iota \oplus \Phi \sim_{as} \iota$, so in particular $x$ and $x \oplus \Phi(x)$ generate the same two-sided, closed ideal in $\multialg{\mathfrak B}$. 
 However, since $x$ is not full in $\multialg{\mathfrak B, \mathfrak J}$, and $x\oplus \Phi(x)$ is full in $\multialg{\mathfrak B, \mathfrak J}$ by Lemma \ref{l:multifullher},
 we get a contradiction, so $(iv)$ can not hold. Thus it remains to prove the above claim.
 
 The proof of the claim is easy in the case where $\mathfrak I \neq 0$.  
 By Lemma \ref{l:fullpropinfproj}, we may find a (necessarily non-zero) projection $P$ which is full in $\multialg{\mathfrak B, \mathfrak I}$.
 Let $(t_n)$ be a sequence of isometries in $\multialg{\mathfrak B}$ such that $\sum t_n t_n^\ast$ converges strictly to $1$, and let $(\lambda_n)$ be a dense sequence in $[0,1]$.
 Let $y = \sum_{n=1}^\infty \lambda_n t_n Pt_n^\ast$. Then $y$ is a positive element in $\multialg{\mathfrak B, \mathfrak I}$ with $\spec(y) = [0,1]$ such that for any non-zero $h \in C_0((0,1])$, $h(y)$ is full in $\multialg{\mathfrak B, \mathfrak I}$.
 
 Let $\overline \iota \colon \corona{\mathfrak B, \mathfrak J} \to \corona{\mathfrak J}$ be the embedding from Lemma \ref{l:coronafullher}. Clearly $\overline \iota (\corona{\mathfrak B, \mathfrak I})$ is full in $\corona{\mathfrak J, \mathfrak I}$.
 We get that $\corona{\mathfrak J}/\corona{\mathfrak J, \mathfrak I} \cong \corona{\mathfrak J/\mathfrak I}$ is non-simple by the equivalence of $(i)$ and $(ii)$.
 Since $\overline \iota (\corona{\mathfrak B, \mathfrak J})$ is a full, hereditary $C^\ast$-subalgebra of $\corona{\mathfrak J}$ by Lemma \ref{l:coronafullelement},
 it follows that $\corona{\mathfrak B, \mathfrak J}/\corona{\mathfrak B, \mathfrak I} \cong \multialg{\mathfrak B, \mathfrak J}/(\multialg{\mathfrak B, \mathfrak I}+\mathfrak J)$ is not simple.
 Thus we may find a positive contraction $x' \in \multialg{\mathfrak B, \mathfrak J}$ which is not full, and which is not in $\multialg{\mathfrak B, \mathfrak I} + \mathfrak J$.
 Let $x$ be a Cuntz sum $x' \oplus y$. We clearly have that $x \in \multialg{\mathfrak B, \mathfrak J}$ is not full, and that $C^\ast(x) \cap \mathfrak B = 0$ by how we chose $y$. 
 This finishes the proof of the claim in the case where $\mathfrak I \neq 0$.
 
 Now suppose that $\mathfrak I =0$ and thus $\mathfrak J$ is simple.
 Since $\mathfrak J$ is stable, $\sigma$-unital and is neither purely infinite nor $\mathbb K$, $\multialg{\mathfrak J}$ contains a non-trivial ideal by \cite{Lin-contscalecorona}. 
 Thus we may find a positive contraction $m \in \multialg{\mathfrak J}$ which is not in $\mathfrak J$ and which is not full in $\multialg{\mathfrak J}$.
 Let $(y_n)$ be a countable approximate identity in $\mathfrak J$ such that $y_n y_{n+1} = y_{n+1} y_n = y_{n}$. 
 By passing to a subsequence we may assume that $\| (1-y_n) m y_k \| < 2^{-(n+k)}$ for $k<n$. Let $d_n = y_n - y_{n-1}$ where we define $y_0 = 0$.
 Exactly as in the proof of Lemma \ref{l:actmulticor} part $(c)$, it follows that
 \[
  m \in \overline{\multialg{\mathfrak J} (\sum_{n=1}^\infty d_n m d_n) \multialg{\mathfrak J}} + \mathfrak J.
 \]
 Note that all our sums converge in the strict topology. Let $m_0 := \sum_{n=1}^\infty d_{2n} m d_{2n}$ and $m_1 := \sum_{n=1}^\infty d_{2n-1} m d_{2n-1}$, and note that $\sum_{n=1}^\infty d_n m d_n = m_0 + m_1$.
 Since
 \[
  (\sum_{k=1}^\infty d_{2k}) m (\sum_{l=1}^\infty d_{2l}) = m_0 + \sum_{k=2}^\infty \sum_{l=1}^{k-1} (d_{2k} m d_{2l} + d_{2l} m d_{2k})
 \]
 one may use the norm estimate (exactly as in the proof of Lemma \ref{l:actmulticor} part $(c)$) to show that the latter sum above is in $\mathfrak J$. 
 Thus $m_0$, and similarly $m_1$, is in $\overline{\multialg{\mathfrak J} m \multialg{\mathfrak J}} + \mathfrak J$. It follows that $m_0 + m_1 + \mathfrak J$ generates the same two-sided, closed ideal in $\corona{\mathfrak J}$ as $m+ \mathfrak J$.
 Since $m_0$ and $m_1$ are positive, it follows that neither $m_0$ nor $m_1$ is full in $\multialg{\mathfrak J}$, and at least one of $m_0$ and $m_1$ is not in $\mathfrak J$. Assume without loss of generality that $m_0$ is not in $\mathfrak J$.
 
 Since $d_{2n}m d_{2n}$ and $d_{2k} m d_{2k}$ are positive and orthogonal for $k \neq n$, it follows that the sequence $(\| d_{2n} m d_{2n}\|)_{n=1}^\infty$ does not tend to zero, for otherwise $m_0 \in \mathfrak J$.
 Hence we may find $0<\epsilon < 1$ and an increasing sequence $(n(k))_{k=1}^\infty$ of positive integers, such that $\| d_{2n(k)} m d_{2n(k)}\| > \epsilon$ for all $k$.
 Let $f,g \colon [0,1] \to [0,1]$ be the continuous functions given by
 \[
  f(t) = \left\{ \begin{array}{ll}
                  0, & \text{ for } t=0 \\
                  \text{affine}, & \text{ for } 0 \leq t \leq \epsilon/2 \\
                  1, & \text{ for } \epsilon /2 \leq t \leq 1
                 \end{array} \right. ,\quad
  g(t) = \left\{ \begin{array}{ll}
                  0, & \text{ for } 0\leq t \leq \epsilon /2 \\
                  \text{affine}, & \text{ for } \epsilon/2 \leq t \leq 1 -  \epsilon/2 \\
                  1, & \text{ for } 1-\epsilon /2 \leq t \leq 1.
                 \end{array} \right.
 \]
 Then $f(m_0) = \sum_{n=1}^\infty f(d_{2n} m d_{2n})$, and $f g = g f = g$.
  
 It is well-known that if $\mathfrak D$ is a simple, $\sigma$-unital $C^\ast$-algebra not stably isomorphic to $\mathbb K$, then $\mathfrak D$ contains a positive element $d$ with $\spec(d) = [0,1]$.
 Since $\epsilon < \| d_{2n(k)} m d_{2n(k)}\| \leq 1$ for every $k$, it follows that that $g(d_{2n(k)} m d_{2n(k)})$ is non-zero. 
 Since $\mathfrak J$ is simple, $\sigma$-unital and not (stably) isomorphic to $\mathbb K$, we may for each $k$ find a positive element $x_k \in \overline{g(d_{2n(k)} m d_{2n(k)}) \mathfrak J g(d_{2n(k)} m d_{2n(k)})}$ such that $\spec(x_k)=[0,1]$.
 Since $fg=gf=g$ we have that $f(m_0)x_k = x_kf(m_0) = x_k$.
 Clearly $\sum_{k=1}^\infty x_k$ converges strictly to an element $x'$, such that $f(m_0) x' = x' f(m_0) = x'$.
 Since $f(m_0)$ is in the $C^\ast$-algebra generated by $m_0$, and is thus not full in $\multialg{\mathfrak J}$, it follows that $x'$ is not full in $\multialg{\mathfrak J}$.
 Since $x_k$ and $x_l$ are positive orthogonal elements, it follows that if $h \in C_0((0,1])$ then $h(x') = \sum_{k=1}^\infty h(x_k)$. Let $h \in C_0((0,1])$ be non-zero. 
 Since $\| h(x_k)\| = \| h\| >0$, it follows that $h(x')$ is not in $\mathfrak J$. 
 Thus $C^\ast(x') \cap \mathfrak J = 0$.
 
 By Kasparov's stabilisation theorem, we may find a projection $P\in \multialg{\mathfrak B}$ such that $P\mathfrak B P$ is full in $\mathfrak J$.
 By replacing $P$ with an infinite repeat, we may assume that $P\mathfrak B P \cong \mathfrak J$, and $P\multialg{\mathfrak B}P \cong \multialg{\mathfrak J}$ is a corner which is full in $\multialg{\mathfrak B, \mathfrak J}$ by Lemma \ref{l:multifullher}.
 Let $x\in \multialg{\mathfrak B, \mathfrak J}$ be the element corresponding to $x'\in \multialg{\mathfrak J}$ by the above isomorphism.
 Since $P\multialg{\mathfrak B}P$ is a full, hereditary $C^\ast$-subalgebra of $\multialg{\mathfrak B, \mathfrak J}$, and $x$ is not full in $P\multialg{\mathfrak B} P$, it follows that $x$ is not full in $\multialg{\mathfrak B, \mathfrak J}$.
 Also, since $P$ is a unit for $C^\ast(x)$, it follows that $C^\ast(x) \cap \mathfrak B = C^\ast(x) \cap P\mathfrak B P = 0$.
 This proves the claim in the case where $\mathfrak I = 0$, and thus finishes the proof of $(iv) \Rightarrow (i)$.
 
 $(i) \Rightarrow (v)$: Let $\mathfrak A$ and $\Phi$ be as in $(v)$. By the equivalence of $(i)$ and $(ii)$ it easily follows that $\pi \circ \Phi$ is $\mathsf X$-full. By Lemma \ref{l:coronafullelement} it follows that $\Phi$ is $\mathsf X$-full.
 Thus, as in $(i) \Rightarrow (iii)$, we obtain $(v)$ by Theorem \ref{t:cfpequiv}.
 
 $(v) \Rightarrow (i)$: The argument is identical to that of $(iv) \Rightarrow (i)$ with $\mathfrak C = \mathfrak A$.
\end{proof}

We also want results for extensions by $\mathsf X$-$C^\ast$-algebras with a tight corona algebra. We do this by introducing the following actions.

\begin{remark}[The induced action of an extension]\label{r:inducedact}
Let $\mathfrak e: 0 \to \mathfrak B \to \mathfrak E \xrightarrow{p} \mathfrak A \to 0$ be an extension of $C^\ast$-algebras, and let $\sigma\colon \mathfrak E \to \multialg{\mathfrak B}$ and $\tau \colon \mathfrak A \to \corona{\mathfrak B}$
be the induced $\ast$-homomorphisms. 
Suppose that $\mathfrak B$ is an $\mathsf X$-$C^\ast$-algebra. Then $\mathfrak e$ induces an $\mathsf X$-$C^\ast$-algebra structure on both $\mathfrak A$ and $\mathfrak E$ by
\begin{eqnarray*}
 \mathfrak A(\mathsf U) &=& \tau^{-1}(\corona{\mathfrak B,\mathfrak B(\mathsf U)}) \\
 \mathfrak E(\mathsf U) &=& \sigma^{-1}(\multialg{\mathfrak B, \mathfrak B(\mathsf U)})
\end{eqnarray*}
for $\mathsf U \in \mathbb O(\mathsf X)$.

There is an easy way of determining $\mathfrak E(\mathsf U)$. Let $\mathsf U \in \mathbb O(\mathsf X)$. 
We clearly have
\[
\mathfrak E(\mathsf{U}) = \overline{\sum_{\mathfrak J \triangleleft \mathfrak E, \mathfrak J \cap \mathfrak B \subset \mathfrak B(\mathsf U)} \mathfrak J}, 
\]
so $\mathfrak E(\mathsf{U})$ is the \emph{unique largest two-sided, closed ideal such that $\mathfrak E(\mathsf U) \cap \mathfrak B = \mathfrak B(\mathsf U)$}.
Thus, if $\mathfrak J$ is any two-sided, closed ideal in $\mathfrak E$ such that $\mathfrak J \cap \mathfrak B \subset \mathfrak B(\mathsf U)$, then $\mathfrak J \subset \mathfrak E(\mathsf U)$.

This description can often be used to determine the action on $\mathfrak E$,
and since $\mathfrak A(\mathsf U) = p(\mathfrak E(\mathsf U))$, we may use this description to determine $\mathfrak A(\mathsf U)$.
\end{remark}

In this paper, we will only consider the induced action when $\mathfrak B$ is a tight $\mathsf X$-$C^\ast$-algebra.

\begin{remark}\label{r:niceinducedact}
Let $\mathfrak e: 0 \to \mathfrak B \to \mathfrak E \to \mathfrak A \to 0$ be an extension of $C^\ast$-algebras. 
It follows from Lemma \ref{l:actmulticor} that the action of $\Prim \mathfrak B$ on $\mathfrak A$ is finitely lower semicontinuous and that the action on $\mathfrak E$ is lower semicontinuous. 

Note that if $\Prim \mathfrak B$ is finite, the by Corollary \ref{c:choieffros} this extension will have a $\Prim \mathfrak B$-equivariant c.p.~split if $\mathfrak A$ is nuclear and separable.
\end{remark}

\begin{remark}\label{r:constructprim}
 Recall, that when $0 \to \mathfrak B \to \mathfrak E \to \mathfrak A \to 0$ is an extension of $C^\ast$-algebras, then $\Prim \mathfrak B$ embeds canonically as an open subset of $\Prim \mathfrak E$, 
 such that $\Prim \mathfrak E \setminus \Prim \mathfrak B$ is canonically homeomorphic to $\Prim \mathfrak A$.
 Thus $\Prim \mathfrak E$ as a \emph{set} may be canonically identified with the disjoint union $\Prim \mathfrak A \sqcup \Prim \mathfrak B$.
 
 If the induced $\Prim \mathfrak B$-$C^\ast$-algebra structure on $\mathfrak A$ is lower semicontinuous, then we may retrieve the topology on $\Prim \mathfrak E$ from the $\Prim \mathfrak B$-$C^\ast$-algebra structure on $\mathfrak A$.
 
 In fact, since $\Prim \mathfrak E = \Prim \mathfrak A \sqcup \Prim \mathfrak B$ any open set will be of the form $\mathsf U \sqcup \mathsf V$, where $\mathsf U \in \mathbb O(\Prim \mathfrak A)$ and $\mathsf V \in \mathbb O(\Prim \mathfrak B)$.
 Let $\mathfrak J_\mathsf{U}$ be the closed, two-sided ideal in $\mathfrak A$ corresponding to $\mathsf U$. 
 If the action of $\Prim \mathfrak B$ on $\mathfrak A$ is lower semicontinuous, then there is a unique smallest open set $\mathsf W_\mathsf{U}$ in $\Prim \mathfrak B$ such that $\mathfrak J_\mathsf{U} \subset \mathfrak A(\mathsf W_\mathsf{U})$.
 It is easily verified that $\mathsf U \sqcup \mathsf V$ is open in $\Prim \mathfrak E$ if and only if $\mathsf W_\mathsf{U} \subset \mathsf V$.
\end{remark}

We will provide two examples. The first example shows that the induced action on the extension algebra and the quotient need not be finitely upper semicontinuous. 
The second example shows that the induced action on the quotient is not necessarily lower semicontinuous.

\begin{example}
Let $\mathfrak E$ be a $C^\ast$-algebra containing exactly three non-trivial, two-sided, closed ideals $\mathfrak B_1$, $\mathfrak B_2$ and $\mathfrak B = \mathfrak B_1 + \mathfrak B_2$ such that $\mathfrak B_1 \cap \mathfrak B_2 = 0$.
Let $\mathfrak A = \mathfrak E/\mathfrak B$ and consider the extension $0 \to \mathfrak B \to \mathfrak E \to \mathfrak A \to 0$.
Since $\mathfrak B \cong \mathfrak B_1 \oplus \mathfrak B_2$, the primitive ideal space $\mathsf X$ of $\mathfrak B$ is homeomorphic to the two-point discrete space $\{ 1,2\}$ with open sets $\mathbb O(\mathsf X) = \{ \emptyset , \{1\}, \{2\}, \mathsf X\}$.
The induced actions on $\mathfrak A$ and $\mathfrak E$ can be determined as in Remark \ref{r:inducedact} and are
\begin{eqnarray*}
 \mathfrak E(\emptyset) = 0, \quad \mathfrak E(\{1\}) = \mathfrak B_1, \quad \mathfrak E(\{ 2\} ) = \mathfrak B_2, \quad \mathfrak E(\mathsf X) = \mathfrak E \\
 \mathfrak A(\emptyset) =0, \quad \mathfrak A(\{ 1\}) = 0, \quad \mathfrak A(\{2\}) = 0, \quad \mathfrak A(\mathsf X) = \mathfrak A.
\end{eqnarray*}
These actions are clearly not finitely upper semicontinuous, since this would imply that $\mathfrak A(\{1\}) + \mathfrak A(\{ 2\}) = \mathfrak A$ and $\mathfrak E(\{1\}) + \mathfrak E(\{2\}) = \mathfrak E$.
\end{example}

\begin{example}
Let $\mathfrak e$ be the extension
\[
 0 \to C_0((0,1]) \to C([0,1]) \to \mathbb C \to 0.
\]
Then the action of $(0,1]$ on $\mathbb C$ induced by $\mathfrak e$ is given by
\[
 \mathbb C(\mathsf U) = \left\{ \begin{array}{ll} 0, & \text{ if } \mathsf U \subset [\epsilon ,1] \text{ for some }\epsilon \in (0,1] \\ \mathbb C, & \text{ otherwise.} \end{array} \right.
\]
Since $\bigcap_{n\in \mathbb N} \mathbb C((0,1/n)) = \mathbb C$, it follows that this action is not lower semicontinuous.
\end{example}

In classification theory it is often desirable to classify tight $\mathsf X$-$C^\ast$-algebras. 
If $0 \to \mathfrak B \to \mathfrak E \xrightarrow{p} \mathfrak A \to 0$ is an extension of $C^\ast$-algebras, then there are induced actions of $\Prim \mathfrak E$ on $\mathfrak A$ and $\mathfrak B$. 
These are simply given by $\mathfrak A (\mathsf U) = p(\mathfrak E(\mathsf U))$ and $\mathfrak B(\mathsf U) = \mathfrak E(\mathsf U) \cap \mathfrak B$, for $\mathsf U\in \mathbb O(\Prim \mathfrak E)$. 

\begin{warning}
 It is important to specify which of these two actions we are using on $\mathfrak A$. E.g.~since $\mathbb O(\Prim \mathfrak B) \subset \mathbb O(\Prim \mathfrak E)$, then $\mathfrak A(\Prim \mathfrak B) = \mathfrak A$ if we use the $\Prim \mathfrak B$ action,
 but $\mathfrak A(\Prim \mathfrak B) = 0$ if we use the $\Prim \mathfrak E$ action.
 
 However, $\mathfrak B(\mathsf U)$ for $\mathsf U\in \mathbb O(\Prim \mathfrak B)$ does not depend on the choice of action.
\end{warning}

The following lemma says that considering this action, or the action induced by the ideal $\mathfrak B$ as above, is essentially the same.

\begin{lemma}\label{l:tightidealvstightext}
 Let $0 \to \mathfrak B \to \mathfrak E \to \mathfrak A \to 0$ be an extension of $C^\ast$-algebras. Then
 \[
  CP(\Prim \mathfrak B; \mathfrak A, \mathfrak B) = CP(\Prim \mathfrak E; \mathfrak A , \mathfrak B).
 \]
\end{lemma}
\begin{proof}
 To avoid confusion we write $\mathfrak A_\mathfrak{B}$ when we mean $\mathfrak A$ with the $\Prim \mathfrak B$ action, and $\mathfrak A_\mathfrak{E}$ when we mean $\mathfrak A$ with the $\Prim \mathfrak E$ action.

 Let $\phi \in CP(\Prim \mathfrak E; \mathfrak A , \mathfrak B)$ and $\mathsf U\in \mathbb O(\Prim \mathfrak B) \subset \mathbb O(\Prim \mathfrak E)$. 
 Let $\mathsf U_\mathfrak{B} \in \mathbb O(\Prim \mathfrak E)$ such that $\mathfrak E(\mathsf U_\mathfrak B) = \mathfrak B$, and define
 \[
  \mathsf V_\mathsf{U} = \bigcup_{\mathsf V \in \mathbb O(\Prim \mathfrak E), \mathsf V \cap \mathsf U_\mathfrak{B} = \mathsf U \cap \mathsf U_{\mathfrak B}} \mathsf V.
 \]
 Then $\mathsf V_\mathsf{U}$ is the largest open subset of $\Prim \mathfrak E$ such that $\mathsf V_\mathsf{U} \cap \mathsf U_\mathfrak{B} = \mathsf U \cap \mathsf U_{\mathfrak B}$.
 
 Let $\sigma \colon \mathfrak E \to \multialg{\mathfrak B}$ and $\tau \colon \mathfrak A \to \corona{\mathfrak B}$ be the canonical maps, and let $a\in \mathfrak A_\mathfrak B (\mathsf U)$.
 Then $a$ lifts to $x\in \mathfrak E$ such that $\sigma (x) \in \multialg{\mathfrak B, \mathfrak B(\mathsf U)}$. As $\sigma^{-1}(\multialg{\mathfrak B, \mathfrak B(\mathsf U)}) \cap \mathfrak B = \mathfrak B(\mathsf U) = \mathfrak E(\mathsf U) \cap \mathfrak B$,
 it follows that $x \in \mathfrak E(\mathsf V_\mathsf{U})$. 
 Hence $a \in \mathfrak A_\mathfrak{E}(\mathsf V_\mathsf{U})$ and thus $\phi(a) \in \mathfrak B(\mathsf V_\mathsf{U}) = \mathfrak B(\mathsf U)$, which implies that $\phi \in CP(\Prim \mathfrak B; \mathfrak A, \mathfrak B)$.
 
 Now, let $\phi \in CP(\Prim \mathfrak B; \mathfrak A, \mathfrak B)$, $\mathsf U\in \mathbb O(\Prim \mathfrak E)$ and $a \in \mathfrak A_\mathfrak{E}(\mathsf U) = p(\mathfrak E(\mathsf U))$. 
 As $\sigma(\mathfrak E(\mathsf U)) \subset \multialg{\mathfrak B, \mathfrak B(\mathsf U)}$ it follows that $\tau (a) \in \corona{\mathfrak B, \mathfrak B( \mathsf U)} = \corona{\mathfrak B, \mathfrak B( \mathsf{U} \cap \mathsf{U}_{\mathfrak{B}} ) }$ 
 and thus $a \in \mathfrak A_\mathfrak{B}(\mathsf U \cap \mathsf U_\mathfrak{B})$.
 Hence $\phi(a) \in \mathfrak B(\mathsf U \cap \mathsf U_\mathfrak{B}) = \mathfrak B(\mathsf U)$, so $\phi \in CP(\Prim \mathfrak E; \mathfrak A, \mathfrak B)$. 
\end{proof}

\begin{remark}\label{r:KKgroups}
 Let $\mathfrak A$ and $\mathfrak B$ be separable $C^\ast$-algebras with actions of $\mathsf X$. The Kasparov group $KK(\mathsf X; \mathfrak A, \mathfrak B)$ is constructed exactly as in the classical case, 
 but by \emph{only} considering Kasparov $\mathfrak A$-$\mathfrak B$-modules 
 \[
 (E, \phi \colon \mathfrak A \to \mathbb B(E), F)
 \]
  for which the c.p.~map
 \[
  a \mapsto \langle x, \phi(a) x\rangle
 \]
 is $\mathsf X$-equivariant for every $x\in E$, i.e.~$\phi$ is weakly in $CP(\mathsf X; \mathfrak A, \mathfrak B)$ by Lemma \ref{p:starhomweakly}. 
 Note that this condition depends \emph{only} on the closed operator convex cone $CP(\mathsf X; \mathfrak A, \mathfrak B)$. 
 Thus if one has a different space $\mathsf Y$ which acts on $\mathfrak A$ and $\mathfrak B$ such that $CP(\mathsf X; \mathfrak A, \mathfrak B) = CP(\mathsf Y; \mathfrak A, \mathfrak B)$, then 
 \[
  KK(\mathsf X; \mathfrak A, \mathfrak B) = KK(\mathsf Y; \mathfrak A, \mathfrak B).
 \]
 
 In particular, it follows from Lemma \ref{l:tightidealvstightext} that if $0 \to \mathfrak B \to \mathfrak E \to \mathfrak A \to 0$ is an extension of separable $C^\ast$-algebra, then
 \[
  KK^1(\Prim \mathfrak B; \mathfrak A, \mathfrak B) = KK^1(\Prim \mathfrak E; \mathfrak A, \mathfrak B).
 \]
\end{remark}

\begin{remark}\label{r:extprim}
 Recall from Remark \ref{r:constructprim}, that we may identify $\Prim \mathfrak E$ and $\Prim \mathfrak A \sqcup \Prim \mathfrak B$ as sets in a canonical way.
 Thus if $0 \to \mathfrak B \to \mathfrak E_i \to \mathfrak A \to 0$ are extensions of $C^\ast$-algebras for $i=1,2$, then by the above identification, it makes sense to consider the identity map $\Prim \mathfrak E_1 = \Prim \mathfrak E_2$.
 Obviously, this map is a homeomorphism if and only if there exists a homeomorphism $\Prim \mathfrak E_1 \xrightarrow{\cong} \Prim \mathfrak E_2$ which acts as the identity on the canonical subsets $\Prim \mathfrak A$ and $\Prim \mathfrak B$.
 So when this is the case, the action of $\Prim \mathfrak E_1$ and of $\Prim \mathfrak E_2$ on $\mathfrak A$ and $\mathfrak B$ respectively, are the same.
\end{remark}

The following should be thought of as a much weaker criterion than absorbing the zero extension, when one wants to appeal to Theorem \ref{t:purelylargefinite}.

\begin{definition}
 Let $\mathfrak e : 0 \to \mathfrak B \to \mathfrak E \to \mathfrak A \to 0$ be an extension of $C^\ast$-algebras and $\sigma \colon \mathfrak E \to \multialg{\mathfrak B}$ be the induced $\ast$-homomorphism. 
 We will say that $\mathfrak e$ is \emph{strongly non-unital} if $\mathfrak E /\mathfrak E(\mathfrak J)$ is non-unital for all $\mathfrak J \in \mathbb I(\mathfrak B)$ where $\mathfrak E(\mathfrak J) = \sigma^{-1}(\multialg{\mathfrak B, \mathfrak J})$.
\end{definition}

Note in particular, that if $\mathfrak A$ has no unital quotients, e.g.~if $\mathfrak A$ is stable, then any extension of $\mathfrak A$ is strongly non-unital.

\begin{theorem}\label{t:classtightcorona}
 Let $\mathfrak e_i : 0 \to \mathfrak B \to \mathfrak E_i \to \mathfrak A \to 0$ be strongly non-unital extensions of $C^\ast$-algebras for $i=1,2$. Suppose that $\mathfrak A$ is separable and nuclear, that $\mathfrak B$ is stable, has finitely many two-sided, closed ideals, 
 which are all $\sigma$-unital, and that $\mathfrak B$ has a $\Prim \mathfrak B$-tight corona algebra. The following are equivalent.
\begin{itemize}
 \item[$(i)$] $\mathfrak e_1$ and $\mathfrak e_2$ are equivalent,
 \item[$(ii)$] $\mathfrak e_1$ and $\mathfrak e_2$ induce the same action of $\Prim \mathfrak B$ on $\mathfrak A$, and 
 \[
 [\mathfrak e_1]=[\mathfrak e_2] \in KK^1(\Prim \mathfrak B; \mathfrak A , \mathfrak B),
 \]
 \item[$(iii)$] under the canonical identification of $\Prim \mathfrak E_i = \Prim \mathfrak A \sqcup \Prim \mathfrak B$, the identity map $\Prim \mathfrak E_1 =  \Prim \mathfrak E_2$ is a homeomorphism, and
 \[
  [\mathfrak e_1] = [\mathfrak e_2] \in KK^1(\Prim \mathfrak E_1; \mathfrak A, \mathfrak B).
 \]
\end{itemize}
\end{theorem}

\begin{proof}
 $(i) \Rightarrow (ii)$: Clear.
 
 $(ii) \Rightarrow (i)$: Suppose that $\mathfrak e_1$ and $\mathfrak e_2$ induce the same action of $\mathsf X := \Prim \mathfrak B$ on $\mathfrak A$. Since $\mathfrak B$ has a tight corona algebra, both extensions are $\mathsf X$-full.
 As in the proof of $(i) \Rightarrow (iii)$ in Theorem \ref{t:WvN}, all two-sided, closed ideals in $\mathfrak B$ have the corona factorisation property.
 Hence by Corollary \ref{c:nucquot} it suffices to show that the extensions are unitisably $\mathsf X$-full. 
 
 Let $\tau$ be the Busby map of $\mathfrak e_1$. Then $\mathfrak e_1$ is unitisably $\mathsf X$-full if and only if $1_{\corona{\mathfrak B}} - \tau(a)$ is full in $\corona{\mathfrak B}$ for all $a\in \mathfrak A$. Suppose for contradiction that there
 is an $a\in \mathfrak A$ such that $1-\tau(a)$ is not full in $\corona{\mathfrak B}$. Since $\mathfrak B$ has a tight corona algebra, $1-\tau(a)$ is full in $\corona{\mathfrak B, \mathfrak J}$ for some proper two-sided, closed ideal $\mathfrak J$ in $\mathfrak B$.
 Lifting $1-\tau(a)$ to $m\in \multialg{\mathfrak B, \mathfrak J}$, $1-m$ is a lift of $\tau(a)$. Hence there is an $x\in \mathfrak E_1$ with $\sigma(x) = 1-m$, and $1-\sigma(x) \in \multialg{\mathfrak B, \mathfrak J}$, 
 where $\sigma \colon \mathfrak E_1 \to \multialg{\mathfrak B}$ is the canonical map. Hence $(1-x) \mathfrak E_1 \subset \mathfrak E_1(\mathfrak J)$.
Note that if $x \in \mathfrak E_1(\mathfrak J)$, then $1 = \sigma(x) + m \in \multialg{\mathfrak B, \mathfrak J}$ which contradicts $\mathfrak J \neq \mathfrak B$. In particular, $\mathfrak E_1(\mathfrak J) \neq \mathfrak E_1$ so $\mathfrak E_1/\mathfrak E_1(\mathfrak J)$ is non-zero.
Now $(1-x) \mathfrak E_1 \subset \mathfrak E_1(\mathfrak J)$ implies that
 $x + \mathfrak E_1(\mathfrak J)$ is a unit for $\mathfrak E_1/\mathfrak E_1(\mathfrak J)$ which is non-zero. This contradicts the strong non-unitality of $\mathfrak e_1$, and thus $\mathfrak e_1$, and similarly $\mathfrak e_2$, is unitisably $\mathsf X$-full.
 
 $(ii) \Leftrightarrow (iii)$: Follows immediately from Lemma \ref{l:tightidealvstightext}, Remark \ref{r:KKgroups} and Remark \ref{r:extprim}.
\end{proof}


\section{Applications}\label{s:applications}

\subsection{The general machinery}
We will use our result in the previous section to classify certain extensions of $C^\ast$-algebras. This classifies all real rank zero extensions of AF algebras by strongly classifiable purely infinite algebras. 
In spirit this means, that given any separable, nuclear, real rank zero $C^\ast$-algebra $\mathfrak A$ and a two-sided, closed ideal $\mathfrak I$ such that $\mathfrak A/\mathfrak I$ is AF, 
and $\mathfrak I$ is a strongly classifiable purely infinite algebra with finitely many two-sided, closed ideals, then $\mathfrak A$ is classified by a $K$-theoretic invariant.

\begin{notation}
 Let $\mathsf Y$ be a finite $T_0$-space. For any point $y\in \mathsf Y$ there is a smallest open set containing $y$, which we denote by $\mathsf U^y$. Note that $(\mathsf U^y)_{y\in \mathsf Y}$ is a basis for the topology $\mathbb O(\mathsf Y)$.
 
 There is a partial order on $\mathsf Y$ given by $x \leq y :\Leftrightarrow \mathsf U^y \subset \mathsf U^x$.
 
 We let $\mathsf Y$-$\mathscr C^\ast$-$cont$ denote the category of continuous $\mathsf Y$-$C^\ast$-algebras with $\mathsf Y$-equivariant $\ast$-homomorphisms as morphisms. 

 Similarly, we let $\mathsf X$-$\mathscr C^\ast$-$lsc$ denote the category of lower semicontinuous $\mathsf X$-$C^\ast$-algebras, with $\mathsf X$-equivariant $\ast$-homomorphisms as morphisms.
\end{notation}

Note that for a finite $T_0$-space $\mathsf Y$, the category $\mathsf Y$-$\mathscr C^\ast$-$cont$ is equivalent to the category of $C^\ast$-algebras over $\mathsf Y$.

\begin{definition}
 Let $\mathsf P$ be a finite partially ordered set. We let $\mathbb Z \mathsf P$ denote the free abelian group generated by elements $i^q_p$ for pairs $(p,q)$ with $p\leq q$, equipped with the bilinear multiplication $i_p^q i_r^s = \delta_{qr} i_p^s$.
 
 We let $\Mod (\mathbb Z \mathsf P)$ denote the category of $\mathbb Z/2$-graded right $\mathbb Z \mathsf P$-modules.

 Define the functor $\mathsf Y K \colon \mathsf Y$-$\mathscr C^\ast$-$cont \to \Mod(\mathbb Z \mathsf Y^{\op})$ by
 \[
  \mathsf Y K(\mathfrak A) = \bigoplus_{y\in \mathsf Y} K_\ast(\mathfrak A(\mathsf U^y))
 \]
 where $i^y_x$ acts by the map $K_\ast(\mathfrak A(\mathsf U^x)) \to K_\ast(\mathfrak A(\mathsf U^y))$ induced by the inclusion $\mathfrak A(\mathsf U^x) \hookrightarrow \mathfrak A(\mathsf U^y)$ on the direct summand corresponding to $x$, 
 and acts as $0$ on all other direct summands.
 
 Similarly we define the functor $\mathbb OK \colon \mathsf X$-$\mathscr C^\ast$-$lsc \to \Mod (\mathbb Z \mathbb O(\mathsf X))$ by
 \[
  \mathbb OK(\mathfrak A) = \bigoplus_{\mathsf U\in \mathbb O(\mathsf X)} K_\ast(\mathfrak A(\mathsf U))
 \]
 where $i_{\mathsf U}^{\mathsf V}$ acts by the map $K_\ast(\mathfrak A(\mathsf U)) \to K_\ast(\mathfrak A(\mathsf V))$ induced by the inclusion $\mathfrak A(\mathsf U) \hookrightarrow \mathfrak A(\mathsf V)$ on the direct summand corresponding to $\mathsf U$, 
 and acts as $0$ on all other direct summands.
\end{definition}

\begin{definition}
 Let $\mathsf X$ be a finite $T_0$-space and $\mathfrak A$ be a lower semicontinuous $\mathsf X$-$C^\ast$-algebra. We say that $\mathfrak A$ has \emph{vanishing boundary maps} if the induced map
 \[
  K_{\ast}(\mathfrak A(\mathsf U)) \to K_\ast(\mathfrak A)
 \]
 is injective for each open subset $\mathsf U$ of $\mathsf X$. If the induced map
 \[
  K_\ast\left(\sum_{k=1}^n \mathfrak A(\mathsf V_k)\right) \to K_\ast(\mathfrak A)
 \]
 is injective for any $\mathsf V_1,\dots,  \mathsf V_n \in \mathbb O(\mathsf X)$, then we say that $\mathfrak A$ has \emph{strongly vanishing boundary maps}.
\end{definition}

Clearly, if $\mathfrak A$ is a continuous $\mathsf Y$-$C^\ast$-algebra with vanishing boundary maps, then it also has strongly vanishing boundary maps.

Let $\mathcal B$ denote the \emph{bootstrap class} in $KK$-theory, i.e.~the class of all separable $C^\ast$-algebras satisfying the universal coefficient theorem (UCT) of Rosenberg and Schochet \cite{RosenbergSchochet-UCT}.
In \cite{MeyerNest-bootstrap} Meyer and Nest introduce an analogues bootstrap class $\mathcal B(\mathsf Y)$ for continuous $\mathsf Y$-$C^\ast$-algebras. 
Instead of recalling the definition of $\mathcal B(\mathsf Y)$, we will simply mention, that if $\mathfrak A$ is a \emph{nuclear}, continuous $\mathsf Y$-$C^\ast$-algebra, 
then $\mathfrak A$ is in $\mathcal B(\mathsf Y)$ if and only if $\mathfrak A(\mathsf U^y)$ is in $\mathcal B$ for every $y\in \mathsf Y$.

Recall, that since we are considering finite $T_0$-spaces, the notions of continuous $\mathsf Y$-$C^\ast$-algebras and $C^\ast$-algebras over $\mathsf Y$ are essentially the same. We will need the following result of Bentmann.

\begin{theorem}[\cite{Bentmann-vanishingbdry}]\label{t:BentmannUCT}
 Let $\mathsf Y$ be a finite $T_0$-space and $\mathfrak A$ and $\mathfrak B$ be separable, continuous $\mathsf Y$-$C^\ast$-algebras. 
 Assume that $\mathfrak A$ has vanishing boundary maps and is in $\mathcal B(\mathsf Y)$. Then there is a natural short exact sequence
 \[
  \Ext_{\Mod(\mathbb Z\mathsf Y^\op)}^1(\mathsf Y K(\mathfrak A), \mathsf Y K(\mathfrak B)) \hookrightarrow KK^1(\mathsf Y; \mathfrak A, \mathfrak B) \twoheadrightarrow \Hom_{\Mod(\mathbb Z\mathsf Y^\op)}(\mathsf YK(\mathfrak A), \mathsf YK(\mathfrak B)[1]).
 \]
\end{theorem}

We will use this result to obtain a UCT with $\mathbb OK$ for lower semicontinuous $\mathsf X$-$C^\ast$-algebras. 
The following lemma is what allows us to consider lower semicontinuous $\mathsf X$-$C^\ast$-algebras instead of continuous $\mathsf Y$-$C^\ast$-algebras.

\begin{lemma}\label{l:lsctocont}
 Let $\mathsf X$ be a finite $T_0$-space. Let $\mathsf Y = \mathbb O(\mathsf X)$ and equip $\mathsf Y$ with the topology which has the sets
 \[
  \mathsf U^{\mathsf V} := \{ \mathsf W \in \mathsf Y : \mathsf W \subset \mathsf V\}
 \]
 as a basis, for $\mathsf V\in \mathsf Y$. Then $\mathsf U^\mathsf{V}$ is the smallest open set containing $\mathsf V$, and there is an isomorphism of categories 
 \[
  F : \mathsf X\textrm{-}\mathscr C^\ast\textrm{-}lsc \to \mathsf Y\textrm{-}\mathscr C^\ast\textrm{-}cont.  
 \]
 Moreover, $\mathsf Y^\op = \mathbb O(\mathsf X)$ as partially ordered sets, and $\mathsf YK \circ F = \mathbb OK$.
\end{lemma}
\begin{proof}
 Suppose that $\mathsf U \in \mathbb O(\mathsf Y)$ is an open set containing $\mathsf V\in \mathsf Y$. There are $\mathsf V_1,\dots,\mathsf V_n\in \mathsf Y$ such that $\mathsf U = \bigcup_{k=1}^n \mathsf U^{\mathsf V_n}$.
 It follows that $\mathsf V \subset \mathsf V_k$ for some $k$, and thus $\mathsf U^\mathsf{V} \subset \mathsf U^{\mathsf V_k}$. Hence $\mathsf U^\mathsf{V} \subset \mathsf U$. 
 Since $\mathsf V \in \mathsf U^{\mathsf V}$ it follows that $\mathsf U^\mathsf{V}$ is the smallest open subset of $\mathsf Y$ containing $\mathsf V$.
 
 Let $\mathfrak A$ be a lower semicontinuous $\mathsf X$-$C^\ast$-algebra. We let $F(\mathfrak A)(\mathsf U^{\mathsf V}) = \mathfrak A(\mathsf V)$ for $\mathsf V\in \mathbb O(\mathsf X)$.
 Then we have
 \[
  \sum_{\mathsf V \in \mathsf Y} F(\mathfrak A)(\mathsf U^{\mathsf V}) = \sum_{\mathsf V \in \mathbb O(\mathsf X)} \mathfrak A(\mathsf V) = \mathfrak A(\mathsf X) = \mathfrak A,
 \]
 and
 \[
  F(\mathfrak A)(\mathsf U^{\mathsf V}) \cap F(\mathfrak A)(\mathsf U^{\mathsf W}) = \mathfrak A(\mathsf V) \cap \mathfrak A(\mathsf W) = \mathfrak A(\mathsf V \cap \mathsf W) = F(\mathfrak A)(\mathsf U^{\mathsf V \cap \mathsf W}) = 
  \sum_{\mathsf Z \in \mathsf U^{\mathsf V \cap \mathsf W}} F(\mathfrak A)(\mathsf U^{\mathsf Z}).
 \]
 for all $\mathsf V, \mathsf W \in \mathsf Y$. 
 Since we clearly have $\mathsf U^{\mathsf V} \cap \mathsf U^{\mathsf W} = \mathsf U^{\mathsf V \cap \mathsf W}$ it follows from \cite[Lemma 2.35]{MeyerNest-bootstrap} that $F(\mathfrak A)$ uniquely determines a continuous $\mathsf Y$-$C^\ast$-algebra structure
 by the above assignment. It is easily seen that this induces the desired functor. The functor is an isomorphism with inverse $G$ defined by $G(\mathfrak A)(\mathsf V) = \mathfrak A(\mathsf U^{\mathsf V})$.
 
 That $\mathsf Y^\op = \mathbb O(\mathsf X)$ as partially ordered sets, and that $\mathsf YK \circ F = \mathbb OK$, is obvious since $\mathsf U^\mathsf{V}$ is the smallest open set containing $\mathsf V$.
\end{proof}

\begin{remark}\label{r:KKlsc}
 As in the above lemma, it follows that the closed operator convex cones $CP(\mathsf X; \mathfrak A, \mathfrak B) = CP(\mathsf Y; F(\mathfrak A), F(\mathfrak B))$ are equal.
 Hence it follows from Remark \ref{r:KKgroups} that 
 \[
  KK(\mathsf X; \mathfrak A, \mathfrak B) = KK(\mathsf Y; F(\mathfrak A), F(\mathfrak B)).  
 \]

 This also shows, that if we construct the category $\mathsf X$-$\mathfrak{KK}$-$lsc$ with objects being separable, lower semicontinuous $\mathsf X$-$C^\ast$-algebras and the morphisms are the $KK(\mathsf X)$-elements, then this category is isomorphic
 to the category $\mathfrak{KK}(\mathsf Y)$. Thus we simply define the bootstrap class $\mathcal B(\mathsf X)_{lsc}$, to be the induced bootstrap class. 
 If $\mathfrak A$ is nuclear then $\mathfrak A$ is in $\mathcal B(\mathsf X)_{lsc}$ if and only if $\mathfrak A(\mathsf U)$ is in $\mathcal B$ for every $\mathsf U\in \mathbb O(\mathsf X)$.
 
 One can of course see, without using the above isomorphism of categories, that $\mathsf X$-$\mathfrak{KK}$-$lsc$ is triangulated, and that $\mathcal B(\mathsf X)_{lsc}$ is the localising subcategory generated by $i_{\mathsf U}(\mathbb C)$ for all
 $\mathsf U\in \mathbb O(\mathsf X)$ (see Section \ref{s:approx} for the notation). 
\end{remark}
 
\begin{corollary}\label{c:lscUCT}
 Let $\mathsf X$ be a finite space and $\mathfrak A$ and $\mathfrak B$ be separable, lower semicontinuous $\mathsf X$-$C^\ast$-algebras. 
 Assume that $\mathfrak A$ has strongly vanishing boundary maps and belongs to $\mathcal B(\mathsf X)_{lsc}$. Then there is a natural short exact sequence
 \[
  \Ext_{\Mod(\mathbb Z\mathbb O(\mathsf X))}^1(\mathbb OK(\mathfrak A), \mathbb OK(\mathfrak B)) \hookrightarrow KK^1(\mathsf X; \mathfrak A, \mathfrak B) \twoheadrightarrow 
  \Hom_{\Mod(\mathbb Z\mathbb O(\mathsf X))}(\mathbb OK(\mathfrak A), \mathbb OK(\mathfrak B)[1]).
 \]
\end{corollary}
\begin{proof}
We may assume without loss of generality that $\mathsf X$ is a $T_0$-space. 
Let $F$ be as in Lemma \ref{l:lsctocont}. Then $F(\mathfrak A)$ is in $\mathcal B(\mathsf Y)$ by definition. 
Note that if $\mathsf U \in \mathbb O(\mathsf Y) \subset 2^{\mathbb O(\mathsf X)}$, then
\[
 F(\mathfrak A)(\mathsf U) = \sum_{\mathsf V \in \mathsf U} F(\mathfrak A)(\mathsf U^\mathsf{V}) = \sum_{\mathsf V \in \mathsf U}\mathfrak A(\mathsf V).
\]
Thus, since $\mathfrak A$ has strongly vanishing boundary maps, it follows that $F(\mathfrak A)$ has vanishing boundary maps.
The above UCT now follows by combining Theorem \ref{t:BentmannUCT}, Lemma \ref{l:lsctocont} and Remark \ref{r:KKlsc}.
\end{proof}

If $\mathfrak e: 0 \to \mathfrak B \to \mathfrak E \to \mathfrak A \to 0$ is an $\mathsf X$-equivariant extension of lower semicontinuous $\mathsf X$-$C^\ast$-algebras, then $\mathbb OK$ induces a cyclic exact sequence
\[
\xymatrix{
 & \mathbb OK(\mathfrak E) \ar[dr]&  \\
\mathbb OK(\mathfrak B) \ar[ur] && \mathbb OK(\mathfrak A) \ar[ll]|{\circ} 
}
\]
where the $\circ$ in one arrow indicates that there is a degree shift of the $\mathbb Z/2$-grading. We denote this three-term exact sequence by $\mathbb OK_\Delta(\mathfrak e)$. 
For two such extensions $\mathfrak e_1$ and $\mathfrak e_2$ of $\mathfrak A$ by $\mathfrak B$ then we say that $\mathbb OK_\Delta(\mathfrak e_1)$ and $\mathbb OK_\Delta(\mathfrak e_2)$ are \emph{congruent}, 
written $\mathbb OK_\Delta(\mathfrak e_1) \equiv \mathbb OK_\Delta(\mathfrak e_2)$, if there is a homomorphism $\eta\colon \mathbb OK(\mathfrak E_1) \to \mathbb OK(\mathfrak E_2)$ such that
\[
\xymatrix{
\mathbb OK(\mathfrak B) \ar[r] \ar@{=}[d] & \mathbb OK(\mathfrak E_1) \ar[r] \ar[d]^\eta & \mathbb OK(\mathfrak A) \ar@{=}[d] \\
\mathbb OK(\mathfrak B) \ar[r] & \mathbb OK(\mathfrak E_2) \ar[r] & \mathbb OK(\mathfrak A)
}
\] 
commutes. Note that any $\eta$ making the diagram commute is an isomorphism by the five lemma.

\begin{proposition}\label{p:KK1class}
Let $\mathsf X$ be a finite space and let $\mathfrak A$ and $\mathfrak B$ be separable, lower semicontinuous $\mathsf X$-$C^\ast$-algebras.
Suppose that $\mathfrak A$ is an AF algebra, and that $\mathfrak e_i : 0 \to \mathfrak B \to \mathfrak E_i \to \mathfrak A \to 0$
for $i=1,2$ are $\mathsf X$-equivariant extensions. Suppose that $\mathfrak E_1$ and $\mathfrak E_2$ have real rank zero.
Then $\mathfrak e_i$ are semisplit, and $[\mathfrak e_1]=[\mathfrak e_2]$ in $KK^1(\mathsf X; \mathfrak A, \mathfrak B)$ if and only if $\mathbb OK_\Delta(\mathfrak e_1) \equiv \mathbb OK_\Delta(\mathfrak e_2)$.

Moreover, $\mathbb OK_\Delta(\mathfrak e_i)$ collapses to a short exact sequence
\[
0 \to \mathbb OK(\mathfrak B) \to \mathbb OK(\mathfrak E_i) \to \mathbb OK(\mathfrak A) \to 0
\]
for $i=1,2$.
\end{proposition}
\begin{proof}
The extensions are semisplit by Corollary \ref{c:choieffros}. Clearly $\mathbb OK_\Delta(\mathfrak e_1) \equiv \mathbb OK_\Delta(\mathfrak e_2)$ if $[\mathfrak e_1]=[\mathfrak e_2]$. 
Since $\mathfrak A$ is an AF algebra it easily follows that $\mathfrak A$ is in $\mathcal B(\mathsf X)_{lsc}$ and has strongly vanishing boundary maps. Thus by Corollary \ref{c:lscUCT} there is a short exact UCT sequence
\[
\Ext_{\Mod(\mathbb Z\mathbb O(\mathsf X))}^1(\mathbb O K(\mathfrak A), \mathbb O K(\mathfrak B)) \hookrightarrow KK^1(\mathsf X; \mathfrak A, \mathfrak B) \twoheadrightarrow 
\Hom_{\Mod(\mathbb Z\mathbb O(\mathsf X))}(\mathbb OK(\mathfrak A), \mathbb OK(\mathfrak B)[1]).
\]

Since $\mathfrak E_i$ has real rank zero it follows that the induced maps $K_0(\mathfrak A(\mathsf U)) \to K_1(\mathfrak B(\mathsf U))$ vanish for every $\mathsf U\in \mathbb O(\mathsf X)$. 
Moreover, since $K_1(\mathfrak A(\mathsf U))=0$ for every $\mathsf U$ it thus follows that $[\mathfrak e_i]\in KK^1(\mathsf X; \mathfrak A, \mathfrak B)$ induces the zero homomorphism in 
$\Hom_{\Mod(\mathbb Z\mathbb O(\mathsf X))}(\mathbb OK(\mathfrak A), \mathbb OK(\mathfrak B)[1])$ for $i=1,2$. It follows that $\mathbb OK_\Delta(\mathfrak e_i)$ collapses to a short exact sequence
\[
0 \to \mathbb OK(\mathfrak B) \to \mathbb OK(\mathfrak E_i) \to \mathbb OK(\mathfrak A) \to 0
\]
for $i=1,2$. 
By the above UCT $[\mathfrak e_i] \in KK^1(\mathsf X; \mathfrak A, \mathfrak B)$ is uniquely determined by the induced element in $\Ext_{\Mod(\mathbb Z \mathbb O(\mathsf X))}^1(\mathbb OK(\mathfrak A), \mathbb OK(\mathfrak B))$ 
which is exactly $\mathbb OK_\Delta(\mathfrak e_i)$.
Thus $[\mathfrak e_1]= [\mathfrak e_2]$ in $KK^1(\mathsf X)$ if $\mathbb OK_\Delta(\mathfrak e_1) \equiv \mathbb OK_\Delta(\mathfrak e_2)$.
\end{proof}

\begin{definition}
Let $F\colon \mathcal A \to \mathcal B$ and $G \colon \mathcal A \to \mathcal C$ be functors.
We say that $F$ is a \emph{finer classification invariant} than $G$ if whenever $\phi,\psi\colon A \to B$ in $\mathcal A$ are isomorphisms such that $F(\phi) = F(\psi)$ then $G(\phi) = G(\psi)$.
\end{definition}

\begin{example}
Let $F\colon \mathcal A \to \mathcal B$ and $H \colon \mathcal B \to \mathcal C$ be functors. Then $F$ is a finer classification invariant than $G := H \circ F$.
\end{example}

Recall, that an invariant $F$ \emph{strongly classifies} the objects $\mathfrak A$ and $\mathfrak B$, if any isomorphism $F(\mathfrak A) \to F(\mathfrak B)$ lifts to an isomorphism $\mathfrak A \to \mathfrak B$.
The following is our main application. In spirit it says that in many cases, if we can separate the finite and infinite part of a $C^\ast$-algebra such that the finite part is the quotient, and both ideal and quotient are strongly classified by $K$-theory, 
then the $C^\ast$-algebra is classified by $K$-theory.

\begin{theorem}\label{t:mainapp}
Let $\mathfrak A_i$ be separable, nuclear, stable $C^\ast$-algebras with real rank zero for $i=1,2$. Suppose that $\mathfrak J_i$ is a closed two-sided ideal in $\mathfrak A_i$ such that
\begin{itemize}
\item $\mathsf X := \Prim \mathfrak J_1 \cong \Prim \mathfrak J_2$ is finite. Equip $\mathfrak J_i$, $\mathfrak A_i$ and $\mathfrak A_i/\mathfrak J_i$ with the induced $\mathsf X$-$C^\ast$-algebra structure as in Remark \ref{r:inducedact}.
\item $\mathfrak A_i/\mathfrak J_i$ are AF algebras.
\item $\mathfrak J_i$ has a tight corona algebra (cf.~Theorem \ref{t:WvN}).
\item There is an invariant $F$ which strongly classifies $\mathfrak J_1$ and $\mathfrak J_2$, and which is a finer classification invariant than $\mathbb OK$.
\end{itemize}
If there exist isomorphisms
\[
\xymatrix{
F(\mathfrak J_1) \ar[r] \ar[d]^\cong & \mathbb OK(\mathfrak A_1) \ar[r] \ar[d]^\cong & \mathbb OK^+(\mathfrak A_1/\mathfrak J_1) \ar[d]^\cong \\
F(\mathfrak J_2) \ar[r] & \mathbb OK(\mathfrak A_2) \ar[r] & \mathbb OK^+(\mathfrak A_2/\mathfrak J_2)
}
\]
such that the diagram commutes, then $\mathfrak A_1 \cong \mathfrak A_2$.
\end{theorem}
\begin{proof}
Fix an isomorphism $\phi \colon \mathfrak J_1 \to \mathfrak J_2$ of $\mathsf X$-$C^\ast$-algebras, 
which lifts $F(\mathfrak J_1) \xrightarrow{\cong} F(\mathfrak J_2)$. Note that this isomorphism also induces the given homeomorphism $\Prim \mathfrak J_1 \cong \Prim \mathfrak J_2$.
By Elliott's classification of AF algebras \cite{Elliott-AFclass} we may also lift $\mathbb OK^+(\mathfrak A_1/\mathfrak J_1) \xrightarrow{\cong} \mathbb OK^+(\mathfrak A_2/\mathfrak J_2)$ to an isomorphism
$\psi \colon \mathfrak A_1/\mathfrak J_1 \to \mathfrak A_2/ \mathfrak J_2$ of $\mathsf X$-$C^\ast$-algebras. Construct the following push-out/pull-back diagram
\[
 \xymatrix{
 \quad 0 \ar[r] & \mathfrak J_1 \ar[r] \ar[d]^{\phi}_\cong & \mathfrak A_1 \ar[r] \ar[d]^{\tilde \phi}_\cong & \mathfrak A_1/\mathfrak J_1 \ar@{=}[d] \ar[r] & 0 \\
 \mathfrak e_1: 0 \ar[r] & \mathfrak J_2 \ar[r] \ar@{=}[d] & \mathfrak E_1 \ar[r] & \mathfrak A_1/\mathfrak J_1 \ar@{=}[d] \ar[r] & 0 \\
 \mathfrak e_2: 0 \ar[r] & \mathfrak J_2 \ar[r] \ar@{=}[d] & \mathfrak E_2 \ar[r] \ar[d]^{\tilde \psi}_\cong & \mathfrak A_1/\mathfrak J_1 \ar[d]^{\psi}_\cong \ar[r] & 0 \\
 \quad 0 \ar[r] & \mathfrak J_2 \ar[r] & \mathfrak A_2 \ar[r] & \mathfrak A_2/\mathfrak J_2 \ar[r] & 0
 }
\]
for which all the rows are short exact sequences. Let $\eta \colon \mathbb OK(\mathfrak A_1) \xrightarrow{\cong} \mathbb OK(\mathfrak A_2)$ be the given isomorphism.
Since $F$ and $\mathbb OK^+$ are finer classification invariants than $\mathbb OK$, it follows that the following diagram commutes
\[
 \xymatrix{
 \mathbb OK(\mathfrak J_1) \ar[r] \ar@{=}[d] & \mathbb OK(\mathfrak E_1) \ar[r] \ar[d]^{f} & \mathbb OK(\mathfrak A_1/\mathfrak J_1) \ar@{=}[d] \\
 \mathbb OK(\mathfrak J_2) \ar[r] & \mathbb OK(\mathfrak E_2) \ar[r] & \mathbb OK(\mathfrak A_1/\mathfrak J_1).
 }
\]
where $f = \mathbb OK(\tilde \psi) \circ \eta \circ \mathbb OK(\tilde \phi^{-1})$. Hence $\mathbb OK_\Delta(\mathfrak e_1) \equiv \mathbb OK_\Delta(\mathfrak e_2)$.
By Proposition \ref{p:KK1class} it follows that $[\mathfrak e_1] = [\mathfrak e_2]$ in $KK^1(\mathsf X; \mathfrak A_1/\mathfrak J_1, \mathfrak J_2)$.

Using that $\phi$ induces the given homeomorphism $\mathsf X=\Prim \mathfrak J_1 \cong \Prim \mathfrak J_2$, it is easily seen that the actions of $\Prim \mathfrak J_2$ on $\mathfrak A_1/\mathfrak J_1$ 
induced by $\mathfrak e_1$ and $\mathfrak e_2$, are both equal to the action of $\mathsf X$ on $\mathfrak A_1/\mathfrak J_1$ considered earlier composed with the map $\mathbb O(\Prim \mathfrak J_2)\to \mathbb O(\mathsf X)$ induced by $\phi^{-1}$.
Since $\mathfrak J_2$ has a tight corona algebra, it follows from Theorem \ref{t:classtightcorona} that $\mathfrak e_1$ and $\mathfrak e_2$ are equivalent extensions. Thus $\mathfrak A_1 \cong \mathfrak E_1 \cong \mathfrak E_2 \cong \mathfrak A_2$.
\end{proof}


\subsection{Applications to graph $C^\ast$-algebras}

We will end with an application to graph $C^\ast$-algebras. This application says that the class of graph $C^\ast$-algebras which contain a purely infinite, two-sided, closed ideal with a finite primitive ideal space, 
such that the quotient is AF, is classified by a $K$-theoretic invariant. 
Moreover, the invariant is computable from the graph. 

This generalises a similar result \cite{EilersRestorffRuiz-classfininf} by the second author, Eilers and Restorff, in which the extensions had to be \emph{full} (in the classical sense). 
Fullness is a huge limitation on the primitive ideal space of the $C^\ast$-algebras for which this is applicable.
In this new result there is \emph{no} limitation on the primitive ideal spaces, as long as the two-sided, closed ideal in question has a finite primitive ideal space.

Bentmann and Meyer introduced in \cite{BentmannMeyer-generalclass} a new way of classifying (separable) continuous $\mathsf Y$-$C^\ast$-algebras up to $KK(\mathsf Y)$-equivalence, 
in the case where the $\mathsf Y$-$C^\ast$-algebras have projective dimension at most 2. They did this by adding a particular obstruction class to the invariant. We will vaguely explain the idea behind this. 
For more details we refer the reader to \cite{BentmannMeyer-generalclass}.

Let $\mathcal M$ be the target category of our invariant $F$ and suppose that $\mathcal M$ is \emph{even}, i.e.~that $\mathcal M$ decomposes as $\mathcal M_+ \times \mathcal M_{-}$ with $\mathcal M_+[-1] = \mathcal M_{-}$ and $\mathcal M_{-}[-1]= \mathcal M_+$.
Note that this is the set-up for both of our invariants $\mathsf YK$ and $\mathbb OK$. 
If $F(\mathfrak A)$ has projective dimension at most 2, then $\mathfrak A$ induces an obstruction class $\delta(\mathfrak A) \in \Ext^2_{\mathcal M}(F(\mathfrak A), F(\mathfrak A)[-1])$.
Thus on the class of elements with projective dimension at most 2, we obtain an invariant $F\delta(\mathfrak A) = (F(\mathfrak A), \delta(\mathfrak A))$. 
An isomorphism $F\delta(\mathfrak A) \xrightarrow{\cong} F\delta(\mathfrak B)$ is an isomorphism $F(\mathfrak A) \xrightarrow{\cong} F(\mathfrak B)$ which intertwines the obstruction classes.

\begin{example}[The obstruction class for graph $C^\ast$-algebras]
 Let $\mathfrak A$ be a lower semi\-continouous $\mathsf X$-$C^\ast$-algebra which is a graph $C^\ast$-algebra $\mathfrak A = C^\ast(E)$, such that every $\mathsf X$-equivariant ideal $\mathfrak A(\mathsf U)$ 
 is invariant under the canonical gauge action $\gamma$ of $\mathbb T$.
 
 Lemma \ref{l:lsctocont} and Remark \ref{r:KKlsc} allows us to do everything done in \cite{BentmannMeyer-generalclass}, for lower semicontinuous $\mathsf X$-$C^\ast$-algebras and the invariant $\mathbb OK$, instead of
 continuous $\mathsf Y$-$C^\ast$-algebras and the invariant $\mathsf YK$. 
 Thus, as explained in \cite{BentmannMeyer-generalclass}, $\mathbb OK(\mathfrak A)$ has projective dimension at most 2, 
 and the obstruction class $\delta \in \Ext^2(\mathbb OK(\mathfrak A), \mathbb OK(\mathfrak A)[-1]) \cong \Ext^2(\mathbb OK_0(\mathfrak A), \mathbb OK_1(\mathfrak A))$ is the one induced by the Pimsner--Voiculescu exact sequence
 \[
  \mathbb OK_1(\mathfrak A) \hookrightarrow \mathbb OK_0(\mathfrak A \rtimes_\gamma \mathbb T) \xrightarrow{1-\gamma_\ast^{-1}} \mathbb OK_0(\mathfrak A \rtimes_\gamma \mathbb T) \twoheadrightarrow \mathbb OK_0(\mathfrak A).
 \]
 In \cite{BentmannMeyer-generalclass} the proof is carried out only for tight $\mathsf Y$-$C^\ast$-algebras, but the case for continuous $\mathsf Y$-$C^\ast$-algebras with all $\mathsf Y$-equivariant ideals gauge invariant has the exact same proof.
 
 Consider the graph $E=(E^0,E^1,r,s)$. 
 Every gauge invariant, two-sided, closed ideal in $\mathfrak A=C^\ast(E)$ corresponds to a hereditary and saturated set $H\subset E^0$. 
 Let $H(\mathsf U)$ denote the hereditary and saturated set corresponding to $\mathfrak A(\mathsf U)$ and write $H(\mathsf U)_{\reg}$ for $H(\mathsf U) \cap E^0_\reg$, where $E^0_\reg$ is the set of regular vertices. 
 By outsplitting every breaking vertex of all $H(\mathsf U)$, we may assume that every $H(\mathsf U)$ has no breaking vertices. 
 Then $E$ induces $\mathbb Z \mathbb O(X)$-modules $\mathbb OK_0(\mathbb Z^{E_\reg})$ and $\mathbb OK_0(\mathbb Z^E)$ by
 \[
  \mathbb OK_0(\mathbb Z^{E_\reg}) := \bigoplus_{\mathsf U \in \mathbb O(\mathsf X)} \mathbb Z^{H(\mathsf U)_\reg}, \qquad \mathbb OK_0(\mathbb Z^E) := \bigoplus_{\mathsf U \in \mathbb O(\mathsf X)} \mathbb Z^{H(\mathsf U)}, 
 \]
 where the module actions on $\mathbb OK_0(\mathbb Z^{E_\reg})$ and $\mathbb OK_0(\mathbb Z^E)$ are the obvious ones, i.e.~if $\mathsf U \subset \mathsf V$, 
 then $i_{\mathsf U}^{\mathsf V}$ acts on the direct summand corresponding to $\mathsf U$ by embedding into the direct summand corresponding to $\mathsf V$,
 and acts as zero on all other direct summands. This is well-defined since $H(\mathsf U) \subset H(\mathsf V)$ whenever $\mathsf U \subset \mathsf V$.
 By decomposing the vertices into $E^0_\reg \sqcup E^0_{\textrm{sing}}$, where $E^0_{\textrm{sing}}$ is the set of singular vertices, the adjacency matrix of $E$ has the form $\begin{lbmatrix}A & \alpha \\ \ast & \ast \end{lbmatrix}$.
 We get a homomorphism $\begin{lbmatrix} 1 - A^t \\ -\alpha^t \end{lbmatrix} \colon \mathbb OK_0(\mathbb Z^{E_\reg}) \to \mathbb OK_0(\mathbb Z^{E})$ in the obvious way, and we let $\mathbb OK_0(E)$ (resp.~$\mathbb OK_1(E))$ denote the cokernel (resp.~kernel)
 of this homomorphism.
 It is easily verified, as in the classical case when the $K$-theory of a graph $C^\ast$-algebra is computed, that we have a commutative diagram
 \[
 \xymatrix{
  \mathbb OK_1(E) \ar[d]^\cong \ar@{>->}[r] & \mathbb OK_0(\mathbb Z^{E_\reg}) \ar[d] \ar[r]^{\begin{lbmatrix} 1-A^t \\ -\alpha^t \end{lbmatrix}} & \mathbb OK_0(\mathbb Z^E) \ar[d] \ar@{->>}[r] & \mathbb OK_0(E) \ar[d]^\cong \\
  \mathbb OK_1(\mathfrak A) \ar@{>->}[r] & \mathbb OK_0(\mathfrak A \rtimes_\gamma \mathbb T) \ar[r]^{1-\gamma_\ast^{-1}} &  \mathbb OK_0(\mathfrak A \rtimes_\gamma \mathbb T) \ar@{->>}[r] & \mathbb OK_0(\mathfrak A).
  }
 \]
 In particular, the top row determines the obstruction class $\delta(\mathfrak A) \in \Ext^2(\mathbb OK(\mathfrak A),\mathbb OK(\mathfrak A)[-1])$ when identified with $\Ext^2(\mathbb OK_0(E), \mathbb OK_1(E))$.
\end{example}

We can now give our main application to graph $C^\ast$-algebras. Note that the invariant in question can easily be computed from the above method.

\begin{theorem}\label{t:graphclass}
Let $\mathfrak A_i$ be graph $C^\ast$-algebras for $i=1,2$. Suppose that $\mathfrak J_i$ is a purely infinite, two-sided, closed ideal in $\mathfrak A_i$, such that $\mathfrak A_i/\mathfrak J_i$ is AF. 
Suppose that $\mathsf X := \Prim \mathfrak J_1 \cong \Prim \mathfrak J_2$ is finite. 
Equip $\mathfrak J_i$, $\mathfrak A_i$ and $\mathfrak A_i/\mathfrak J_i$ with the induced $\mathsf X$-$C^\ast$-algebra structure as in Remark \ref{r:inducedact}.
If there exist isomorphisms
\[
\xymatrix{
\mathbb OK\delta(\mathfrak J_1) \ar[r] \ar[d]^\cong & \mathbb OK(\mathfrak A_1) \ar[r] \ar[d]^\cong & \mathbb OK^+(\mathfrak A_1/\mathfrak J_1) \ar[d]^\cong \\
\mathbb OK\delta(\mathfrak J_2) \ar[r] & \mathbb OK(\mathfrak A_2) \ar[r] & \mathbb OK^+(\mathfrak A_2/\mathfrak J_2)
}
\]
such that the diagram commutes, then $\mathfrak A_1 \otimes \mathbb K \cong \mathfrak A_2 \otimes \mathbb K$.
\end{theorem}
\begin{proof}
Without loss of generality we may assume that $\mathfrak A_i$ is stable. It is well known to experts of graph $C^\ast$-algebras, that $\mathfrak A_i$ will have real rank zero, since it contains a purely infinite, two-sided, closed ideal with an AF quotient.
Moreover, it is well-known that $\mathfrak A_i$ is separable and nuclear, and that $\mathfrak J_i$ is itself a graph $C^\ast$-algebra. By Theorem \ref{t:WvN} $\mathfrak J_i$ has a tight corona algebra. 
Clearly $\mathbb OK\delta$ is a finer classification invariant than $\mathbb OK$ so by Theorem \ref{t:mainapp} it suffices to show that $\mathbb OK\delta$ classifies $\mathfrak J_1$ and $\mathfrak J_2$ strongly.

It is proven in \cite[Section 5.3]{BentmannMeyer-generalclass}, that for tight $\mathsf Y$-$C^\ast$-algebras which are graph $C^\ast$-algebras,
an isomorphism of the invariant $\mathsf YK\delta$ lifts to a $KK(\mathsf Y)$-equivalence. 
The proof in the continuous (not necessarily tight) case is identical, if we furthermore assume that all $\mathsf Y$-equivariant ideals are gauge invariant. This is automatic in the tight case, since we assume that $\mathsf Y$ is finite.

Let $\mathsf Y= \mathbb O(\mathsf X)$ and $F$ be as in Lemma \ref{l:lsctocont}. Then $F(\mathfrak J_i)$ are continuous $\mathsf Y$-$C^\ast$-algebras which are graph $C^\ast$-algebras in which every two-sided, closed ideal is gauge invariant.
Thus the isomorphism 
\[
 \mathbb OK\delta(\mathfrak J_1) = \mathsf YK\delta(F(\mathfrak J_1)) \xrightarrow{\cong} \mathsf YK\delta(F(\mathfrak J_2)) = \mathbb OK\delta(\mathfrak J_2)
\]
lifts to a $KK(\mathsf Y)$-equivalence $\alpha\in KK(\mathsf Y; F(\mathfrak J_1), F(\mathfrak J_2))$. 
Since $KK(\mathsf Y; F(\mathfrak C), F(\mathfrak D))=KK(\mathsf X; \mathfrak C,\mathfrak D)$ for all separable, lower semicontinouous $\mathsf X$-$C^\ast$-algebras $\mathfrak C$ and $\mathfrak D$, 
it follows that $\alpha$ is also a $KK(\mathsf X)$-equivalence which lifts $\mathbb OK\delta(\mathfrak J_1) \xrightarrow{\cong} \mathbb OK\delta(\mathfrak J_2)$.
It follows from Kirchberg's classification \cite{Kirchberg-non-simple} that there is an isomorphism $\phi$ of $\mathsf X$-$C^\ast$-algebras which lifts $\alpha$.
Hence $\mathbb OK\delta$ classifies $\mathfrak J_1$ and $\mathfrak J_2$ strongly.
\end{proof}

\begin{remark}[Classification with filtered $K$-theory]
So far, the classification of non-simple graph $C^\ast$-algebras has usually used the invariant \emph{ordered filtered $K$-theory} $FK^+_\mathsf{Y}$ (also known as ordered filtrated $K$-theory).
Let $\mathfrak A_i$ and $\mathfrak J_i$ be as in Theorem \ref{t:graphclass} such that $\mathsf Y := \Prim \mathfrak A_1 \cong \Prim \mathfrak A_2$ is \emph{finite}.
We can obtain classification with ordered filtered $K$-theory \emph{if} we furthermore assume that $\mathfrak J_1$ and $\mathfrak J_2$ are classified strongly by filtered $K$-theory $FK_\mathsf{X}$.
In fact, suppose that $FK_\mathsf{Y}^+(\mathfrak A_1) \cong FK_{\mathsf Y}^+(\mathfrak A_2)$.
Since the ordered filtered $K$-theory $FK_\mathsf{Y}^+$ contains \emph{all} natural transformations of the ordered $K$-theory of all subquotients, there is an induced diagram
\[
\xymatrix{
FK_{\mathsf X}(\mathfrak J_1) \ar[r] \ar[d]^\cong & \mathbb OK(\mathfrak A_1) \ar[r] \ar[d]^\cong & \mathbb OK^+(\mathfrak A_1/\mathfrak J_1) \ar[d]^\cong \\
FK_\mathsf{X}(\mathfrak J_2) \ar[r] & \mathbb OK(\mathfrak A_2) \ar[r] & \mathbb OK^+(\mathfrak A_2/\mathfrak J_2),
}
\]
where $\mathbb OK$ is the invariant on lower semicontinuous $\mathsf X$-$C^\ast$-algebras (not $\mathsf Y$-$C^\ast$-algebras).
Since $FK_\mathsf{X}$ is a finer classification invariant than $\mathbb OK$, it follows from Theorem \ref{t:mainapp} that $\mathfrak A_1 \otimes \mathbb K \cong \mathfrak A_2 \otimes \mathbb K$.
\end{remark}

\subsection{A final remark}

It is possible to obtain classification results similar to Theorem \ref{t:mainapp} with much more general, two-sided, closed ideals than purely infinite ones and not necessarily AF quotients.
To obtain these one should apply Corollary \ref{c:nucquot} instead of Theorem \ref{t:classtightcorona}, and thus require that the extensions are $\mathsf X$-full.
If one wants to use the same method, one should have an invariant $G$, and a universal coefficient theorem for separable, lower semicontinuous $\mathsf X$-$C^\ast$-algebras with respect to $G$.
Moreover, one needs to require that the homomorphism induced by the $KK^1(\mathsf X)$-class of the extension vanishes, so that $G_\Delta$ determines the $KK^1(\mathsf X)$-class uniquely, as done in the proof of Proposition \ref{p:KK1class}.
In the classical case where $G = K_\ast$, which was done by Rørdam \cite{Rordam-classsixterm}, one did not need vanishing of this homomorphism. 
However, Rørdam uses that the kernel and cokernel of a map on the invariant can be realised by $C^\ast$-algebras. It is not obvious whether or not this can be done with the invariant $\mathbb OK$.

\subsection*{Acknowledgement}

Some of the work in this paper was carried out while the first named author visited the University of Hawai'i at Hilo. 
He would like to thank them for their hospitality during his stay. He would also like to thank his Ph.D.~advisor Søren Eilers for some helpful comments.  
The second named author thanks the Department of Mathematical Science at the University of Copenhagen for providing excellent working environment during his visit in Fall 2013.

\end{document}